\documentclass{amsart}
\usepackage{amssymb}
\usepackage{amsmath}
\usepackage{amsfonts}
\usepackage{diagrams}
\usepackage{hyperref}

\newtheorem{theorem}{Theorem}[section]

\newtheorem{conjecture}[theorem]{Conjecture}
\newtheorem{corollary}[theorem]{Corollary}

\newtheorem{definition}[theorem]{Definition}
\newtheorem{example}[theorem]{Example}

\newtheorem{lemma}[theorem]{Lemma}
\newtheorem{notation}[theorem]{Notation}

\newtheorem{proposition}[theorem]{Proposition}
\newtheorem{remark}[theorem]{Remark}

\input{tcilatex}
\DeclareMathOperator{\coker}{coker}
\DeclareMathOperator{\Hom}{Hom}

\DeclareMathOperator{\Imm}{Im}
\numberwithin{equation}{section}

\begin{document}
\title[Cosheaf homology]{Cosheaf homology}
\author{Andrei V. Prasolov}
\address{Institute of Mathematics and Statistics\\
The University of Troms\o\ - The Arctic University of Norway\\
N-9037 Troms\o , Norway}
\email{andrei.prasolov@uit.no}
\urladdr{http://serre.mat-stat.uit.no/ansatte/andrei/Welcome.html}
\date{}

\begin{abstract}
In this paper the cosheaf homology is investigated from different
viewpoints: the behavior under site morphisms, connections with 
\u{C}ech
homology via spectral sequences, and description of cosheaf homology using
hypercoverings.

It is proved that in the case of Hausdorff paracompact spaces, the cosheaf
homology in general is isomorphic to the 
\u{C}ech
homology, and for a constant cosheaf is isomorphic to the shape pro-homology.

In the case of Alexandroff spaces, including finite and locally finite
spaces, the cosheaf homology is isomorphic to the singular homology.
\end{abstract}

\maketitle
\tableofcontents

\setcounter{section}{-1}

\section{Introduction}

This paper continues a series of papers (\cite%
{Prasolov-smooth-cosheaves-2012-MR2879363}, \cite%
{Prasolov-Cosheafification-2016-MR3660525}, \cite%
{Prasolov-Cosheaves-2021-MR4347662}) of the author concerning (pre)cosheaves.

In Section \ref{Sec-Site-morphisms} we study the behavior of cosheaves and
cosheaf homology under site morphisms. For a site morphism $f:X\rightarrow Y$%
, we construct and study in Proposition \ref{Prop-Site-morphism-cosheaves}
and Corollary \ref{Cor-Satellites-site-morphisms} the functors and their
left satellites%
\begin{eqnarray*}
f_{\ast } &:&\text{Cosheaves on }X\text{ }\longrightarrow \text{Cosheaves on 
}Y\text{,} \\
L_{s}f_{\ast } &:&\text{Cosheaves on }X\text{ }\longrightarrow \text{%
Cosheaves on }Y\text{, }s\geq 0, \\
f^{\ast } &:&\text{Cosheaves on }Y\text{ }\longrightarrow \text{Cosheaves on 
}X\text{.}
\end{eqnarray*}
Various spectral sequences are constructed in Theorem \ref%
{Th-Grothendieck-spectral-sequence-(pre)cosheaves}:

\begin{enumerate}
\item In (\ref{Th-Grothendieck-spectral-sequence-compare-to-Cech}) the two
spectral sequences comparing cosheaf homology with 
\u{C}ech
homology.

\item In (\ref{Th-Grothendieck-spectral-sequence-Leray}) the Leray-type
spectral sequence comparing the homology on $X$ and $Y$.

\item In (\ref{Th-Grothendieck-spectral-sequence-Grothendieck}) the
Grothendieck-type spectral sequence comparing the left satellites $%
L_{\bullet }f_{\ast }$, $L_{\bullet }g_{\ast }$ and $L_{\bullet }\left(
g\circ f\right) _{\ast }$.
\end{enumerate}

In Section \ref{Sec-Hypercoverings} we give an explicit construction of
cosheaf homology using hypercoverings. It is proved in Theorem \ref%
{Th-Hypercoverings-vs-resolutions} that the cosheaf homology may be
calculated as a limit of 
\u{C}ech
homology for hypercoverings:%
\begin{equation*}
H_{\bullet }\left( U,\mathcal{A}\right) 
\simeq%
\underset{\mathcal{K}_{\bullet }\in \mathcal{HR}\left( U\right) }{%
\underleftarrow{\lim }}\check{H}_{\bullet }\left( \mathcal{K}_{\bullet },%
\mathcal{A}\right)
\end{equation*}%
where $\mathcal{K}_{\bullet }$ runs through the hypercoverings (modulo
simplicial homotopy) of $U$.

In Section \ref{Sec-topological-spaces} we study cosheaf homology on
topological spaces. For any cosheaf $\mathcal{A}$ on a general site $X$
there is an epimorphism $\mathcal{P}\rightarrow \mathcal{A}$ where $\mathcal{%
P}$ is quasi-projective. The proof in \cite[Theorem 3.4.1(1)]%
{Prasolov-Cosheaves-2021-MR4347662} is rather complicated. For topological
spaces we give a simpler proof in Proposition \ref%
{Prop-Enough-quasi-projective}.

Let now $X$ be a Hausdorff paracompact space. It is proved in Corollary \ref%
{Cor-Homology-Paracompact} that naturally%
\begin{equation*}
H_{\bullet }\left( X,\mathcal{A}_{\#}\right) 
\simeq%
\check{H}_{\bullet }\left( X,\mathcal{A}\right)
\end{equation*}%
where $\mathcal{A}$ is a precosheaf, $\mathcal{A}_{\#}$ is its
cosheafification, and $\check{H}_{\bullet }$ is the 
\u{C}ech
homology.

The main result of Section \ref{Sec-topological-spaces} is Theorem \ref%
{Th-Paracompact-vs-Shape} saying that for an abelian group $G$ (or a $k$%
-module $G$)%
\begin{equation*}
H_{\bullet }\left( X,G_{\#}\right) 
\simeq%
pro\text{-}H_{\bullet }\left( X,G\right)
\end{equation*}%
where $pro$-$H_{\bullet }$ is the shape pro-homology. That theorem proves 
\cite[Conjecture 1.0.3(1)]{Prasolov-Cosheaves-2021-MR4347662}.

In Section \ref{Sec-A-spaces} we study sheaf \textbf{co}homology and \textbf{%
co}sheaf homology on Alexandroff spaces (shortly: $A$-spaces). That class of
topological spaces includes finite and locally finite spaces.

Theorem \ref{Th-Poset-vs-sheaves} establishes deep connections between
sheaves on an $A$-space $X$, and functors from the corresponding pre-ordered
set $\left( X,\leq \right) ^{op}$. In \cite{Jensen-MR0407091} such
connections were established for \textbf{directed posets}. In (\ref%
{Th-Poset-vs-sheaves-categories-equivalent}) it is proved that the
corresponding categories (of functors and of sheaves) are equivalent. In (%
\ref{Th-Poset-vs-sheaves-higher-limits}) it is proved that the sheaf \textbf{%
co}homology groups on $X$ are isomorphic to the higher limits $%
\underleftarrow{\lim }^{\bullet }$ (the \textbf{right} satellites of the
limit functor $\underleftarrow{\lim }$) of functors from $\left( X,\leq
\right) ^{op}$. In (\ref{Th-Poset-vs-sheaves-Bar-construction}) the
bar-construction for calculating of $\underleftarrow{\lim }^{\bullet }$ is
given and proved.

Theorem \ref{Th-Poset-vs-cosheaves} establishes deep connections between 
\textbf{co}sheaves on an $A$-space $X$, and functors from the corresponding
pre-ordered set $\left( X,\leq \right) $. In (\ref%
{Th-Poset-vs-cosheaves-categories-equivalent}) it is proved that the
corresponding categories (of functors and of \textbf{co}sheaves) are
equivalent. In (\ref{Th-Poset-vs-cosheaves-higher-colimits}) it is proved
that the \textbf{co}sheaf homology groups on $X$ are isomorphic to the
higher colimits $\underrightarrow{\lim }_{\bullet }$ (the \textbf{left}
satellites of the colimit functor $\underrightarrow{\lim }$) of functors
from $\left( X,\leq \right) $. In (\ref{Th-Poset-vs-sheaves-Bar-construction}%
) the bar-construction for calculating of $\underrightarrow{\lim }_{\bullet
} $ is given and proved.

In Theorem \ref{Th-Cosheaf-homology-for-A-spaces} it is proved that for an $%
A $-space and a $k$-module $G$ there are natural isomorphisms%
\begin{equation*}
H_{n}\left( X,G_{\#}\right) 
\simeq%
H_{n}^{sing}\left( \left\vert \mathcal{K}\left( X\right) \right\vert
,G\right) 
\simeq%
H_{n}^{sing}\left( X,G\right)
\end{equation*}%
where $\left\vert \mathcal{K}\left( X\right) \right\vert $ is the polyheder
corresponding to $X$. That theorem proves \cite[Conjecture 1.0.5(1)]%
{Prasolov-Cosheaves-2021-MR4347662}.

In Section \ref{Sec-classical-cosheaves} we give some examples of
\textquotedblleft bad\textquotedblright\ behavior of classical cosheaves
(with values in sets and groups \textbf{instead} of \textbf{pro}-sets and 
\textbf{pro}-groups). The author is skeptical to the prospects of building
an appropriate theory for such (\textbf{rudimentary}) cosheaves. We believe
that the only way is building such a theory for pro-sets (pro-groups,
pro-modules), hoping that sometimes the theory produces \textbf{rudimentary}
pro-objects or (pre)cosheaves. In that case one may apply Proposition \ref%
{Prop-Rudimentary-precosheave}.

The category of cosheaves with values in $\mathbf{Pro}\left( \mathbf{Ab}%
\right) $ is abelian. The key ingredient in the proof of that fact in \cite[%
Theorem 3.3.1(1)]{Prasolov-Cosheaves-2021-MR4347662} was the \textbf{%
exactness} of the \textbf{co}sheafification functor $\left( \bullet \right)
_{\#}$. For classical cosheaves, that functor is \textbf{not} in general
exact, see Example \ref{Ex-Non-exactness-Set-valued} and \ref%
{Ex-Non-exactness-Ab-valued}. Therefore we conjecture that the classical
cosheaves do \textbf{not} form an abelian category, see Conjecture \ref%
{Conj-Classical-not-abelian} and Remark \ref{Rem-Classical-not-abelian}.

Several examples are given in Section \ref{Sec-Examples}.

The categories of modules $\mathbf{Mod}\left( k\right) $ and pro-modules $%
\mathbf{Pro}\left( k\right) $ are very important in this paper. For our
purposes, it has to be assumed that the ring $k$ is \textbf{quasi-noetherian}%
. It was proved earlier in \cite[Proposition 2.28]%
{Prasolov-universal-coefficients-formula-2013-MR3095217} that any noetherian
ring is quasi-noetherian. It happens, however, that the two classes
(noetherian and quasi-noetherian rings) coincide, see Theorem \ref%
{Th-(Quasi-)noetherian}.

In Section \ref{Sec-Derived-categories} we give several useful facts about
derived categories:

\begin{enumerate}
\item Left satellites, see Proposition \ref{Prop-Left-derived}.

\item Cartan-Eilenberg resolutions, see Proposition \ref%
{Prop-Cartan-Eilenberg}.

\item The Grothendieck spectral sequence, see Theorem \ref%
{Th-Grothendieck-spectral-sequence}.
\end{enumerate}

\section{Preliminaries}

\subsection{Categories}

Below are several notations, definitions and facts about categories that are
used in this paper.

\begin{notation}
\label{Not-Categories}~

\begin{enumerate}
\item Let $\mathbf{C}$ be a category.

\begin{enumerate}
\item $Ob\left( \mathbf{C}\right) $ is the \textbf{class} (the \textbf{set},
if the category is \textbf{small}) of objects of $\mathbf{C}$.

\item $Mor\left( \mathbf{C}\right) $ is the \textbf{class} (the \textbf{set}%
, if the category is \textbf{small}) of morphisms of $\mathbf{C}$.

\item We will simply write $X\in \mathbf{C}$ if $X$ is an object of $\mathbf{%
C}$ ($X\in Ob\left( \mathbf{C}\right) $).
\end{enumerate}

\item ~

\begin{enumerate}
\item Denote by $\mathbf{Set}$ ($\mathbf{Set}_{\ast }$), $\mathbf{Top}$ ($%
\mathbf{Top}_{\ast }$), $\mathbf{Pol}$ ($\mathbf{Pol}_{\ast }$), $\mathbf{%
Group}$ and $\mathbf{Ab}$ the categories of (pointed) sets, (pointed)
topological spaces, (pointed) polyhedra, groups and abelian groups,
respectively.

\item Denote by $H\left( \mathbf{Top}\right) $, $H\left( \mathbf{Top}_{\ast
}\right) $, $H\left( \mathbf{Pol}\right) $, and $H\left( \mathbf{Pol}_{\ast
}\right) $ the corresponding homotopy categories.
\end{enumerate}

\item For categories $\mathbf{C}$ and $\mathbf{D}$, let $\mathbf{D}^{\mathbf{%
C}}$ be the category of functors $\mathbf{C}\rightarrow \mathbf{D}$.

\begin{enumerate}
\item If both $\mathbf{C}$ and $\mathbf{D}$ are small, then $\mathbf{D}^{%
\mathbf{C}}$ is a small category.

\item If $\mathbf{C}$ is small, then $\mathbf{D}^{\mathbf{C}}$ is a (\textbf{%
large}) category. For $F,G\in \mathbf{D}^{\mathbf{C}}$, $%
\Hom%
_{\mathbf{D}^{\mathbf{C}}}\left( F,G\right) $ is a \textbf{set}.

\item Strictly speaking, $\mathbf{D}^{\mathbf{C}}$ is in general not a
category, because $%
\Hom%
_{\mathbf{D}^{\mathbf{C}}}\left( F,G\right) $ needs not be a set. Let us
call such \textquotedblleft categories\textquotedblright\ \textbf{huge}.
\end{enumerate}

\item A \textbf{diagram} in $\mathbf{C}$ is a functor $D:\mathbf{%
I\rightarrow C}$ where the \textbf{index category} $\mathbf{I}$ is always
assumed to be \textbf{small}.

\item Limits (inverse/projective limits) will be denoted by $\underleftarrow{%
\lim }$, while colimits (direct/inductive limits) will be denoted by $%
\underrightarrow{\lim }$.

\item \label{Def-(Co)-complete}\label{Not-Categories-(Co)-complete}A
category $\mathbf{C}$ is called \textbf{(co)complete} iff it admits all
small (co)limits $\underleftarrow{\lim }$ ($\underrightarrow{\lim }$). This
implies that $\mathbf{C}$ contains a terminal (initial) object as a
(co)limit of an empty diagram.

\item \label{Not-Categories-(Co)-continuous}Let $F:\mathbf{C}\rightarrow 
\mathbf{D}$ be a functor.

\begin{enumerate}
\item $F$ is \textbf{(co)continuous} iff it preserves small (co)limits $%
\underleftarrow{\lim }$ ($\underrightarrow{\lim }$). Equivalently, iff $F$
preserves \textbf{(co)kernels} and small \textbf{(co)products}.

\item $F$ is left (right) \textbf{exact} iff it preserves \textbf{finite}
(co)limits $\underleftarrow{\lim }$ ($\underrightarrow{\lim }$).
Equivalently, iff $F$ preserves (co)kernels and \textbf{finite} (co)products.

\item $F$ is \textbf{exact} iff it is \textbf{both} left and right exact.
\end{enumerate}

\item \label{Not-Categories-Yoneda}The Yoneda functors will be denoted $%
h_{\bullet }$ and $h^{\bullet }$. If $X\in \mathbf{C}$ then:

\begin{enumerate}
\item $h_{X}%
{:=}%
\Hom%
_{\mathbf{C}}\left( \bullet ,X\right) \in \mathbf{Set}^{\mathbf{C}^{op}}$.

\item $h^{X}%
{:=}%
\Hom%
_{\mathbf{C}}\left( X,\bullet \right) \in \mathbf{Set}^{\mathbf{C}}$.

\item We will also consider the third Yoneda functor%
\begin{equation*}
\left( h^{\bullet }\right) ^{op}:\mathbf{C=}\left( \mathbf{C}^{op}\right)
^{op}\rightarrow \left( \mathbf{Set}^{\mathbf{C}}\right) ^{op}.
\end{equation*}

\item $h_{\bullet }$, $h^{\bullet }$ and $\left( h^{\bullet }\right) ^{op}$
are full embeddings due to Yoneda's lemma.
\end{enumerate}

\item \label{Not-Categories-(co)filtered}A category $\mathbf{I}$ is called 
\textbf{filtered} iff:

\begin{enumerate}
\item It is not empty.

\item For every two objects $i,j\in \mathbf{I}$ there exists an object $k$
and two morphisms $i\rightarrow k$, $j\rightarrow k$.

\item For every two parallel morphisms $u,v:i\rightarrow j$ there exists an
object $k$ and a morphism $w:j\longrightarrow k$, such that $w\circ u=w\circ
v$.
\end{enumerate}

\item A category $\mathbf{I}$ is called \textbf{cofiltered} if $\mathbf{I}%
^{op}$ is filtered. A diagram $D:\mathbf{I\rightarrow C}$ is called
(co)filtered if $\mathbf{I}$ is a (co)filtered category.

\item \label{Not-Categories-Category-for-poset}For any poset (and even a 
\textbf{pre}-ordered set) $\left( I,\leq \right) $ there corresponds the
category $\mathbf{I}$ with $Ob\left( \mathbf{I}\right) =I$ and $%
\Hom%
_{\mathbf{I}}\left( i,j\right) $ is either a singleton (if $i\leq j$) or
empty (otherwise).

\item A poset $\left( I,\leq \right) $ is called \textbf{(co)directed} iff
the corresponding category $\mathbf{I}$ is (co)filtered.

\item See, e.g., \cite[Chapter IX.1]{Mac-Lane-Categories-1998-MR1712872} for
filtered, and \cite[Chapter I.1.4]{Mardesic-Segal-MR676973} for cofiltered
categories. In \cite{Kashiwara-Categories-MR2182076}, such categories and
diagrams are called \textbf{(co)filtrant}.

\item Let $\varphi :\mathbf{C\rightarrow D}$ be a functor, and let $d\in 
\mathbf{D}$.

\begin{enumerate}
\item The \textbf{comma-category} $\varphi \downarrow d$ is defined as
follows:%
\begin{eqnarray*}
&&Ob\left( \varphi \downarrow d\right) 
{:=}%
\left\{ \left( \varphi \left( c\right) \rightarrow d\right) \in 
\Hom%
_{\mathbf{D}}\left( \varphi \left( c\right) ,d\right) \right\} , \\
&&%
\Hom%
_{\varphi \downarrow d}\left( \left( \alpha _{1}:\varphi \left( c_{1}\right)
\rightarrow d\right) ,\left( \alpha _{2}:\varphi \left( c_{2}\right)
\rightarrow d\right) \right) 
{:=}%
\left\{ \beta :c_{1}\rightarrow c_{2}~|~\alpha _{2}\circ \varphi \left(
\beta \right) =\alpha _{1}\right\} .
\end{eqnarray*}

\item Another \textbf{comma-category} $d\downarrow \varphi =\left( \varphi
^{op}\downarrow d\right) ^{op}$ is defined as follows:%
\begin{eqnarray*}
&&Ob\left( d\downarrow \varphi \right) 
{:=}%
\left\{ \left( d\rightarrow \varphi \left( c\right) \right) \in 
\Hom%
_{\mathbf{D}}\left( d,\varphi \left( c\right) \right) \right\} , \\
&&%
\Hom%
_{\varphi \downarrow d}\left( \left( \alpha _{1}:d\rightarrow \varphi \left(
c_{1}\right) \right) ,\left( \alpha _{2}:d\rightarrow \varphi \left(
c_{2}\right) \right) \right) 
{:=}%
\left\{ \beta :c_{1}\rightarrow c_{2}~|~\varphi \left( \beta \right) \circ
\alpha _{1}=\alpha _{2}\right\} .
\end{eqnarray*}
\end{enumerate}

\item \label{Not-Categories-(co)cofinal}\label{Def-Co-cofinal}See \cite[%
Definition 2.5.1]{Kashiwara-Categories-MR2182076}.

\begin{enumerate}
\item A functor $\varphi :\mathbf{J\longrightarrow I}$ is \textbf{cofinal}
if the comma-category $i\downarrow \varphi $ is \textbf{connected} \cite[p.
13]{Kashiwara-Categories-MR2182076} for any $i\in \mathbf{I}$.

\item A functor $\varphi :\mathbf{J\longrightarrow I}$ is \textbf{co-cofinal}
if $\varphi ^{op}:\mathbf{J}^{op}\longrightarrow \mathbf{I}^{op}$ is
cofinal, that is, if the comma-category $\varphi \downarrow i$ is connected
for any $i\in \mathbf{I}$.
\end{enumerate}

\item \label{Not-Categories-adjoint}~

\begin{enumerate}
\item Denote by $F\dashv G$ a pair of adjoint functors (\textbf{left}
adjoint $F$ and \textbf{right} adjoint $G$)%
\begin{equation*}
\left( F:\mathbf{C}\longrightarrow \mathbf{D,}G:\mathbf{D}\longrightarrow 
\mathbf{C}\right)
\end{equation*}%
such that there exists an isomorphism of functors%
\begin{equation*}
\Hom%
_{\mathbf{D}}\left( F\left( \bullet \right) ,\bullet \right) 
\simeq%
\Hom%
_{\mathbf{C}}\left( \bullet ,G\left( \bullet \right) \right) :\mathbf{C}%
^{op}\times \mathbf{D}\longrightarrow \mathbf{Set.}
\end{equation*}

\item \label{Not-Categories-(co)reflective}\label{Def-(co)reflective}A
subcategory $\mathbf{C\subseteq D}$ is called \textbf{reflective}
(respectively \textbf{coreflective}) iff the embedding $\mathbf{%
C\hookrightarrow D}$ is a right (respectively left) adjoint. The left
(respectively right) adjoint $\mathbf{D\rightarrow C}$ is called a \textbf{%
reflection} (respectively \textbf{coreflection}).
\end{enumerate}

\item \label{Def-Quarrable}A morphism $V\rightarrow U$ in a category $%
\mathbf{C}$ is called \textbf{base-changeable} (\textquotedblleft
quarrable\textquotedblright\ in (\cite[Def. II.1.3]{SGA4-1-MR0354652}), iff
for every other morphism $U^{\prime }\rightarrow U$ the fiber product $%
V\times _{U}U^{\prime }$ exists.

\item \label{Def-Kan-extensions}Let $\mathbf{I}$ and $\mathbf{J}$ be small
categories and let $\mathbf{C}$ be an arbitrary category.

\begin{enumerate}
\item \label{Def-Kan-extensions-direct}For $\varphi :\mathbf{J\rightarrow I}$
denote by $\varphi _{\ast }$ the following functor:%
\begin{equation*}
\varphi _{\ast }:\mathbf{C}^{\mathbf{I}}\longrightarrow \mathbf{C}^{\mathbf{J%
}}~\left( \varphi _{\ast }\left( f\right) 
{:=}%
f\circ \varphi \right) ,
\end{equation*}%
where $f:\mathbf{I\rightarrow C}$ is an arbitrary diagram.

\item \label{Def-Kan-extensions-left}The following \textbf{left} adjoint ($%
\varphi ^{\dag }\dashv \varphi _{\ast }$) $\varphi ^{\dag }:\mathbf{C}^{%
\mathbf{J}}\rightarrow \mathbf{C}^{\mathbf{I}}$ to $\varphi _{\ast }$ (if
exists!) is called the \textbf{left} Kan extension of $\varphi $. See \cite[%
Definition 2.3.1(ii)]{Kashiwara-Categories-MR2182076}. Assume that $\beta
^{\prime }\left( i\right) 
{:=}%
\underrightarrow{\lim }_{\left( \varphi \left( j\right) \rightarrow i\right)
\in \varphi \downarrow i}\beta \left( j\right) $ exists in $\mathbf{C}$ for
any $i\in \mathbf{I}$. Then $\varphi ^{\dag }\beta $ exists, and we have $%
\varphi ^{\dag }\beta \left( i\right) =\beta ^{\prime }\left( i\right) $ for 
$i\in \mathbf{I}$. See \cite[Theorem 2.3.3(i)]%
{Kashiwara-Categories-MR2182076}.

\item \label{Def-Kan-extensions-right}The following \textbf{right} adjoint ($%
\varphi _{\ast }\dashv \varphi ^{\ddag }$) $\varphi ^{\ddag }:\mathbf{C}^{%
\mathbf{J}}\rightarrow \mathbf{C}^{\mathbf{I}}$ to $\varphi _{\ast }$ (if
exists!) is called the \textbf{right} Kan extension of $\varphi $. See \cite[%
Definition 2.3.1(iii)]{Kashiwara-Categories-MR2182076}. Assume that $\beta
^{\prime \prime }\left( i\right) 
{:=}%
\underleftarrow{\lim }_{\left( i\rightarrow \varphi \left( j\right) \right)
\in i\downarrow \varphi }\beta \left( j\right) $ exists in $\mathbf{C}$ for
any $i\in \mathbf{I}$. Then $\varphi ^{\ddag }\beta $ exists, and we have $%
\varphi ^{\ddag }\beta \left( i\right) =\beta ^{\prime \prime }\left(
i\right) $ for $i\in \mathbf{I}$. See \cite[Theorem 2.3.3(ii)]%
{Kashiwara-Categories-MR2182076}.
\end{enumerate}
\end{enumerate}
\end{notation}

\subsection{Pro-objects}

The main reference is \cite[Chapter 6]{Kashiwara-Categories-MR2182076} where
the $\mathbf{Ind}$-objects are considered. The $\mathbf{Pro}$-objects used
in this paper are dual to the $\mathbf{Ind}$-objects%
\begin{equation*}
\mathbf{Pro}\left( \mathbf{C}\right) 
\simeq%
\left( \mathbf{Ind}\left( \mathbf{C}^{op}\right) \right) ^{op}.
\end{equation*}%
See \cite[the remark after Lemma 6.1.2]{Kashiwara-Categories-MR2182076}.

\begin{definition}
\label{Def-Pro-category}\label{Def-Pro-C}Let $\mathbf{K}$ be a category. The
pro-category $\mathbf{Pro}\left( \mathbf{K}\right) $ (see \cite[Definition
6.1.1]{Kashiwara-Categories-MR2182076}, \cite[Remark I.1.4]%
{Mardesic-Segal-MR676973}, or \cite[Appendix]{Artin-Mazur-MR883959}) is the
full subcategory of $\left( \mathbf{Set}^{\mathbf{K}}\right) ^{op}$
consisting of functors that are cofiltered limits of representable functors,
i.e., limits of diagrams of the form%
\begin{equation*}
\mathbf{I}\overset{\mathbf{X}}{\longrightarrow }\mathbf{K}\overset{\left(
h^{\bullet }\right) ^{op}}{\longrightarrow }\left( \mathbf{Set}^{\mathbf{K}%
}\right) ^{op}
\end{equation*}%
where $\mathbf{I}$ is a cofiltered category, $\mathbf{X}:\mathbf{I}%
\rightarrow \mathbf{K}$ is a diagram, and $\left( h^{\bullet }\right) ^{op}$
is the third Yoneda embedding. We will simply denote such diagrams by $%
\mathbf{X}=\left( X_{i}\right) _{i\in \mathbf{I}}$.
\end{definition}

\begin{remark}
\label{Rem-Pro-objects-morphisms}See \cite[Lemma 6.1.2 and formula (2.6.4)]%
{Kashiwara-Categories-MR2182076}:

\begin{enumerate}
\item \textbf{Co}filtered \textbf{limits} in the category $\left( \mathbf{Set%
}^{\mathbf{K}}\right) ^{op}$ are equivalent to filtered \textbf{co}limits in 
$\mathbf{Set}^{\mathbf{K}}.$

\item Let two pro-objects be defined by the diagrams $\mathbf{X}=\left(
X_{i}\right) _{i\in \mathbf{I}}$ and $\mathbf{Y}=\left( Y_{j}\right) _{j\in 
\mathbf{J}}$. Then%
\begin{equation*}
\Hom%
_{\mathbf{Pro}\left( \mathbf{K}\right) }\left( \mathbf{X},\mathbf{Y}\right)
=~\underset{j\in \mathbf{J}}{\underleftarrow{\lim }}~\underset{i\in \mathbf{I%
}}{\underrightarrow{\lim }}~%
\Hom%
_{\mathbf{K}}\left( X_{i},Y_{j}\right) .
\end{equation*}

\item $\mathbf{Pro}\left( \mathbf{K}\right) $ is indeed a category even
though $\left( \mathbf{Set}^{\mathbf{K}}\right) ^{op}$ is a
\textquotedblleft huge\textquotedblright\ category: $%
\Hom%
_{\mathbf{Pro}\left( \mathbf{K}\right) }\left( \mathbf{X},\mathbf{Y}\right) $
is a \textbf{set} for any $\mathbf{X}$ and $\mathbf{Y}$.
\end{enumerate}
\end{remark}

\begin{remark}
\label{Rem-Trivial-pro-object}\label{Rem-Rudimentary}The category $\mathbf{K}
$ is a full subcategory of $\mathbf{Pro}\left( \mathbf{K}\right) $: any
object $X\in \mathbf{K}$ gives rise to the singleton%
\begin{equation*}
\left( X\right) \in \mathbf{Pro}\left( \mathbf{K}\right)
\end{equation*}%
with a trivial index category $\mathbf{I=}\left( \left\{ i\right\} ,\mathbf{1%
}_{i}\right) $. A pro-object $\mathbf{X}$ is called \textbf{rudimentary} 
\cite[\S I.1.1]{Mardesic-Segal-MR676973} iff it is isomorphic to an object
of $\mathbf{K}$:%
\begin{equation*}
\mathbf{X}%
\simeq%
Z\in \mathbf{K}\subseteq \mathbf{Pro}\left( \mathbf{K}\right) .
\end{equation*}
\end{remark}

Below is the list of several useful facts about pro-categories. See \cite[%
Chapter 6]{Kashiwara-Categories-MR2182076}, \cite[\S I.1 and \S II.2]%
{Mardesic-Segal-MR676973}, \cite[Appendix]{Artin-Mazur-MR883959}, and \cite[%
Section 2.1]{Prasolov-Cosheaves-2021-MR4347662}.

\begin{proposition}
\label{Prop-Facts-pro-objects}Let $\mathbf{X=}\left( X_{i}\right) _{i\in 
\mathbf{I}}\in \mathbf{Pro}\left( \mathbf{K}\right) $.

\begin{enumerate}
\item \label{Prop-Facts-pro-objects-Co-cofinal}Let $\mathbf{J}$ be a
cofiltered (small!) index category, and $\varphi $:$~\mathbf{J\rightarrow I}$
a co-cofinal functor. Then%
\begin{equation*}
\mathbf{X}%
\simeq%
\mathbf{X}\circ \varphi 
{:=}%
\left( X_{\varphi \left( j\right) }\right) _{j\in \mathbf{J}}.
\end{equation*}

\item \label{Prop-Facts-pro-objects-Cofinite}There exists a \textbf{directed}
(\textbf{not} \textbf{co}directed, due to the notations in the cited book
below!) poset $\left( M,\leq \right) $ and a co-cofinal functor $\mathbf{M}%
^{op}\rightarrow \mathbf{I}$ \cite[Theorem I.1.4]{Mardesic-Segal-MR676973}.
Moreover, the poset $M$ could be made \textbf{cofinite}, i.e., such that for
any $m\in M$, the set $\left\{ x\in M~|~x\leq m\right\} $ is finite.

\item \label{Prop-Facts-pro-objects-Mor(K)-finite}If $Mor\left( \mathbf{I}%
\right) $ is finite, there exists an initial object $i\in \mathbf{I}$, and $%
\mathbf{X}$ is rudimentary: $\mathbf{X}%
\simeq%
X_{i}\in \mathbf{K}$.

\item \label{Prop-Facts-pro-objects-Mor(K)-countable}If $Mor\left( \mathbf{I}%
\right) $ is infinite countable, then $\mathbf{X}$ is either rudimentary, or 
$\mathbf{X}%
\simeq%
\mathbf{Y}$ for some \textbf{tower} $\mathbf{Y}$:%
\begin{equation*}
\mathbf{Y=}\left( Y_{0}\overset{\rho _{0}}{\longleftarrow }Y_{1}\overset{%
\rho _{1}}{\longleftarrow }Y_{2}\overset{\rho _{2}}{\longleftarrow }\dots 
\overset{\rho _{i}}{\longleftarrow }Y_{i}\overset{\rho _{i+1}}{%
\longleftarrow }\dots \right) .
\end{equation*}%
In fact, towers could also be rudimentary when, say, all $\rho _{i}$ (or all 
$\rho _{i}$ with $i\geq i_{0}$, for some $i_{0}$) are isomorphisms.

\item \label{Prop-Facts-pro-objects-Mor(K)-UNcountable}For uncountable index
posets, see, e.g., \cite[Section 14, especially Lemma 14.15]%
{Mardesic-Strong-shape-and-homology-MR1740831}.

\item \label{Prop-Facts-pro-objects-Diagrams-levelized}The following
diagrams in $\mathbf{Pro}\left( \mathbf{K}\right) $ may be \textquotedblleft
levelized\textquotedblright\ \cite[Definition 2.1.10(3)]%
{Prasolov-Cosheaves-2021-MR4347662}.

\begin{enumerate}
\item \label{Prop-Facts-pro-objects-Diagrams-without-loops}All finite
diagrams without loops \cite[dual to Theorem 6.4.3]%
{Kashiwara-Categories-MR2182076}.

\item \label{Prop-Facts-pro-objects-Diagrams-One-morphism}One morphism \cite[%
dual to Corollary 6.1.14]{Kashiwara-Categories-MR2182076}.

\item \label{Prop-Facts-pro-objects-Diagrams-Pair-of-morphisms}A pair of
parallel morphisms \cite[dual to Corollary 6.1.15]%
{Kashiwara-Categories-MR2182076}.
\end{enumerate}

\item \label{Prop-Facts-pro-objects-admits-(co)limits}$~$

\begin{enumerate}
\item \label{Prop-Facts-pro-objects-admits-cofiltered-limits}$\mathbf{Pro}%
\left( \mathbf{K}\right) $ admits small \textbf{co}filtered limits $%
\underleftarrow{\lim }$ \cite[dual to Theorem 6.1.8]%
{Kashiwara-Categories-MR2182076}. See Example \ref{Ex-Limits-in-Pro(K)}(\ref%
{Ex-Limits-in-Pro(K)-cofiltered-limits}).

\item \label{Prop-Facts-pro-objects-admits-kernels}$\mathbf{Pro}\left( 
\mathbf{K}\right) $ admits kernels if $\mathbf{K}$ admits kernels \cite[dual
to Proposition 6.1.18(i)]{Kashiwara-Categories-MR2182076} See Example \ref%
{Ex-Limits-in-Pro(K)}(\ref{Ex-Limits-in-Pro(K)-kernels}).

\item \label{Prop-Facts-pro-objects-admits-products}$\mathbf{Pro}\left( 
\mathbf{K}\right) $ admits \textbf{small} products if $\mathbf{K}$ admits 
\textbf{finite} products \cite[dual to Proposition 6.1.18(ii)]%
{Kashiwara-Categories-MR2182076}. See Example \ref{Ex-Limits-in-Pro(K)}(\ref%
{Ex-Limits-in-Pro(K)-products}).

\item \label{Prop-Facts-pro-objects-admits-small-limits}$\mathbf{Pro}\left( 
\mathbf{K}\right) $ admits \textbf{small} limits $\underleftarrow{\lim }$ if 
$\mathbf{K}$ admits \textbf{finite} limits $\underleftarrow{\lim }$ \cite[%
dual to Proposition 6.1.18(iii)]{Kashiwara-Categories-MR2182076}.

\item \label{Prop-Facts-pro-objects-admits-cokernels}$\mathbf{Pro}\left( 
\mathbf{K}\right) $ admits cokernels if $\mathbf{K}$ admits cokernels \cite[%
dual to Proposition 6.1.16(i)]{Kashiwara-Categories-MR2182076}. See Example %
\ref{Ex-Colimits-in-Pro(K)}(\ref{Ex-Colimits-in-Pro(K)-cokernels}).

\item \label{Prop-Facts-pro-objects-admits-small-coproducts}$\mathbf{Pro}%
\left( \mathbf{K}\right) $ admits small coproducts if $\mathbf{K}$ admits
small coproducts \cite[dual to Proposition 6.1.16(ii)]%
{Kashiwara-Categories-MR2182076}. See Example \ref{Ex-Colimits-in-Pro(K)}(%
\ref{Ex-Colimits-in-Pro(K)-coproducts}).

\item \label{Prop-Facts-pro-objects-admits-small-colimits}$\mathbf{Pro}%
\left( \mathbf{K}\right) $ admits finite (small) colimits $\underrightarrow{%
\lim }$ if $\mathbf{K}$ admits finite (small) colimits $\underrightarrow{%
\lim }$ \cite[dual to Corollary 6.1.17]{Kashiwara-Categories-MR2182076}. See
Example \ref{Ex-Colimits-in-Pro(K)}( \ref%
{Ex-Colimits-in-Pro(K)-small-colimits}).
\end{enumerate}

\item \label{Prop-Facts-pro-objects-cofltered-limits-exact}Assume that $%
\mathbf{K}$ admits \textbf{finite} limits $\underleftarrow{\lim }$ and 
\textbf{finite} \textbf{co}limits $\underrightarrow{\lim }$. Then \textbf{%
small} \textbf{co}filtered limits $\underleftarrow{\lim }$ are \textbf{exact}
in $\mathbf{Pro}\left( \mathbf{C}\right) $ \cite[dual to Proposition 6.1.19]%
{Kashiwara-Categories-MR2182076}.

\item \label{Prop-Facts-pro-objects-Embedding}Let $\iota _{\mathbf{K}}:%
\mathbf{K}\rightarrow \mathbf{Pro}\left( \mathbf{K}\right) $ be the full
embedding from Remark \ref{Rem-Rudimentary}, i.e., $\iota _{\mathbf{K}%
}\left( X\right) 
{:=}%
$the singleton $\left( X\right) \in \mathbf{Pro}\left( \mathbf{K}\right) $.

\begin{enumerate}
\item \label{Prop-Facts-pro-objects-Embedding-cocontinuous}$\iota _{\mathbf{K%
}}$ is \textbf{co}continuous \cite[dual to Corollary 6.1.17(ii)]%
{Kashiwara-Categories-MR2182076}, therefore \textbf{right} exact.

\item \label{Prop-Facts-pro-objects-Embedding-NOT-continuous}$\iota _{%
\mathbf{K}}$ is \textbf{not} in general continuous. See Example \ref%
{Ex-Limits-in-Pro(K)}.

\item \label{Prop-Facts-pro-objects-Embedding-exact}$\iota _{\mathbf{K}}$ is 
\textbf{left }exact \cite[dual to Corollary 6.1.17(i)]%
{Kashiwara-Categories-MR2182076}, therefore \textbf{exact}.

\item \label{Prop-Facts-pro-objects-Embedding-Coreflective}See \cite[dual to
Proposition 6.3.1]{Kashiwara-Categories-MR2182076}: assume that $\mathbf{K}$
admits \textbf{small} \textbf{co}filtered limits $\underleftarrow{\lim }$.

\begin{enumerate}
\item There exists $\sigma _{\mathbf{K}}\left( =\underleftarrow{\lim }%
\right) \mathbf{:Pro}\left( \mathbf{K}\right) \rightarrow \mathbf{K}\ $such
that $\iota _{\mathbf{K}}\dashv \sigma _{\mathbf{K}}$.

\item We have $\sigma _{\mathbf{K}}\circ \iota _{\mathbf{C}}%
\simeq%
\mathbf{1}_{\mathbf{K}}$. In other words, $\mathbf{K}$ is a \textbf{co}%
reflective subcategory of $\mathbf{Pro}\left( \mathbf{K}\right) $.
\end{enumerate}
\end{enumerate}

\item If $\mathbf{K}$ is abelian, $\mathbf{Pro}\left( \mathbf{K}\right) $ is
abelian as well \cite[dual to Theorem 8.6.5(i)]%
{Kashiwara-Categories-MR2182076}.
\end{enumerate}
\end{proposition}

\subsection{Pro-modules}

We fix a commutative ring $k$, and let $\mathbf{Mod}\left( k\right) $ be the
category of $k$-modules. We will use a shorter notation $\mathbf{Pro}\left(
k\right) $ for the category $\mathbf{Pro}\left( \mathbf{Mod}\left( k\right)
\right) $ of pro-modules over $k$. Clearly the following categories are
equivalent:%
\begin{eqnarray*}
&&\mathbf{Mod}\left( \mathbb{Z}\right) 
\simeq%
\mathbf{Ab}, \\
&&\mathbf{Pro}\left( \mathbb{Z}\right) 
\simeq%
\mathbf{Pro}\left( \mathbf{Ab}\right) .
\end{eqnarray*}%
The category $\mathbf{Pro}\left( k\right) $ is abelian, and satisfies the $%
AB5^{\ast }$ axiom (\cite[1.5]{Grothendieck-Tohoku-1957-MR0102537}, \cite[%
Ch. 5.8]{Bucur-Deleanu-1968-Introduction-categories-functors-MR0236236}):
cofiltered limits are exact in the category $\mathbf{Pro}\left( k\right) $.
It satisfies also $AB4$, but not $AB5$. See \cite[Proposition A.2.8 (1,10,12)%
]{Prasolov-Cosheaves-2021-MR4347662}. Unfortunately, $\mathbf{Pro}\left(
k\right) $ does \textbf{not} have enough projectives. Fortunately, it 
\textbf{does} have enough \textbf{quasi-projectives} \cite[Proposition
A.2.8(5)]{Prasolov-Cosheaves-2021-MR4347662}.

\begin{definition}
\label{Def-quasi-projective}\cite[dual to Definition 15.2.1]%
{Kashiwara-Categories-MR2182076} A pro-module $\mathbf{P}\in \mathbf{Pro}%
\left( k\right) $ is called \textbf{quasi-projective} if the functor%
\begin{equation*}
\Hom%
_{\mathbf{Pro}\left( k\right) }\left( \mathbf{P},\bullet \right) :\mathbf{Mod%
}\left( k\right) \longrightarrow \mathbf{Mod}\left( k\right)
\end{equation*}%
is exact.
\end{definition}

\begin{proposition}
\label{Prop-quasi-projective}A pro-module $\mathbf{P}\in \mathbf{Pro}\left(
k\right) $ is quasi-projective iff it can be represented as%
\begin{equation*}
\mathbf{P}=\left( P_{i}\right) _{i\in I}
\end{equation*}%
where all $P_{i}$ are projective modules.
\end{proposition}

\begin{proof}
See \cite[dual to Proposition 15.2.3 (ii)]{Kashiwara-Categories-MR2182076}.
\end{proof}

In order to build a satisfactory homology theory for (pre)cosheaves (with
values in $\mathbf{Pro}\left( k\right) $), we have to reduce the class of
rings we are using: from now on, let $k$ be a \textbf{quasi-noetherian} ring
(Definition \ref{Def-quasi-noetherian}). After the reduction, the class is
still very large: it includes all \textbf{noetherian} rings \cite[%
Proposition 2.28]{Prasolov-universal-coefficients-formula-2013-MR3095217}.
It became clear, however, that the two classes \textbf{coincide}: a ring $k$
is quasi-noetherian \textbf{iff} it is noetherian (Theorem \ref%
{Th-(Quasi-)noetherian}).

\subsection{Sites}

\begin{definition}
\label{Def-Site}\cite[Chapter 3.1, p. 31]%
{Jardine-2015-Local-homotopy-theory-MR3309296} A Grothendieck site $X$ is a
pair $X=\left( \mathbf{C}_{X},Cov\left( X\right) \right) $ where $\mathbf{C}%
_{X}$ is a category, and%
\begin{equation*}
Cov\left( X\right) =\left( Cov\left( U\right) \right) _{U\in \mathbf{C}_{X}}
\end{equation*}%
is a set (or class) of \textbf{sieves}, i.e., subfunctors of $h_{U}$
satisfying:

\begin{enumerate}
\item (base change) If $\alpha :V\rightarrow U$ is a morphism in $\mathbf{C}%
_{X}$ and $R\in Cov\left( U\right) $, then%
\begin{equation*}
\left( h_{\alpha }\right) ^{-1}\left( R\right) \in Cov\left( V\right) .
\end{equation*}

\item (local character) Suppose that $R$,$R^{\prime }\subseteq h_{U}$ are
sieves over $U$, and $R$ is covering. Assume that%
\begin{equation*}
\left( h_{\alpha }\right) ^{-1}\left( R^{\prime }\right) \in Cov\left(
V\right)
\end{equation*}%
for any%
\begin{equation*}
\left( \alpha :V\longrightarrow U\right) \in R\left( V\right) .
\end{equation*}%
Then $R^{\prime }\in Cov\left( U\right) $.

\item $h_{U}\in Cov\left( U\right) $.
\end{enumerate}

The sieves from $Cov\left( X\right) $ are called \textbf{covering} sieves.
The site is called \textbf{small} iff $\mathbf{C}_{X}$ is a small category.
The class (or a set, if $X$ is small) $Cov\left( X\right) $ is called the 
\textbf{topology} on $X$.
\end{definition}

\begin{remark}
\label{Rem-Only-small-sites}\label{Rem-Extra-axiom-Kashiwara}~

\begin{enumerate}
\item We will consider only \textbf{small} sites in this paper.

\item \cite[Definition 16.1.2]{Kashiwara-Categories-MR2182076} includes also
the following condition in the definition of a site:
\end{enumerate}

(GT2) If $R_{1}\subseteq R_{2}\subseteq h_{U}$ and $R_{1}\in Cov\left(
U\right) $, then $R_{2}\in Cov\left( U\right) $.

It follows however from \cite[Lemma 3.12(1)]%
{Jardine-2015-Local-homotopy-theory-MR3309296} that the three conditions
from Definition \ref{Def-Site} imply (GT2).
\end{remark}

\begin{definition}
\label{Def-Chaotic}Given a site $X=\left( \mathbf{C}_{X},Cov\left( X\right)
\right) $, we say that the topology is \textbf{chaotic} iff for any $U\in 
\mathbf{C}_{X}$, $Cov\left( U\right) =h_{U}$.
\end{definition}

\begin{definition}
\label{Def-Pretopology}We say that the topology on a small site $X$ is
induced by a \textbf{pretopology} if each object $U\in \mathbf{C}_{X}$ is
supplied with base-changeable (Definition \ref{Def-Quarrable}) \textbf{covers%
} $\left\{ U_{i}\rightarrow U\right\} _{i\in I}$, satisfying \cite[%
Definition II.1.3]{SGA4-1-MR0354652} (compare to \cite[Definition 16.1.5]%
{Kashiwara-Categories-MR2182076}), and the covering sieves $R\in Cov\left(
X\right) $ are \textbf{generated} by covers:%
\begin{equation*}
R=R_{\left\{ U_{i}\rightarrow U\right\} }\subseteq h_{U},
\end{equation*}%
where $R_{\left\{ U_{i}\rightarrow U\right\} }\left( V\right) $ consists of
morphisms $\left( V\rightarrow U\right) \in h_{U}\left( V\right) $ admitting
a decomposition%
\begin{equation*}
\left( V\rightarrow U\right) =\left( V\rightarrow U_{j}\rightarrow U\right)
\end{equation*}%
for some cover $\left\{ U_{i}\rightarrow U\right\} _{i\in I}$ and a $j\in I$.
\end{definition}

\begin{definition}
\label{Def-(Pre)cosheaves}(compare to \cite[Definition 2.2.1]%
{Prasolov-Cosheaves-2021-MR4347662}) Given a site $X=\left( \mathbf{C}%
_{X},Cov\left( X\right) \right) $.

\begin{enumerate}
\item A \textbf{precosheaf} $\mathcal{A}$ on $X$ with values in $\mathbf{K}$
is a functor $\mathcal{A}:\mathbf{C}_{X}\rightarrow \mathbf{K}$.

\item For any $U\in \mathbf{C}_{X}$ and a covering sieve (Definition \ref%
{Def-Site}) $R$ over $U$ there is a natural morphism%
\begin{equation*}
\varphi \left( U,R\right) :\left( \underset{\left( V\rightarrow U\right) \in
R\left( V\right) }{\underrightarrow{\lim }}\mathcal{A}\left( V\right)
\right) \longrightarrow \mathcal{A}\left( U\right)
\end{equation*}

\begin{enumerate}
\item A precosheaf $\mathcal{A}$ on $X$ is \textbf{coseparated} provided $%
\varphi \left( U,R\right) $ is an \textbf{epimorphism} for any $U\in \mathbf{%
C}_{X}$ and for any covering sieve $R\in Cov\left( U\right) $.

\item A precosheaf $\mathcal{A}$ on $X$ is a \textbf{cosheaf} provided $%
\varphi \left( U,R\right) $ is an \textbf{isomorphism} for any $U\in \mathbf{%
C}_{X}$ and for any covering sieve $R\in Cov\left( U\right) $.
\end{enumerate}
\end{enumerate}
\end{definition}

\begin{remark}
\label{Rem-Cosheaf-via-pairing}~

\begin{enumerate}
\item An equivalent definition of $\varphi \left( U,R\right) $ is the
following:%
\begin{equation*}
\varphi \left( U,R\right) :\mathcal{A}\otimes _{\mathbf{Set}^{\mathbf{C}%
_{X}}}R\longrightarrow \mathcal{A}\left( U\right)
\end{equation*}%
where $\otimes _{\mathbf{Set}^{\mathbf{C}_{X}}}$is the pairing defined in 
\cite[Definition A.1.1(3-5)]{Prasolov-Cosheaves-2021-MR4347662}.

\item If the topology is defined using a pretopology (Definition \ref%
{Def-Pretopology}), and a sieve $R$ is generated by a family $\left\{
U_{i}\rightarrow U\right\} _{i\in I}$ then $\varphi \left( U,R\right) $
becomes%
\begin{equation*}
\varphi \left( U,R\right) :%
\coker%
\left( \dcoprod\limits_{i,j\in I}\mathcal{A}\left( U_{i}\underset{U}{\times }%
U_{j}\right) \rightrightarrows \dcoprod\limits_{i\in I}\mathcal{A}\left(
U_{i}\right) \right) \longrightarrow \mathcal{A}\left( U\right) .
\end{equation*}
\end{enumerate}
\end{remark}

\begin{notation}
\label{Not-Categories-of-(pre)(co)sheaves}~

\begin{enumerate}
\item $\mathbf{pCS}\left( X,\mathbf{K}\right) $ is the category of pre%
\textbf{co}sheaves with values in $\mathbf{K}$.

\item $\mathbf{CS}\left( X,\mathbf{K}\right) $ is the category of \textbf{co}%
sheaves with values in $\mathbf{K}$.

\item $\mathbf{pS}\left( X,\mathbf{K}\right) $ is the category of presheaves
with values in $\mathbf{K}$.

\item $\mathbf{S}\left( X,\mathbf{K}\right) $ is the category of sheaves
with values in $\mathbf{K}$.
\end{enumerate}
\end{notation}

\begin{remark}
\label{Rem-Chaotic}If the topology on $X$ is chaotic (Definition \ref%
{Def-Chaotic}), then any pre(co)sheaf is a (co)sheaf, therefore%
\begin{eqnarray*}
\mathbf{pCS}\left( X,\mathbf{K}\right) &=&\mathbf{CS}\left( X,\mathbf{K}%
\right) , \\
\mathbf{pS}\left( X,\mathbf{K}\right) &=&\mathbf{S}\left( X,\mathbf{K}\right)
\end{eqnarray*}%
in this case.
\end{remark}

\subsection{Pairing}

There is an important pairing%
\begin{eqnarray*}
\left\langle \bullet ,\bullet \right\rangle &:&\mathbf{Pro}\left( k\right)
^{op}\times \mathbf{Mod}\left( k\right) \longrightarrow \mathbf{Mod}\left(
k\right) : \\
&&\left\langle \mathbf{A},T\right\rangle 
{:=}%
\Hom%
_{\mathbf{Pro}\left( k\right) }\left( \mathbf{A},T\right) .
\end{eqnarray*}%
The pairing can be easily extended to (pre)cosheaves:%
\begin{eqnarray*}
\left\langle \bullet ,\bullet \right\rangle &:&\mathbf{pCS}\left( X,\mathbf{%
Pro}\left( k\right) \right) ^{op}\times \mathbf{Mod}\left( k\right)
\longrightarrow \mathbf{pS}\left( X,\mathbf{Mod}\left( k\right) \right) : \\
&&\left\langle \mathcal{A},T\right\rangle 
{:=}%
\left( U\longmapsto \left\langle \mathcal{A}\left( U\right) ,T\right\rangle
\right) .
\end{eqnarray*}%
See \cite[Definition A.1.1(1,2) ]{Prasolov-Cosheaves-2021-MR4347662}. The
importance lies in the following principle:

\begin{proposition}
\label{Prop-Duality}A statement on a pro-module $\mathbf{A}$ (respectively a
(pre)\textbf{co}sheaf $\mathcal{B}$) is correct iff for any \textbf{injective%
} $T\in \mathbf{Mod}\left( k\right) $ the dual statement is correct for a $k$%
-module $\left\langle \mathbf{A},T\right\rangle $ (respectively a (pre)sheaf 
$\left\langle \mathcal{B},T\right\rangle $).
\end{proposition}

\begin{remark}
Quite often the requirement that $T$ is injective may be dropped.
\end{remark}

See \cite{Prasolov-Cosheaves-2021-MR4347662}: Theorem 2.2.6 (5-7), Theorem
3.1.1 (2-8), Theorem 3.2.1 (4), Theorem 3.3.1 (2-7), Theorem 3.4.1 (4a),
Proposition A.2.8 (2-4, 7, 8).

The following definition is perfectly consistent with the above principle.

\begin{definition}
\label{Def-quasi-projective-(pre)cosheaf}~

\begin{enumerate}
\item A precosheaf $\mathcal{A}\in \mathbf{pCS}\left( X,\mathbf{Pro}\left(
k\right) \right) $ is called \textbf{quasi-projective} iff for any injective 
$T\in \mathbf{Mod}\left( k\right) $ the presheaf $\left\langle \mathcal{A}%
,T\right\rangle \in \mathbf{pS}\left( X,\mathbf{Mod}\left( k\right) \right) $
is injective.

\item A cosheaf $\mathcal{A}\in \mathbf{CS}\left( X,\mathbf{Pro}\left(
k\right) \right) $ is called \textbf{quasi-projective} iff for any injective 
$T\in \mathbf{Mod}\left( k\right) $ the sheaf $\left\langle \mathcal{A}%
,T\right\rangle \in \mathbf{S}\left( X,\mathbf{Mod}\left( k\right) \right) $
is injective.
\end{enumerate}
\end{definition}

\begin{definition}
\label{Def-Cech-H0}Let $\mathcal{A}\in \mathbf{pCS}\left( X,\mathbf{Pro}%
\left( k\right) \right) $, $U\in \mathbf{C}_{X}$ and $R\subseteq h_{U}$ be a
sieve. Define:

\begin{enumerate}
\item 
\begin{equation*}
H_{0}\left( R,\mathcal{A}\right) 
{:=}%
\underset{\left( V\rightarrow U\right) \in R\left( V\right) }{%
\underrightarrow{\lim }}~\mathcal{A}\left( R\right) 
\simeq%
\mathcal{A}\otimes _{\mathbf{Set}^{\mathbf{C}_{X}}}R\in \mathbf{Mod}\left(
k\right) .
\end{equation*}

\item 
\begin{equation*}
\check{H}_{0}\left( U,\mathcal{A}\right) 
{:=}%
\underset{R\in Cov\left( U\right) }{\underleftarrow{\lim }}~H_{0}\left( R,%
\mathcal{A}\right) \in \mathbf{Mod}\left( k\right) .
\end{equation*}

\item 
\begin{equation*}
\mathcal{A}_{+}%
{:=}%
\left( U\longmapsto \check{H}_{0}\left( U,\mathcal{A}\right) \right) \in 
\mathbf{pCS}\left( X,\mathbf{Pro}\left( k\right) \right) .
\end{equation*}

\item 
\begin{equation*}
\mathcal{A}_{\#}%
{:=}%
\mathcal{A}_{++}\in \mathbf{pCS}\left( X,\mathbf{Pro}\left( k\right) \right)
.
\end{equation*}
\end{enumerate}

See \cite[Definition B.2.5]{Prasolov-Cosheaves-2021-MR4347662}.
\end{definition}

\begin{proposition}
\label{Prop-Cosheafification}For any injective $T\in \mathbf{Mod}\left(
k\right) $:

\begin{enumerate}
\item 
\begin{equation*}
\left\langle H_{0}\left( R,\mathcal{A}\right) ,T\right\rangle 
\simeq%
H^{0}\left( R,\left\langle \mathcal{A},T\right\rangle \right) .
\end{equation*}

\item 
\begin{equation*}
\left\langle \check{H}_{0}\left( U,\mathcal{A}\right) ,T\right\rangle 
\simeq%
\check{H}^{0}\left( U,\left\langle \mathcal{A},T\right\rangle \right) .
\end{equation*}

\item 
\begin{equation*}
\left\langle \mathcal{A}_{+},T\right\rangle 
\simeq%
\left\langle \mathcal{A},T\right\rangle ^{+}.
\end{equation*}

\item 
\begin{equation*}
\left\langle \mathcal{A}_{\#},T\right\rangle 
\simeq%
\left\langle \mathcal{A},T\right\rangle ^{\#}.
\end{equation*}

\item $\mathcal{A}_{+}$ is coseparated.

\item If $\mathcal{A}$ is coseparated, then $\mathcal{A}_{+}$ is a cosheaf.

\item The functor $\left( \bullet \right) _{+}$ is right exact (preserves
finite colimits).

\item The functor $\left( \bullet \right) _{\#}$ is right adjoint to the
inclusion%
\begin{equation*}
\iota :\mathbf{CS}\left( X,\mathbf{Pro}\left( k\right) \right)
\longrightarrow \mathbf{pCS}\left( X,\mathbf{Pro}\left( k\right) \right) .
\end{equation*}

\item \label{Prop-Cosheafification-exact}The functor $\left( \bullet \right)
_{\#}$ is exact. In fact, it preserves arbitrary (small) limits and finite
colimits.
\end{enumerate}
\end{proposition}

\begin{proof}
(1-4) follows from Proposition \ref{Prop-Duality}. (5-9) follows from (1-4),
Proposition \ref{Prop-Duality}, and the corresponding properties for
(pre)sheaves. See also \cite[Theorem 2.2.6]%
{Prasolov-Cosheaves-2021-MR4347662} and \cite[Theorem 3.1 and Corollary 3.2]%
{Prasolov-Cosheafification-2016-MR3660525}.
\end{proof}

\subsection{(Pre)cosheaf homology}

Since the categories $\mathbf{pCS}\left( X,\mathbf{Pro}\left( k\right)
\right) $ and $\mathbf{CS}\left( X,\mathbf{Pro}\left( k\right) \right) $ do
not have enough projective objects, one cannot construct a \textquotedblleft
classical\textquotedblright\ homology theory for (pre)cosheaves. However,
the machinery of derived categories helps us to build the desired homology
theories. In a \textquotedblleft classical\textquotedblright\ homology
theory, one deals with projective resolutions and chain homotopies. Any
additive functor $F$ preserves chain homotopies, therefore it is always
possible to define the left satellites $L_{i}F$. If $F$ is right exact then $%
L_{0}F%
\simeq%
F$. In our situation, we can only use quasi-projective resolutions, and
quasi-isomorphisms between chain complexes.

\begin{notation}
$~$

\begin{enumerate}
\item $Q\left( \mathbf{pCS}\left( X,\mathbf{Pro}\left( k\right) \right)
\right) \subseteq \mathbf{pCS}\left( X,\mathbf{Pro}\left( k\right) \right) $
is the full subcategory consisting of quasi-projective precosheaves.

\item $Q\left( \mathbf{CS}\left( X,\mathbf{Pro}\left( k\right) \right)
\right) \subseteq \mathbf{CS}\left( X,\mathbf{Pro}\left( k\right) \right) $
is the full subcategory consisting of quasi-projective cosheaves.
\end{enumerate}
\end{notation}

The following proposition is of great importance for us.

\begin{proposition}
\label{Prop-Generating}~

\begin{enumerate}
\item $Q\left( \mathbf{CS}\left( X,\mathbf{Pro}\left( k\right) \right)
\right) $ is generating and quasi-projective (see Definition \ref%
{Def-F-projective}).

\item $Q\left( \mathbf{pCS}\left( X,\mathbf{Pro}\left( k\right) \right)
\right) $ is generating and quasi-projective.
\end{enumerate}
\end{proposition}

\begin{proof}
~

\begin{enumerate}
\item See \cite[Theorem 3.4.1(2)]{Prasolov-Cosheaves-2021-MR4347662}.

\item See \cite[Theorem 3.2.1(2)]{Prasolov-Cosheaves-2021-MR4347662}. In
fact, the statement follows also from (1) and Remark \ref{Rem-Chaotic}.
\end{enumerate}
\end{proof}

\begin{remark}
The proof is much simpler in the case of topological spaces, see Proposition %
\ref{Prop-Enough-quasi-projective}.
\end{remark}

\begin{theorem}
\label{Th-Three-F-projective}$Q\left( \mathbf{CS}\left( X,\mathbf{Pro}\left(
k\right) \right) \right) $ is $F$-projective \cite[ Definition A.3.1]%
{Prasolov-Cosheaves-2021-MR4347662} with respect to the functors:

\begin{enumerate}
\item 
\begin{equation*}
F\left( \bullet \right) =\Gamma \left( U,\bullet \right) 
{:=}%
\bullet \left( U\right) .
\end{equation*}

\item 
\begin{equation*}
F=\iota :\mathbf{CS}\left( X,\mathbf{Pro}\left( k\right) \right)
\hookrightarrow \mathbf{pCS}\left( X,\mathbf{Pro}\left( k\right) \right)
\end{equation*}%
(inclusion of cosheaves into precosheaves).

\item 
\begin{equation*}
F=f_{\ast }:\mathbf{CS}\left( X,\mathbf{Pro}\left( k\right) \right)
\longrightarrow \mathbf{CS}\left( Y,\mathbf{Pro}\left( k\right) \right)
\end{equation*}%
(see Proposition \ref{Prop-Site-morphism-cosheaves}).
\end{enumerate}
\end{theorem}

\begin{proof}
(1) and (2) were proved in \cite[ Theorem 3.4.1(2)]%
{Prasolov-Cosheaves-2021-MR4347662}. (3) is proved in Proposition \ref%
{Prop-Site-morphism-satellites} below.
\end{proof}

In the three cases above, one can define the left satellite functors (see
Definition \ref{Def-Left-satellites})%
\begin{eqnarray*}
\mathbb{L}F &:&D^{\ast }\left( \mathbf{C}\right) \longrightarrow D^{\ast
}\left( \mathbf{E}\right) , \\
L_{s}F &:&D^{\ast }\left( \mathbf{C}\right) \longrightarrow \mathbf{E},s\geq
0,
\end{eqnarray*}%
where $\ast =+$ or $\geq 0$, $\mathbf{C=CS}\left( X,\mathbf{Pro}\left(
k\right) \right) $, and $\mathbf{E}$ is either $\mathbf{Pro}\left( k\right) $%
, $\mathbf{pCS}\left( X,\mathbf{Pro}\left( k\right) \right) $ or $\mathbf{CS}%
\left( Y,\mathbf{Pro}\left( k\right) \right) $.

\begin{remark}
The same construction as in Definition \ref{Def-Left-satellites} can be
applied to \textbf{pre}cosheaves.
\end{remark}

\begin{notation}
\label{Not-Various-left-satellites}~

\begin{enumerate}
\item If $F=\Gamma \left( U,\bullet \right) =\bullet \left( U\right) $, let%
\begin{equation*}
H_{s}\left( U,\bullet \right) 
{:=}%
L_{s}F\left( \bullet \right) :\mathbf{CS}\left( X,\mathbf{Pro}\left(
k\right) \right) \longrightarrow \mathbf{Mod}\left( k\right) .
\end{equation*}

\item 
\begin{equation*}
\mathcal{H}_{s}\left( \bullet \right) 
{:=}%
L_{s}\iota \left( \bullet \right) :\mathbf{CS}\left( X,\mathbf{Pro}\left(
k\right) \right) \longrightarrow \mathbf{pCS}\left( X,\mathbf{Pro}\left(
k\right) \right)
\end{equation*}%
where $\iota $ is the inclusion functor $\mathbf{CS}\left( X,\mathbf{Pro}%
\left( k\right) \right) \subseteq \mathbf{pCS}\left( X,\mathbf{Pro}\left(
k\right) \right) $.

\item If $F=H_{0}\left( R,\bullet \right) $, where $R\subseteq h_{U}$, let%
\begin{equation*}
H_{s}\left( R,\bullet \right) 
{:=}%
L_{s}F\left( \bullet \right) :\mathbf{pCS}\left( X,\mathbf{Pro}\left(
k\right) \right) \longrightarrow \mathbf{Mod}\left( k\right) .
\end{equation*}

\item If $F=\check{H}_{0}\left( \bullet \right) $, let%
\begin{equation*}
\check{H}_{s}\left( \bullet \right) 
{:=}%
L_{s}F\left( \bullet \right) :\mathbf{pCS}\left( X,\mathbf{Pro}\left(
k\right) \right) \longrightarrow \mathbf{Mod}\left( k\right) .
\end{equation*}
\end{enumerate}
\end{notation}

\begin{proposition}
\label{Prop-Various-left-satellites}Let $T\in \mathbf{Mod}\left( k\right) $
be injective. Then:

\begin{enumerate}
\item $\left\langle H_{s}\left( U,\bullet \right) ,T\right\rangle 
\simeq%
H^{s}\left( U,\left\langle \bullet ,T\right\rangle \right) $.

\item $\left\langle \mathcal{H}_{s}\left( \bullet \right) ,T\right\rangle 
\simeq%
\mathcal{H}^{s}\left( \left\langle \bullet ,T\right\rangle \right) $.

\item $\left\langle L_{s}f_{\ast }\left( \bullet \right) ,T\right\rangle 
\simeq%
R^{s}F_{\ast }\left\langle \bullet ,T\right\rangle $ (see Proposition \ref%
{Prop-Site-morphism-cosheaves}).

\item $\left\langle H_{s}\left( R,\bullet \right) ,T\right\rangle 
\simeq%
H^{s}\left( R,\left\langle \bullet ,T\right\rangle \right) $.

\item $\left\langle \check{H}_{s}\left( \bullet \right) ,T\right\rangle 
\simeq%
\check{H}^{s}\left( \left\langle \bullet ,T\right\rangle \right) $.
\end{enumerate}
\end{proposition}

\begin{proof}
Follows from Proposition \ref{Prop-Duality}.
\end{proof}

\section{\label{Sec-Site-morphisms}Site morphisms}

\subsection{Satellites $L_{s}f_{\ast }$}

Following \cite[p. 409]%
{MacLane-Moerdijk-1994-Sheaves-in-geometry-and-logic-MR1300636}, define a
morphism by the following.

\begin{definition}
\label{Def-Site-morphism}Let $X=\left( \mathbf{C}_{X},Cov\left( X\right)
\right) $ and $Y=\left( \mathbf{C}_{Y},Cov\left( Y\right) \right) $ be
sites, and let $F:\mathbf{C}_{Y}\longrightarrow \mathbf{C}_{X}$ be a
functor. We say that $F$ is a \textbf{site morphism} (in the sense of \cite[%
Theorem VII.10.2]%
{MacLane-Moerdijk-1994-Sheaves-in-geometry-and-logic-MR1300636}), iff:

\begin{enumerate}
\item $\mathbf{C}_{Y}$ admits finite limits.

\item $F$ is left exact, i.e., preserves finite limits.

\item For any sieve $R\in Cov\left( U\right) $, $U\in \mathbf{C}_{Y}$ the
sieve generated by $F\left( R\right) $ is covering (belongs to $Cov\left(
F\left( U\right) \right) $).
\end{enumerate}
\end{definition}

\begin{proposition}
\label{Prop-Site-morphism}The functor $F$ from Definition \ref%
{Def-Site-morphism} induces a geometric morphism%
\begin{equation*}
\left( \varphi ^{\ast },\varphi _{\ast }\right) :\mathbf{S}\left( X,\mathbf{%
Set}\right) \longrightarrow \mathbf{S}\left( Y,\mathbf{Set}\right)
\end{equation*}%
where:

\begin{enumerate}
\item $\varphi _{\ast }\mathcal{A}%
{:=}%
\mathcal{A}\circ F^{op}$, $\mathcal{A}\in \mathbf{S}\left( X,\mathbf{Set}%
\right) $.

\item $\varphi ^{\ast }\mathcal{B}%
{:=}%
\left( F^{\dag }\mathcal{B}\right) ^{\#}$, $\mathcal{B}\in \mathbf{S}\left(
Y,\mathbf{Set}\right) $, where $\left( \bullet \right) ^{\dag }$ is the left
Kan extension, see Notation \ref{Not-Categories}(\ref{Def-Kan-extensions}).

\item $\varphi ^{\ast }$ is exact.
\end{enumerate}
\end{proposition}

\begin{proof}
See \cite[Theorem VII.10.2]%
{MacLane-Moerdijk-1994-Sheaves-in-geometry-and-logic-MR1300636}.
\end{proof}

\begin{proposition}
\label{Prop-Site-morphism-cosheaves}Let%
\begin{equation*}
F:\mathbf{C}_{Y}\longrightarrow \mathbf{C}_{X}
\end{equation*}%
be a site morphism (Definition \ref{Def-Site-morphism})%
\begin{equation*}
X=\left( \mathbf{C}_{X},Cov\left( X\right) \right) \longrightarrow Y=\left( 
\mathbf{C}_{Y},Cov\left( Y\right) \right) .
\end{equation*}%
Denote by%
\begin{equation*}
f_{\ast }:\mathbf{pCS}\left( X,\mathbf{Pro}\left( k\right) \right)
\longrightarrow \mathbf{pCS}\left( Y,\mathbf{Pro}\left( k\right) \right)
\end{equation*}%
the functor $\mathcal{A}\longmapsto \mathcal{A}\circ F^{op}$, by%
\begin{equation*}
F_{\ast }:\mathbf{pS}\left( X,\mathbf{Pro}\left( k\right) \right)
\longrightarrow \mathbf{pS}\left( Y,\mathbf{Pro}\left( k\right) \right)
\end{equation*}%
the functor $\mathcal{B}\longmapsto \mathcal{B}\circ F^{op}$, by 
\begin{equation*}
f^{\ast }\left( \mathcal{A}\right) 
{:=}%
\left( F^{\ddag }\left( \mathcal{A}\right) \right) _{\#},
\end{equation*}%
where%
\begin{equation*}
F^{\ddag }:\mathbf{pCS}\left( Y,\mathbf{Pro}\left( k\right) \right)
\longrightarrow \mathbf{pCS}\left( X,\mathbf{Pro}\left( k\right) \right)
\end{equation*}%
is the right Kan extension, see Notation \ref{Not-Categories}(\ref%
{Def-Kan-extensions}), and by%
\begin{equation*}
F^{\ast }\left( \mathcal{B}\right) 
{:=}%
\left( F^{\dag }\left( \mathcal{B}\right) \right) ^{\#},
\end{equation*}%
where%
\begin{equation*}
F^{\dag }:\mathbf{pS}\left( Y,\mathbf{Mod}\left( k\right) \right)
\longrightarrow \mathbf{pS}\left( X,\mathbf{Mod}\left( k\right) \right)
\end{equation*}%
is the left Kan extension, see Notation \ref{Not-Categories}(\ref%
{Def-Kan-extensions}). Let also $T$ be an \textbf{arbitrary} injective $k$%
-module. Then:

\begin{enumerate}
\item $\left\langle f_{\ast }\bullet ,T\right\rangle 
\simeq%
F_{\ast }\left\langle \bullet ,T\right\rangle $.

\item $\left\langle f^{\ast }\bullet ,T\right\rangle 
\simeq%
F^{\ast }\left\langle \bullet ,T\right\rangle $.

\item $F_{\ast }$ sends sheaves to sheaves.

\item $f_{\ast }$ sends \textbf{co}sheaves to \textbf{co}sheaves.

\item $F^{\ast }\dashv F_{\ast }$, i.e., $F^{\ast }$ is left adjoint to $%
F_{\ast }$. Moreover,%
\begin{equation*}
F^{\ast }\circ \left( \bullet \right) ^{\#}%
\simeq%
\left( \bullet \right) ^{\#}\circ F^{\dag }:\mathbf{pS}\left( Y,\mathbf{Mod}%
\left( k\right) \right) \longrightarrow \mathbf{S}\left( X,\mathbf{Mod}%
\left( k\right) \right) .
\end{equation*}

\item \label{Prop-Site-morphism-cosheaves-right-adjoint}$f_{\ast }\dashv
f^{\ast }$, i.e., $f^{\ast }$ is right adjoint to $f_{\ast }$. Moreover,%
\begin{equation*}
f^{\ast }\circ \left( \bullet \right) _{\#}%
\simeq%
\left( \bullet \right) _{\#}\circ f^{\ddag }:\mathbf{pCS}\left( Y,\mathbf{Pro%
}\left( k\right) \right) \longrightarrow \mathbf{CS}\left( X,\mathbf{Pro}%
\left( k\right) \right) .
\end{equation*}

\item $F^{\ast }$ is exact.

\item \label{Prop-Site-morphism-cosheaves-exact}$f^{\ast }$ is exact.

\item $F_{\ast }$ sends injective sheaves to injective sheaves.

\item \label{Prop-Site-morphism-cosheaves-quasi-projective}$f_{\ast }$ sends
quasi-projective cosheaves to quasi-projective cosheaves.
\end{enumerate}
\end{proposition}

\begin{proof}
~

\begin{enumerate}
\item Evident.

\item Let $\mathcal{A}\in \mathbf{pCS}\left( X,\mathbf{Pro}\left( k\right)
\right) $. Then%
\begin{equation*}
\left\langle f^{\ast }\mathcal{A},T\right\rangle 
\simeq%
\left\langle \left( F^{\ddag }\left( \mathcal{A}\right) \right)
_{\#},T\right\rangle 
\simeq%
\left\langle F^{\ddag }\left( \mathcal{A}\right) ,T\right\rangle ^{\#}%
\simeq%
\left( F^{\dag }\left\langle \mathcal{A},T\right\rangle \right) ^{\#}%
\simeq%
F^{\ast }\left\langle \mathcal{A},T\right\rangle .
\end{equation*}

\item Follows from Proposition \ref{Prop-Site-morphism}.

\item Let $\mathcal{A}\in \mathbf{CS}\left( X,\mathbf{Pro}\left( k\right)
\right) $. It follows that $\left\langle \mathcal{A},T\right\rangle $ is a
sheaf. Then $\left\langle f_{\ast }\mathcal{A},T\right\rangle =F_{\ast
}\left\langle \mathcal{A},T\right\rangle $ is a sheaf as well, therefore $%
f_{\ast }\mathcal{A}$ is a \textbf{co}sheaf.

\item Let $\mathcal{A}\in \mathbf{S}\left( X,\mathbf{Mod}\left( k\right)
\right) $, and $\mathcal{B}\in \mathbf{S}\left( Y,\mathbf{Mod}\left(
k\right) \right) $. Then, naturally in $\mathcal{A}$ and $\mathcal{B}$,%
\begin{eqnarray*}
&&%
\Hom%
_{\mathbf{S}\left( Y,\mathbf{Mod}\left( k\right) \right) }\left( \mathcal{A}%
,F_{\ast }\mathcal{B}\right) 
\simeq%
\Hom%
_{\mathbf{pS}\left( Y,\mathbf{Mod}\left( k\right) \right) }\left( \mathcal{A}%
,F_{\ast }\mathcal{B}\right) 
\simeq%
\Hom%
_{\mathbf{pS}\left( X,\mathbf{Mod}\left( k\right) \right) }\left( F^{\dag }%
\mathcal{A},\mathcal{B}\right) 
\simeq
\\
&&%
\Hom%
_{\mathbf{S}\left( X,\mathbf{Mod}\left( k\right) \right) }\left( \left(
F^{\dag }\left( \mathcal{A}\right) \right) ^{\#},\mathcal{B}\right) 
\simeq%
\Hom%
_{\mathbf{S}\left( X,\mathbf{Mod}\left( k\right) \right) }\left( F^{\ast
}\left( \mathcal{A}\right) ,\mathcal{B}\right) .
\end{eqnarray*}%
Moreover, consider the following commutative (up to an isomorphism) diagram
of functors%
\begin{equation*}
\begin{diagram}[size=3.0em,textflow]
\mathbf{S}\left( X,\mathbf{Mod}\left( k\right) \right) & \rTo^{F_{\ast }} & \mathbf{S}\left( Y,\mathbf{Mod}\left( k\right) \right) \\
\dTo^{\iota _{X}} & & \dTo^{\iota _{Y}} \\
\mathbf{pS}\left( X,\mathbf{Mod}\left( k\right) \right) & \rTo^{F_{\ast }} & \mathbf{pS}\left( Y,\mathbf{Mod}\left( k\right) \right) \\
\end{diagram}%
\end{equation*}%
It follows that the corresponding diagram of \textbf{left} adjoints%
\begin{equation*}
\begin{diagram}[size=3.0em,textflow]
\mathbf{S}\left( X,\mathbf{Mod}\left( k\right) \right) & \lTo^{F^{\ast }} & \mathbf{S}\left( Y,\mathbf{Mod}\left( k\right) \right) \\
\uTo_{\left( \bullet \right)^{\#}} & & \uTo_{\left( \bullet \right)^{\#}} \\
\mathbf{pS}\left( X,\mathbf{Mod}\left( k\right) \right) & \lTo^{F^{\dag }} & \mathbf{pS}\left( Y,\mathbf{Mod}\left( k\right) \right) \\
\end{diagram}
%
\end{equation*}%
is commutative up to an isomorphism, too.

\item Let $\mathcal{A}\in \mathbf{CS}\left( X,\mathbf{Pro}\left( k\right)
\right) $ and $\mathcal{B}\in \mathbf{CS}\left( Y,\mathbf{Pro}\left(
k\right) \right) $. Then, naturally in $\mathcal{A}$ and $\mathcal{B}$,%
\begin{eqnarray*}
&&%
\Hom%
_{\mathbf{CS}\left( Y,\mathbf{Pro}\left( k\right) \right) }\left( f_{\ast }%
\mathcal{A},\mathcal{B}\right) 
\simeq%
\Hom%
_{\mathbf{pCS}\left( Y,\mathbf{Pro}\left( k\right) \right) }\left( f_{\ast }%
\mathcal{A},\mathcal{B}\right) 
\simeq%
\Hom%
_{\mathbf{pCS}\left( X,\mathbf{Pro}\left( k\right) \right) }\left( \mathcal{A%
},f^{\ddag }\left( \mathcal{B}\right) \right) 
\simeq
\\
&&%
\Hom%
_{\mathbf{CS}\left( X,\mathbf{Pro}\left( k\right) \right) }\left( \mathcal{A}%
,\left( f^{\ddag }\left( \mathcal{B}\right) \right) _{\#}\right) 
\simeq%
\Hom%
_{\mathbf{CS}\left( X,\mathbf{Pro}\left( k\right) \right) }\left( \mathcal{A}%
,f^{\ast }\left( \mathcal{B}\right) \right) .
\end{eqnarray*}%
Moreover, consider the following commutative (up to an isomorphism) diagram
of functors%
\begin{equation*}
\begin{diagram}[size=3.0em,textflow]
\mathbf{S}\left( X,\mathbf{Pro}\left( k\right) \right) & \rTo^{f_{\ast }} & \mathbf{S}\left( Y,\mathbf{Pro}\left( k\right) \right) \\
\dTo^{\iota _{X}} & & \dTo^{\iota _{Y}} \\
\mathbf{pS}\left( X,\mathbf{Pro}\left( k\right) \right) & \rTo^{f_{\ast }} & \mathbf{pS}\left( Y,\mathbf{Pro}\left( k\right) \right) \\
\end{diagram}
%
\end{equation*}%
It follows that the corresponding diagram of \textbf{right} adjoints%
\begin{equation*}
\begin{diagram}[size=3.0em,textflow]
\mathbf{S}\left( X,\mathbf{Pro}\left( k\right) \right) & \lTo^{f^{\ast }} & \mathbf{S}\left( Y,\mathbf{Pro}\left( k\right) \right) \\
\uTo_{\left( \bullet \right)^{\#}} & & \uTo_{\left( \bullet \right)^{\#}} \\
\mathbf{pS}\left( X,\mathbf{Pro}\left( k\right) \right) & \lTo^{f^{\ddag }} & \mathbf{pS}\left( Y,\mathbf{Pro}\left( k\right) \right) \\
\end{diagram}
%
\end{equation*}%
is commutative up to an isomorphism, too.

\item ~

\begin{enumerate}
\item Since $F^{\ast }$ is left adjoint, it preserves arbitrary small (e.g.,
finite) colimits, i.e., $F^{\ast }$ is \textbf{right} exact.

\item $F^{\ast }$ is the composition $\left( \bullet \right) ^{\#}\circ
F^{\dag }\circ \iota $ of three functors%
\begin{eqnarray*}
\iota &:&\mathbf{S}\left( Y,\mathbf{Mod}\left( k\right) \right)
\longrightarrow \mathbf{pS}\left( Y,\mathbf{Mod}\left( k\right) \right) , \\
F^{\dag } &:&\mathbf{pS}\left( Y,\mathbf{Mod}\left( k\right) \right)
\longrightarrow \mathbf{pS}\left( X,\mathbf{Mod}\left( k\right) \right) , \\
\left( \bullet \right) ^{\#} &:&\mathbf{pS}\left( X,\mathbf{Mod}\left(
k\right) \right) \longrightarrow \mathbf{S}\left( X,\mathbf{Mod}\left(
k\right) \right) .
\end{eqnarray*}%
The inclusion $\iota $, being right adjoint to $\left( \bullet \right) ^{\#}$%
, preserves arbitrary small (e.g., finite) limits, therefore is \textbf{left}
exact. It is well-known that the functor $\left( \bullet \right) ^{\#}$ is
exact.

\item It remains to show that $F^{\dag }$ is left exact. Since $F$ preserves
finite limits, it follows from \cite[Proposition 3.3.3 and Definition 3.3.1]%
{Kashiwara-Categories-MR2182076}, that the comma-category $U\downarrow F$ is 
\textbf{co}filtered for any $U\in \mathbf{C}_{X}$. Therefore the category%
\begin{equation*}
\left( U\downarrow F\right) ^{op}%
\simeq%
F^{op}\downarrow U
\end{equation*}%
is filtered. Remind that%
\begin{equation*}
F^{\dag }\mathcal{A}\left( U\right) =\underset{F\left( V\right) \rightarrow V%
}{\underrightarrow{\lim }}~\mathcal{A}\left( V\right) \in \mathbf{Mod}\left(
k\right) .
\end{equation*}%
Filtered colimits are exact in the category $\mathbf{Ab}$, therefore $%
F^{\dag }$ and $F^{\ast }$ are exact.
\end{enumerate}

\item 
\begin{equation*}
\left\langle f^{\ast }\bullet ,T\right\rangle 
\simeq%
F^{\ast }\left\langle \bullet ,T\right\rangle :\mathbf{S}\left( Y,\mathbf{Mod%
}\left( k\right) \right) \longrightarrow \mathbf{S}\left( X,\mathbf{Mod}%
\left( k\right) \right)
\end{equation*}%
is exact, therefore%
\begin{equation*}
f^{\ast }:\mathbf{CS}\left( Y,\mathbf{Pro}\left( k\right) \right)
\longrightarrow \mathbf{CS}\left( X,\mathbf{Pro}\left( k\right) \right)
\end{equation*}%
is exact as well.

\item Let $\mathcal{B}\in \mathbf{S}\left( Y,\mathbf{Mod}\left( k\right)
\right) $ be an injective sheaf. The functor%
\begin{equation*}
\Hom%
_{\mathbf{S}\left( X,\mathbf{Mod}\left( k\right) \right) }\left( \bullet
,F_{\ast }\mathcal{B}\right) 
\simeq%
\Hom%
_{\mathbf{S}\left( Y,\mathbf{Mod}\left( k\right) \right) }\left( F^{\ast
}\bullet ,\mathcal{B}\right)
\end{equation*}%
is exact because $F^{\ast }$ is exact and $\mathcal{B}$ is injective.
Therefore $F_{\ast }\mathcal{B}$ is injective.

\item Let $\mathcal{A}\in \mathbf{CS}\left( Y,\mathbf{Pro}\left( k\right)
\right) $ be quasi-projective, i.e., $\left\langle \mathcal{A}%
,T\right\rangle $ is an injective sheaf. Then%
\begin{equation*}
\left\langle f_{\ast }\mathcal{A},T\right\rangle 
\simeq%
F_{\ast }\left\langle \mathcal{A},T\right\rangle
\end{equation*}%
is injective, thus $f_{\ast }\mathcal{A}$ is quasi-projective.
\end{enumerate}
\end{proof}

\begin{proposition}
\label{Prop-Site-morphism-satellites}The full subcategory of
quasi-projective cosheaves%
\begin{equation*}
\mathbf{P=}Q\left( \mathbf{CS}\left( X,\mathbf{Pro}\left( k\right) \right)
\right) \subseteq \mathbf{CS}\left( X,\mathbf{Pro}\left( k\right) \right)
\end{equation*}%
is $f_{\ast }$-projective where%
\begin{equation*}
f_{\ast }:\mathbf{CS}\left( X,\mathbf{Pro}\left( k\right) \right)
\longrightarrow \mathbf{CS}\left( Y,\mathbf{Pro}\left( k\right) \right) .
\end{equation*}
\end{proposition}

\begin{proof}
We have already proved that $\mathbf{P}$ is generating and quasi-projective
(see Proposition \ref{Prop-Generating}). Consider now an exact sequence%
\begin{equation*}
0\longrightarrow \mathcal{A}^{\prime }\longrightarrow \mathcal{A}%
\longrightarrow \mathcal{A}^{\prime \prime }\longrightarrow 0
\end{equation*}%
in $\mathbf{CS}\left( X,\mathbf{Pro}\left( k\right) \right) $ with $\mathcal{%
A}$, $\mathcal{A}^{\prime \prime }\in \mathbf{P}$. For any injective $T\in 
\mathbf{Mod}\left( k\right) $, the sequence in $\mathbf{S}\left( X,\mathbf{%
Pro}\left( k\right) \right) $%
\begin{equation*}
0\longrightarrow \left\langle \mathcal{A}^{\prime \prime },T\right\rangle
\longrightarrow \left\langle \mathcal{A},T\right\rangle \longrightarrow
\left\langle \mathcal{A}^{\prime },T\right\rangle \longrightarrow 0
\end{equation*}%
is exact in $\mathbf{S}\left( X,\mathbf{Mod}\left( k\right) \right) $, while 
$\left\langle \mathcal{A}^{\prime \prime },T\right\rangle $ and $%
\left\langle \mathcal{A},T\right\rangle $ are injective. Therefore the above
sequence splits, and%
\begin{equation*}
\left\langle \mathcal{A},T\right\rangle 
\simeq%
\left\langle \mathcal{A}^{\prime \prime },T\right\rangle \oplus \left\langle 
\mathcal{A}^{\prime },T\right\rangle .
\end{equation*}%
Apply the functor $\mathcal{B}\mapsto F_{\ast }\mathcal{B}$ to the split
exact sequence above, and get the following split exact sequences in $%
\mathbf{Mod}\left( k\right) $%
\begin{equation*}
\begin{diagram}[size=3.0em,textflow]
0 & \rTo & F_{\ast }\left\langle \mathcal{A}^{\prime \prime},T\right\rangle & \rTo & F_{\ast }\left\langle \mathcal{A}^{\prime },T\right\rangle & \rTo &  F_{\ast }\left\langle \mathcal{A}^{\prime },T\right\rangle & \rTo & 0  \\
& & \dTo_{=} & & \dTo_{=} & &  \dTo_{=} \\ 
0 & \rTo & \left\langle f_{\ast }\mathcal{A}^{\prime \prime},T\right\rangle & \rTo & \left\langle f_{\ast }\mathcal{A}^{\prime },T\right\rangle & \rTo &  \left\langle f_{\ast }\mathcal{A}^{\prime },T\right\rangle & \rTo & 0  \\
\end{diagram}
%
\end{equation*}%
Due to Proposition \ref{Prop-Duality}, the sequence%
\begin{equation*}
0\longrightarrow f_{\ast }\mathcal{A}^{\prime }\longrightarrow f_{\ast }%
\mathcal{A}\longrightarrow f_{\ast }\mathcal{A}^{\prime \prime
}\longrightarrow 0
\end{equation*}%
is exact in $\mathbf{CS}\left( Y,\mathbf{Pro}\left( k\right) \right) $.
\end{proof}

\begin{corollary}
\label{Cor-Satellites-site-morphisms}~

\begin{enumerate}
\item The left satellites%
\begin{equation*}
L_{s}f_{\ast }:\mathbf{CS}\left( X,\mathbf{Pro}\left( k\right) \right)
\longrightarrow \mathbf{pCS}\left( Y,\mathbf{Pro}\left( k\right) \right)
\end{equation*}%
are well defined.

\item $L_{0}f_{\ast }%
\simeq%
f_{\ast }$.
\end{enumerate}
\end{corollary}

\begin{proof}
~

\begin{enumerate}
\item Follows from Proposition \ref{Prop-Site-morphism-satellites}.

\item $f_{\ast }$, being \textbf{left} adjoint to $f^{\ast }$, is \textbf{%
right} exact.
\end{enumerate}
\end{proof}

\subsection{Spectral sequences}

The main result of this section is the following

\begin{theorem}
\label{Th-Grothendieck-spectral-sequence-(pre)cosheaves}Let $f:X\rightarrow
Y $ and $g:Y\rightarrow Z$ be morphisms of sites induced by functors $F:%
\mathbf{C}_{Y}\rightarrow \mathbf{C}_{X}$ and $G:\mathbf{C}_{Z}\rightarrow 
\mathbf{C}_{Y}$, respectively.

\begin{enumerate}
\item If $\mathcal{A}\in \mathbf{CS}\left( X,\mathbf{Pro}\left( k\right)
\right) $ is quasi-projective as a cosheaf, then $\iota \mathcal{A}\in 
\mathbf{pCS}\left( X,\mathbf{Pro}\left( k\right) \right) $ is
quasi-projective as a precosheaf.

\item \label{Th-Grothendieck-spectral-sequence-compare-to-Cech}Let $U\in 
\mathbf{C}_{X}$ and $\mathcal{A}\in \mathbf{CS}\left( X,\mathbf{Pro}\left(
k\right) \right) $

\begin{enumerate}
\item \label{Th-Grothendieck-spectral-sequence-compare-to-Cech-sieve}For any
covering sieve $R$ on $U$ there exists a natural spectral sequence%
\begin{equation*}
E_{s,t}^{2}%
{:=}%
H_{t}\left( R,\mathcal{H}_{s}\left( \mathcal{A}\right) \right) \implies
H_{s+t}\left( U,\mathcal{A}\right) .
\end{equation*}

\item \label{Th-Grothendieck-spectral-sequence-compare-to-Cech-limit}There
exists a natural spectral sequence%
\begin{equation*}
E_{s,t}^{2}=\check{H}_{t}\left( U,\mathcal{H}_{s}\left( \mathcal{A}\right)
\right) \implies H_{s+t}\left( U,\mathcal{A}\right) .
\end{equation*}
\end{enumerate}

\item \label{Th-Grothendieck-spectral-sequence-Leray}(Leray) For each $U\in 
\mathbf{C}_{Y}$ and $\mathcal{A}\in \mathbf{CS}\left( X,\mathbf{Pro}\left(
k\right) \right) $ there exists a natural spectral sequence%
\begin{equation*}
E_{s,t}^{2}%
{:=}%
H_{t}\left( U,L_{s}f_{\ast }\left( \mathcal{A}\right) \right) \implies
H_{s+t}\left( F\left( U\right) ,\mathcal{A}\right) .
\end{equation*}

\item \label{Th-Grothendieck-spectral-sequence-Grothendieck}(Grothendieck)
For each $\mathcal{A}\in \mathbf{CS}\left( X,\mathbf{Pro}\left( k\right)
\right) $ there exists a natural spectral sequence%
\begin{equation*}
E_{s,t}^{2}%
{:=}%
L_{t}g_{\ast }\left( L_{s}f_{\ast }\left( \mathcal{A}\right) \right)
\implies L_{s+t}\left( g\circ f\right) _{\ast }\left( \mathcal{A}\right) .
\end{equation*}
\end{enumerate}
\end{theorem}

\begin{proof}
~

\begin{enumerate}
\item Let $Y=\left( \left( \mathbf{C}_{Y}=\mathbf{C}_{X}\right) ,\text{%
chaotic}\right) $ be the site on the same category $\mathbf{C}_{X}$, but
with the chaotic topology (Definition \ref{Def-Chaotic}), and let $%
h:X\rightarrow Y$ be the morphism induced by the identity functor%
\begin{equation*}
H=1_{\mathbf{C}_{X}}:\mathbf{C}_{Y}\longrightarrow \mathbf{C}_{X}.
\end{equation*}%
Due to Remark \ref{Rem-Chaotic}, $h_{\ast }$ is in fact the inclusion%
\begin{equation*}
h_{\ast }=\iota :\mathbf{CS}\left( X,\mathbf{Pro}\left( k\right) \right)
\longrightarrow \mathbf{pCS}\left( X,\mathbf{Pro}\left( k\right) \right) .
\end{equation*}%
It follows from Proposition \ref{Prop-Site-morphism-cosheaves}(\ref%
{Prop-Site-morphism-cosheaves-quasi-projective}) that $\iota \mathcal{A}%
=h_{\ast }\mathcal{A}$ is quasi-projective.

\item For an alternative proof, see \cite[Theorem 3.4.1(6ab)]%
{Prasolov-Cosheaves-2021-MR4347662}.

\begin{enumerate}
\item Consider the functors%
\begin{eqnarray*}
\iota &:&\mathbf{CS}\left( X,\mathbf{Pro}\left( k\right) \right)
\longrightarrow \mathbf{pCS}\left( X,\mathbf{Pro}\left( k\right) \right) , \\
H_{0}\left( R,\bullet \right) &:&\mathbf{pCS}\left( X,\mathbf{Pro}\left(
k\right) \right) \longrightarrow \mathbf{Pro}\left( k\right) , \\
H_{0}\left( R,\bullet \right) \circ \iota &=&H_{0}\left( R,\bullet \right)
=\Gamma \left( U,\bullet \right) : \\
\mathbf{CS}\left( X,\mathbf{Pro}\left( k\right) \right) &\longrightarrow &%
\mathbf{Pro}\left( k\right) .
\end{eqnarray*}%
Since $\iota $ sends quasi-projective cosheaves to quasi-projective
(therefore $H_{0}\left( R,\bullet \right) $-acyclic) precosheaves, one may
apply Theorem \ref{Th-Grothendieck-spectral-sequence}. Since%
\begin{eqnarray*}
L_{s}\iota &=&\mathcal{H}_{s}, \\
L_{t}H_{0}\left( R,\bullet \right) &=&H_{t}\left( R,\bullet \right) , \\
L_{n}\left( H_{0}\left( R,\bullet \right) \circ \iota \right) &=&H_{n}\left(
U,\mathcal{A}\right) .
\end{eqnarray*}%
one obtains a spectral sequence%
\begin{equation*}
E_{s,t}^{2}=L_{t}H_{0}\left( R,L_{s}\iota \mathcal{A}\right) =H_{t}\left( R,%
\mathcal{H}_{s}\mathcal{A}\right) \implies H_{s+t}\left( U,\mathcal{A}%
\right) .
\end{equation*}

\item Consider the functors%
\begin{eqnarray*}
\check{H}_{0}\left( U,\bullet \right) &:&\mathbf{pCS}\left( X,\mathbf{Pro}%
\left( k\right) \right) \longrightarrow \mathbf{Pro}\left( k\right) , \\
\check{H}_{0}\left( U,\bullet \right) \circ \iota &=&H_{0}\left( U,\bullet
\right) =\Gamma \left( U,\bullet \right) : \\
\mathbf{CS}\left( X,\mathbf{Pro}\left( k\right) \right) &\longrightarrow &%
\mathbf{Pro}\left( k\right) .
\end{eqnarray*}%
Apply again Theorem \ref{Th-Grothendieck-spectral-sequence}. Since%
\begin{eqnarray*}
L_{t}\check{H}_{0}\left( U,\bullet \right) &=&H_{t}\left( U,\bullet \right) ,
\\
L_{n}\left( H_{0}\left( U,\bullet \right) \circ \iota \right) &=&H_{n}\left(
U,\mathcal{A}\right)
\end{eqnarray*}%
one obtains a spectral sequence%
\begin{equation*}
E_{s,t}^{2}=L_{t}\check{H}_{0}\left( U,L_{s}\iota \mathcal{A}\right) =\check{%
H}_{t}\left( U,\mathcal{H}_{s}\mathcal{A}\right) \implies H_{s+t}\left( U,%
\mathcal{A}\right) .
\end{equation*}
\end{enumerate}

\item Consider the functors%
\begin{eqnarray*}
f_{\ast } &:&\mathbf{CS}\left( X,\mathbf{Pro}\left( k\right) \right)
\longrightarrow \mathbf{CS}\left( Y,\mathbf{Pro}\left( k\right) \right) , \\
\Gamma \left( U,\bullet \right) &:&\mathbf{CS}\left( Y,\mathbf{Pro}\left(
k\right) \right) \longrightarrow \mathbf{Pro}\left( k\right) , \\
\Gamma \left( U,\bullet \right) \circ f_{\ast } &=&\Gamma \left( F\left(
U\right) ,\bullet \right) :\mathbf{CS}\left( X,\mathbf{Pro}\left( k\right)
\right) \longrightarrow \mathbf{Pro}\left( k\right) .
\end{eqnarray*}%
Since $f_{\ast }$ sends quasi-projective cosheaves on $X$ to
quasi-projective (therefore $\Gamma \left( U,\bullet \right) $-acyclic)
cosheaves on $Y$ (Proposition \ref{Prop-Site-morphism-cosheaves}(\ref%
{Prop-Site-morphism-cosheaves-quasi-projective})), one may apply Theorem \ref%
{Th-Grothendieck-spectral-sequence}. Since%
\begin{eqnarray*}
L_{t}\Gamma \left( U,\bullet \right) &=&H_{t}\left( U,\bullet \right) , \\
L_{n}\left( \Gamma \left( U,\bullet \right) \circ f_{\ast }\right)
&=&H_{n}\left( F\left( U\right) ,\bullet \right) .
\end{eqnarray*}%
one obtains a spectral sequence%
\begin{equation*}
E_{s,t}^{2}=L_{t}\Gamma \left( U,L_{s}f_{\ast }\mathcal{A}\right)
=H_{t}\left( U,\left( L_{s}f_{\ast }\right) \mathcal{A}\right) \implies
H_{s+t}\left( F\left( U\right) ,\mathcal{A}\right) .
\end{equation*}

\item Consider the functors%
\begin{eqnarray*}
f_{\ast } &:&\mathbf{CS}\left( X,\mathbf{Pro}\left( k\right) \right)
\longrightarrow \mathbf{CS}\left( Y,\mathbf{Pro}\left( k\right) \right) , \\
g_{\ast } &:&\mathbf{CS}\left( Y,\mathbf{Pro}\left( k\right) \right)
\longrightarrow \mathbf{CS}\left( Z,\mathbf{Pro}\left( k\right) \right) , \\
g_{\ast }\circ f_{\ast } &=&\left( g\circ f\right) _{\ast }:\mathbf{CS}%
\left( X,\mathbf{Pro}\left( k\right) \right) \longrightarrow \mathbf{CS}%
\left( Z,\mathbf{Pro}\left( k\right) \right) .
\end{eqnarray*}%
Since $f_{\ast }$ sends quasi-projective cosheaves on $X$ to
quasi-projective (therefore $g_{\ast }$-acyclic) cosheaves on $Y$, one may
apply Theorem \ref{Th-Grothendieck-spectral-sequence} and obtain a spectral
sequence%
\begin{equation*}
E_{s,t}^{2}=L_{t}g_{\ast }\left( L_{s}f_{\ast }\left( \mathcal{A}\right)
\right) \implies L_{s+t}\left( g\circ f\right) _{\ast }\left( \mathcal{A}%
\right) .
\end{equation*}
\end{enumerate}
\end{proof}

\section{\label{Sec-Hypercoverings}Hypercoverings}

Throughout this section, $X=\left( \mathbf{C}_{X},Cov\left( X\right) \right) 
$ is a site such that $\mathbf{C}_{X}$ admits fibre products.

\begin{definition}
\label{Def-Tensor-Set}Given $A\in \mathbf{Set}$ and $B\in \mathbf{K}$, let%
\begin{equation*}
A\otimes _{\mathbf{Set}}B%
{:=}%
\dcoprod\limits_{A}B\in \mathbf{K}
\end{equation*}%
\cite[Definition A.1.1(3)]{Prasolov-Cosheaves-2021-MR4347662}. Clearly, $%
\otimes _{\mathbf{Set}}$is a functor%
\begin{equation*}
\bullet \otimes _{\mathbf{Set}}\bullet :\mathbf{Set}\times \mathbf{K}%
\longrightarrow \mathbf{K.}
\end{equation*}
\end{definition}

\begin{definition}
\label{Def-Tensor-Presheaves}Given $\mathcal{A}\in \mathbf{pS}\left( X,%
\mathbf{Set}\right) $ and $\mathcal{B}\in \mathbf{pCS}\left( X,\mathbf{Pro}%
\left( k\right) \right) $, let $\mathcal{A\otimes }_{\mathbf{pS}\left( X,%
\mathbf{Set}\right) }\mathcal{B}$ be the following pro-module:%
\begin{equation*}
\mathcal{A\otimes }_{\mathbf{pS}\left( X,\mathbf{Set}\right) }\mathcal{B}%
{:=}%
\coker%
\left( \dbigoplus\limits_{\substack{ V,W\in \mathbf{C}_{X}  \\ \left(
V\rightarrow W\right) \in \mathbf{C}_{X}\left( V\rightarrow W\right) }}%
\mathcal{A}\left( W\right) \otimes _{\mathbf{Set}}\mathcal{B}\left( V\right) 
\overset{\alpha ,\beta }{\rightrightarrows }\dbigoplus\limits_{U\in \mathbf{C%
}_{X}}\mathcal{A}\left( U\right) \otimes _{\mathbf{Set}}\mathcal{B}\left(
U\right) \right)
\end{equation*}%
where $\alpha $ and $\beta $ are the compositions%
\begin{eqnarray*}
\alpha &=&\left( 
\begin{diagram}[size=3.0em,textflow]
\mathcal{A}\left( W\right) \otimes _{\mathbf{Set}}\mathcal{B}\left( V\right) & \rTo^{\mathcal{A}\left( V\rightarrow W\right) \otimes _{\mathbf{Set}}\mathbf{1}_{\mathcal{B}\left( V\right) }} & \mathcal{A}\left( V\right) \otimes _{\mathbf{Set}}\mathcal{B}\left( V\right) & \rTo^{\iota _{V}} & \dbigoplus\limits_{U\in \mathbf{C}_{X}}\mathcal{A}\left( U\right) \otimes _{\mathbf{Set}}\mathcal{B}\left( U\right) \\
\end{diagram}
%
\right) , \\
\beta &=&\left( 
\begin{diagram}[size=3.0em,textflow]
\mathcal{A}\left( W\right) \otimes _{\mathbf{Set}}\mathcal{B}\left( V\right) & \rTo^{\mathbf{1}_{\mathcal{A}\left( W\right) }\otimes _{\mathbf{Set}}\mathcal{A}\left( V\rightarrow W\right) } & \mathcal{A}\left( W\right) \otimes _{\mathbf{Set}}\mathcal{B}\left( W\right) & \rTo^{\iota _{W}} & \dbigoplus\limits_{U\in \mathbf{C}_{X}}\mathcal{A}\left( U\right) \otimes _{\mathbf{Set}}\mathcal{B}\left( U\right) \\
\end{diagram}
%
\right) ,
\end{eqnarray*}%
while $\iota _{V}$ and $\iota _{W}$ are the corresponding embeddings.
\end{definition}

\begin{proposition}
\label{Prop-Dual-to-tensor}For any injective $T\in \mathbf{Mod}\left(
k\right) $%
\begin{equation*}
\left\langle \mathcal{A\otimes }_{\mathbf{pS}\left( X,\mathbf{Set}\right) }%
\mathcal{B},T\right\rangle 
\simeq%
\Hom%
_{\mathbf{pS}\left( X,\mathbf{Set}\right) }\left( \mathcal{A},\left\langle 
\mathcal{B},T\right\rangle \right)
\end{equation*}%
naturally in $\mathcal{A}$, $\mathcal{B}$ and $T$.
\end{proposition}

\begin{proof}
Clearly%
\begin{equation*}
\Hom%
_{\mathbf{Pro}\left( k\right) }\left( X\otimes _{\mathbf{Set}}Y,T\right) 
\simeq%
\dprod\limits_{X}%
\Hom%
_{\mathbf{Set}}\left( X,%
\Hom%
_{\mathbf{Pro}\left( k\right) }\left( Y,T\right) \right) ,
\end{equation*}%
naturally in $X\in \mathbf{Set}$, $Y\in \mathbf{Pro}\left( k\right) $ and $%
T\in \mathbf{Mod}\left( k\right) $. Therefore%
\begin{eqnarray*}
&&\left\langle \mathcal{A\otimes }_{\mathbf{pS}\left( X,\mathbf{Set}\right) }%
\mathcal{B},T\right\rangle 
\simeq
\\
&&\ker \left( \dprod\limits_{U\in \mathbf{C}_{X}}%
\Hom%
_{\mathbf{Set}}\left( \mathcal{A}\left( U\right) ,\left\langle \mathcal{B}%
\left( U\right) ,T\right\rangle \right) \overset{\left\langle \alpha
,T\right\rangle ,\left\langle \beta ,T\right\rangle }{\rightrightarrows }%
\dprod\limits_{\substack{ V,W\in \mathbf{C}_{X}  \\ \left( V\rightarrow
W\right) \in \mathbf{C}_{X}\left( V\rightarrow W\right) }}%
\Hom%
_{\mathbf{Set}}\left( \mathcal{A}\left( W\right) ,\left\langle \mathcal{B}%
\left( V\right) ,T\right\rangle \right) \right) ,
\end{eqnarray*}%
naturally in $\mathcal{A}\in \mathbf{pS}\left( X,\mathbf{Set}\right) $, $%
\mathcal{B}\in \mathbf{pCS}\left( X,\mathbf{Pro}\left( k\right) \right) $
and $T\in \mathbf{Mod}\left( k\right) $. The latter set consists of families%
\begin{equation*}
\left( \varphi _{U}\in 
\Hom%
_{\mathbf{Set}}\left( \mathcal{A}\left( U\right) ,\left\langle \mathcal{B}%
\left( U\right) ,T\right\rangle \right) \right) ,
\end{equation*}%
such that the diagrams%
\begin{equation*}
\begin{diagram}[size=3.0em,textflow]
\mathcal{A}\left( W\right) & \rTo & \left\langle \mathcal{B} \left( W\right) ,T\right\rangle \\ 
\dTo &  & \dTo \\ 
\mathcal{A}\left( V\right) & \rTo & \left\langle \mathcal{B} \left( V\right) ,T\right\rangle \\
\end{diagram}%
\end{equation*}%
commute for all $\left( V\rightarrow W\right) \in 
\Hom%
_{\mathbf{C}_{X}}\left( V,W\right) $. It follows that such families $\left(
\varphi _{U}\right) _{U\in \mathbf{C}_{X}}$ are \textbf{exactly} the functor
morphisms $\mathcal{A}\rightarrow \left\langle \mathcal{B},T\right\rangle $,
i.e. the morphisms of presheaves.
\end{proof}

\begin{definition}
\label{Def-Semi-representable}A presheaf $\mathcal{A}\in \mathbf{pS}\left( X,%
\mathbf{Set}\right) $ is called semi-representable iff $\mathcal{A}$ is a
coproduct of representable presheaves:%
\begin{equation*}
\mathcal{A}%
\simeq%
\dcoprod\limits_{i\in I}h_{U_{i}}.
\end{equation*}%
See \cite[V.7.3.0]{SGA4-2-MR0354653}.
\end{definition}

\begin{proposition}
\label{Prop-Semi-representable}Let%
\begin{equation*}
\mathcal{A}%
\simeq%
\dcoprod\limits_{i\in I}h_{U_{i}}
\end{equation*}%
and $\mathcal{B}\in \mathbf{pCS}\left( X,\mathbf{Pro}\left( k\right) \right) 
$. Then%
\begin{equation*}
\mathcal{A\otimes }_{\mathbf{pS}\left( X,\mathbf{Set}\right) }\mathcal{B}%
\simeq%
\dbigoplus\limits_{i\in I}\mathcal{B}\left( U_{i}\right) .
\end{equation*}
\end{proposition}

\begin{proof}
Let $T\in \mathbf{Mod}\left( k\right) $ be an \textbf{arbitrary} injective
module. Define the following morphism%
\begin{equation*}
\dbigoplus\limits_{U\in \mathbf{C}_{X}}\mathcal{A}\left( U\right) \otimes _{%
\mathbf{Set}}\mathcal{B}\left( U\right) \longrightarrow
\dbigoplus\limits_{i\in I}\mathcal{B}\left( U_{i}\right) .
\end{equation*}%
For each summand let%
\begin{equation*}
\mathcal{A}\left( U\right) \otimes _{\mathbf{Set}}\mathcal{B}\left( U\right)
\longrightarrow \dbigoplus\limits_{i\in I}\mathcal{B}\left( U_{i}\right)
\end{equation*}%
be%
\begin{equation*}
\dbigoplus\limits_{\substack{ i\in I  \\ f\in 
\Hom%
_{\mathbf{C}_{X}\left( U,U_{i}\right) }}}\varphi
_{i,f}:\dbigoplus\limits_{i\in I}h_{U_{i}}\left( U\right) \otimes _{\mathbf{%
Set}}\mathcal{B}\left( U\right) =\dbigoplus\limits_{i\in
I}\dbigoplus\limits_{f\in 
\Hom%
_{\mathbf{C}_{X}\left( U,U_{i}\right) }}\mathcal{B}\left( U\right)
\longrightarrow \dbigoplus\limits_{i\in I}\mathcal{B}\left( U_{i}\right)
\end{equation*}
where $\varphi _{i,f}$ is the composition%
\begin{equation*}
\mathcal{B}\left( U\right) \overset{\mathcal{B}\left( f\right) }{%
\longrightarrow }\mathcal{B}\left( U_{i}\right) \longrightarrow
\dbigoplus\limits_{i\in I}\mathcal{B}\left( U_{i}\right) .
\end{equation*}%
Apply $\left\langle \bullet ,T\right\rangle $, and obtain Yoneda's
isomorphism%
\begin{eqnarray*}
&&\left\langle \mathcal{A\otimes }_{\mathbf{pS}\left( X,\mathbf{Set}\right) }%
\mathcal{B},T\right\rangle 
\simeq%
\Hom%
_{\mathbf{pS}\left( X,\mathbf{Set}\right) }\left( \mathcal{A},\left\langle 
\mathcal{B},T\right\rangle \right) 
\simeq
\\
&&%
\simeq%
\dprod\limits_{i\in I}%
\Hom%
_{\mathbf{pS}\left( X,\mathbf{Set}\right) }\left( h_{U_{i}},\left\langle 
\mathcal{B},T\right\rangle \right) 
\simeq%
\dprod\limits_{i\in I}\left\langle \mathcal{B},T\right\rangle \left(
U_{i}\right) 
\simeq%
\left\langle \dbigoplus\limits_{i\in I}\mathcal{B}\left( U_{i}\right)
,T\right\rangle .
\end{eqnarray*}%
It follows from Proposition \ref{Prop-Duality} that%
\begin{equation*}
\mathcal{A\otimes }_{\mathbf{pS}\left( X,\mathbf{Set}\right) }\mathcal{B}%
\simeq%
\dbigoplus\limits_{i\in I}\mathcal{B}\left( U_{i}\right) .
\end{equation*}
\end{proof}

\begin{definition}
(see \cite[Definition 14.12.1]{Stacks-Project-2024})

\begin{enumerate}
\item The category $\Delta $ is the category with objects%
\begin{equation*}
\left[ n\right] 
{:=}%
\left\{ 0<1<\dots <n\right\}
\end{equation*}%
and monotone mappings.

\item The category $\Delta _{\leq n}$ is the full subcategory of $\Delta $
with%
\begin{equation*}
Ob\left( \Delta _{\leq n}\right) =\left\{ \left[ 0\right] ,\left[ 1\right]
,\dots ,\left[ n\right] \right\} .
\end{equation*}

\item The category $Simp\left( \mathbf{K}\right) $ of \textbf{simplicial
objects} of $\mathbf{K}$ is the category of functors $A_{\bullet }:\Delta
^{op}\rightarrow \mathbf{K}$. We denote $A\left( \left[ n\right] \right) $
simply by $A_{n}$.

\item The category $Simp_{n}\left( \mathbf{K}\right) $ of \textbf{truncated
simplicial objects} of $\mathbf{K}$ is the category of functors $\left(
\Delta _{\leq n}\right) ^{op}\rightarrow \mathbf{K}$.

\item For a simplicial object%
\begin{equation*}
A:\Delta ^{op}\rightarrow \mathbf{K}
\end{equation*}%
let $sk_{n}A$ be its restriction to $\left( \Delta _{\leq n}\right) ^{op}$.
We have the \textbf{skeleton} functor%
\begin{equation*}
sk_{n}:Simp\left( \mathbf{K}\right) \longrightarrow Simp_{n}\left( \mathbf{K}%
\right) .
\end{equation*}

\item (see \cite[14.19]{Stacks-Project-2024}) The \textbf{coskeleton} functor%
\begin{equation*}
cosk_{n}:Simp_{n}\left( \mathbf{K}\right) \longrightarrow Simp\left( \mathbf{%
K}\right)
\end{equation*}%
(if exists) is the right adjoint to $sk_{n}$.
\end{enumerate}
\end{definition}

\begin{definition}
\label{Def-local-epi(mono,iso)morphism}A morphism of presheaves%
\begin{equation*}
f\in 
\Hom%
_{\mathbf{pS}\left( X,\mathbf{K}\right) }\left( \mathcal{A},\mathcal{B}%
\right)
\end{equation*}%
is called a \textbf{local} \textbf{epi}(respectively \textbf{mono}, \textbf{%
iso})morphism iff it satisfies condition 1 (respectively 2, 3) in \cite[p.
38-39]{Jardine-2015-Local-homotopy-theory-MR3309296}.
\end{definition}

\begin{lemma}
\label{Lemma-local-epi(mono,iso)morphism}$f\in 
\Hom%
_{\mathbf{pS}\left( X,\mathbf{K}\right) }\left( \mathcal{A},\mathcal{B}%
\right) $ is a local epi(respectively mono, iso)morphism iff the
sheafification%
\begin{equation*}
f^{\#}\in 
\Hom%
_{\mathbf{S}\left( X,\mathbf{K}\right) }\left( \mathcal{A}^{\#},\mathcal{B}%
^{\#}\right)
\end{equation*}%
is an epi(respectively mono, iso)morphism.
\end{lemma}

\begin{proof}
See \cite[Lemma 3.16]{Jardine-2015-Local-homotopy-theory-MR3309296}.
\end{proof}

\begin{definition}
\label{Def-Hypercovering}(see \cite[Definition 25.6.1]{Stacks-Project-2024}
and \cite[V.7.3.1(HR1 and HR3)]{SGA4-2-MR0354653}) A simplicial presheaf $%
\mathcal{K}_{\bullet }\in Simp\left( \mathbf{pS}\left( X,\mathbf{Set}\right)
\right) $ endowed with an augmentation $\mathcal{K}_{\bullet }\rightarrow
h_{U}$ is called a \textbf{hypercovering} of $U$ iff:

\begin{enumerate}
\item $\mathcal{K}_{0}\rightarrow h_{U}$ is a local epimorphism.

\item $\mathcal{K}_{1}\rightarrow \mathcal{K}_{0}\times _{h_{U}}\mathcal{K}%
_{0}$ is a local epimorphism.

\item $\mathcal{K}_{n+1}\rightarrow \left( cosk_{n}sk_{n}\mathcal{K}\right)
_{n+1}$ for $n\geq 1$ is a local epimorphism.
\end{enumerate}
\end{definition}

The simplest hypercovering of $U$ is the 
\u{C}ech
hypercovering.

\begin{example}
(see \cite[Example 25.3.6]{Stacks-Project-2024}) Let $\left(
U_{i}\rightarrow U\right) _{i\in I}$ be a covering. Define%
\begin{equation*}
\mathcal{K}_{n}%
{:=}%
\dcoprod\limits_{i\left( 0\right) ,i\left( 1\right) ,\dots ,i\left( n\right)
}h_{U_{i\left( 0\right) }\underset{U}{\times }U_{i\left( 1\right) }\underset{%
U}{\times }\dots \underset{U}{\times }U_{i\left( n\right) }}.
\end{equation*}%
Let%
\begin{equation*}
\varphi :\left[ m\right] \longrightarrow \left[ n\right]
\end{equation*}%
be a monotone mapping. To define a morphism from $V$ to $U_{i\left( 0\right)
}\underset{U}{\times }U_{i\left( 1\right) }\underset{U}{\times }\dots 
\underset{U}{\times }U_{i\left( m\right) }$, it is enough to define
morphisms $\psi _{k}$ from $V$ to each $U_{i\left( k\right) }$. Let%
\begin{equation*}
V=U_{j\left( 0\right) }\underset{U}{\times }U_{j\left( 1\right) }\underset{U}%
{\times }\dots \underset{U}{\times }U_{j\left( n\right) }.
\end{equation*}%
For each $\left( i\left( 0\right) ,i\left( 1\right) ,\dots ,i\left( m\right)
\right) $ with $j\left( \varphi \left( k\right) \right) =i\left( k\right) $
define $\psi _{k}$ as%
\begin{equation*}
\psi _{k}:U_{j\left( 0\right) }\underset{U}{\times }U_{j\left( 1\right) }%
\underset{U}{\times }\dots \underset{U}{\times }U_{j\left( n\right)
}\longrightarrow U_{j\left( \varphi \left( k\right) \right) }=U_{i\left(
k\right) }.
\end{equation*}%
This gives us%
\begin{equation*}
\varphi ^{\ast }:U_{j\left( 0\right) }\underset{U}{\times }U_{j\left(
1\right) }\underset{U}{\times }\dots \underset{U}{\times }U_{j\left(
n\right) }\longrightarrow U_{i\left( 0\right) }\underset{U}{\times }%
U_{i\left( 1\right) }\underset{U}{\times }\dots \underset{U}{\times }%
U_{i\left( m\right) },
\end{equation*}%
and, finally,%
\begin{equation*}
h_{\varphi ^{\ast }}:h_{U_{j\left( 0\right) }\underset{U}{\times }U_{j\left(
1\right) }\underset{U}{\times }\dots \underset{U}{\times }U_{j\left(
n\right) }}\longrightarrow h_{U_{i\left( 0\right) }\underset{U}{\times }%
U_{i\left( 1\right) }\underset{U}{\times }\dots \underset{U}{\times }%
U_{i\left( m\right) }}.
\end{equation*}%
Combining all those $h_{\varphi ^{\ast }}$, one obtains the desired morphism%
\begin{equation*}
\mathcal{K}_{\varphi }:\mathcal{K}_{n}\longrightarrow \mathcal{K}_{m},
\end{equation*}%
and the simplicial structure on $\mathcal{K}_{\bullet }$.
\end{example}

\begin{definition}
\label{Def-Cech-homology-hypercovering}Let $\mathcal{K}_{\bullet }\in
Simp\left( \mathbf{pS}\left( X,\mathbf{Set}\right) \right) $ be a
hypercovering, and $\mathcal{A}\in \mathbf{pCS}\left( X,\mathbf{Pro}\left(
k\right) \right) $ a precosheaf. The tensor product%
\begin{equation*}
\mathcal{K}_{\bullet }\otimes _{\mathbf{pS}\left( X,\mathbf{Set}\right) }%
\mathcal{A}\in Simp\left( \mathbf{Pro}\left( k\right) \right)
\end{equation*}%
has a structure of a simplicial pro-module. Let%
\begin{eqnarray*}
d_{s} &:&\mathcal{K}_{n}\otimes _{\mathbf{pS}\left( X,\mathbf{Set}\right) }%
\mathcal{A}\longrightarrow \mathcal{K}_{n-1}\otimes _{\mathbf{pS}\left( X,%
\mathbf{Set}\right) }\mathcal{A}, \\
0 &\leq &s\leq n,
\end{eqnarray*}%
be the face morphisms. Denote by%
\begin{equation*}
\check{C}_{\bullet }\left( \mathcal{K}_{\bullet },\mathcal{A}\right)
\end{equation*}%
the corresponding chain complex%
\begin{eqnarray*}
&&\check{C}_{n}\left( \mathcal{K}_{\bullet },\mathcal{A}\right) 
{:=}%
\mathcal{K}_{n}\otimes _{\mathbf{pS}\left( X,\mathbf{Set}\right) }\mathcal{A}%
, \\
d &=&\dsum\limits_{s=0}^{n}\left( -1\right) ^{s}d_{s}:\check{C}_{n}\left( 
\mathcal{K}_{\bullet },\mathcal{A}\right) \longrightarrow \check{C}%
_{n-1}\left( \mathcal{K}_{\bullet },\mathcal{A}\right) .
\end{eqnarray*}%
The homology of that chain complex will be called the 
\u{C}ech
homology:%
\begin{equation*}
\check{H}_{n}\left( \mathcal{K}_{\bullet },\mathcal{A}\right) 
{:=}%
H_{n}\left( \check{C}_{\bullet }\left( \mathcal{K}_{\bullet },\mathcal{A}%
\right) \right) .
\end{equation*}
\end{definition}

\begin{remark}
\label{Rem-Cech-homology-hypercovering}It is easy to check, using
Proposition \ref{Prop-Dual-to-tensor} and \ref{Prop-Semi-representable},
that, for any injective module $T\in \mathbf{Mod}\left( k\right) $.%
\begin{equation*}
\left\langle \check{H}_{n}\left( \mathcal{K}_{\bullet },\mathcal{A}\right)
,T\right\rangle 
\simeq%
\check{H}^{n}\left( \mathcal{K}_{\bullet },\left\langle \mathcal{A}%
,T\right\rangle \right)
\end{equation*}%
where $\check{H}^{n}$ is the 
\u{C}ech
\textbf{co}homology from \cite[25.5]{Stacks-Project-2024}.
\end{remark}

\begin{lemma}
\label{Lemma-Acyclicity}Let $\mathcal{K}_{\bullet }$ be a hypercovering of $%
U\in \mathbf{C}_{X}$, and $\mathcal{P}\in \mathbf{CS}\left( X,\mathbf{Pro}%
\left( k\right) \right) $ be a quasi-projective cosheaf. Then%
\begin{equation*}
\check{H}_{n}\left( \mathcal{K}_{\bullet },\mathcal{P}\right) =\left\{ 
\begin{array}{ccc}
\mathcal{P}\left( U\right) & \text{if} & n=0 \\ 
0 & \text{if} & n>0%
\end{array}%
\right.
\end{equation*}
\end{lemma}

\begin{proof}
Let $T\in \mathbf{Mod}\left( k\right) $ be an arbitrary injective module.
Due to \cite[Lemma 25.5.2]{Stacks-Project-2024},%
\begin{equation*}
\left\langle \check{H}_{n}\left( \mathcal{K}_{\bullet },\mathcal{P}\right)
,T\right\rangle =\check{H}^{n}\left( \mathcal{K}_{\bullet },\left\langle 
\mathcal{P},T\right\rangle \right) =\left\{ 
\begin{array}{ccc}
\left\langle \mathcal{P},T\right\rangle \left( U\right) =\left\langle 
\mathcal{P}\left( U\right) ,T\right\rangle & \text{if} & n=0 \\ 
0 & \text{if} & n>0%
\end{array}%
\right.
\end{equation*}%
In fact, the cited lemma deals with a sheaf $\mathcal{I}\in \mathbf{S}\left(
X,\mathbf{Ab}\right) $ of abelian groups (not $k$-modules). However, the
proof may be easily applied to $k$-modules. Just replace the constant
sheaves $\left( \mathbb{Z}_{U_{i}}\right) ^{\#}$ with the constant sheaves $%
\left( k_{U_{i}}\right) ^{\#}$. It follows that $\check{H}_{n}\left( 
\mathcal{K}_{\bullet },\mathcal{P}\right) =0$ if $n>0$. Moreover, the
augmentation $\mathcal{K}_{0}\rightarrow h_{U}$ determines a morphism%
\begin{equation*}
\check{H}_{0}\left( \mathcal{K}_{\bullet },\mathcal{P}\right) =%
\coker%
\left( \check{C}_{1}\left( \mathcal{K}_{\bullet },\mathcal{A}\right)
\rightarrow \check{C}_{0}\left( \mathcal{K}_{\bullet },\mathcal{A}\right)
\right) \longrightarrow h_{U}\left( \mathcal{P}\right) =\mathcal{P}\left(
U\right) .
\end{equation*}%
Since, due to \cite[Lemma 25.5.2]{Stacks-Project-2024},%
\begin{equation*}
\left\langle \mathcal{P}\left( U\right) ,T\right\rangle \longrightarrow
\left\langle \check{H}_{0}\left( \mathcal{K}_{\bullet },\mathcal{P}\right)
,T\right\rangle
\end{equation*}%
is an isomorphism,%
\begin{equation*}
\check{H}_{0}\left( \mathcal{K}_{\bullet },\mathcal{P}\right)
\longrightarrow \mathcal{P}\left( U\right)
\end{equation*}%
is an isomorphism as well.
\end{proof}

\begin{lemma}
\label{Lemma-Spectral-sequence-hypercovering}~

\begin{enumerate}
\item \label{Lemma-Spectral-sequence-hypercovering-morphism}Let $\mathcal{K}%
_{\bullet }$ be a hypercovering of $U\in \mathbf{C}_{X}$. There is a morphism%
\begin{equation*}
H_{\bullet }\left( U,\mathcal{A}\right) \longrightarrow \check{H}_{\bullet
}\left( \mathcal{K}_{\bullet },\mathcal{A}\right) ,
\end{equation*}%
natural in $U$, $\mathcal{A}\in \mathbf{CS}\left( X,\mathbf{Pro}\left(
k\right) \right) $ and $\mathcal{K}_{\bullet }$.

\item Moreover, there exists a natural spectral sequence%
\begin{equation*}
E_{s,t}^{2}=H_{s}\left( \mathcal{K}_{\bullet },\mathcal{H}_{t}\mathcal{A}%
\right) \implies H_{s+t}\left( U,\mathcal{A}\right)
\end{equation*}%
where $\mathcal{H}_{t}$ are left satellites $L_{s}\iota $ of the full
embedding of $\left\{ \text{cosheaves}\right\} $ into $\left\{ \text{\textbf{%
pre}cosheaves}\right\} $:%
\begin{equation*}
\iota :\mathbf{CS}\left( X,\mathbf{Pro}\left( k\right) \right)
\longrightarrow \mathbf{pCS}\left( X,\mathbf{Pro}\left( k\right) \right) .
\end{equation*}
\end{enumerate}
\end{lemma}

\begin{proof}
Compare to \cite[Lemma 25.5.3]{Stacks-Project-2024}. Let%
\begin{equation*}
0\longleftarrow \mathcal{A}\longleftarrow \mathcal{P}_{0}\longleftarrow 
\mathcal{P}_{1}\longleftarrow \mathcal{\dots }\longleftarrow \mathcal{P}%
_{n}\longleftarrow \mathcal{\dots }
\end{equation*}%
be a quasi-projective resolution. Consider the following bicomplex $\left(
B_{\bullet \bullet },d_{\bullet },\delta _{\bullet }\right) $%
\begin{equation*}
B_{s,t}=\check{C}_{s}\left( \mathcal{K}_{\bullet },\mathcal{P}_{t}\right)
\end{equation*}%
where $d_{\bullet }$ is inherited from $\check{C}_{\bullet }$ and $\delta
_{\bullet }$ is inherited from $\mathcal{P}_{\bullet }$. Consider also the
corresponding total complex $Tot_{\bullet }\left( B_{\bullet \bullet
}\right) $:%
\begin{equation*}
Tot_{n}\left( B_{\bullet \bullet }\right) =\dbigoplus\limits_{s+t=n}B_{s,t}%
\simeq%
\dprod\limits_{s+t=n}B_{s,t}
\end{equation*}%
with the differential%
\begin{equation*}
\partial _{n}:Tot_{n+1}\left( B_{\bullet \bullet }\right) \longrightarrow
Tot_{n}\left( B_{\bullet \bullet }\right) ,
\end{equation*}%
given by%
\begin{equation*}
\partial _{n}\circ \iota _{s,t}=\iota _{s-1,t}\circ d+\left( -1\right)
^{s}\iota _{s,t-1}\circ \delta ,
\end{equation*}%
where%
\begin{equation*}
\iota _{s,t}:B_{s,t}\rightarrowtail Tot_{n}\left( B_{\bullet \bullet }\right)
\end{equation*}%
is the natural embedding into the coproduct.

\begin{enumerate}
\item We will prove below that%
\begin{equation*}
H_{n}\left( Tot_{\bullet }\left( B_{\bullet \bullet }\right) \right) 
\simeq%
H_{n}\left( U,\mathcal{A}\right) .
\end{equation*}%
Consider the compositions%
\begin{equation*}
\varphi _{s}:Tot_{s}\left( B_{\bullet \bullet }\right) \longrightarrow
B_{s,0}\longrightarrow \check{C}_{s}\left( \mathcal{K}_{\bullet },\mathcal{P}%
_{0}\right) \longrightarrow \check{C}_{s}\left( \mathcal{K}_{\bullet },%
\mathcal{A}\right) .
\end{equation*}%
Clearly $\varphi _{s}\circ \partial _{n}=d_{s-1}\circ \varphi _{s-1}$. This
gives us the desired natural morphisms%
\begin{equation*}
H_{n}\left( U,\mathcal{A}\right) 
\simeq%
H_{n}\left( Tot_{\bullet }\left( B_{\bullet \bullet }\right) \right)
\longrightarrow H_{n}\left( \check{C}_{\bullet }\left( \mathcal{K}_{\bullet
},\mathcal{A}\right) \right) =\check{H}_{n}\left( \mathcal{K}_{\bullet },%
\mathcal{A}\right) .
\end{equation*}

\item Consider the two spectral sequences from \cite[Theorem A.4.3]%
{Prasolov-Cosheaves-2021-MR4347662}.%
\begin{eqnarray*}
~^{ver}E_{s,t}^{2} &=&~^{hor}H_{s}\left( ^{ver}H_{t}B_{\bullet \bullet
}\right) \implies H_{s+t}\left( Tot\left( B_{\bullet \bullet }\right)
\right) , \\
~^{hor}E_{s,t}^{2} &=&~^{ver}H_{t}\left( ^{hor}H_{s}B_{\bullet \bullet
}\right) \implies H_{s+t}\left( Tot\left( B_{\bullet \bullet }\right)
\right) .
\end{eqnarray*}%
It follows from Lemma \ref{Lemma-Acyclicity} that%
\begin{equation*}
^{hor}H_{s}B_{\bullet ,t}=\check{H}_{s}\left( \mathcal{K}_{\bullet },%
\mathcal{P}_{t}\right) =\left\{ 
\begin{array}{ccc}
\mathcal{P}_{t}\left( U\right) =\Gamma \left( U,\mathcal{P}_{t}\right) & 
\text{if} & s=0 \\ 
0 & \text{if} & s>0%
\end{array}%
\right.
\end{equation*}%
and%
\begin{equation*}
~^{hor}E_{s,t}^{2}=\left\{ 
\begin{array}{ccc}
H_{t}\left( U,\mathcal{A}\right) & \text{if} & s=0 \\ 
0 & \text{if} & s>0%
\end{array}%
\right.
\end{equation*}%
The spectral sequence degenerates from $E^{2}$ on, and%
\begin{equation*}
H_{n}\left( Tot\left( B_{\bullet \bullet }\right) \right) 
\simeq%
H_{t}\left( U,\mathcal{A}\right) .
\end{equation*}%
On the other hand, due to Proposition \ref{Prop-Semi-representable},%
\begin{equation*}
B_{s,t}=\dbigoplus\limits_{i\in I_{s}}\mathcal{P}_{t}\left( U_{s,i}\right)
\end{equation*}%
where%
\begin{equation*}
\mathcal{K}_{s}=\dcoprod\limits_{i\in I_{s}}h_{U_{s,i}}.
\end{equation*}%
Consider now the cosheaves $\mathcal{P}_{t}$ as \textbf{pre}cosheaves. It
follows that%
\begin{equation*}
^{ver}H_{t}B_{\bullet \bullet }=\dbigoplus\limits_{i\in I_{s}}\mathcal{H}_{t}%
\mathcal{A}\left( U_{s,i}\right) =\mathcal{K}_{s}\otimes _{\mathbf{pS}\left(
X,\mathbf{Set}\right) }\mathcal{H}_{t}\mathcal{A}
\end{equation*}%
and%
\begin{equation*}
~^{ver}E_{s,t}^{2}=\check{H}_{s}\left( \mathcal{K}_{\bullet },\mathcal{H}_{t}%
\mathcal{A}\right) \implies H_{s+t}\left( Tot\left( B_{\bullet \bullet
}\right) \right) =H_{s+t}\left( U,\mathcal{A}\right) .
\end{equation*}
\end{enumerate}
\end{proof}

\begin{notation}
\label{Not-Hypercoverings-up-to-homotopy}Denote by $HR\left( U\right) $ the
category of hypercoverings of $U\in \mathbf{C}_{X}$, and by $\mathcal{HR}%
\left( U\right) $ the category of hypercoverings of $U$ up to simplicial
homotopy.
\end{notation}

\begin{remark}
\label{Rem-Hypercoverings-up-to-homotopy}(see \cite[Lemmas 25.3.7, 25.7.2
and 25.9.2]{Stacks-Project-2024})

\begin{enumerate}
\item $\mathcal{HR}\left( U\right) $ is a \textbf{small} category.

\item $\mathcal{HR}\left( U\right) $ is \textbf{co}filtered.
\end{enumerate}
\end{remark}

\begin{theorem}
\label{Th-Hypercoverings-vs-resolutions}Let $\mathcal{A}\in \mathbf{CS}%
\left( X,\mathbf{Pro}\left( k\right) \right) $.

\begin{enumerate}
\item If the two morphisms%
\begin{equation*}
f,g:\mathcal{K}_{\bullet }\longrightarrow \mathcal{L}_{\bullet }
\end{equation*}%
are simplicially homotopic, then%
\begin{equation*}
\check{H}_{\bullet }\left( f,\mathcal{A}\right) =\check{H}_{\bullet }\left(
f,\mathcal{A}\right) :\check{H}_{\bullet }\left( \mathcal{K}_{\bullet },%
\mathcal{A}\right) \longrightarrow \check{H}_{\bullet }\left( \mathcal{L}%
_{\bullet },\mathcal{A}\right) .
\end{equation*}

\item Let%
\begin{equation*}
\check{H}_{\bullet }^{\mathcal{HR}}\left( U,\mathcal{A}\right) 
{:=}%
\underset{\mathcal{K}_{\bullet }\in \mathcal{HR}\left( U\right) }{%
\underleftarrow{\lim }}\check{H}_{\bullet }\left( \mathcal{K}_{\bullet },%
\mathcal{A}\right) .
\end{equation*}%
Then%
\begin{equation*}
\check{H}_{\bullet }^{\mathcal{HR}}\left( U,\mathcal{A}\right) 
\simeq%
H_{\bullet }\left( U,\mathcal{A}\right) ,
\end{equation*}%
naturally in $U$ and $\mathcal{A}$.
\end{enumerate}
\end{theorem}

\begin{proof}
~

\begin{enumerate}
\item See \cite[25.10]{Stacks-Project-2024}.

\item Due to Lemma \ref{Lemma-Spectral-sequence-hypercovering}(\ref%
{Lemma-Spectral-sequence-hypercovering-morphism}) there exist morphisms%
\begin{equation*}
\varphi _{\mathcal{K}_{\bullet }}:H_{\bullet }\left( U,\mathcal{A}\right)
\rightarrow \check{H}_{\bullet }\left( \mathcal{K}_{\bullet },\mathcal{A}%
\right)
\end{equation*}%
which define a morphism%
\begin{equation*}
\varphi :H_{\bullet }\left( U,\mathcal{A}\right) \rightarrow \underset{%
\mathcal{K}_{\bullet }\in \mathcal{HR}\left( U\right) }{\underleftarrow{\lim 
}}\check{H}_{\bullet }\left( \mathcal{K}_{\bullet },\mathcal{A}\right) =%
\check{H}_{\bullet }^{\mathcal{HR}}\left( U,\mathcal{A}\right) .
\end{equation*}%
Let $T\in \mathbf{Mod}\left( k\right) $ be an \textbf{arbitrary} injective
module. Apply $\left\langle \bullet ,T\right\rangle $ to the above morphism:%
\begin{eqnarray*}
\left\langle \varphi ,T\right\rangle &:&\left\langle \check{H}_{\bullet }^{%
\mathcal{HR}}\left( U,\mathcal{A}\right) ,T\right\rangle =\left\langle 
\underset{\mathcal{K}_{\bullet }\in \mathcal{HR}\left( U\right) }{%
\underleftarrow{\lim }}\check{H}_{\bullet }\left( \mathcal{K}_{\bullet },%
\mathcal{A}\right) ,T\right\rangle = \\
&=&\underset{\mathcal{K}_{\bullet }\in \mathcal{HR}\left( U\right) }{%
\underrightarrow{\lim }}\left\langle \check{H}_{\bullet }\left( \mathcal{K}%
_{\bullet },\mathcal{A}\right) ,T\right\rangle =\underset{\mathcal{K}%
_{\bullet }\in \mathcal{HR}\left( U\right) }{\underrightarrow{\lim }}\check{H%
}_{\bullet }\left( \mathcal{K}_{\bullet },\left\langle \mathcal{A}%
,T\right\rangle \right) \\
&\longrightarrow &\left\langle H_{\bullet }\left( U,\mathcal{A}\right)
,T\right\rangle =H^{\bullet }\left( U,\left\langle \mathcal{A}%
,T\right\rangle \right) .
\end{eqnarray*}%
The latter morphism is an \textbf{iso}morphism, due to \cite[Theorem 25.10.1]%
{Stacks-Project-2024}. Proposition \ref{Prop-Duality} implies that $\varphi $
is an isomorphism as well.
\end{enumerate}
\end{proof}

\section{\label{Sec-topological-spaces}(Pre)cosheaves on topological spaces}

\subsection{General properties}

\begin{notation}
\label{Not-Topological-spaces-as-sites}~

\begin{enumerate}
\item For a topological space $X$ denote by the same letter $X$ the
following site $X=\left( \mathbf{C}_{X},Cov\left( X\right) \right) $ where $%
\mathbf{C}_{X}=Open\left( X\right) $ is the poset of open subsets of $X$ (a
poset could always be considered as a category).

\item The topology $Cov\left( X\right) $ is generated by the following
pre-topology: $\left\{ U_{i}\rightarrow U\right\} _{i\in I}$ is a covering
iff%
\begin{equation*}
\dbigcup\limits_{i\in I}U_{i}=U.
\end{equation*}

\item The corresponding sieve $R\subseteq h_{U}$ is the following: for $%
V\subseteq U$%
\begin{equation*}
R\left( V\right) =\left\{ 
\begin{array}{ccc}
\text{One-point set }\left\{ \left( V,U\right) \right\} & \text{if} & \left(
\exists i\in I\right) \left[ V\subseteq U_{i}\right] \\ 
\varnothing & \text{otherwise} & 
\end{array}%
\right. .
\end{equation*}

\item For a continuous map $f:X\rightarrow Y$ denote by the same letter $f$
the morphism of sites defined via the functor%
\begin{eqnarray*}
F &:&\mathbf{C}_{Y}\longrightarrow \mathbf{C}_{X}, \\
&&F\left( U\right) 
{:=}%
f^{-1}\left( U\right) .
\end{eqnarray*}
\end{enumerate}
\end{notation}

\begin{remark}
\label{Rem-Topological-spaces-as-sites}For a continuous map $f:X\rightarrow
Y $:

\begin{enumerate}
\item The corresponding functors are:

\begin{enumerate}
\item 
\begin{equation*}
f_{\ast }:\mathbf{pCS}\left( X,\mathbf{Pro}\left( k\right) \right)
\longrightarrow \mathbf{pCS}\left( Y,\mathbf{Pro}\left( k\right) \right)
\end{equation*}%
and%
\begin{eqnarray*}
f_{\ast } &:&\mathbf{CS}\left( X,\mathbf{Pro}\left( k\right) \right)
\longrightarrow \mathbf{CS}\left( Y,\mathbf{Pro}\left( k\right) \right) : \\
&&f_{\ast }\mathcal{A}\left( U\right) 
{:=}%
\mathcal{A}\left( f^{-1}U\right) .
\end{eqnarray*}

\item 
\begin{eqnarray*}
f^{\ddag } &:&\mathbf{pCS}\left( Y,\mathbf{Pro}\left( k\right) \right)
\longrightarrow \mathbf{pCS}\left( X,\mathbf{Pro}\left( k\right) \right) : \\
f^{\ddag }\mathcal{A}\left( U\right) &=&\underset{U\subseteq f^{-1}\left(
V\right) }{\underleftarrow{\lim }}~\mathcal{A}\left( V\right) =\underset{%
f\left( U\right) \subseteq V}{\underleftarrow{\lim }}~\mathcal{A}\left(
V\right)
\end{eqnarray*}%
Clearly, $f_{\ast }\dashv f^{\ddag }$, since $f^{\ddag }$ is the right Kan
extension of the functor%
\begin{equation*}
F=f^{-1}:Open\left( Y\right) \longrightarrow Open\left( X\right) .
\end{equation*}

\item 
\begin{eqnarray*}
f\ast &:&\mathbf{CS}\left( Y,\mathbf{Pro}\left( k\right) \right)
\longrightarrow \mathbf{CS}\left( X,\mathbf{Pro}\left( k\right) \right) : \\
&&f^{\ast }\mathcal{A}%
{:=}%
\left( f^{\ddag }\mathcal{A}\right) _{\#}.
\end{eqnarray*}%
Clearly, $f_{\ast }\dashv f^{\ast }$ ($f^{\ast }$ is right adjoint to $%
f_{\ast }$), see Proposition \ref{Prop-Site-morphism-cosheaves}(\ref%
{Prop-Site-morphism-cosheaves-right-adjoint}).
\end{enumerate}

\item $f^{\ast }$ is exact, see Proposition \ref%
{Prop-Site-morphism-cosheaves}(\ref{Prop-Site-morphism-cosheaves-exact}).

\item $f_{\ast }$ sends quasi-projective sheaves on $X$ to quasi-projective
sheaves on $Y$, see Proposition \ref{Prop-Site-morphism-cosheaves}(\ref%
{Prop-Site-morphism-cosheaves-quasi-projective}).
\end{enumerate}
\end{remark}

\subsection{(Co)stalks}

Let $X$ be a topological space, and $x\in X$. Let $\mathbf{D}$ and $\mathbf{E%
}$ be categories. Assume that $\mathbf{D}$ is \textbf{co}complete and admits 
\textbf{co}filtered limits $\underleftarrow{\lim }$, while $\mathbf{E}$ is
complete and admits filtered \textbf{co}limits $\underrightarrow{\lim }$.

\begin{definition}
\label{Def-(Co)stalk}Let $\mathcal{A}$ be a precosheaf, $\mathcal{A}\in 
\mathbf{pCS}\left( X,\mathbf{D}\right) $ and $\mathcal{B}$ be a presheaf, $%
\mathcal{B}\in \mathbf{pS}\left( X,\mathbf{E}\right) $. The \textbf{costalk} 
$\mathcal{A}^{x}$ of $\mathcal{A}$ at $x$\ is%
\begin{equation*}
\mathcal{A}^{x}%
{:=}%
\underset{x\in U\in \mathbf{C}_{X}=Open\left( X\right) }{\underleftarrow{%
\lim }}~\mathcal{A}\left( U\right) .
\end{equation*}%
The \textbf{stalk} $\mathcal{B}_{x}$ of $\mathcal{B}$ at $x$\ is%
\begin{equation*}
\mathcal{B}_{x}%
{:=}%
\underset{x\in U\in \mathbf{C}_{X}=Open\left( X\right) }{\underrightarrow{%
\lim }}~\mathcal{B}\left( U\right) .
\end{equation*}
\end{definition}

\begin{proposition}
\label{Prop-(Co)stalks}Assume that a reflection (Notation \ref%
{Not-Categories}(\ref{Def-(co)reflective}))%
\begin{equation*}
\left( \bullet \right) ^{\#}:\mathbf{pS}\left( X,\mathbf{E}\right)
\longrightarrow \mathbf{S}\left( X,\mathbf{E}\right)
\end{equation*}%
and a coreflection (Notation \ref{Not-Categories}(\ref{Def-(co)reflective})%
\begin{equation*}
\left( \bullet \right) _{\#}:\mathbf{pCS}\left( X,\mathbf{D}\right)
\longrightarrow \mathbf{CS}\left( X,\mathbf{D}\right)
\end{equation*}%
exist, i.e., $\left( \bullet \right) ^{\#}\dashv \iota $ and $\iota \dashv
\left( \bullet \right) _{\#}$ where $\iota $ are the canonical embeddings of
(co)sheaves into pre(co)sheaves. Then, for any precosheaf $\mathcal{A}\in 
\mathbf{pCS}\left( X,\mathbf{D}\right) $ and any presheaf $\mathcal{B}\in 
\mathbf{pS}\left( X,\mathbf{E}\right) $ the canonical morphisms%
\begin{eqnarray*}
\left( \mathcal{A}_{\#}\right) ^{x} &\longrightarrow &\mathcal{A}^{x}, \\
\mathcal{B}_{x} &\longrightarrow &\left( \mathcal{B}^{\#}\right) _{x}
\end{eqnarray*}%
are isomorphisms.
\end{proposition}

\begin{proof}
See \cite[Proposition 2.21]{Prasolov-Cosheafification-2016-MR3660525} for pre%
\textbf{co}sheaves. For presheaves, the proof is analogous (and in fact
well-known).
\end{proof}

\begin{proposition}
\label{Prop-(Co)stalks-pro-modules}Let $\mathcal{C}$ be a precosheaf, $%
\mathcal{C}\in \mathbf{pCS}\left( X,\mathbf{Pro}\left( k\right) \right) $,
and $\mathcal{A}$ and $\mathcal{B}$ be cosheaves, $\mathcal{A},\mathcal{B}%
\in \mathbf{CS}\left( X,\mathbf{Pro}\left( k\right) \right) $. Let also $%
T\in \mathbf{Mod}\left( k\right) $ be an \textbf{arbitrary} injective $k$%
-module.

\begin{enumerate}
\item $\left\langle \mathcal{C}^{x},T\right\rangle 
\simeq%
\left\langle \mathcal{C},T\right\rangle _{x}$.

\item $\mathcal{C}_{+}\rightarrow \mathcal{C}$ induces an isomorphism $%
\left( \mathcal{C}_{+}\right) ^{x}\rightarrow \mathcal{C}^{x}$.

\item \label{Prop-(Co)stalks-pro-modules-morphisms}$f:\mathcal{A}\rightarrow 
\mathcal{B}$ is a monomorphism (epimorphism, isomorphism) iff%
\begin{equation*}
f^{x}:\mathcal{A}^{x}\longrightarrow \mathcal{B}^{x}
\end{equation*}%
is a monomorphism (epimorphism, isomorphism) for all $x\in X$.
\end{enumerate}
\end{proposition}

\begin{proof}
~

\begin{enumerate}
\item Follows from Proposition \ref{Prop-Duality}.

\item 
\begin{equation*}
\left\langle \mathcal{C}^{x},T\right\rangle 
\simeq%
\left\langle \mathcal{C},T\right\rangle _{x}\longrightarrow \left\langle 
\mathcal{C}_{+},T\right\rangle _{x}%
\simeq%
\left\langle \left( \mathcal{C}_{+}\right) ^{x},T\right\rangle
\end{equation*}%
is an isomorphism (well-known fact for presheaves), therefore $\left( 
\mathcal{C}_{+}\right) ^{x}\rightarrow \mathcal{C}^{x}$ is an isomorphism,
too.

\item 
\begin{equation*}
\left\langle f,T\right\rangle :\left\langle \mathcal{B},T\right\rangle
\longrightarrow \left\langle \mathcal{A},T\right\rangle
\end{equation*}%
is an epimorphism (monomorphism, isomorphism) of sheaves iff 
\begin{equation*}
\left[ \left\langle f,T\right\rangle _{x}=\left\langle f^{x},T\right\rangle %
\right] :\left[ \left\langle \mathcal{B},T\right\rangle _{x}%
\simeq%
\left\langle \mathcal{B}^{x},T\right\rangle \right] \longrightarrow \left[
\left\langle \mathcal{A},T\right\rangle _{x}%
\simeq%
\left\langle \mathcal{A}^{x},T\right\rangle \right]
\end{equation*}%
is an epimorphism (monomorphism, isomorphism) of modules for all $x\in X$
(well-known fact for sheaves). The result then follows from Proposition \ref%
{Prop-Duality}.
\end{enumerate}
\end{proof}

For topological spaces, the result of \cite[Theorem 3.4.1(1)]%
{Prasolov-Cosheaves-2021-MR4347662} can be significantly simplified:

\begin{proposition}
\label{Prop-Enough-quasi-projective}For an arbitrary cosheaf $\mathcal{A}\in 
\mathbf{CS}\left( X,\mathbf{Pro}\left( k\right) \right) $, there exists a
functorial epimorphism%
\begin{equation*}
\sigma \left( \mathcal{A}\right) :\mathcal{R}\left( \mathcal{A}\right)
\twoheadrightarrow \mathcal{A},
\end{equation*}%
where $\mathcal{R}\left( \mathcal{A}\right) $ is quasi-projective.
\end{proposition}

\begin{proof}
Consider the space $X^{\delta }$ which is $X$ with the discrete topology. It
is easy to check:

\begin{enumerate}
\item ~

\begin{enumerate}
\item Cosheaves $\mathcal{A}\in \mathbf{CS}\left( X^{\delta },\mathbf{Pro}%
\left( k\right) \right) $ on $X^{\delta }$ are completely determined by
families%
\begin{equation*}
\left( \mathcal{A}^{x}\right) _{x\in X}
\end{equation*}%
of their costalks. Indeed, for any such cosheaf, and any $U\in Open\left(
X\right) $%
\begin{equation*}
\mathcal{A}\left( U\right) 
\simeq%
\dcoprod\limits_{x\in U}\mathcal{A}^{x}.
\end{equation*}

\item Sheaves $\mathcal{B}\in \mathbf{S}\left( X^{\delta },\mathbf{Mod}%
\left( k\right) \right) $ on $X^{\delta }$ are completely determined by
families%
\begin{equation*}
\left( \mathcal{B}_{x}\right) _{x\in X}
\end{equation*}%
of their stalks. Indeed, for any such sheaf, and any $U\in Open\left(
X\right) $%
\begin{equation*}
\mathcal{B}\left( U\right) 
\simeq%
\dprod_{x\in U}\mathcal{B}_{x}.
\end{equation*}
\end{enumerate}

\item ~

\begin{enumerate}
\item A cosheaf $\mathcal{A}\in \mathbf{CS}\left( X^{\delta },\mathbf{Pro}%
\left( k\right) \right) $ is quasi-projective iff all the costalks $\mathcal{%
A}^{x}\in \mathbf{Pro}\left( k\right) $ are quasi-projective pro-modules.

\item A sheaf $\mathcal{B}\in \mathbf{S}\left( X^{\delta },\mathbf{Mod}%
\left( k\right) \right) $ is injective iff all the stalks $\mathcal{B}%
_{x}\in \mathbf{Mod}\left( k\right) $ are injective modules.
\end{enumerate}

\item For an arbitrary pro-module $\mathbf{M}\in \mathbf{Pro}\left( k\right) 
$, there exists a functorial epimorphism%
\begin{equation*}
\mathbf{F}\left( \mathbf{M}\right) \twoheadrightarrow \mathbf{M},
\end{equation*}%
where $\mathbf{F}\left( \mathbf{M}\right) $ is quasi-projective, see \cite[%
Proposition A.2.8(5)]{Prasolov-Cosheaves-2021-MR4347662}.
\end{enumerate}

Consider now the family $\left( \mathbf{F}\left( \mathcal{A}^{x}\right)
\right) _{x\in X}$ of quasi-projective pro-modules. The family determines
the quasi-projective cosheaf $\mathcal{P}_{\mathcal{A}}$ on $X^{\delta }$:%
\begin{equation*}
\mathcal{P}_{\mathcal{A}}\left( U\right) =\dcoprod\limits_{x\in U}\left( 
\mathbf{F}\left( \mathcal{A}^{x}\right) \right) .
\end{equation*}%
Let $f:X^{\delta }\rightarrow X$ be the canonical continuous mapping. Due to
Proposition \ref{Prop-Site-morphism-cosheaves}(\ref%
{Prop-Site-morphism-cosheaves-quasi-projective}),%
\begin{equation*}
\mathcal{R}\left( \mathcal{A}\right) 
{:=}%
f_{\ast }\left( \mathcal{P}_{\mathcal{A}}\right) \in \mathbf{CS}\left( X,%
\mathbf{Pro}\left( k\right) \right)
\end{equation*}%
is a quasi-projective cosheaf. More precisely,%
\begin{equation*}
\mathcal{R}\left( \mathcal{A}\right) \left( U\right) =\mathcal{P}_{\mathcal{A%
}}\left( U\right) =\dcoprod\limits_{x\in U}\left( \mathbf{F}\left( \mathcal{A%
}^{x}\right) \right) .
\end{equation*}%
There are canonical morphisms%
\begin{equation*}
\sigma \left( \mathcal{A}\right) \left( U\right) :\mathcal{R}\left( \mathcal{%
A}\right) \left( U\right) =\dcoprod\limits_{x\in U}\left( \mathbf{F}\left( 
\mathcal{A}^{x}\right) \right) \longrightarrow \mathcal{A}\left( U\right) ,
\end{equation*}%
defined on coproduct components by the following compositions:%
\begin{equation*}
\mathbf{F}\left( \mathcal{A}^{x}\right) \twoheadrightarrow \mathcal{A}^{x}%
\overset{\varphi _{U}}{\longrightarrow }\mathcal{A}\left( U\right)
\end{equation*}%
where $\varphi _{U}$ are he canonical compositions%
\begin{equation*}
\varphi _{U}:\mathcal{A}^{x}=\underset{x\in V\subseteq U}{\underleftarrow{%
\lim }}~\mathcal{A}\left( V\right) \overset{\pi _{V}}{\longrightarrow }%
\mathcal{A}\left( V\right) \overset{\mathcal{A}\left( V\rightarrow U\right) }%
{\longrightarrow }\mathcal{A}\left( U\right)
\end{equation*}%
where $\pi _{V}$ are the canonical projections. Clearly, the morphisms%
\begin{equation*}
\sigma \left( \mathcal{A}\right) ^{x}:\mathcal{R}\left( \mathcal{A}\right)
^{x}=\mathbf{F}\left( \mathcal{A}^{x}\right) \longrightarrow \mathcal{A}^{x}
\end{equation*}%
are epimorphisms, therefore $\sigma \left( \mathcal{A}\right) $ is an
epimorphism, due to Proposition \ref{Prop-(Co)stalks-pro-modules}(\ref%
{Prop-(Co)stalks-pro-modules-morphisms}).
\end{proof}

\subsection{Pro-homotopy and pro-homology}

One of the most important tools in strong shape theory is a \emph{strong
expansion} (see \cite{Mardesic-Strong-shape-and-homology-MR1740831},
conditions (S1) and (S2) on p. 129). In this paper, it is sufficient to use
a weaker notion: an $H\left( \mathbf{Top}\right) $-\emph{expansion} (\cite[%
\S I.4.1]{Mardesic-Segal-MR676973}, conditions (E1) and (E2)). Those two
conditions are equivalent to the following

\begin{definition}
\label{HTOP-extension}\label{Def-HTOP-extension}Let $X$ be a topological
space. A morphism $X\rightarrow \left( Y_{j}\right) _{j\in \mathbf{I}}$ in $%
\mathbf{Pro}\left( H\left( \mathbf{Top}\right) \right) $ is called an $%
H\left( \mathbf{Top}\right) $-expansion (or simply \textbf{expansion}) if
for any polyhedron $P$ the following mapping%
\begin{equation*}
\underrightarrow{\lim }_{j}\left[ Y_{j},P\right] =\underrightarrow{\lim }%
_{j}Hom_{H\left( \mathbf{Top}\right) }\left( Y_{j},P\right) \longrightarrow
Hom_{H\left( \mathbf{Top}\right) }\left( X,P\right) =\left[ X,P\right]
\end{equation*}%
is bijective where $\left[ Z,P\right] $ is the set of homotopy classes of
continuous mappings from $Z$ to $P$.

An expansion is called \textbf{polyhedral} (or an $H\left( \mathbf{Pol}%
\right) $-expansion) if all $Y_{j}$ are polyhedra.
\end{definition}

\begin{remark}
\label{Rem-HPOL-extension}~

\begin{enumerate}
\item The pointed version of this notion (an $H\left( \mathbf{Pol}_{\ast
}\right) $-expansion) is defined similarly.

\item For any (pointed) topological space $X$ there exists an $H\left( 
\mathbf{Pol}\right) $-expansion (an $H\left( \mathbf{Pol}_{\ast }\right) $%
-expansion), see \cite[Theorem I.4.7 and I.4.10]{Mardesic-Segal-MR676973}.

\item Any two $H\left( \mathbf{Pol}\right) $-expansions ($H\left( \mathbf{Pol%
}_{\ast }\right) $-expansions) of a (pointed) topological space $X$ are
isomorphic in the category $\mathbf{Pro}\left( H\left( \mathbf{Pol}\right)
\right) $ ($\mathbf{Pro}\left( H\left( \mathbf{Pol}_{\ast }\right) \right) $%
), see \cite[Theorem I.2.6]{Mardesic-Segal-MR676973}.
\end{enumerate}
\end{remark}

\begin{definition}
\label{Def-Normal-covering}An open covering is called \textbf{normal} \cite[%
\S I.6.2]{Mardesic-Segal-MR676973}, iff there is a partition of unity
subordinated to it.
\end{definition}

\begin{remark}
Theorem 8 from \cite[App.1, \S 3.2]{Mardesic-Segal-MR676973}, shows that an $%
H\left( \mathbf{Pol}\right) $- or an $H\left( \mathbf{Pol}_{\ast }\right) $%
-expansion for $X$ can be constructed using nerves of normal (see Definition %
\ref{Def-Normal-covering}) open coverings of $X$.
\end{remark}

Pro-homotopy is defined in \cite[p. 121]{Mardesic-Segal-MR676973}:

\begin{definition}
\label{Def-Pro-homotopy-groups}\label{Pro-homotopy-groups}For a (pointed)
topological space $X$, define its pro-homotopy pro-sets%
\begin{equation*}
pro\text{-}\pi _{n}\left( X\right) 
{:=}%
\left( \pi _{n}\left( Y_{j}\right) \right) _{j\in \mathbf{J}}
\end{equation*}%
where $X\rightarrow \left( Y_{j}\right) _{j\in \mathbf{J}}$ is an $H\left( 
\mathbf{Pol}\right) $-expansion if $n=0$, and an $H\left( \mathbf{Pol}_{\ast
}\right) $-expansion if $n\geq 1$.

Similar to the \textquotedblleft usual\textquotedblright\ algebraic
topology, $pro$-$\pi _{0}$ is a pro-set (an object of $\mathbf{Pro}\left( 
\mathbf{Set}\right) $), $pro$-$\pi _{1}$ is a pro-group (an object of $%
\mathbf{Pro}\left( \mathbf{Group}\right) $), and $pro$-$\pi _{n}$ are
abelian pro-groups (objects of $\mathbf{Pro}\left( \mathbf{Ab}\right) $) for 
$n\geq 2$.
\end{definition}

Pro-homology groups are defined in \cite[\S II.3.2]{Mardesic-Segal-MR676973}%
, as follows:

\begin{definition}
\label{Def-Pro-homology-groups}\label{Pro-homology-groups}For a topological
space $X$, and an abelian group $G$, define its pro-homology groups as%
\begin{equation*}
pro\text{-}H_{n}\left( X,G\right) 
{:=}%
\left( H_{n}\left( Y_{j},G\right) \right) _{j\in \mathbf{J}}
\end{equation*}%
where $X\rightarrow \left( Y_{j}\right) _{j\in \mathbf{J}}$ is a polyhedral
expansion.
\end{definition}

\subsection{Constant (pre)cosheaves}

\begin{lemma}
\label{Lemma-Description-of-costalk}Let $x\in X$, let $\iota _{x}:\left\{
x\right\} \rightarrow X$ be the embedding, and let $\mathcal{A}\in \mathbf{%
pCS}\left( X,\mathbf{Pro}\left( k\right) \right) $, $\mathcal{B}\in \mathbf{%
CS}\left( X,\mathbf{Pro}\left( k\right) \right) $. Then

\begin{enumerate}
\item $\left( \left( \iota _{x}\right) ^{\ddag }\mathcal{A}\right) \left(
\left\{ x\right\} \right) 
\simeq%
\mathcal{A}^{x}$.

\item $\left( \left( \iota _{x}\right) ^{\ast }\mathcal{A}_{\#}\right)
\left( \left\{ x\right\} \right) 
\simeq%
\mathcal{A}^{x}$.

\item $\left( \left( \iota _{x}\right) ^{\ast }\mathcal{B}\right) \left(
\left\{ x\right\} \right) 
\simeq%
\mathcal{B}^{x}$.
\end{enumerate}
\end{lemma}

\begin{proof}
~

\begin{enumerate}
\item For an open subset $U\subseteq X$%
\begin{equation*}
\left( \iota _{x}\right) ^{-1}\left( U\right) =\left\{ 
\begin{array}{ccc}
\left\{ x\right\} & \text{if} & x\in U \\ 
\varnothing & \text{otherwise} & 
\end{array}%
\right. .
\end{equation*}%
Therefore, due to Notation \ref{Not-Categories}(\ref%
{Def-Kan-extensions-right})%
\begin{equation*}
\left( \iota _{x}\right) ^{\ddag }\mathcal{A}\left( \left\{ x\right\}
\right) =\underset{\left\{ x\right\} \subseteq \left( \iota _{x}\right)
^{-1}U}{\underleftarrow{\lim }}~\mathcal{A}\left( U\right) =\underset{x\in U}%
{\underleftarrow{\lim }}~\mathcal{A}\left( U\right) =\mathcal{A}^{x}.
\end{equation*}

\item Since all precosheaves on $\left\{ x\right\} $ are cosheaves, $\left(
\left( \iota _{x}\right) ^{\ddag }\mathcal{A}_{\#}\right) _{\#}=\left( \iota
_{x}\right) ^{\ddag }\mathcal{A}_{\#},$\ and $\left( \iota _{x}\right)
^{\ast }\mathcal{A}_{\#}\left( \left\{ x\right\} \right) =\left( \left(
\iota _{x}\right) ^{\ddag }\mathcal{A}_{\#}\right) _{\#}\left( \left\{
x\right\} \right) =\left( \left( \iota _{x}\right) ^{\ddag }\mathcal{A}%
_{\#}\right) \left( \left\{ x\right\} \right) =\left( \mathcal{A}%
_{\#}\right) ^{x}%
\simeq%
\mathcal{A}^{x}$.

\item Follows easily from the above, because $\mathcal{B}_{\#}%
\simeq%
\mathcal{B}$.
\end{enumerate}
\end{proof}

\begin{definition}
\label{Def-Constant-(pre)cosheaf}Let $G\in \mathbf{Pro}\left( k\right) $.

\begin{enumerate}
\item Denote by $G_{X}$ (or simply $G$) the following \textbf{constant}
precosheaf on $X$:%
\begin{equation*}
G_{X}\left( U\right) 
{:=}%
G.
\end{equation*}

\item The cosheaf $\left( G_{X}\right) _{\#}$ (or simply $G_{\#}$) will be
called the \textbf{constant} cosheaf on $X$.
\end{enumerate}
\end{definition}

\begin{remark}
\label{Rem-Constant-cosheaf-pro-H0}Let $G\in \mathbf{Mod}\left( k\right) $.
It follows from \cite[Theorem 3.11(3)]%
{Prasolov-Cosheafification-2016-MR3660525} that%
\begin{equation*}
G_{\#}%
\simeq%
pro\text{-}H_{0}\left( \bullet ,G\right) ,
\end{equation*}%
i.e., $G_{\#}\left( U\right) 
\simeq%
pro$-$H_{0}\left( U,G\right) $ for all $U\in \mathbf{C}_{X}$. In fact, in
the cited paper the statement was proved for $G\in \mathbf{Ab}$, but the
proof can be easily extended to the case $G\in \mathbf{Mod}\left( k\right) $.
\end{remark}

\begin{proposition}
\label{Prop-Action-of-continuous-mappings}Let $f:X\rightarrow Y$ be a
continuous mapping, and let $\mathcal{A}\in \mathbf{CS}\left( Y,\mathbf{Pro}%
\left( k\right) \right) $.

\begin{enumerate}
\item For any $x\in X$, $\left( f^{\ast }\mathcal{A}\right) ^{x}%
\simeq%
\mathcal{A}^{f\left( x\right) }$.

\item For $G\in \mathbf{Pro}\left( k\right) $%
\begin{equation*}
f^{\ast }\left( \left( G_{Y}\right) _{\#}\right) 
\simeq%
\left( G_{X}\right) _{\#}.
\end{equation*}
\end{enumerate}
\end{proposition}

\begin{proof}
~

\begin{enumerate}
\item Let $\kappa _{x}:\left\{ x\right\} \rightarrow Y$ be the composition%
\begin{equation*}
\kappa _{x}=f\circ \iota _{x}:\left\{ x\right\} \overset{\iota _{x}}{%
\longrightarrow }X\overset{f}{\longrightarrow }Y.
\end{equation*}%
Then%
\begin{equation*}
\left( \kappa _{x}\right) ^{\ast }%
\simeq%
\left( \iota _{x}\right) ^{\ast }\circ f^{\ast },
\end{equation*}%
and%
\begin{equation*}
\mathcal{A}^{f\left( x\right) }=\left( \left( \kappa _{x}\right) ^{\ast }%
\mathcal{A}\right) \left( \left\{ x\right\} \right) =\left( \left( \iota
_{x}\right) ^{\ast }f^{\ast }\mathcal{A}\right) \left( \left\{ x\right\}
\right) =\left( f^{\ast }\mathcal{A}\right) ^{x}.
\end{equation*}

\item Calculate $f^{\ddag }\left( G_{Y}\right) $:%
\begin{equation*}
f^{\ddag }\left( G_{Y}\right) \left( U\right) =\underset{U\subseteq
f^{-1}\left( V\right) }{\underleftarrow{\lim }}~G_{Y}\left( V\right) =%
\underset{U\subseteq f^{-1}\left( V\right) }{\underleftarrow{\lim }}~G=G,
\end{equation*}%
i.e., $f^{\ddag }\left( G_{Y}\right) 
\simeq%
G_{X}$. Proposition \ref{Prop-Site-morphism-cosheaves}(\ref%
{Prop-Site-morphism-cosheaves-right-adjoint}, second part) reads%
\begin{equation*}
f^{\ast }\circ \left( \bullet \right) _{\#}%
\simeq%
\left( \bullet \right) _{\#}\circ f^{\ddag }:\mathbf{pCS}\left( Y,\mathbf{Pro%
}\left( k\right) \right) \longrightarrow \mathbf{CS}\left( X,\mathbf{Pro}%
\left( k\right) \right) .
\end{equation*}%
It follows that%
\begin{equation*}
f^{\ast }\left( \left( G_{Y}\right) _{\#}\right) 
\simeq%
\left( f^{\ddag }\left( G_{Y}\right) \right) _{\#}%
\simeq%
\left( G_{X}\right) _{\#}.
\end{equation*}
\end{enumerate}
\end{proof}

It is possible to obtain a simpler description of $pro$-$H_{0}\left( \bullet
,G\right) $ and $G_{\#}$.

\begin{definition}
\label{Def-Decomposition}\label{Def-Partition}Let $X$ be a topological
space. A \textbf{partition} $\mathcal{U}$ of $X$ is a representation of $X$
as a disjoint union of non-empty open subsets, i.e.,

\begin{enumerate}
\item $\mathcal{U}=\left\{ \varnothing \neq U_{i}\in Open\left( X\right)
\right\} _{i\in I}$.

\item 
\begin{equation*}
X=\dbigcup\limits_{i\in I}U_{i}.
\end{equation*}

\item $U_{i}\cap U_{j}=\varnothing $ if $i\neq j$.
\end{enumerate}

We say that another partition $\mathcal{V=}\left\{ \mathcal{V}_{j}\right\}
_{j\in J}$ \textbf{refines} $\mathcal{U}$ ($\mathcal{U\preceq V}$) iff for
any $j\in J$, $V_{j}\subseteq U_{i}$ for some (in fact, unique!) $i\in I$.
Denote by $Part\left( X\right) $ the poset $\left( \left\{ \text{Partitions
of }X\right\} ,\mathcal{\preceq }\right) $.
\end{definition}

\begin{remark}
\label{Rem-Decomposition}\label{Rem-Partition}$Part\left( X\right) $ is a 
\textbf{directed} poset. Indeed, if%
\begin{eqnarray*}
\mathcal{V} &\mathcal{=}&\left\{ V_{j}\right\} _{j\in J}, \\
\mathcal{W} &\mathcal{=}&\left\{ W_{k}\right\} _{k\in K}, \\
\mathcal{U} &\mathcal{\preceq }&\mathcal{V},\mathcal{W},
\end{eqnarray*}%
then%
\begin{equation*}
\mathcal{V}\cap \mathcal{W=}\left\{ V_{j}\cap W_{k}\neq \varnothing \right\}
_{j\in J,k\in K}
\end{equation*}%
clearly refines both $\mathcal{V}$ and $\mathcal{W}$.
\end{remark}

\begin{notation}
\label{Not-H0-for-partition}Let $G\in \mathbf{Mod}\left( k\right) $ and $%
\mathcal{U}=\left\{ U_{i}\right\} _{i\in I}\in Part\left( X\right) $. Denote
by%
\begin{equation*}
H_{0}\left( \mathcal{U},G\right) 
{:=}%
\dbigoplus\limits_{i\in I}G.
\end{equation*}
\end{notation}

\begin{proposition}
\label{Prop-Simpler-description-G-sharp}Let $G\in \mathbf{Mod}\left(
k\right) $. Then the functor%
\begin{equation*}
pro\text{-}H_{0}\left( \bullet ,G\right) :\mathbf{Top}\longrightarrow 
\mathbf{Pro}\left( k\right)
\end{equation*}%
admits the following description:

\begin{enumerate}
\item 
\begin{equation*}
pro\text{-}H_{0}\left( X,G\right) =\left( H_{0}\left( \mathcal{U},G\right)
\right) _{\mathcal{U}=\left\{ U_{i}\right\} _{i\in I}\in Part\left( X\right)
}
\end{equation*}

\item If $f:X\rightarrow Y$ is a continuous mapping, then%
\begin{eqnarray*}
h &=&pro\text{-}H_{0}\left( f,G\right) \in 
\Hom%
_{\mathbf{Pro}\left( k\right) }\left( pro\text{-}H_{0}\left( X,G\right) ,pro%
\text{-}H_{0}\left( Y,G\right) \right) = \\
&=&\underset{\mathcal{V}=\left\{ V_{j}\right\} _{j\in J}\in Part\left(
Y\right) }{\underleftarrow{\lim }}~\underset{\mathcal{U}=\left\{
U_{i}\right\} _{i\in I}\in Part\left( X\right) }{\underrightarrow{\lim }}~%
\Hom%
_{\mathbf{Mod}\left( k\right) }\left( H_{0}\left( \mathcal{U},G\right)
,H_{0}\left( \mathcal{V},G\right) \right)
\end{eqnarray*}%
is given by the following data (see \cite[Remark 2.1.9]%
{Prasolov-Cosheaves-2021-MR4347662}):

\begin{enumerate}
\item $\varphi :Part\left( Y\right) \rightarrow Part\left( X\right) $:%
\begin{equation*}
\varphi \left( \mathcal{V}=\left\{ V_{j}\right\} _{j\in J}\right) 
{:=}%
f^{-1}\left( \mathcal{V}\right) =\left\{ f^{-1}\left( V_{j}\right) \neq
\varnothing \right\} _{j\in J}=\left\{ U_{i}\right\} _{i\in I}.
\end{equation*}

\item 
\begin{equation*}
h_{\mathcal{V}}:H_{0}\left( \mathcal{U},G\right) =\dbigoplus\limits_{i\in
I}G\rightarrow H_{0}\left( \mathcal{V},G\right) =\dbigoplus\limits_{j\in J}G
\end{equation*}%
maps identically the $i$-th summand of the first direct sum to the $j$-th
summand of the second one, where%
\begin{equation*}
\varnothing \neq U_{i}=f^{-1}\left( V_{j}\right) .
\end{equation*}
\end{enumerate}
\end{enumerate}
\end{proposition}

\begin{proof}
Due to \cite[Appendix 1.3, p. 324]{Mardesic-Segal-MR676973}, there are the 
\u{C}ech
$H\left( \mathbf{Pol}\right) $-expansions%
\begin{eqnarray*}
X &\longrightarrow &\mathbf{X}\mathcal{=}\left( \left\vert N\mathcal{U}%
\right\vert \right) _{\mathcal{U}\in Norm\left( X\right) }\in \mathbf{Pro}%
\left( H\left( \mathbf{Pol}\right) \right) , \\
Y &\longrightarrow &\mathbf{Y}\mathcal{=}\left( \left\vert N\mathcal{V}%
\right\vert \right) _{\mathcal{V}\in Norm\left( Y\right) }\in \mathbf{Pro}%
\left( H\left( \mathbf{Pol}\right) \right) ,
\end{eqnarray*}%
where $\mathcal{U}$ ($\mathcal{V}$) runs through the set $Norm\left(
X\right) $ ($Norm\left( Y\right) $) of\ \textbf{normal} coverings
(Definition \ref{Def-Normal-covering}) of $X$ ($Y$), and $\left\vert N%
\mathcal{U}\right\vert $ ($\left\vert N\mathcal{V}\right\vert $) is the
geometric realization of the 
\u{C}ech
nerve $N\mathcal{U}$ ($N\mathcal{V}$) of $\mathcal{U}$ ($\mathcal{V}$).

\begin{enumerate}
\item Any \textbf{partition} $\mathcal{U=}\left\{ U_{i}\right\} _{i\in I}$
is clearly a normal covering, because it admits the following \textbf{%
partition of unity}, subordinated to it, i.e., the family of continuous
mappings $\left( \chi _{i}\right) _{i\in I}$, satisfying:%
\begin{eqnarray*}
\chi _{i} &:&X\longrightarrow \left[ 0,1\right] , \\
\chi _{i}\left( x\right) &=&\left\{ 
\begin{array}{ccc}
1 & \text{if} & x\in U_{i} \\ 
0 & \text{otherwise} & 
\end{array}%
\right. , \\
\dsum\limits_{i\in I}\chi _{i}\left( x\right) &\equiv &1.
\end{eqnarray*}%
On the other hand, consider a normal covering $\mathcal{V=}\left\{
V_{j}\right\} _{j\in J}$ of $X$. It follows that%
\begin{equation*}
H_{0}\left( \left\vert N\mathcal{V}\right\vert ,G\right) 
\simeq%
\dbigoplus\limits_{\pi _{0}\left( \left\vert N\mathcal{V}\right\vert \right)
}G.
\end{equation*}%
The $0$-simplices of $N\mathcal{V}$ are the sets $V_{j}$, while the $n$%
-simplices are $\left( n+1\right) $-tuples%
\begin{equation*}
\left( V_{j_{0}},V_{j_{1}},\dots ,V_{j_{n}}\right)
\end{equation*}%
with non-empty intersection. It is well-known that%
\begin{equation*}
\pi _{0}\left( \left\vert N\mathcal{V}\right\vert \right) =\pi _{0}\left( N%
\mathcal{V}\right)
\end{equation*}%
is the partition of $J$ into the equivalence classes for the equivalence
relation $R$ generated by pairs $\left( i,j\right) $ with%
\begin{equation*}
V_{i}\cap V_{j}\neq \varnothing .
\end{equation*}%
For each $j\in J$ consider%
\begin{equation*}
W_{j}=\dbigcup\limits_{iRj}V_{i}.
\end{equation*}%
It is clear that for any pair $\left( s,t\right) $ the sets $W_{s}$ and $%
W_{t}$ are either equal or have empty intersection. Choose one subset from
each group of equal subsets, and get a \textbf{partition} in our sense:%
\begin{equation*}
\mathcal{W}=\left\{ W_{k}\right\} _{k\in K\subseteq J}.
\end{equation*}%
The assignment%
\begin{equation*}
\mathcal{V=}\left\{ V_{j}\right\} _{j\in J}\mapsto \mathcal{W}=\left\{
W_{k}\right\} _{k\in K\subseteq J}
\end{equation*}%
gives a surjective (therefore \textbf{cofinal}) monotone mapping%
\begin{equation*}
\alpha _{X}:Norm\left( X\right) \longrightarrow Part\left( X\right) ,
\end{equation*}%
inducing the desired isomorphism of pro-modules.

\item Since $X\rightarrow \mathbf{X}$ and $Y\rightarrow \mathbf{Y}$ are $%
H\left( \mathbf{Pol}\right) $-expansions, there exists a \textbf{unique} $%
\mathbf{Pro}\left( H\left( \mathbf{Pol}\right) \right) $-morphism $\mathbf{%
f:X}\rightarrow \mathbf{Y}$ such that the diagram%
\begin{equation*}
\begin{diagram}[size=3.0em,textflow]
X & \rTo & \mathbf{X} \\ 
\dTo_{f} &  & \dTo_{\mathbf{f}} \\ 
Y & \rTo & \mathbf{Y} \\
\end{diagram}%
\end{equation*}%
commutes. Let us construct $\mathbf{f}$ given by the following data (see 
\cite[Remark 2.1.9]{Prasolov-Cosheaves-2021-MR4347662}. It is clear that%
\begin{equation*}
\mathcal{V}=\left\{ V_{j}\right\} _{j\in J}\in Norm\left( Y\right) 
\end{equation*}%
implies%
\begin{equation*}
\mathcal{U}=\left\{ f^{-1}\left( V_{j}\right) \neq \varnothing \right\}
_{j\in J}\in Norm\left( X\right) .
\end{equation*}%
The assignments $\mathcal{V}\mapsto \mathcal{U}$ and%
\begin{equation*}
\left( f^{-1}\left( V_{0}\right) ,f^{-1}\left( V_{1}\right) ,\dots
,f^{-1}\left( V_{n}\right) \right) \longmapsto \left( V_{0},V_{1},\dots
,V_{n}\right) 
\end{equation*}%
give us the desired morphism%
\begin{equation*}
\mathbf{f}=\left( \varphi :Norm\left( Y\right) \rightarrow Norm\left(
X\right) ,\left( f_{\mathcal{V}}:N\mathcal{U}\rightarrow N\mathcal{V}\right)
_{\mathcal{V}\in Norm\left( Y\right) }\right) :\mathbf{X}\longrightarrow 
\mathbf{Y.}
\end{equation*}%
Apply the cofinal monotone mappings%
\begin{eqnarray*}
\alpha _{X} &:&Norm\left( X\right) \longrightarrow Part\left( X\right) , \\
\alpha _{Y} &:&Norm\left( Y\right) \longrightarrow Part\left( Y\right) ,
\end{eqnarray*}%
and $H_{0}\left( \bullet ,G\right) $ to obtain the commutative diagram%
\begin{equation*}
\begin{diagram}[size=3.0em,textflow]
H_{0}\left( \mathbf{X},G\right) & \lTo_{\TeXButton{ISO}{\simeq}} & \left( H_{0}\left( \mathcal{U},G\right) \right) _{\mathcal{U}\in Part\left( X\right) } \\ 
\dTo_{H_{0}\left( \mathbf{f},G\right)} &  & \dTo_{\mathbf{f}^{\prime }} \\ 
H_{0}\left( \mathbf{Y},G\right) & \lTo{\TeXButton{ISO}{\simeq}} & \left( H_{0}\left( \mathcal{V},G\right) \right) _{\mathcal{V}\in Part\left( X\right) }\\
\end{diagram}
%
\end{equation*}%
in $\mathbf{Pro}\left( k\right) $, where%
\begin{equation*}
\mathbf{f}^{\prime }=\left( \varphi ^{\prime }:Part\left( Y\right)
\rightarrow Part\left( X\right) ,\left( f_{\mathcal{V}}:H_{0}\left(
f^{-1}\left( \mathcal{V}\right) ,G\right) \rightarrow H_{0}\left( \mathcal{V}%
,G\right) \right) \right) 
\end{equation*}%
is given by%
\begin{equation*}
\varphi ^{\prime }\left( \mathcal{V}=\left\{ V_{j}\right\} _{j\in J}\right)
=f^{-1}\left( \mathcal{V}\right) 
{:=}%
\left\{ f^{-1}\left( V_{j}\right) \neq \varnothing \right\} _{j\in J}
\end{equation*}%
and $f_{\mathcal{V}}$ that sends the $i$-th ($U_{i}=f^{-1}\left(
V_{j}\right) $) summand $G$ of $H_{0}\left( \mathcal{U},G\right) $
identically to the $j$-th summand $G$ of $H_{0}\left( \mathcal{V},G\right) $.
\end{enumerate}
\end{proof}

\begin{corollary}
\label{Cor-Simpler-description-G-sharp}Let $G\in \mathbf{Mod}\left( k\right) 
$, and $G_{X}$ be the corresponding constant precosheaf on $X$. Then the
cosheafification $H_{X,G}=\left( G_{X}\right) _{\#}$ admits the following
description:%
\begin{equation*}
H_{X,G}\left( U\right) =\left( H_{0}\left( \mathcal{U},G\right)
=\dbigoplus\limits_{i\in I}G\right) _{\mathcal{U}=\left\{ U_{i}\right\}
_{i\in I}\in Part\left( U\right) }
\end{equation*}%
with the morphisms $G_{X}\left( U\rightarrow V\right) $ for $U\subseteq V$
like in Proposition \ref{Prop-Simpler-description-G-sharp}.
\end{corollary}

\begin{remark}
\label{Rem-Simpler-description-G-sharp}Below is a shorter proof of Corollary %
\ref{Cor-Simpler-description-G-sharp}. Let $H_{X,G}\rightarrow G$ be the
morphism of precosheaves mapping each copy of $G$ in $H_{0}\left( \mathcal{U}%
,G\right) $ identically to $G$. Assume that $T\in \mathbf{Mod}\left(
k\right) $ is an arbitrary injective $k$-module, and examine the morphism $%
\left\langle G,T\right\rangle \rightarrow \left\langle
H_{X,G},T\right\rangle $:%
\begin{eqnarray*}
\left\langle H_{X,G},T\right\rangle &=&\left( U\longmapsto \underset{%
\mathcal{U}=\left\{ U_{i}\right\} _{i\in I}\in Part\left( U\right) }{%
\underrightarrow{\lim }}%
\Hom%
_{\mathbf{Mod}\left( k\right) }\left( \dbigoplus\limits_{i\in I}G,T\right)
\right) = \\
&=&\left( U\longmapsto \left\{ \text{locally constant mappings }U\rightarrow
\left\langle G,T\right\rangle \right\} \right) .
\end{eqnarray*}%
It is well-known that for any $A\in \mathbf{Mod}\left( k\right) $ the
sheafification of the constant presheaf $A$ is:%
\begin{equation*}
A^{\#}=\left( U\longmapsto \left\{ \text{locally constant mappings }%
U\rightarrow A\right\} \right) ,
\end{equation*}%
and the morphism $A\rightarrow A^{\#}$ is the inclusion of constant mappings
into locally constant mappings. Since $\left\langle H_{X,G},T\right\rangle
=\left\langle G,T\right\rangle ^{\#}$ is a sheaf (for any $T$!), $H_{X,G}$
is a \textbf{co}sheaf. There exists a unique splitting $H_{X,G}\rightarrow
\left( G_{X}\right) _{\#}\rightarrow G$ of the morphism of precosheaves $%
H_{X,G}\rightarrow G$. We have therefore the following commutative diagram:%
\begin{equation*}
\begin{diagram}[size=3.0em,textflow]
\left\langle G,T\right\rangle & \rTo & \left\langle
G,T\right\rangle ^{\#} & 
\TeXButton{ISO}{\simeq} & \left\langle \left( G_{X}\right) _{\#},T\right\rangle \\ 
\dTo_{=} &  & \dTo_{ 
\TeXButton{ISO}{\simeq} } & \ldTo &  \\ 
\left\langle G,T\right\rangle & \longrightarrow & \left\langle
H_{X,G},T\right\rangle &  & \\
\end{diagram}%
\end{equation*}%
It follows that $\left\langle \left( G_{X}\right) _{\#},T\right\rangle
\rightarrow \left\langle H_{X,G},T\right\rangle $ is an isomorphism (for any 
$T$!), therefore%
\begin{equation*}
H_{X,G}\rightarrow \left( G_{X}\right) _{\#}
\end{equation*}%
is an isomorphism as well.
\end{remark}

\subsection{Hausdorff paracompact spaces}

The \textquotedblleft classical\textquotedblright\ definition of 
\u{C}ech
\textbf{co}homology is based on cochains%
\begin{equation*}
\check{C}_{red}^{n}%
{:=}%
\dprod\limits_{\substack{ \left( i_{0},i_{1},\dots ,i_{n}\right)  \\ %
U_{i_{0}}\cap U_{i_{1}}\cap \dots \cap U_{i_{n}}\neq \varnothing }}\mathcal{A%
}\left( U_{i_{0}}\cap U_{i_{1}}\cap \dots \cap U_{i_{n}}\right)
\end{equation*}%
corresponding to $\left( n+1\right) $-tuples of open subsets with \textbf{%
non-empty} intersection. Our definition for topological spaces arises from a
more general definition for arbitrary Grothendieck sites where \textbf{%
non-emptiness} is not defined:%
\begin{equation*}
\check{C}^{n}%
{:=}%
\dprod\limits_{\left( i_{0},i_{1},\dots ,i_{n}\right) }\mathcal{A}\left(
U_{i_{0}}\cap U_{i_{1}}\cap \dots \cap U_{i_{n}}\right) .
\end{equation*}%
Of course, if $\mathcal{A}\left( \varnothing \right) =0$, the two
definitions coincide. If $\mathcal{A}$ is a sheaf, then we have $\mathcal{A}%
\left( \varnothing \right) =0$ automatically. We need, however, to handle
the case of \textbf{pre}sheaves $\mathcal{A}$ with $\mathcal{A}\left(
\varnothing \right) \neq 0$.

\begin{proposition}
\label{Prop-Reduced-Complex}Let $\mathcal{A}$ be a presheaf, $\mathcal{U}%
=\left\{ U_{i}\rightarrow U\right\} _{i\in I}$ be a covering of $U$, and $%
\check{H}^{\bullet }\left( \mathcal{U},\mathcal{A}\right) $ be the 
\u{C}ech
\textbf{co}homology.

\begin{enumerate}
\item If $U=\varnothing $, then $\check{H}^{n}\left( \mathcal{U},\mathcal{A}%
\right) =0$ for $n\geq 1$, and $\check{H}^{0}\left( \mathcal{U},\mathcal{A}%
\right) =\mathcal{A}\left( \varnothing \right) $.

\item If $U\neq \varnothing $, then $\check{H}^{\bullet }\left( \mathcal{U},%
\mathcal{A}\right) $ does not depend on $\mathcal{A}\left( \varnothing
\right) $.
\end{enumerate}
\end{proposition}

\begin{proof}
Let $\check{C}^{\bullet }\left( \mathcal{U},\mathcal{A}\right) $
(respectively $\check{C}_{red}^{\bullet }\left( \mathcal{U},\mathcal{A}%
\right) $) be the \textbf{large} (respectively the \textbf{reduced}) 
\u{C}ech
cochain complex, and let $\pi :\check{C}^{\bullet }\rightarrow \check{C}%
_{red}^{\bullet }$ be the projection:%
\begin{eqnarray*}
&&\check{C}^{n}\left( \mathcal{U},\mathcal{A}\right) 
{:=}%
\dprod\limits_{\left( i_{0},\dots ,i_{n}\right) \in I^{n+1}}\mathcal{A}%
\left( U_{i_{0}}\cap \dots \cap U_{i_{n}}\right) , \\
&&\check{C}_{red}^{n}\left( \mathcal{U},\mathcal{A}\right) 
{:=}%
\dprod\limits_{\substack{ \left( i_{0},\dots ,i_{n}\right) \in I^{n+1}  \\ %
U_{i_{0}}\cap \dots \cap U_{i_{n}}\neq \varnothing }}\mathcal{A}\left(
U_{i_{0}}\cap \dots \cap U_{i_{n}}\right) .
\end{eqnarray*}

\begin{enumerate}
\item If $U=\varnothing $, and $U_{i}=\varnothing $ (even with repetition),
then%
\begin{equation*}
\check{C}^{\bullet }\left( \mathcal{U},\mathcal{A}\right) =\left[
\dprod\limits_{I}\mathcal{A}\left( \varnothing \right) \longrightarrow
\dprod\limits_{I^{2}}\mathcal{A}\left( \varnothing \right) \longrightarrow
\dprod\limits_{I^{3}}\mathcal{A}\left( \varnothing \right) \longrightarrow
\dots \right] ,
\end{equation*}%
and $\check{H}^{\bullet }\left( \mathcal{U},\mathcal{A}\right) $ is the
cohomology $H^{\bullet }\left( U,\mathcal{A}\left( \varnothing \right)
\right) $ (which is the desired) of the contractible full infinite simplex
with the set of vertices $I$. More precisely, consider the augmented complex%
\begin{equation*}
\left[ C^{-1}=\mathcal{A}\left( \varnothing \right) \right] \longrightarrow %
\left[ C^{0}=\dprod\limits_{I}\mathcal{A}\left( \varnothing \right) \right]
\longrightarrow \left[ C^{1}=\dprod\limits_{I^{2}}\mathcal{A}\left(
\varnothing \right) \right] \longrightarrow \dots ,
\end{equation*}%
and build the contraction%
\begin{eqnarray*}
S^{n} &:&C^{n+1}\longrightarrow C^{n}: \\
\left( S^{n}a\right) \left( i_{0},i_{1},\dots ,i_{n}\right) &=&\left(
-1\right) ^{n+1}a\left( i_{0},i_{1},\dots ,i_{n},\ast \right) ,
\end{eqnarray*}%
where $\ast \in I$ is an arbitrary chosen index. $S^{\bullet }$ is indeed a
contraction:%
\begin{eqnarray*}
\left[ S^{n}d^{n}a\right] \left( i_{0},i_{1},\dots ,i_{n}\right) &=&\left(
-1\right) ^{n+1}\left[ da\right] \left( i_{0},i_{1},\dots ,i_{n},\ast \right)
\\
&=&\dsum\limits_{k=0}^{n}\left( -1\right) ^{k+n+1}a\left( i_{0},i_{1},\dots ,%
\widehat{i_{k}},\dots ,i_{n},\ast \right) +a\left( i_{0},i_{1},\dots
,i_{n}\right) , \\
\left[ d^{n-1}S^{n-1}a\right] \left( i_{0},i_{1},\dots ,i_{n}\right)
&=&\dsum\limits_{k=0}^{n}\left( -1\right) ^{k}\left[ Sa\right] \left(
i_{0},i_{1},\dots ,\widehat{i_{k}},\dots ,i_{n}\right) = \\
&=&\dsum\limits_{k=0}^{n}\left( -1\right) ^{k+n}\left[ a\right] \left(
i_{0},i_{1},\dots ,\widehat{i_{k}},\dots ,i_{n},\ast \right) , \\
\left[ \left( Sd+dS\right) a\right] \left( i_{0},i_{1},\dots ,i_{n}\right)
&=&a\left( i_{0},i_{1},\dots ,i_{n}\right) .
\end{eqnarray*}%
Since the augmented complex is contractible, and therefore acyclic, the
desired calculation for the original complex follows easily.

\item If $U\neq \varnothing $, let $K^{\bullet }=\ker \pi $. We claim that $%
K^{\bullet }$ is contractible. Let us build the contraction%
\begin{equation*}
S^{n}:K^{n+1}\longrightarrow K^{n},n\geq -1.
\end{equation*}%
Choose an index $\ast \in I$, such that $U_{\ast }\neq \varnothing $.

\begin{enumerate}
\item $S^{-1}=0$.

\item Let $a\in K^{0}$. Define $S^{0}:K^{1}\rightarrow K^{0}$ by%
\begin{equation*}
\left[ S^{0}b\right] \left( i\right) =b\left( \ast ,i\right) .
\end{equation*}%
It follows that%
\begin{eqnarray*}
\left[ Sda\right] \left( i\right) &=&da\left( \ast ,i\right) =a\left(
i\right) -a\left( \ast \right) =a\left( i\right) , \\
\left[ dSa\right] \left( i\right) &=&0, \\
\left[ \left[ Sd+dS\right] \left( a\right) \right] \left( i\right)
&=&a\left( i\right) ,\text{ i.e.,} \\
Sd+dS &=&1_{K^{0}}.
\end{eqnarray*}

\item If $a\in K^{1}$, then%
\begin{eqnarray*}
\left[ a-dSa\right] \left( i_{0},i_{1}\right) &=&a\left( i_{0},i_{1}\right)
-Sa\left( i_{1}\right) +Sa\left( i_{0}\right) = \\
&=&a\left( i_{0},i_{1}\right) -a\left( \ast ,i_{1}\right) +a\left( \ast
,i_{0}\right) .
\end{eqnarray*}%
Define $S^{1}:K^{2}\rightarrow K^{1}$ by%
\begin{equation*}
\left[ S^{1}b\right] \left( i_{0},i_{1}\right) =b\left( \ast
,i_{0},i_{1}\right) .
\end{equation*}

\item For any $n$, define $S^{n}:K^{n+1}\rightarrow K^{n}$ by%
\begin{equation*}
S^{n}b\left( i_{0},\dots ,i_{n}\right) =b\left( \ast ,i_{0},\dots
,i_{n}\right) .
\end{equation*}%
It follows that for $a\in K^{n}$%
\begin{eqnarray*}
Sda\left( i_{0},\dots ,i_{n}\right) &=&da\left( \ast ,i_{0},\dots
,i_{n}\right) = \\
&=&a\left( i_{0},\dots ,i_{n}\right) +\dsum\limits_{k=0}^{n}\left( -1\right)
^{k+1}a\left( i_{0},i_{1},\dots ,\widehat{i_{k}},\dots ,i_{n}\right) , \\
dSa\left( i_{0},\dots ,i_{n}\right) &=&\dsum\limits_{k=0}^{n}\left(
-1\right) ^{k}\left[ Sa\right] \left( i_{0},i_{1},\dots ,\widehat{i_{k}}%
,\dots ,i_{n}\right) = \\
&=&\dsum\limits_{k=0}^{n}\left( -1\right) ^{k}a\left( \ast
,i_{0},i_{1},\dots ,\widehat{i_{k}},\dots ,i_{n}\right) , \\
\left[ \left[ Sd+dS\right] a\right] \left( i_{0},\dots ,i_{n}\right)
&=&a\left( i_{0},\dots ,i_{n}\right) ,\text{ i.e.,} \\
Sd+dS &=&1_{K^{n}}.
\end{eqnarray*}
\end{enumerate}
\end{enumerate}
\end{proof}

\begin{proposition}
\label{Prop-Zero-Cech-homology}\label{Cor-Zero-Cech-homology}Let $X\neq
\varnothing $ be a Hausdorff paracompact topological space, and let $%
\mathcal{A}\in \mathbf{pCS}\left( X,\mathbf{Pro}\left( k\right) \right) $ be
a pre\textbf{co}sheaf of \textbf{pro}-modules on $X$. Assume that the 
\textbf{co}sheafification $\mathcal{A}_{\#}=0$. Then%
\begin{equation*}
\check{H}_{n}\left( X,\mathcal{A}\right) =0
\end{equation*}%
for all $n\geq 0$.
\end{proposition}

\begin{proof}
Let $T\in \mathbf{Mod}\left( k\right) $ be an arbitrary injective module.
Define the presheaf%
\begin{equation*}
\mathcal{B}%
{:=}%
\left\langle \mathcal{A},T\right\rangle \in \mathbf{pS}\left( X,\mathbf{Mod}%
\left( k\right) \right) .
\end{equation*}%
Clearly the sheafification%
\begin{equation*}
\mathcal{B}^{\#}=\left\langle \mathcal{A},T\right\rangle ^{\#}%
\simeq%
\left\langle \mathcal{A}_{\#},T\right\rangle =0.
\end{equation*}%
Apply \cite[Th\'{e}or\`{e}me II.5.10.2]{Godement-MR0102797} or \cite[Theorem
III.4.4 ]{Bredon-Book-MR1481706} to the presheaf $\mathcal{B}$. In the cited
books, the reduced 
\u{C}ech
cohomology is used, therefore the cited theorems read%
\begin{equation*}
\check{H}_{red}^{n}\left( X,\mathcal{B}\right) =0
\end{equation*}%
in our terminology. However, Proposition \ref{Prop-Reduced-Complex} implies%
\begin{equation*}
\left\langle \check{H}_{n}\left( X,\mathcal{A}\right) ,T\right\rangle 
\simeq%
\check{H}^{n}\left( X,\left\langle \mathcal{A},T\right\rangle \right) 
\simeq%
\check{H}_{red}^{n}\left( X,\mathcal{B}\right) =0
\end{equation*}%
for \textbf{any} injective $T$. It follows from Proposition \ref%
{Prop-Duality} that%
\begin{equation*}
\check{H}_{n}\left( X,\mathcal{A}\right) =0.
\end{equation*}
\end{proof}

\begin{proposition}
\label{Prop-Homology-like-Paracompact}Let $X$ be a topological space. Assume
that $\check{H}_{\ast }\left( X,\mathcal{B}\right) =0$ for any precosheaf of
pro-modules $\mathcal{B}\in \mathbf{pCS}\left( X,\mathbf{Pro}\left( k\right)
\right) $ with $\mathcal{B}_{\#}=0$. Let $\mathcal{A}\in \mathbf{pCS}\left(
X,\mathbf{Pro}\left( k\right) \right) $ be a precosheaf. Then the natural
(in $\mathcal{A}$) homomorphism%
\begin{equation*}
H_{n}\left( X,\mathcal{A}_{\#}\right) \longrightarrow \check{H}_{n}\left( X,%
\mathcal{A}\right)
\end{equation*}%
is an isomorphism for all $n$.
\end{proposition}

\begin{proof}
Consider the exact sequence of precosheaves%
\begin{equation*}
0\longrightarrow \ker \varphi \longrightarrow \mathcal{A}_{\#}\overset{%
\varphi }{\longrightarrow }\mathcal{A}\longrightarrow 
\coker%
\varphi \longrightarrow 0.
\end{equation*}%
Apply the exact functor $\left( \bullet \right) _{\#}$. It follows that%
\begin{equation*}
\left( \ker \varphi \right) _{\#}=\left( 
\coker%
\varphi \right) _{\#}=0,
\end{equation*}%
and, due to the assumption,%
\begin{equation*}
\check{H}_{n}\left( X,\ker \varphi \right) =\check{H}_{n}\left( X,%
\coker%
\varphi \right) =0
\end{equation*}%
for all $n$. Consider the two short exact sequences of precosheaves:%
\begin{eqnarray*}
0 &\longrightarrow &\ker \varphi \longrightarrow \mathcal{A}%
_{\#}\longrightarrow 
\Imm%
\left( \varphi \right) \longrightarrow 0, \\
0 &\longrightarrow &%
\Imm%
\left( \varphi \right) \overset{\varphi }{\longrightarrow }\mathcal{A}%
\longrightarrow 
\coker%
\varphi \longrightarrow 0,
\end{eqnarray*}%
and the two corresponding long exact sequences of $\check{H}_{n}$. It
follows that%
\begin{equation*}
\check{H}_{n}\left( X,\mathcal{A}_{\#}\right) \longrightarrow \check{H}%
_{n}\left( X,%
\Imm%
\left( \varphi \right) \right) \longrightarrow \check{H}_{n}\left( X,%
\mathcal{A}\right)
\end{equation*}%
are isomorphisms for all $n$. Apply now the spectral sequence from Theorem %
\ref{Th-Grothendieck-spectral-sequence-(pre)cosheaves}(\ref%
{Th-Grothendieck-spectral-sequence-compare-to-Cech-limit}) to the cosheaf $%
\mathcal{A}_{\#}$:%
\begin{equation*}
E_{s,t}^{2}=\check{H}_{s}\left( X,\mathcal{H}_{t}\left( \mathcal{A}%
_{\#}\right) \right) \implies H_{s+t}\left( X,\mathcal{A}\right) .
\end{equation*}%
Due to \cite[Theorem 3.4.1(5)]{Prasolov-Cosheaves-2021-MR4347662},%
\begin{equation*}
\left( \mathcal{H}_{t}\left( \mathcal{A}_{\#}\right) \right) _{\#}=0,t\geq 1,
\end{equation*}%
and the spectral sequence degenerates, implying%
\begin{equation*}
E_{s,0}^{2}=\check{H}_{s}\left( X,\mathcal{H}_{0}\left( \mathcal{A}%
_{\#}\right) \right) =\check{H}_{s}\left( X,\mathcal{A}_{\#}\right) 
\simeq%
H_{s}\left( X,\mathcal{A}_{\#}\right) .
\end{equation*}%
Finally,%
\begin{equation*}
\check{H}_{s}\left( X,\mathcal{A}\right) 
\simeq%
\check{H}_{s}\left( X,\mathcal{A}_{\#}\right) 
\simeq%
H_{s}\left( X,\mathcal{A}_{\#}\right)
\end{equation*}%
for all $s$.
\end{proof}

\begin{corollary}
\label{Cor-Homology-Paracompact}Let $X$ be a Hausdorff paracompact
topological space, and let $\mathcal{A}\in \mathbf{pCS}\left( X,\mathbf{Pro}%
\left( k\right) \right) $ be a precosheaf of pro-modules on $X$. Then the
natural (in $\mathcal{A}$) homomorphism%
\begin{equation*}
H_{n}\left( X,\mathcal{A}_{\#}\right) \longrightarrow \check{H}_{n}\left( X,%
\mathcal{A}\right)
\end{equation*}%
is an isomorphism for all $n$.
\end{corollary}

\begin{proof}
Due to Proposition \ref{Prop-Zero-Cech-homology}, we may apply Proposition %
\ref{Prop-Homology-like-Paracompact}.
\end{proof}

\begin{theorem}
\label{Cor-Paracompact-vs-Shape}\label{Th-Paracompact-vs-Shape}Let $X$ be a
Hausdorff paracompact topological space, let $G\in \mathbf{Ab}$, and let $%
G_{\#}\in \mathbf{CS}\left( X,\mathbf{Ab}\right) $ be the corresponding
constant cosheaf. Then for all $n$ there is a natural (in $X$ and $G$)
isomorphism%
\begin{equation*}
H_{n}\left( X,G_{\#}\right) \longrightarrow pro\text{-}H_{n}\left(
X,G\right) ,
\end{equation*}%
where $pro$-$H_{n}$ is the shape pro-homology (Definition \ref%
{Def-Pro-homology-groups}).
\end{theorem}

\begin{proof}
Due to \cite[Appendix 1.3, p. 324]{Mardesic-Segal-MR676973}, there is the 
\u{C}ech
$H\left( \mathbf{Pol}\right) $-expansion%
\begin{equation*}
X\longrightarrow \mathbf{Y}\mathcal{=}\left( \left\vert N\mathcal{U}%
\right\vert \right) _{\mathcal{U}}\in \mathbf{Pro}\left( H\left( \mathbf{Pol}%
\right) \right) ,
\end{equation*}%
where $\mathcal{U}$ runs through \textbf{normal} coverings of $X$, and $%
\left\vert N\mathcal{U}\right\vert $ is the geometric realization of the 
\u{C}ech
nerve $N\mathcal{U}$ of $\mathcal{U}$.

\begin{enumerate}
\item Since all coverings of a Hausdorff paracompact space are normal \cite[%
Corollary App. 1.1, p. 325]{Mardesic-Segal-MR676973}, $pro$-$H_{n}\left(
X,G\right) \in \mathbf{Pro}\left( \mathbf{Ab}\right) $ is the following
pro-group.

\begin{enumerate}
\item For each $\mathcal{U}=\left\{ U_{i}\rightarrow X\right\} _{i\in I}$
the (singular) homology of $\left\vert N\mathcal{U}\right\vert $ is
naturally isomorphic to the \textbf{cellular} homology of $N\mathcal{U}$,
i.e.,%
\begin{equation*}
H_{n}\left( \left\vert N\mathcal{U}\right\vert ,G\right) 
\simeq%
H_{n}\left( C_{\bullet }^{\mathcal{U}}\right)
\end{equation*}%
where%
\begin{equation*}
C_{n}^{\mathcal{U}}=\dbigoplus\limits_{\substack{ \left( i_{0},i_{1},\dots
,i_{n}\right) \in I^{n+1}  \\ U_{i_{0}}\cap \dots \cap U_{i_{n}}\neq
\varnothing }}G
\end{equation*}%
with evident differentials.

\item Finally,%
\begin{equation*}
pro\text{-}H_{n}\left( X,G\right) 
\simeq%
\left( H_{n}\left( C_{\bullet }^{\mathcal{U}}\right) \right) _{\mathcal{U}}%
\simeq%
\underset{\mathcal{U}}{\underleftarrow{\lim }}~H_{n}\left( C_{\bullet }^{%
\mathcal{U}}\right)
\end{equation*}%
where $\underleftarrow{\lim }$ is taken in the \textbf{pro}category $\mathbf{%
Pro}\left( \mathbf{Ab}\right) $.
\end{enumerate}

\item Corollary \ref{Cor-Homology-Paracompact} and Proposition \ref%
{Prop-Reduced-Complex} establish the natural isomorphism%
\begin{equation*}
H_{n}\left( X,G_{\#}\right) \longrightarrow \check{H}_{n}\left( X,G\right) 
\simeq%
\check{H}_{n}^{red}\left( X,G\right) .
\end{equation*}

\item It follows, due to \cite[Theorem 3.2.1(3ab)]%
{Prasolov-Cosheaves-2021-MR4347662},\ that%
\begin{equation*}
H_{n}\left( X,G_{\#}\right) 
\simeq%
\check{H}_{n}^{red}\left( X,G\right) 
\simeq%
\underset{\mathcal{U}}{\underleftarrow{\lim }}~H_{n}\left( C_{\bullet
}^{red}\left( \mathcal{U},G\right) \right)
\end{equation*}%
where $C_{\bullet }^{red}\left( \mathcal{U},G\right) $ is the following
chain complex:%
\begin{equation*}
C_{n}^{red}\left( \mathcal{U},G\right) =\dbigoplus\limits_{\substack{ \left(
i_{0},i_{1},\dots ,i_{n}\right) \in I^{n+1}  \\ U_{i_{0}}\cap \dots \cap
U_{i_{n}}\neq \varnothing }}G.
\end{equation*}

\item Notice that all $H_{n}\left( C_{\bullet }^{red}\left( \mathcal{U}%
,G\right) \right) $ are \textbf{rudimentary} pro-groups naturally isomorphic
to \textbf{rudimentary} pro-groups $H_{n}\left( \left\vert N\mathcal{U}%
\right\vert ,G\right) $.

\item Finally, take the limit $\underleftarrow{\lim }_{\mathcal{U}}$ in the 
\textbf{pro}category $\mathbf{Pro}\left( \mathbf{Ab}\right) $:%
\begin{equation*}
H_{n}\left( X,G_{\#}\right) 
\simeq%
\underset{\mathcal{U}}{\underleftarrow{\lim }}~H_{n}\left( C_{\bullet
}^{red}\left( \mathcal{U},G\right) \right) 
\simeq%
\underset{\mathcal{U}}{\underleftarrow{\lim }}~H_{n}\left( C_{\bullet }^{%
\mathcal{U}}\right) 
\simeq%
pro\text{-}H_{n}\left( X,G\right) ,
\end{equation*}%
and we are done.
\end{enumerate}
\end{proof}

\begin{corollary}
\label{Cor-Cosheaf-homology-polyhedra}If $X\in \mathbf{Pol}$, then $%
H_{n}\left( X,G_{\#}\right) 
\simeq%
H_{n}^{sing}\left( X,G\right) $ where $H_{\bullet }^{sing}$ is the singular
homology.
\end{corollary}

\begin{remark}
\label{Rem-Cosheaf-homology-polyhedra}When $X\in \mathbf{Pol}$, all $%
H_{n}\left( X,G_{\#}\right) $ are \textbf{rudimentary} pro-objects.
\end{remark}

\section{\label{Sec-A-spaces}(Pre)cosheaves on Alexandroff spaces}

The goal of this section is to calculate $H_{\ast }\left( X,G_{\#}\right) $
for $A$-spaces. Those spaces were introduced in \cite%
{Alexandroff-1937-Diskrete-Raume} under the name \textquotedblleft discrete
spaces\textquotedblright . Mc Cord \cite{McCord-MR0196744} argues that now
\textquotedblleft discrete spaces\textquotedblright\ mean those spaces in
which all subsets are open, and call the spaces from \cite%
{Alexandroff-1937-Diskrete-Raume} \textquotedblleft Alexandroff
spaces\textquotedblright\ or shortly \textquotedblleft $A$%
-spaces\textquotedblright :

\begin{definition}
\label{Def-A-space}A topological space $X$ is an $A$-space iff any
intersection of open subsets of $X$ is open. An $A$-space that satisfies $%
T_{0}$ is called $T_{0}A$-space.
\end{definition}

\begin{remark}
\label{Rem-A-spaces}(see \cite[p. 466]{McCord-MR0196744}) There are \textbf{%
proper} inclusions of the following classes of topological spaces:%
\begin{equation*}
\left\{ \text{finite spaces}\right\} \underset{\neq }{\subset }\left\{ \text{%
locally finite spaces}\right\} \underset{\neq }{\subset }\left\{ A\text{%
-spaces}\right\} .
\end{equation*}%
\textquotedblleft Locally finite spaces\textquotedblright\ are those spaces $%
X$, that admit a finite open neighborhood for any point $x\in X$.
\end{remark}

\subsection{General properties}

Below is the list of facts about $A$-spaces.

\begin{proposition}
\label{Prop-Facts-A-spaces}Let $X$ be an $A$-space.

\begin{enumerate}
\item For any $x\in X$ there exists a minimal open neighborhood $U_{x}$ of $%
x $. Namely, $U_{x}$ is the intersection of all open neighborhoods of $x$.

\item There is an equivalence between the two categories $\mathbf{P}$ and $%
\mathbf{A}$, where $\mathbf{P}$ is the category of pre-ordered sets and
monotone mappings, while $\mathbf{A}$ is the category of $A$-spaces and
continuous mappings.

\begin{enumerate}
\item The functor $\mathbf{A}\rightarrow \mathbf{P}$. Let $X\in \mathbf{A}$.
There is a pre-ordering (transitive and reflective binary relation) $\leq $
on the set of points of $X$:%
\begin{equation*}
x\leq y\iff x\in U_{y}\iff U_{x}\subseteq U_{y}.
\end{equation*}

\item The functor $\mathbf{P}\rightarrow \mathbf{A}$. Let $\left( X,\leq
\right) \in \mathbf{P}$. Introduce a topology on $X$ with the basis%
\begin{equation*}
\left\{ V_{x}=\left\{ y\in X~|~y\leq x\right\} \right\} _{x\in X}.
\end{equation*}%
A subset $W\subseteq X$ is open iff $W$ contains all predecessors of any $%
x\in W$. $X$ becomes clearly an $A$-space, and $U_{x}=V_{x}$ for any $x\in X$%
.
\end{enumerate}

\item $X$ satisfies $T_{0}$ iff the above pre-ordering is anti-symmetric,
i.e. $\left( X,\leq \right) $ is a poset.

\item Let $R$ be the following equivalence relation:%
\begin{equation*}
xRy\iff \left( x\leq y\right) \&\left( y\leq x\right) ,
\end{equation*}%
and let%
\begin{equation*}
\nu _{X}:X\longrightarrow \widehat{X}%
{:=}%
X/R
\end{equation*}%
be the projection of $X$ onto the corresponding quotient space. Then:

\begin{enumerate}
\item $\widehat{X}$ is a $T_{0}A$-space.

\item \label{Prop-Facts-A-spaces-Weak-equivalence}$\nu _{X}$ is a homotopy
equivalence.
\end{enumerate}

\item \label{Prop-Facts-A-spaces-Nerve}For a $T_{0}A$-space $X$, let $%
\mathcal{K}\left( X\right) 
{:=}%
N\mathbf{X}$ be the nerve of the category $\mathbf{X}$ corresponding to the
poset $X$, see Notation \ref{Not-Categories}(\ref%
{Not-Categories-Category-for-poset}). There is a weak equivalence 
\begin{equation*}
f_{X}:\left\vert \mathcal{K}\left( X\right) \right\vert \longrightarrow X,
\end{equation*}%
where $\left\vert \mathcal{K}\left( X\right) \right\vert $ is the geometric
realization of $\mathcal{K}\left( X\right) $. $f_{X}$ is defined as follows:

\begin{enumerate}
\item If $u\in \left\vert \mathcal{K}\left( X\right) \right\vert $ is a
vertex $u=\left( x\right) $, then $f_{X}\left( u\right) =x$.

\item If $u$ belongs to the interior of the $n$-simplex%
\begin{equation*}
\left( x_{0}<x_{1}<\dots <x_{n}\right) ,
\end{equation*}%
then $f_{X}\left( u\right) =x_{0}$. See the nice picture \cite[Fig. 1.3]%
{Barmak-2011-Algebraic-topology-finite-spaces-MR3024764}.
\end{enumerate}
\end{enumerate}
\end{proposition}

\begin{proof}
See \cite{McCord-MR0196744}.
\end{proof}

\subsection{Sheaves on $A$-spaces}

In \cite[\S 1]{Jensen-MR0407091} the higher limits $\underleftarrow{\lim }%
^{s}$ are defined for functors%
\begin{equation*}
A:\mathbf{I}^{op}\longrightarrow \mathbf{Ab}
\end{equation*}%
where $I$ is a \textbf{directed} poset. There is also established a
bijection between such functors and $\mathbf{Ab}$-valued sheaves on the
corresponding topological space $X\left( I\right) $. It is proved that%
\begin{equation*}
\underleftarrow{\lim }^{s}A%
\simeq%
H^{s}\left( X\left( I\right) ,\mathcal{A}\right)
\end{equation*}%
where $\mathcal{A}$ is the corresponding sheaf on $X\left( I\right) $.
However, it is possible to prove the above two statements for \textbf{%
arbitrary} (\textbf{not} necessarily \textbf{directed}) posets, and even for 
\textbf{pre-ordered} sets, like in Proposition \ref{Prop-Facts-A-spaces}.

\begin{theorem}
\label{Th-Poset-vs-sheaves}Let $\left( X,\leq \right) $ be a pre-ordered
set. Denote by the same letter $X$ the corresponding $A$-space, and by $%
\mathbf{X}$ the corresponding category, see Notation \ref{Not-Categories}(%
\ref{Not-Categories-Category-for-poset}).

\begin{enumerate}
\item \label{Th-Poset-vs-sheaves-categories-equivalent}There is an
equivalence between the category $\mathbf{Mod}\left( k\right) ^{\mathbf{X}%
^{op}}$ of functors $\mathbf{X}^{op}\rightarrow \mathbf{Mod}\left( k\right) $
and the category $\mathbf{S}\left( X,\mathbf{Mod}\left( k\right) \right) $
of sheaves on $X$.

\item \label{Th-Poset-vs-sheaves-higher-limits}There is a natural (in $X$
and $F$) isomorphism%
\begin{equation*}
\underleftarrow{\lim }^{\bullet }F%
\simeq%
H^{\bullet }\left( X,\mathcal{F}\right)
\end{equation*}%
where $\mathcal{F}$ is the sheaf corresponding to the functor $F$.

\item \label{Th-Poset-vs-sheaves-Bar-construction}\label%
{Rem-Bar-construction}There is a description (the \textbf{bar-construction})
of $\underleftarrow{\lim }^{\bullet }~F$ for $F\in \mathbf{Mod}\left(
k\right) ^{\mathbf{I}^{op}}$ when $\left( I,\leq \right) $ is a poset, or a
pre-ordered set, or even when $\mathbf{I}$ is an \textbf{arbitrary} small
category:%
\begin{equation*}
\underleftarrow{\lim }^{\bullet }~F=H^{\bullet }\left( C^{\bullet }\left( 
\mathbf{I},F\right) \right)
\end{equation*}%
where%
\begin{equation*}
C^{n}\left( \mathbf{I},F\right) =\dprod\limits_{i_{0}\rightarrow
i_{1}\rightarrow \dots \rightarrow i_{n}}F\left( i_{0}\right)
\end{equation*}%
and $d^{n}:C^{n}\left( \mathbf{I},F\right) \rightarrow C^{n+1}\left( \mathbf{%
I},F\right) $ is given by%
\begin{eqnarray*}
&&\left( d^{n}\varphi \right) \left( i_{0}\rightarrow i_{1}\rightarrow \dots
\rightarrow i_{n+1}\right) \\
&=&F\left( i_{0}\rightarrow i_{1}\right) \left( \varphi \left(
i_{1}\rightarrow i_{2}\rightarrow \dots \rightarrow i_{n+1}\right) \right)
+\dsum\limits_{k=1}^{n+1}\left( -1\right) ^{k}\varphi \left(
i_{0}\rightarrow \dots \rightarrow \widehat{i_{k}}\rightarrow \dots
\rightarrow i_{n+1}\right) .
\end{eqnarray*}%
Compare to \cite[Definition B.2.1(1)]{Prasolov-Cosheaves-2021-MR4347662}. 
\textbf{Remark}: in fact, we do not require here $k$ to be quasi-noetherian
(or even commutative).

\item \label{Th-Poset-vs-sheaves-Quotient-space}Define the binary relation $%
R $ on $X$ by%
\begin{equation*}
xRy\iff \left( x\leq y\right) \&\left( y\leq x\right) ,
\end{equation*}%
let $\widehat{X}=X/R$ be the corresponding poset, and let%
\begin{equation*}
\nu _{X}:X\longrightarrow \widehat{X}
\end{equation*}%
be the projection. Denote by the same letter $\widehat{X}$ the corresponding 
$T_{0}A$-space, and by $\widehat{\mathbf{X}}$ the corresponding category.
Then:

\begin{enumerate}
\item $\left( \nu _{X}\right) ^{-1}$ establishes a bijection between open
subsets of $\widehat{X}$ and $X$.

\item $\left( \nu _{X}\right) ^{-1}$ establishes a categorical equivalence $%
\mathcal{F}\longleftrightarrow \widehat{\mathcal{F}}$ between sheaves on $X$
and sheaves on $\widehat{X}$.

\item The above bijection preserves $H^{\bullet }\left( \bullet ,\bullet
\right) $, i.e., there is a natural (in $X$ and $\mathcal{F}$) isomorphism%
\begin{equation*}
H^{\bullet }\left( \widehat{X},\widehat{\mathcal{F}}\right) 
\simeq%
H^{\bullet }\left( X,\mathcal{F}\right) .
\end{equation*}

\item For any $\widehat{F}:\left( \widehat{\mathbf{X}}\right)
^{op}\rightarrow \mathbf{Mod}\left( k\right) $%
\begin{equation*}
\left( \nu _{X}\right) ^{\ast }:\underleftarrow{\lim }^{\bullet }\widehat{F}%
\longrightarrow \underleftarrow{\lim }^{\bullet }\left( \widehat{F}\circ
\left( \nu _{X}\right) ^{op}\right)
\end{equation*}%
is an isomorphism.
\end{enumerate}
\end{enumerate}
\end{theorem}

\begin{proof}
~

\begin{enumerate}
\item Given a functor $F:\mathbf{X}^{op}\rightarrow \mathbf{Mod}\left(
k\right) $, let $\mathcal{F}$ be the following \textbf{pre}sheaf (in fact, a
sheaf, which will be seen later):%
\begin{equation*}
\mathcal{F}\left( U\right) 
{:=}%
\underleftarrow{\lim }~\left( F|_{U}\right) .
\end{equation*}%
Given a sheaf $\mathcal{G}$, let $\mathcal{G}_{x}$ be the stalk at $x$.
Since $U_{x}$ is a minimal open neighborhood of $x$,%
\begin{equation*}
\mathcal{G}_{x}%
\simeq%
\mathcal{G}\left( U_{x}\right) .
\end{equation*}%
If $x\leq y$, i.e., $U_{x}\subseteq U_{y}$, one has the homomorphism%
\begin{equation*}
\mathcal{G}\left( U_{x}\rightarrow U_{y}\right) :\mathcal{G}_{y}=\mathcal{G}%
\left( U_{y}\right) \longrightarrow \mathcal{G}\left( U_{x}\right) =\mathcal{%
G}_{x}.
\end{equation*}%
We have clearly obtained a functor%
\begin{equation*}
G:\mathbf{X}^{op}\rightarrow \mathbf{Mod}\left( k\right) .
\end{equation*}

\begin{enumerate}
\item $\mathcal{F}$ is a sheaf. Indeed, let $\mathcal{U}=\left( U_{i}\right)
_{i\in I}$, and%
\begin{equation*}
U=\dbigcup\limits_{i\in I}U_{i}.
\end{equation*}%
Consider the homomorphism%
\begin{eqnarray*}
\alpha _{\mathcal{U}} &:&\underleftarrow{\lim }~\left( F|_{U}\right) =%
\mathcal{F}\left( U\right) \longrightarrow \ker \left( \dprod\limits_{i\in I}%
\mathcal{F}\left( U_{i}\right) \rightrightarrows \dprod\limits_{i,j\in I}%
\mathcal{F}\left( U_{i}\cap U_{j}\right) \right) = \\
&=&\ker \left( \dprod\limits_{i\in I}\underleftarrow{\lim }~\left(
F|_{U_{i}}\right) \rightrightarrows \dprod\limits_{i,j\in I}\underleftarrow{%
\lim }~\left( F|_{U_{i}\cap U_{j}}\right) \right) .
\end{eqnarray*}%
\ 

\begin{enumerate}
\item $\alpha _{\mathcal{U}}$ is a \textbf{mono}morphism. Indeed, let $f\in 
\underleftarrow{\lim }~\left( F|_{U}\right) $ be given by the $U$-tuple $%
\left( f_{x}\in F\left( x\right) \right) _{x\in U}$ and assume that $\alpha
_{\mathcal{U}}\left( f\right) =0$. It follows that $f_{x}=0$ for any $i\in I$
and $x\in U_{i}$. Since $U=\cup _{i\in I}U_{i}$, $f_{x}=0$ for all $x\in U$,
and $f=0$.

\item $\alpha _{\mathcal{U}}$ is an \textbf{epi}morphism. Indeed, let%
\begin{equation*}
g\in \ker \left( \dprod\limits_{i\in I}\underleftarrow{\lim }~\left(
F|_{U_{i}}\right) \rightrightarrows \dprod\limits_{i,j\in I}\underleftarrow{%
\lim }~\left( F|_{U_{i}\cap U_{j}}\right) \right)
\end{equation*}%
be given by the $I$-tuple $\left( g_{i}\right) _{i\in I}$, where%
\begin{equation*}
g_{i}=\left( h_{i,x}\in F\left( x\right) \right) _{x\in U_{i}}\in 
\underleftarrow{\lim }~\left( F|_{U_{i}}\right) .
\end{equation*}%
It follows that $h_{i,x}=h_{j,x}$ for any $i,j\in I$ and $x\in U_{i}\cap
U_{j}$, therefore there is an $U$-tuple $\left( p_{x}\in F\left( x\right)
\right) _{x\in U}$ such that $p_{x}=h_{i,x}$ for any $i\in I$ and $x\in
U_{i} $. Let $x,y\in U$ and $x\leq y$. There exist $i,j\in I$ such that $%
x\in U_{i} $ and $y\in U_{j}$. Since $U_{j}$ is open (therefore contains all
predecessors of $y$), $x\in U_{i}\cap U_{j}$. The $U_{j}$-tuple%
\begin{equation*}
\left( p_{z}=h_{j,z}\right) _{z\in U_{j}}
\end{equation*}%
belongs to $\underleftarrow{\lim }~\left( F|_{U_{j}}\right) $, therefore%
\begin{equation*}
p_{x}=h_{j,x}=F\left( x\leq y\right) \left( h_{j,y}\right) =F\left( x\leq
y\right) \left( p_{y}\right) .
\end{equation*}%
It follows that%
\begin{equation*}
\left( p_{x}\right) _{x\in U}\in \underleftarrow{\lim }~\left( F|_{U}\right)
,
\end{equation*}%
and $\alpha _{\mathcal{U}}\left( \left( p_{x}\right) _{x\in U}\right) =g$.
\end{enumerate}

\item ~

\begin{enumerate}
\item The composition $F\mapsto \mathcal{F=G}\mapsto G$ is the identity (up
to an isomorphism), i.e., $G%
\simeq%
F$. Indeed, for any $x\in X$%
\begin{equation*}
G\left( x\right) =\mathcal{G}_{x}=\mathcal{F}_{x}=\mathcal{F}\left(
U_{x}\right) =\underleftarrow{\lim }~\left( F|_{U_{x}}\right) =F\left(
x\right)
\end{equation*}%
because $x$ is a \textbf{terminal} object of the category $\mathbf{U}_{x}$
(therefore an \textbf{initial} object of $\left( \mathbf{U}_{x}\right) ^{op}$%
), corresponding to the pre-ordered set $U_{x}$.

\item The composition $\mathcal{F}\mapsto F=G\mapsto \mathcal{G}$ is the
identity (up to an isomorphism), i.e., $\mathcal{G}%
\simeq%
\mathcal{F}$. Indeed, for any open subset $U\subseteq X$%
\begin{equation*}
\mathcal{G}\left( U\right) =\underleftarrow{\lim }~\left( G|_{U}\right) =%
\underleftarrow{\lim }~\left( F|_{U}\right) .
\end{equation*}%
The elements of $\underleftarrow{\lim }~\left( F|_{U}\right) $ are
represented by $U$-tuples%
\begin{equation*}
\left( f_{x}\in \mathcal{F}_{x}=\mathcal{F}\left( U_{x}\right) \right)
_{x\in U},
\end{equation*}%
satisfying%
\begin{equation*}
f_{x}=\mathcal{F}\left( U_{x}\rightarrow U_{y}\right) \left( f_{y}\right)
\end{equation*}%
for any pair $x\leq y$. Since $\mathcal{F}$ is a sheaf,%
\begin{equation*}
\alpha _{\mathcal{U}}:\mathcal{F}\left( U\right) \longrightarrow \ker \left(
\dprod\limits_{x\in U}\mathcal{F}\left( U_{x}\right) \rightrightarrows
\dprod\limits_{x,y\in U}\mathcal{F}\left( U_{x}\cap U_{y}\right) \right)
\end{equation*}%
is an isomorphism for the open covering $\mathcal{U=}\left( U_{x}\right)
_{x\in U}$ of $U$. It follows that the elements of $\mathcal{F}\left(
U\right) $ are represented by $U$-tuples%
\begin{equation*}
\left( g_{x}\in \mathcal{F}_{x}=\mathcal{F}\left( U_{x}\right) \right)
_{x\in U},
\end{equation*}%
satisfying%
\begin{equation*}
\mathcal{F}\left( U_{x}\cap U_{y}\rightarrow U_{x}\right) \left(
g_{x}\right) =\mathcal{F}\left( U_{x}\cap U_{y}\rightarrow U_{y}\right)
\left( g_{y}\right) \in \mathcal{F}\left( U_{x}\cap U_{y}\right)
\end{equation*}%
for any pair $x,y\in U$. We may therefore assume that both $\underleftarrow{%
\lim }~\left( F|_{U}\right) $ and $\mathcal{F}\left( U\right) $ lie in the
same set:%
\begin{equation}
\mathcal{F}\left( U\right) \subseteq \underleftarrow{\lim }~\left(
F|_{U}\right) \subseteq \dprod\limits_{x\in U}\mathcal{F}\left( U_{x}\right)
.  \label{Line-F(U)-vs-lim-F}
\end{equation}
\end{enumerate}

\item The first inclusion follows from the \textbf{stronger} conditions on $%
\left( g_{x}\right) $ than on $\left( f_{x}\right) $. The conditions on $%
\left( f_{x}\right) $ are valid only for pairs $x\leq y$, while the other
conditions are valid for \textbf{all} pairs $x,y\in U$. Consider $x,y\in X$
and define%
\begin{equation*}
\xi \left( x,y\right) =\mathcal{F}\left( U_{x}\cap U_{y}\rightarrow
U_{x}\right) \left( f_{x}\right) -\mathcal{F}\left( U_{x}\cap
U_{y}\rightarrow U_{y}\right) \left( f_{y}\right) \in \mathcal{F}\left(
U_{x}\cap U_{y}\right) .
\end{equation*}%
For each $z\in U_{x}\cap U_{y}$%
\begin{equation*}
\xi \left( x,y\right) |_{U_{z}}=\xi \left( x,z\right) -\xi \left( y,z\right)
=0-0=0\in \mathcal{F}\left( U_{z}\right) .
\end{equation*}%
Since $\mathcal{F}$ is a sheaf,%
\begin{equation*}
\beta :\mathcal{F}\left( U_{x}\cap U_{y}\right) \longrightarrow
\dprod\limits_{z\in U_{x}\cap U_{y}}\mathcal{F}\left( U_{z}\right)
\end{equation*}%
is a \textbf{mono}morphism. $\beta \left( \xi \left( x,y\right) \right) =0$
implies $\xi \left( x,y\right) =0$. It follows that $\left( f_{x}\right)
_{x\in U}\in \mathcal{F}\left( U\right) $,%
\begin{equation*}
\mathcal{G}\left( U\right) =\underleftarrow{\lim }~\left( F|_{U}\right) =%
\mathcal{F}\left( U\right)
\end{equation*}%
and $\mathcal{G}=\mathcal{F}$.
\end{enumerate}

\item Since the two abelian categories $\mathbf{Mod}\left( k\right) ^{%
\mathbf{X}^{op}}$ and $\mathbf{S}\left( X,\mathbf{Mod}\left( k\right)
\right) $ are equivalent, and the following diagram%
\begin{equation*}
\begin{diagram}[size=3.0em,textflow]
\mathbf{Mod}\left( k\right) ^{\mathbf{X}^{op}} & 
\TeXButton{ISO}{\simeq} & \mathbf{S}\left( X,\mathbf{Mod}\left( k\right) \right) \\ 
\dTo_{\underleftarrow{\lim }} &  & \dTo_{\Gamma \left( X,\bullet
\right) =\bullet \left( X\right)} \\ 
\mathbf{Mod}\left( k\right) & = & \mathbf{Mod}\left( k\right)\\
\end{diagram}%
\end{equation*}%
commutes up to an isomorphism of functors, the two right satellites%
\begin{eqnarray*}
\underleftarrow{\lim }^{\bullet }~ &=&R^{\bullet }~\underleftarrow{\lim }~:%
\mathbf{Mod}\left( k\right) ^{\mathbf{X}^{op}}\longrightarrow \mathbf{Mod}%
\left( k\right) , \\
H^{\bullet }\left( X,\bullet \right) &=&R^{\bullet }\Gamma \left( X,\bullet
\right) :\mathbf{S}\left( X,\mathbf{Mod}\left( k\right) \right)
\longrightarrow \mathbf{Mod}\left( k\right) ,
\end{eqnarray*}%
are naturally isomorphic.

\item ~

\begin{enumerate}
\item Apply the \textbf{cosimplicial replacement} $\Pi ^{\bullet }$ \cite[%
Section XI.5.2]{Bousfield-Kan-MR0365573} to the functor $F\in \mathbf{Mod}%
\left( k\right) ^{\mathbf{I}^{op}}$ and then consider its cohomotopy groups $%
\pi ^{\bullet }$ \cite[Section X.7.1]{Bousfield-Kan-MR0365573}. Then%
\begin{equation*}
\pi ^{\bullet }\left( \Pi ^{\bullet }F\right) 
\simeq%
H^{\bullet }\left( C^{\bullet }\left( \mathbf{I},F\right) \right) .
\end{equation*}%
Due to \cite[Proposition XI.6.2]{Bousfield-Kan-MR0365573}%
\begin{equation*}
\underleftarrow{\lim }^{\bullet }~F%
\simeq%
\pi ^{\bullet }\left( \Pi ^{\bullet }F\right) 
\simeq%
H^{\bullet }\left( C^{\bullet }\left( \mathbf{I},F\right) \right) ,
\end{equation*}%
and we are done.

\item However, since the proof of the cited proposition is only sketched,
let us give here the complete proof. Moreover, the proposition is valid for
functors $F\in \mathbf{Ab}^{\mathbf{I}^{op}}=\mathbf{Mod}\left( \mathbb{Z}%
\right) ^{\mathbf{I}^{op}}$, while we use functors $F\in \mathbf{Mod}\left(
k\right) ^{\mathbf{I}^{op}}$ for a \textbf{general} ring $k$. Notice that in 
\cite[Corollary 11.47]{Mardesic-Strong-shape-and-homology-MR1740831} the
desired statement is proved for $F\in \mathbf{Mod}\left( k\right) ^{\mathbf{I%
}^{op}}$ where $k$ is not necessarily commutative. However, $\mathbf{I}$
there is the category corresponding to a \textbf{directed poset} $\left(
I,\leq \right) $.

\item Let $i\in \mathbf{I}$ and%
\begin{equation*}
\rho _{i}%
{:=}%
\left( F\longmapsto F\left( i\right) \right) :\mathbf{Mod}\left( k\right) ^{%
\mathbf{I}^{op}}\longrightarrow \mathbf{Mod}\left( k\right) .
\end{equation*}%
Define 
\begin{equation*}
\left( \rho _{i}\right) ^{\ast }:\mathbf{Mod}\left( k\right) \longrightarrow 
\mathbf{Mod}\left( k\right) ^{\mathbf{I}^{op}}
\end{equation*}%
by the following. Let $X\in \mathbf{Mod}\left( k\right) $, $Y\in \mathbf{Mod}%
\left( k\right) ^{\mathbf{I}^{op}}$, and let%
\begin{equation*}
\left( \rho _{i}\right) ^{\ast }X%
{:=}%
Z\in \mathbf{Mod}\left( k\right) ^{\mathbf{I}^{op}}
\end{equation*}%
be the following functor:%
\begin{equation*}
Z\left( j\right) =\dprod\limits_{i\rightarrow j}X=\left\{ \text{tuples }%
\left( x_{i\rightarrow j}\in X\right) _{\left( i\rightarrow j\right) \in 
\Hom%
_{\mathbf{I}}\left( i,j\right) }\right\} .
\end{equation*}%
If $\left( k\rightarrow j\right) \in 
\Hom%
_{\mathbf{I}}\left( k,j\right) $, let%
\begin{equation*}
Z\left( k\rightarrow j\right) \left( \left( x_{i\rightarrow j}\right)
_{i\rightarrow j}\right) =\left( z_{i\rightarrow k}%
{:=}%
x_{i\rightarrow k\rightarrow j}\right) _{i\rightarrow k}\in Z\left( k\right)
.
\end{equation*}%
\newline
Clearly, $Z\in \mathbf{Mod}\left( k\right) ^{\mathbf{I}^{op}}$. We claim
that $\rho _{i}\dashv \left( \rho _{i}\right) ^{\ast }$.

\begin{enumerate}
\item Given%
\begin{equation*}
\alpha :Y\left( i\right) =\rho _{i}\left( Y\right) \longrightarrow X,
\end{equation*}%
let $\beta =\alpha ^{\ast }:Y\rightarrow Z$ be the following morphism:%
\begin{equation*}
\beta _{j}\left( y\in Y\left( j\right) \right) 
{:=}%
\left( x_{i\rightarrow j}%
{:=}%
\alpha \left( Y\left( i\rightarrow j\right) \left( y\right) \right) \right)
_{i\rightarrow j}\in Z\left( j\right) .
\end{equation*}%
$\beta $ is a morphism of functors. Indeed, the diagram%
\begin{equation*}
\begin{diagram}[size=3.0em,textflow]
Y\left( j\right) & \rTo^{\beta _{j}} & Z\left( j\right)
\\ 
\dTo^{Y\left( k\rightarrow j\right)} &  & \dTo_{Z\left(
k\rightarrow j\right)} \\ 
Y\left( k\right) & \lTo^{\beta _{k}} & Z\left( k\right) \\
\end{diagram}
%
\end{equation*}%
is commutative because%
\begin{eqnarray*}
&=&Z\left( k\rightarrow j\right) \circ \beta _{j}\left( y\in Y\left(
j\right) \right) =Z\left( k\rightarrow j\right) \left( \left(
x_{i\rightarrow j}=\alpha \left( Y\left( i\rightarrow j\right) \left(
y\right) \right) \right) _{i\rightarrow j}\right) \\
&=&\left( z_{i\rightarrow k}%
{:=}%
x_{i\rightarrow k\rightarrow j}=\alpha \left( Y\left( i\rightarrow
k\rightarrow j\right) \left( y\right) \right) \right) _{i\rightarrow k} \\
&=&\left( \alpha \left( Y\left( i\rightarrow k\right) \circ Y\left(
k\rightarrow j\right) \left( y\right) \right) \right) _{i\rightarrow
k}=\beta _{k}\left( Y\left( k\rightarrow j\right) \left( y\right) \right) .
\end{eqnarray*}

\item Conversely, given $\beta :Y\rightarrow Z$, since%
\begin{equation*}
\beta _{i}\left( y\in Y\left( i\right) \right) =\left( x_{i\rightarrow i}\in
X\right) _{i\rightarrow i},
\end{equation*}%
define%
\begin{eqnarray*}
\alpha &:&Y\left( i\right) \longrightarrow X, \\
\alpha \left( y\right) &=&\left( \beta _{i}\left( y\right) \right) _{\mathbf{%
1}_{i}:i\rightarrow i}.
\end{eqnarray*}

\item Clearly, the correspondences $\alpha \mapsto \beta $ and $\beta
\mapsto \alpha $ define the natural isomorphism%
\begin{eqnarray*}
&&%
\Hom%
_{\mathbf{Mod}\left( k\right) }\left( \rho _{i}\left( Y\right) =Y\left(
i\right) ,X\right) 
\simeq%
\Hom%
_{\mathbf{Mod}\left( k\right) ^{\mathbf{I}^{op}}}\left( Y,\left( \rho
_{i}\right) ^{\ast }X\right) \\
&:&\left( \mathbf{Mod}\left( k\right) ^{\mathbf{I}^{op}}\right) ^{op}\times 
\mathbf{Mod}\left( k\right) \longrightarrow \mathbf{Set,}
\end{eqnarray*}%
and $\rho _{i}\dashv \left( \rho _{i}\right) ^{\ast }$ as desired.
\end{enumerate}

\item Since $\rho _{i}$ is exact, $\left( \rho _{i}\right) ^{\ast }$
converts injective modules into injective functors.

\item The construction of $\left( \rho _{i}\right) ^{\ast }$ shows that it
is exact. Therefore $\rho _{i}$ converts projective functors into projective
modules.

\item For any $j\in \mathbf{I}$%
\begin{equation*}
H^{s}\left( C^{\bullet }\left( \mathbf{I},\left( \rho _{j}\right) ^{\ast
}X\right) \right) =\left\{ 
\begin{array}{ccc}
0 & \text{if} & s>0 \\ 
X & \text{if} & s=0%
\end{array}%
\right.
\end{equation*}%
Indeed, the elements of $C^{n}\left( \mathbf{I},\left( \rho _{j}\right)
^{\ast }X\right) $ can be represented by tuples%
\begin{equation*}
\left( \varphi \left( j\rightarrow i_{0}\rightarrow \dots \rightarrow
i_{n}\right) \in X\right) _{j\rightarrow i_{0}\rightarrow \dots \rightarrow
i_{n}}.
\end{equation*}%
Moreover,%
\begin{eqnarray*}
&&\left( d^{n}\varphi \right) \left( j\rightarrow i_{0}\rightarrow \dots
\rightarrow i_{n+1}\right) \\
&=&\dsum\limits_{s=0}^{n+1}\left( -1\right) ^{s}\varphi \left( j\rightarrow
i_{0}\rightarrow \dots \rightarrow \widehat{i_{s}}\rightarrow \dots
\rightarrow i_{n+1}\right) .
\end{eqnarray*}%
Consider the augmented cochain complex%
\begin{equation*}
0\longrightarrow \left( C^{-1}\left( \mathbf{I},\left( \rho _{j}\right)
^{\ast }X\right) 
{:=}%
X\right) \overset{d^{-1}}{\longrightarrow }C^{0}\left( \mathbf{I},\left(
\rho _{j}\right) ^{\ast }X\right) \longrightarrow \dots \longrightarrow
C^{n}\left( \mathbf{I},\left( \rho _{j}\right) ^{\ast }X\right)
\longrightarrow \dots
\end{equation*}%
where $d^{-1}$ sends $x\in X$ to the constant function%
\begin{equation*}
\left( \varphi \left( i_{0}\right) \left( j\rightarrow i_{0}\right)
=x\right) _{i_{0}\in \mathbf{I}}.
\end{equation*}%
We claim that it is cochain contractible. Let us build the contraction $%
S^{\bullet }$:

\begin{enumerate}
\item For $n\geq 1$, define $S^{n-1}:C^{n}\left( \mathbf{I},\left( \rho
_{j}\right) ^{\ast }X\right) \longrightarrow C^{n-1}\left( \mathbf{I},\left(
\rho _{j}\right) ^{\ast }X\right) $. Given%
\begin{equation*}
\left( \varphi \left( j\rightarrow i_{0}\rightarrow \dots \rightarrow
i_{n}\right) \in X\right) _{j\rightarrow i_{0}\rightarrow \dots \rightarrow
i_{n}}\in C^{n}\left( \mathbf{I},\left( \rho _{j}\right) ^{\ast }X\right) ,
\end{equation*}%
let%
\begin{eqnarray*}
&&\left( S^{n-1}\varphi \right) \left( j\rightarrow i_{0}\rightarrow \dots
\rightarrow i_{n-1}\right) \\
&=&\varphi \left( j\overset{\mathbf{1}_{j}}{\rightarrow }j\rightarrow
i_{0}\rightarrow \dots \rightarrow i_{n-1}\right) .
\end{eqnarray*}

\item It follows that 
\begin{eqnarray*}
&&\left( d^{n-1}\circ S^{n-1}\varphi \right) \left( j\rightarrow
i_{0}\rightarrow \dots \rightarrow i_{n}\right) \\
&=&\varphi \left( j\rightarrow i_{0}\rightarrow \dots \rightarrow
i_{n}\right) + \\
&&-\dsum\limits_{s=1}^{n}\left( -1\right) ^{s}\varphi \left( j\overset{%
\mathbf{1}_{j}}{\rightarrow }j\rightarrow i_{0}\rightarrow \dots \rightarrow 
\widehat{i_{s}}\rightarrow \dots \rightarrow i_{n}\right)
\end{eqnarray*}%
and%
\begin{eqnarray*}
&&\left( S^{n}\circ d^{n}\varphi \right) \left( j\rightarrow
i_{0}\rightarrow \dots \rightarrow i_{n}\right) \\
&=&\dsum\limits_{s=1}^{n}\left( -1\right) ^{s}\varphi \left( j\overset{%
\mathbf{1}_{j}}{\rightarrow }j\rightarrow i_{0}\rightarrow \dots \rightarrow 
\widehat{i_{s}}\rightarrow \dots \rightarrow i_{n}\right) .
\end{eqnarray*}%
Therefore,%
\begin{equation*}
d^{n-1}\circ S^{n-1}+S^{n}\circ d^{n}=\mathbf{1}_{C^{n}\left( \mathbf{I}%
,\left( \rho _{j}\right) ^{\ast }X\right) }
\end{equation*}%
when $n\geq 1$.

\item Let $n=0$. Given 
\begin{equation*}
\left( \varphi \left( j\rightarrow i_{0}\right) \in X\right) _{i_{0}\in 
\mathbf{I}}\in C^{0}\left( \mathbf{I},\left( \rho _{j}\right) ^{\ast
}X\right) ,
\end{equation*}%
let%
\begin{eqnarray*}
&&S^{-1}\varphi 
{:=}%
\varphi \left( \mathbf{1}_{j}:j\rightarrow j\right) , \\
S^{-1} &:&C^{0}\left( \mathbf{I},\left( \rho _{j}\right) ^{\ast }X\right)
\longrightarrow C^{-1}\left( \mathbf{I},\left( \rho _{j}\right) ^{\ast
}X\right) =X.
\end{eqnarray*}%
Then%
\begin{eqnarray*}
\left( S^{0}\circ d^{0}\varphi \right) \left( j\rightarrow i_{0}\right)
&=&\left( d^{0}\varphi \right) \left( j\overset{\mathbf{1}_{j}}{\rightarrow }%
j\rightarrow i_{0}\right) =\varphi \left( j\rightarrow i_{0}\right) -\varphi
\left( j\overset{\mathbf{1}_{j}}{\rightarrow }j\right) , \\
\left( d^{-1}\circ S^{-1}\varphi \right) \left( j\rightarrow i_{0}\right)
&=&\varphi \left( \mathbf{1}_{j}:j\rightarrow j\right) =\varphi \left( j%
\overset{\mathbf{1}_{j}}{\rightarrow }j\right) , \\
S^{0}\circ d^{0}+d^{-1}\circ S^{-1} &=&\mathbf{1}_{C^{0}\left( \mathbf{I}%
,\left( \rho _{j}\right) ^{\ast }X\right) }.
\end{eqnarray*}

\item Finally%
\begin{equation*}
S^{-1}\circ d^{-1}=\mathbf{1}_{X}=\mathbf{1}_{C^{-1}\left( \mathbf{I},\left(
\rho _{j}\right) ^{\ast }X\right) },
\end{equation*}
and the contraction is built.
\end{enumerate}

\item Let now $X\in \mathbf{Mod}\left( k\right) ^{\mathbf{I}^{op}}$, and
define%
\begin{equation*}
\rho ^{\ast }\left( X\right) 
{:=}%
\dprod\limits_{i\in \mathbf{I}}\left( \rho _{i}\right) ^{\ast }X\left(
i\right)
\end{equation*}%
It follows easily that%
\begin{equation*}
H^{s}\left( C^{\bullet }\left( \mathbf{I},\rho ^{\ast }X\right) \right)
=\left\{ 
\begin{array}{ccc}
0 & \text{if} & s>0 \\ 
\dprod\limits_{i\in \mathbf{I}}X\left( i\right) & \text{if} & s=0%
\end{array}%
\right.
\end{equation*}

\item There is a canonical embedding%
\begin{equation*}
\gamma _{X}:X\longrightarrow \rho ^{\ast }\left( X\right)
\end{equation*}%
given by%
\begin{eqnarray*}
\gamma _{X}\left( j\right) &:&X\left( j\right) \longrightarrow \rho ^{\ast
}\left( X\right) \left( j\right) =\dprod\limits_{i\in \mathbf{I}}\left(
\left( \rho _{i}\right) ^{\ast }X\left( i\right) \right) \left( j\right)
=\dprod\limits_{i\in \mathbf{I}}\dprod\limits_{i\rightarrow j}X\left(
i\right) = \\
&=&\left\{ \text{tuples }\left( x_{i\rightarrow j}\in X\left( i\right)
\right) _{\left( i\rightarrow j\right) \in 
\Hom%
_{\mathbf{I}}\left( i,j\right) }\right\} , \\
&&\gamma _{X}\left( j\right) \left( x\in X\left( j\right) \right) 
{:=}%
\left( x_{i\rightarrow j}%
{:=}%
X\left( i\rightarrow j\right) \left( x\right) \right) _{i\rightarrow j}.
\end{eqnarray*}%
$\gamma _{X}$ is clearly a monomorphism. Indeed, if $x,y\in X\left( j\right) 
$ and $x\neq y$, then%
\begin{equation*}
X\left( i\overset{\mathbf{1}_{i}}{\rightarrow }i\right) \left( y\right)
=y\neq x=X\left( i\overset{\mathbf{1}_{i}}{\rightarrow }i\right) \left(
x\right) ,
\end{equation*}%
thus, $\gamma _{X}\left( j\right) \left( x\right) \neq \gamma _{X}\left(
j\right) \left( y\right) $.

\item If $X\in \mathbf{Mod}\left( k\right) ^{\mathbf{I}^{op}}$ is an \textbf{%
injective} functor, then%
\begin{equation*}
H^{s}\left( C^{\bullet }\left( \mathbf{I},X\right) \right) =\left\{ 
\begin{array}{ccc}
0 & \text{if} & s>0 \\ 
\underleftarrow{\lim }~X & \text{if} & s=0%
\end{array}%
\right.
\end{equation*}%
Indeed, let 
\begin{equation*}
\gamma _{X}:X\longrightarrow \rho ^{\ast }\left( X\right)
\end{equation*}%
be the canonical embedding. Since $X$ is injective, $\gamma _{X}$ is a split
monomorphism, therefore%
\begin{equation*}
\rho ^{\ast }\left( X\right) 
\simeq%
X\oplus 
\coker%
\gamma _{X}.
\end{equation*}%
For any $s>0$%
\begin{equation*}
0=H^{s}\left( C^{\bullet }\left( \mathbf{I},\rho ^{\ast }\left( X\right)
\right) \right) 
\simeq%
H^{s}\left( C^{\bullet }\left( \mathbf{I},X\right) \right) \oplus
H^{s}\left( C^{\bullet }\left( \mathbf{I},%
\coker%
\gamma _{X}\right) \right)
\end{equation*}%
and%
\begin{equation*}
H^{s}\left( C^{\bullet }\left( \mathbf{I},X\right) \right) =0.
\end{equation*}%
On the other hand,%
\begin{equation*}
H^{0}\left( C^{\bullet }\left( \mathbf{I},X\right) \right) =\ker \left(
d^{0}:C^{0}\left( \mathbf{I},X\right) \longrightarrow C^{1}\left( \mathbf{I}%
,X\right) \right) .
\end{equation*}%
It is easy to check that the latter module is isomorphic to $\underleftarrow{%
\lim }~X$.

\item Finally, given a functor $X\in \mathbf{Mod}\left( k\right) ^{\mathbf{I}%
^{op}}$, choose an injective coresolution%
\begin{equation*}
0\longrightarrow X\longrightarrow J^{0}\longrightarrow J^{1}\longrightarrow
\dots \longrightarrow J^{n}\longrightarrow \dots
\end{equation*}%
and build the cochain bicomplex%
\begin{equation*}
B^{s,t}%
{:=}%
C^{s}\left( \mathbf{I},J^{t}\right) .
\end{equation*}%
Consider the two spectral sequences for $B^{\bullet \bullet }$:%
\begin{eqnarray*}
^{hor}E_{2}^{s,t} &=&~^{ver}H^{t}~^{hor}H^{s}\left( B^{\bullet \bullet
}\right) \implies H^{s+t}\left( Tot\left( B^{\bullet \bullet }\right)
\right) , \\
^{ver}E_{2}^{s,t} &=&~^{hor}H^{s}~^{ver}H^{t}\left( B^{\bullet \bullet
}\right) \implies H^{s+t}\left( Tot\left( B^{\bullet \bullet }\right)
\right) .
\end{eqnarray*}%
It follows that%
\begin{eqnarray*}
&&~^{hor}H^{s}\left( B^{\bullet \bullet }\right) 
\simeq%
\left\{ 
\begin{array}{ccc}
0 & \text{if} & s>0 \\ 
\underleftarrow{\lim }~J^{t} & \text{if} & s=0%
\end{array}%
\right. , \\
&&^{hor}E_{2}^{s,t}%
\simeq%
\left\{ 
\begin{array}{ccc}
0 & \text{if} & s>0 \\ 
H^{t}\left( \underleftarrow{\lim }~J^{\bullet }\right) 
\simeq%
\underleftarrow{\lim }^{t}~X & \text{if} & s=0%
\end{array}%
\right.
\end{eqnarray*}%
$^{hor}E_{2}^{s,t}$ degenerates from $E_{2}$ on, and%
\begin{equation*}
H^{n}\left( Tot\left( B^{\bullet \bullet }\right) \right) 
\simeq%
\underleftarrow{\lim }^{n}~X.
\end{equation*}%
On the other hand,%
\begin{eqnarray*}
&&~^{ver}H^{t}\left( B^{\bullet \bullet }\right) 
\simeq%
\left\{ 
\begin{array}{ccc}
0 & \text{if} & t>0 \\ 
C^{s}\left( \mathbf{I},X\right) & \text{if} & t=0%
\end{array}%
\right. , \\
&&^{ver}E_{2}^{s,t}%
\simeq%
\left\{ 
\begin{array}{ccc}
0 & \text{if} & t>0 \\ 
H^{s}\left( C^{\bullet }\left( \mathbf{I},X\right) \right) & \text{if} & t=0%
\end{array}%
\right.
\end{eqnarray*}%
It follows that%
\begin{equation*}
H^{n}\left( C^{\bullet }\left( \mathbf{I},X\right) \right) 
\simeq%
H^{n}\left( Tot\left( B^{\bullet \bullet }\right) \right) 
\simeq%
\underleftarrow{\lim }^{n}~X.
\end{equation*}
\end{enumerate}

\item Denote $\nu 
{:=}%
\nu _{X}$.

\begin{enumerate}
\item For any $x\in X$, $U_{x}=\nu ^{-1}\left( U_{\nu \left( x\right)
}\right) $. We have therefore a bijection between the bases of $X$ and $%
\widehat{X}$. This bijection extends clearly to the bijection between
topologies on $X$ and $\widehat{X}$.

\item Given a sheaf $\mathcal{F}$ on $X$, let $\widehat{\mathcal{F}}$ be the
following sheaf on $\widehat{X}$:%
\begin{equation*}
\widehat{\mathcal{F}}\left( U\right) 
{:=}%
\mathcal{F}\left( \nu ^{-1}\left( U\right) \right) .
\end{equation*}%
The assignment $\mathcal{F}\mapsto \widehat{\mathcal{F}}$ is the desired
equivalence between the abelian categories $\mathbf{S}\left( X,\mathbf{Mod}%
\left( k\right) \right) $ and $\mathbf{S}\left( \widehat{X},\mathbf{Mod}%
\left( k\right) \right) $.

\item Clearly $\Gamma \left( X,\mathcal{F}\right) 
\simeq%
\Gamma \left( \widehat{X},\widehat{\mathcal{F}}\right) $, and the above
equivalence of categories establishes the natural isomorphism%
\begin{equation*}
H^{\bullet }\left( X,\mathcal{F}\right) 
\simeq%
R^{\bullet }\Gamma \left( X,\mathcal{F}\right) 
\simeq%
R^{\bullet }\Gamma \left( \widehat{X},\widehat{\mathcal{F}}\right) 
\simeq%
H^{\bullet }\left( \widehat{X},\widehat{\mathcal{F}}\right) .
\end{equation*}

\item Let $F%
{:=}%
\widehat{F}\circ \left( \nu _{X}\right) ^{op}$, and let $\mathcal{F}$ and $%
\widehat{\mathcal{F}}$ be the corresponding sheaves. Then%
\begin{equation*}
\underleftarrow{\lim }^{\bullet }~F%
\simeq%
H^{\bullet }\left( X,\mathcal{F}\right) 
\simeq%
H^{\bullet }\left( \widehat{X},\widehat{\mathcal{F}}\right) 
\simeq%
\underleftarrow{\lim }^{\bullet }~\widehat{F}.
\end{equation*}
\end{enumerate}
\end{enumerate}
\end{proof}

\subsection{Cosheaves on $A$-spaces}

\begin{theorem}
\label{Th-Poset-vs-cosheaves}Let $\left( X,\leq \right) $ be a pre-ordered
set. Denote by the same letter $X$ the corresponding $A$-space, and by $%
\mathbf{X}$ the corresponding category, see Notation \ref{Not-Categories}(%
\ref{Not-Categories-Category-for-poset}).

\begin{enumerate}
\item \label{Th-Poset-vs-cosheaves-categories-equivalent}There is an
equivalence between the category $\mathbf{Pro}\left( k\right) ^{\mathbf{X}}$
of functors $\mathbf{X}\rightarrow \mathbf{Pro}\left( k\right) $ and the
category $\mathbf{CS}\left( X,\mathbf{Pro}\left( k\right) \right) $ of 
\textbf{co}sheaves on $X$.

\item \label{Th-Poset-vs-cosheaves-higher-colimits}There is a natural (in $X$
and $F$) isomorphism%
\begin{equation*}
\underrightarrow{\lim }_{\bullet }~F%
\simeq%
H_{\bullet }\left( X,\mathcal{F}\right)
\end{equation*}%
where $\mathcal{F}$ is the \textbf{co}sheaf corresponding to the functor $F$.

\item \label{Th-Poset-vs-cosheaves-Bar-construction}There is a description
(the \textbf{bar-construction}) of $\underrightarrow{\lim }^{\bullet }~F$
for $F\in \mathbf{Pro}\left( k\right) ^{\mathbf{I}}$ when $\left( I,\leq
\right) $ is a poset, or a pre-ordered set, or even when $\mathbf{I}$ is an 
\textbf{arbitrary} small category:%
\begin{equation*}
\underrightarrow{\lim }_{\bullet }~F=H_{\bullet }\left( C_{\bullet }\left( 
\mathbf{I},F\right) \right)
\end{equation*}%
where%
\begin{equation*}
C_{n}\left( \mathbf{I},F\right) =\dbigoplus\limits_{i_{0}\rightarrow
i_{1}\rightarrow \dots \rightarrow i_{n}}F\left( i_{0}\right)
\end{equation*}%
and $d_{n}:C_{n+1}\left( \mathbf{I},F\right) \rightarrow C_{n}\left( \mathbf{%
I},F\right) $ is given by%
\begin{equation*}
d_{n}\circ \rho _{i_{0}\rightarrow i_{1}\rightarrow \dots \rightarrow
i_{n+1}}=\rho _{i_{1}\rightarrow i_{2}\rightarrow \dots \rightarrow
i_{n+1}}\circ F\left( i_{0}\rightarrow i_{1}\right)
+\dsum\limits_{k=1}^{n+1}\left( -1\right) ^{k}\rho _{i_{0}\rightarrow \dots
\rightarrow \widehat{i_{k}}\rightarrow \dots \rightarrow i_{n+1}},
\end{equation*}%
where%
\begin{equation*}
\rho _{i_{0}\rightarrow i_{1}\rightarrow \dots \rightarrow i_{m}}:F\left(
i_{0}\right) \longrightarrow \dbigoplus\limits_{i_{0}\rightarrow
i_{1}\rightarrow \dots \rightarrow i_{m}}F\left( i_{0}\right)
\end{equation*}%
are the natural embeddings. Compare to \cite[Definition B.2.4]%
{Prasolov-Cosheaves-2021-MR4347662} and to the proof of Theorem \ref%
{Th-Poset-vs-sheaves}(\ref{Th-Poset-vs-sheaves-Bar-construction}).

\item Define the binary relation $R$ on $X$ by%
\begin{equation*}
xRy\iff \left( x\leq y\right) \&\left( y\leq x\right) ,
\end{equation*}%
let $\widehat{X}=X/R$ be the corresponding poset, and let%
\begin{equation*}
\nu _{X}:X\longrightarrow \widehat{X}
\end{equation*}%
be the projection. Denote by the same letter $\widehat{X}$ the corresponding 
$T_{0}A$-space, and by $\widehat{\mathbf{X}}$ the corresponding category.
Then:

\begin{enumerate}
\item $\left( \nu _{X}\right) ^{-1}$ establishes a categorical equivalence $%
\mathcal{F}\longleftrightarrow \widehat{\mathcal{F}}$ between \textbf{co}%
sheaves on $X$ and \textbf{co}sheaves on $\widehat{X}$.

\item The above bijection preserves $H_{\bullet }\left( \bullet ,\bullet
\right) $, i.e., there is a natural (in $X$ and $\mathcal{F}$) isomorphism%
\begin{equation*}
H_{\bullet }\left( \widehat{X},\widehat{\mathcal{F}}\right) 
\simeq%
H_{\bullet }\left( X,\mathcal{F}\right) .
\end{equation*}

\item For any $\widehat{F}:\widehat{\mathbf{X}}\rightarrow \mathbf{Mod}%
\left( k\right) $%
\begin{equation*}
\left( \nu _{X}\right) ^{\ast }:\underrightarrow{\lim }_{\bullet }~\widehat{F%
}\longrightarrow \underrightarrow{\lim }_{\bullet }~\left( \widehat{F}\circ
\nu _{X}\right)
\end{equation*}%
is an isomorphism.
\end{enumerate}
\end{enumerate}
\end{theorem}

\begin{proof}
Let $T\in \mathbf{Mod}\left( k\right) $ be an \textbf{arbitrary} injective
module.

\begin{enumerate}
\item Given a functor $F:\mathbf{X}\rightarrow \mathbf{Pro}\left( k\right) $%
, let $\mathcal{F}$ be the following \textbf{preco}sheaf (in fact, a \textbf{%
co}sheaf, which will be seen later):%
\begin{equation*}
\mathcal{F}\left( U\right) 
{:=}%
\underrightarrow{\lim }~\left( F|_{U}\right) .
\end{equation*}%
Given a \textbf{co}sheaf $\mathcal{G}$, let $\mathcal{G}^{x}$ be the \textbf{%
co}stalk at $x$. Since $U_{x}$ is a minimal open neighborhood of $x$,%
\begin{equation*}
\mathcal{G}^{x}%
\simeq%
\mathcal{G}\left( U_{x}\right) .
\end{equation*}%
If $x\leq y$, i.e., $U_{x}\subseteq U_{y}$, one has the homomorphism%
\begin{equation*}
\mathcal{G}\left( U_{x}\rightarrow U_{y}\right) :\mathcal{G}^{x}=\mathcal{G}%
\left( U_{x}\right) \longrightarrow \mathcal{G}\left( U_{y}\right) =\mathcal{%
G}^{y}.
\end{equation*}%
We have clearly obtained a functor%
\begin{equation*}
G:\mathbf{X}\rightarrow \mathbf{Mod}\left( k\right) .
\end{equation*}

\begin{enumerate}
\item Theorem \ref{Th-Poset-vs-sheaves} gives us an equivalence between $%
\mathbf{Mod}\left( k\right) ^{\mathbf{X}^{op}}$ and $\mathbf{S}\left( X,%
\mathbf{Mod}\left( k\right) \right) $. Under that equivalence the functor $%
\left\langle F,T\right\rangle :\mathbf{X}^{op}\rightarrow \mathbf{Mod}\left(
k\right) $ is mapped to the sheaf%
\begin{equation*}
U\longmapsto \underleftarrow{\lim }~\left( \left\langle F,T\right\rangle
|_{U}\right) 
\simeq%
\left\langle \underrightarrow{\lim }~\left( F|_{U}\right) ,T\right\rangle .
\end{equation*}%
Therefore $\mathcal{F}=\left( U\mapsto \underrightarrow{\lim }~\left(
F|_{U}\right) \right) $ is a \textbf{co}sheaf.

\item ~

\begin{enumerate}
\item The composition $F\mapsto \mathcal{F=G}\mapsto G$ is the identity (up
to an isomorphism), i.e., $G%
\simeq%
F$. Indeed, for any $x\in X$%
\begin{equation*}
G\left( x\right) =\mathcal{G}^{x}=\mathcal{F}^{x}=\mathcal{F}\left(
U_{x}\right) =\underrightarrow{\lim }~\left( F|_{U_{x}}\right) =F\left(
x\right)
\end{equation*}%
because $x$ is a \textbf{terminal} object of the category $\mathbf{U}_{x}$,
corresponding to the pre-ordered set $U_{x}$.

\item The composition $\mathcal{F}\mapsto F=G\mapsto \mathcal{G}$ is the
identity (up to an isomorphism), i.e., $\mathcal{G}%
\simeq%
\mathcal{F}$. Indeed, for any open subset $U\subseteq X$%
\begin{equation*}
\mathcal{G}\left( U\right) =\underrightarrow{\lim }~\left( G|_{U}\right) =%
\underrightarrow{\lim }~\left( F|_{U}\right) .
\end{equation*}%
Since $\mathcal{F}$ is a cosheaf,%
\begin{equation*}
\mathcal{F}\left( U\right) 
\simeq%
\coker%
\left( \dbigoplus\limits_{x,y\in U}\mathcal{F}\left( U_{x}\cap U_{y}\right)
\rightrightarrows \dbigoplus\limits_{y\in U}\mathcal{F}\left( U_{y}\right)
\right) ,
\end{equation*}%
while%
\begin{equation*}
\underrightarrow{\lim }~\left( F|_{U}\right) 
\simeq%
\coker%
\left( \dbigoplus\limits_{x\leq y\in U}\mathcal{F}\left( U_{x}=U_{x}\cap
U_{y}\right) \rightrightarrows \dbigoplus\limits_{y\in U}\mathcal{F}\left(
U_{y}\right) \right) .
\end{equation*}%
There is a commutative diagram%
\begin{equation*}
\begin{diagram}[size=3.0em,textflow]
\dbigoplus\limits_{x\leq y\in U}\mathcal{F}\left( U_{x}=U_{x}\cap
U_{y}\right) &\pile{\rTo\\ \rTo} & \dbigoplus\limits_{y\in U}\mathcal{F}\left( U_{y}\right) & \rOnto & \underrightarrow{\lim }~\left(
F|_{U}\right) \\ 
\dInto &  & \dTo_{=} &  & \dOnto \\ 
\dbigoplus\limits_{x,y\in U}\mathcal{F}\left( U_{x}\cap U_{y}\right) &  & 
\dbigoplus\limits_{y\in U}\mathcal{F}\left( U_{y}\right) & \rOnto 
& \mathcal{F}\left( U\right) \\
\end{diagram}
%
\end{equation*}%
Apply $\left\langle \bullet ,T\right\rangle $ and get the diagram (\ref%
{Line-F(U)-vs-lim-F}) from the proof of Theorem \ref{Th-Poset-vs-sheaves}%
\begin{eqnarray*}
&&\left( \left\langle \mathcal{F}\left( U\right) ,T\right\rangle
=\left\langle \mathcal{F},T\right\rangle \left( U\right) \right) \overset{%
\varphi }{\rightarrowtail }\left( \underleftarrow{\lim }~\left( \left\langle 
\mathcal{F},T\right\rangle |_{U}\right) =\left\langle \underrightarrow{\lim }%
~\left( F|_{U}\right) ,T\right\rangle \right) \\
&\rightarrowtail &\left( \dprod\limits_{x\in U}\left\langle \mathcal{F}%
,T\right\rangle \left( U_{x}\right) =\left\langle \dbigoplus\limits_{y\in U}%
\mathcal{F}\left( U_{y}\right) ,T\right\rangle \right) .
\end{eqnarray*}%
It was proved in Theorem \ref{Th-Poset-vs-sheaves} that $\varphi $ is an
isomorphism, therefore 
\begin{equation*}
\mathcal{G}\left( U\right) =\underrightarrow{\lim }~\left( F|_{U}\right) =%
\mathcal{F}\left( U\right)
\end{equation*}%
and $\mathcal{G}=\mathcal{F}$.
\end{enumerate}
\end{enumerate}

\item The two abelian categories $\mathbf{Pro}\left( k\right) ^{\mathbf{X}}$
and $\mathbf{CS}\left( X,\mathbf{Pro}\left( k\right) \right) $ are
equivalent, and the following diagram%
\begin{equation*}
\begin{diagram}[size=3.0em,textflow]
\mathbf{Pro}\left( k\right) ^{\mathbf{X}} & 
\TeXButton{ISO}{\simeq} & \mathbf{CS}\left( X,\mathbf{Pro}\left( k\right) \right) \\ 
\dTo^{\underrightarrow{\lim }} &  & \dTo_{\Gamma \left( X,\bullet
\right) =\bullet \left( X\right)} \\ 
\mathbf{Pro}\left( k\right) & = & \mathbf{Pro}\left( k\right)\\
\end{diagram}
%
\end{equation*}%
commutes up to an isomorphism of functors. Moreover, $F\in \mathbf{Pro}%
\left( k\right) ^{\mathbf{X}}$ is quasi-projective $\iff $ $\left\langle
F,T\right\rangle $ is injective for all injective $T$ $\iff $ $\left\langle 
\mathcal{F},T\right\rangle $ is injective for all injective $T$ $\iff $ $%
\mathcal{F}\in \mathbf{CS}\left( X,\mathbf{Pro}\left( k\right) \right) $ is
quasi-projective. Since the left satellites are defined using
quasi-projective resolutions, the two left satellites%
\begin{eqnarray*}
\underrightarrow{\lim }_{\bullet }~ &=&L_{\bullet }~\underrightarrow{\lim }~:%
\mathbf{Pro}\left( k\right) ^{\mathbf{X}}\longrightarrow \mathbf{Pro}\left(
k\right) , \\
H_{\bullet }\left( X,\bullet \right) &=&L_{\bullet }\Gamma \left( X,\bullet
\right) :\mathbf{CS}\left( X,\mathbf{Pro}\left( k\right) \right)
\longrightarrow \mathbf{Pro}\left( k\right) ,
\end{eqnarray*}%
are naturally isomorphic.

\item Let $F\in \mathbf{Pro}\left( k\right) ^{\mathbf{I}}$. Choose a
quasi-projective resolution%
\begin{equation*}
0\longleftarrow F\longleftarrow P_{0}\longleftarrow P_{1}\longleftarrow
\dots \longleftarrow P_{n}\longleftarrow P_{n+1}\longleftarrow \dots
\end{equation*}%
and build the chain bicomplex%
\begin{equation*}
B_{s,t}%
{:=}%
C_{s}\left( \mathbf{I},P_{t}\right) .
\end{equation*}%
Consider the two spectral sequences for $B_{\bullet \bullet }$:%
\begin{eqnarray*}
^{hor}E_{s,t}^{2} &=&~^{ver}H_{t}~^{hor}H_{s}\left( B_{\bullet \bullet
}\right) \implies H_{s+t}\left( Tot\left( B_{\bullet \bullet }\right)
\right) , \\
^{ver}E_{s,t}^{2} &=&~^{hor}H_{s}~^{ver}H_{t}\left( B_{\bullet \bullet
}\right) \implies H_{s+t}\left( Tot\left( B_{\bullet \bullet }\right)
\right) .
\end{eqnarray*}%
It follows that%
\begin{equation*}
0\longrightarrow \left\langle F,T\right\rangle \longrightarrow \left\langle
P_{0},T\right\rangle \longrightarrow \left\langle P_{1},T\right\rangle
\longrightarrow \dots \longrightarrow \left\langle P_{n},T\right\rangle
\longrightarrow \dots
\end{equation*}%
is an \textbf{injective} coresolution for the functor $\left\langle
F,T\right\rangle $, therefore, due to the proof of Theorem \ref%
{Th-Poset-vs-sheaves}(\ref{Th-Poset-vs-sheaves-Bar-construction}),%
\begin{eqnarray*}
&&\left\langle ~^{hor}H_{s}\left( B_{\bullet \bullet }\right)
,T\right\rangle 
\simeq%
\left\{ 
\begin{array}{ccc}
0 & \text{if} & s>0 \\ 
\underleftarrow{\lim }~\left\langle P_{t},T\right\rangle 
\simeq%
\left\langle \underrightarrow{\lim }~P_{t},T\right\rangle & \text{if} & s=0%
\end{array}%
\right. , \\
&&^{hor}E_{2}^{s,t}%
\simeq%
\left\{ 
\begin{array}{ccc}
0 & \text{if} & s>0 \\ 
H_{t}\left( \underrightarrow{\lim }~P_{\bullet }\right) 
\simeq%
\underrightarrow{\lim }_{t}~F & \text{if} & s=0%
\end{array}%
\right.
\end{eqnarray*}%
$^{hor}E_{s,t}^{2}$ degenerates from $E^{2}$ on, and%
\begin{equation*}
H_{n}\left( Tot\left( B_{\bullet \bullet }\right) \right) 
\simeq%
\underrightarrow{\lim }_{n}~F.
\end{equation*}%
On the other hand,%
\begin{eqnarray*}
&&~^{ver}H_{t}\left( B_{\bullet \bullet }\right) 
\simeq%
\left\{ 
\begin{array}{ccc}
0 & \text{if} & t>0 \\ 
C_{s}\left( \mathbf{I},F\right) & \text{if} & t=0%
\end{array}%
\right. , \\
&&^{ver}E_{s,t}^{2}%
\simeq%
\left\{ 
\begin{array}{ccc}
0 & \text{if} & t>0 \\ 
H_{s}\left( C_{\bullet }\left( \mathbf{I},F\right) \right) & \text{if} & t=0%
\end{array}%
\right.
\end{eqnarray*}%
It follows that%
\begin{equation*}
H_{n}\left( C_{\bullet }\left( \mathbf{I},F\right) \right) 
\simeq%
H_{n}\left( Tot\left( B_{\bullet \bullet }\right) \right) 
\simeq%
\underrightarrow{\lim }_{n}~F.
\end{equation*}

\item Analogously to Theorem \ref{Th-Poset-vs-sheaves}(\ref%
{Th-Poset-vs-sheaves-Quotient-space}).
\end{enumerate}
\end{proof}

\begin{theorem}
\label{Th-Cosheaf-homology-for-A-spaces}Let $X$ be an $A$-space, and $G\in 
\mathbf{Mod}\left( k\right) $. Then, naturally in $X$ and $G$,%
\begin{equation*}
H_{n}\left( X,G_{\#}\right) 
\simeq%
H_{n}^{sing}\left( \left\vert \mathcal{K}\left( X\right) \right\vert
,G\right) 
\simeq%
H_{n}^{sing}\left( X,G\right) ,
\end{equation*}%
where $H_{\bullet }^{sing}$ is the singular homology, $\mathcal{K}\left(
X\right) =N\mathbf{X}$ is the 
\u{C}ech
nerve of the category $\mathbf{X}$ corresponding to the pre-ordered set from
Proposition \ref{Prop-Facts-A-spaces}(\ref{Prop-Facts-A-spaces-Nerve}), and $%
\left\vert \bullet \right\vert $ is the geometric realization.
\end{theorem}

\begin{remark}
\label{Rem-Cosheaf-homology-for-A-spaces}All $H_{n}\left( X,G_{\#}\right) $
are \textbf{rudimentary} pro-objects.
\end{remark}

\begin{proof}
Apply Theorem \ref{Th-Poset-vs-cosheaves} to the \textbf{co}sheaf $\mathcal{F%
}=G_{\#}$. Since for any $x\in X$ the costalk $\left( G_{\#}\right) ^{x}%
\simeq%
G$, the corresponding functor $F\in \mathbf{Pro}\left( k\right) ^{\mathbf{X}%
} $ is constant (and rudimentary!). We know that%
\begin{equation*}
H_{n}\left( X,\mathcal{F}\right) 
\simeq%
\underrightarrow{\lim }_{n}~F%
\simeq%
H_{n}\left( C_{\bullet }\left( \mathbf{X},F\right) \right) .
\end{equation*}%
The explicit construction of $H_{n}\left( C_{\bullet }\left( \mathbf{X}%
,F\right) \right) $ coincides with the construction of the \textbf{cellular}
homology of $H_{n}^{cell}\left( N\mathbf{X=}\mathcal{K}\left( X\right)
,G\right) $ which, since $\left\vert \mathcal{K}\left( X\right) \right\vert $
is a polyheder, is isomorphic to $H_{n}^{sing}\left( \left\vert \mathcal{K}%
\left( X\right) \right\vert ,G\right) $. Due to Proposition \ref%
{Prop-Facts-A-spaces}(\ref{Prop-Facts-A-spaces-Weak-equivalence} and \ref%
{Prop-Facts-A-spaces-Nerve}), $\left\vert \mathcal{K}\left( X\right)
\right\vert \rightarrow X$ is a weak equivalence, implying the desired
isomorphisms%
\begin{equation*}
H_{n}\left( X,G_{\#}\right) 
\simeq%
H_{n}^{cell}\left( N\mathbf{X},G\right) 
\simeq%
H_{n}^{sing}\left( \left\vert \mathcal{K}\left( X\right) \right\vert
,G\right) 
\simeq%
H_{n}^{sing}\left( X,G\right) .
\end{equation*}
\end{proof}

See Example \ref{Ex-4-point-circle} and \ref{Ex-6-point-2-sphere}.

\section{\label{Sec-classical-cosheaves}\textquotedblleft
Bad\textquotedblright\ classical cosheaves}

\subsection{General remarks}

Glen E. Bredon in \cite{Bredon-MR0226631} and \cite[Chapter VI]%
{Bredon-Book-MR1481706} was rather skeptical concerning building the
homology theory for cosheaves with values in the category $\mathbf{Ab}$. He
wrote in \cite{Bredon-MR0226631}: \textquotedblleft \emph{The most basic
concept in sheaf theory is that of a sheaf generated by a given presheaf. In
categorical terminology this is the concept of a reflector from presheaves
to sheaves. We believe that there is not much hope for the existence of a
reflector from precosheaves to cosheaves.}\textquotedblright

In fact, the coreflection functor%
\begin{equation*}
\left( \bullet \right) _{\#}:\mathbf{pCS}\left( X,\mathbf{Ab}\right)
\longrightarrow \mathbf{CS}\left( X,\mathbf{Ab}\right)
\end{equation*}%
\textbf{does} exist \cite[Theorem 3.1(1)]%
{Prasolov-Cosheafification-2016-MR3660525}, but does not have a nice
description. It seems also that the category $\mathbf{CS}\left( X,\mathbf{Ab}%
\right) $ is not good for building an appropriate homology theory. In \cite%
{Bredon-MR0226631} and \cite[Chapter VI]{Bredon-Book-MR1481706} such a
theory was built for a very restricted class of topological spaces and a
very restricted class of cosheaves on them.

The reasoning below shows that the category of \emph{classical} cosheaves $%
\mathbf{CS}\left( X,\mathbf{Ab}\right) $ is hardly abelian. The crucial fact
that was used in the proof of \cite[Theorem 3.3.1(1)]%
{Prasolov-Cosheaves-2021-MR4347662} was this: the coreflection functor%
\begin{equation*}
\left( \bullet \right) _{\#}:\mathbf{pCS}\left( X,\mathbf{Pro}\left( \mathbf{%
Ab}\right) \right) \longrightarrow \mathbf{CS}\left( X,\mathbf{Pro}\left( 
\mathbf{Ab}\right) \right)
\end{equation*}%
is \textbf{exact}.

It is shown below (Example \ref{Ex-Non-exactness-Ab-valued}) that the \emph{%
classical} coreflection functor%
\begin{equation*}
\left( \bullet \right) _{\#}:\mathbf{pCS}\left( X,\mathbf{Ab}\right)
\longrightarrow \mathbf{CS}\left( X,\mathbf{Ab}\right)
\end{equation*}%
is \textbf{not} in general exact.

\begin{conjecture}
\label{Conj-Classical-not-abelian}The category of classical cosheaves $%
\mathbf{CS}\left( X,\mathbf{Ab}\right) $ is \textbf{not} in general abelian.
\end{conjecture}

\begin{remark}
\label{Rem-Classical-not-abelian}The critical point in the proof of \cite[%
Theorem 3.3.1(1)]{Prasolov-Cosheaves-2021-MR4347662} was the following.
Since $\left( \bullet \right) _{\#}$ is exact,%
\begin{equation*}
\left( 
\coker%
\left( h\right) \right) _{\#}%
\simeq%
\coker%
\left( h_{\#}\right) ,
\end{equation*}%
where%
\begin{equation*}
h:\ker \left( \iota f\right) \longrightarrow \mathcal{A},
\end{equation*}%
$\iota $ is the embedding%
\begin{equation*}
\iota :\mathbf{CS}\left( X,\mathbf{Ab}\right) \longrightarrow \mathbf{pCS}%
\left( X,\mathbf{Ab}\right) ,
\end{equation*}%
and%
\begin{equation*}
f:\mathcal{A\longrightarrow B}
\end{equation*}%
is a morphism of cosheaves. Therefore, to find a counter-example to the
abelianity of $\mathbf{CS}\left( X,\mathbf{Ab}\right) $, it is enough to
find a \textbf{mono}morphism of \textbf{pre}cosheaves%
\begin{equation*}
h:\mathcal{C}\longrightarrow \mathcal{A},
\end{equation*}%
where $\mathcal{A}$ is a cosheaf, such that%
\begin{equation*}
\left( 
\coker%
\left( h\right) \right) _{\#}\not\approxeq 
\coker%
\left( h_{\#}\right) .
\end{equation*}%
Unfortunately, Example \ref{Ex-Non-exactness-Ab-valued} cannot serve as a
counter-example, because $h$ there is \textbf{not} a monomorphism.
\end{remark}

\subsection{Classical vs \textquotedblleft new\textquotedblright\
(pre)cosheaves}

The category $\mathbf{Mod}\left( k\right) $ is a full subcategory of $%
\mathbf{Pro}\left( k\right) =\mathbf{Pro}\left( \mathbf{Mod}\left( k\right)
\right) $. It is clear that the category of classical cosheaves $\mathbf{pCS}%
\left( X,\mathbf{Mod}\left( k\right) \right) $ is a full subcategory of $%
\mathbf{pCS}\left( X,\mathbf{Pro}\left( k\right) \right) $.

\begin{definition}
\label{Def-Rudimentary-precosheave}A precosheaf $\mathcal{A}\in \mathbf{pCS}%
\left( X,\mathbf{Pro}\left( k\right) \right) $ is called \textbf{rudimentary}
iff it is isomorphic to a cosheaf $\mathcal{B}\in \mathbf{pCS}\left( X,%
\mathbf{Mod}\left( k\right) \right) $.
\end{definition}

\begin{remark}
\label{Rem-Rudimentary-precosheave}A precosheaf $\mathcal{A}\in \mathbf{pCS}%
\left( X,\mathbf{Pro}\left( k\right) \right) $ is rudimentary iff $\mathcal{A%
}\left( U\right) $ is a rudimentary pro-module for any $U\in \mathbf{C}_{X}$.
\end{remark}

Let us formulate a list of properties of rudimentary precosheaves. Since
many properties and constructions depend on whether they are
built/formulated in $\mathbf{pCS}\left( X,\mathbf{Pro}\left( k\right)
\right) $ or $\mathbf{pCS}\left( X,\mathbf{Mod}\left( k\right) \right) $,
let us mark them with $\mathbf{Pro}\left( k\right) $ or $\mathbf{Mod}\left(
k\right) $ respectively.

\begin{proposition}
\label{Prop-Rudimentary-precosheave}Let $\mathcal{A}\in \mathbf{pCS}\left( X,%
\mathbf{Mod}\left( k\right) \right) $ and $\mathcal{B}\in \mathbf{pCS}\left(
X,\mathbf{Pro}\left( k\right) \right) $.

\begin{enumerate}
\item $\mathcal{A}$ is $\mathbf{Mod}\left( k\right) $-coseparated iff it is $%
\mathbf{Pro}\left( k\right) $-coseparated.

\item $~$

\begin{enumerate}
\item $\mathcal{A}$ is a $\mathbf{Mod}\left( k\right) $-cosheaf iff it is a $%
\mathbf{Pro}\left( k\right) $-cosheaf.

\item $\mathbf{CS}\left( X,\mathbf{Mod}\left( k\right) \right) $ is a full
subcategory of $\mathbf{CS}\left( X,\mathbf{Pro}\left( k\right) \right) $.
\end{enumerate}

\item There exists a right adjoint 
\begin{equation*}
\left( \bullet \right) _{\#}^{\mathbf{Mod}\left( k\right) }:\mathbf{pCS}%
\left( X,\mathbf{Mod}\left( k\right) \right) \longrightarrow \mathbf{CS}%
\left( X,\mathbf{Mod}\left( k\right) \right)
\end{equation*}%
($\iota \dashv \left( \bullet \right) _{\#}^{\mathbf{Mod}\left( k\right) }$)
to the embedding%
\begin{equation*}
\iota :\mathbf{CS}\left( X,\mathbf{Mod}\left( k\right) \right)
\longrightarrow \mathbf{pCS}\left( X,\mathbf{Mod}\left( k\right) \right) .
\end{equation*}

\item If $\left( \mathcal{A}\right) _{\mathbf{\#}}^{\mathbf{Pro}\left(
k\right) }$ is rudimentary, then $\left( \mathcal{A}\right) _{\mathbf{\#}}^{%
\mathbf{Mod}\left( k\right) }%
\simeq%
\left( \mathcal{A}\right) _{\mathbf{\#}}^{\mathbf{Pro}\left( k\right) }$.

\item The embedding%
\begin{equation*}
\mathbf{pCS}\left( X,\mathbf{Mod}\left( k\right) \right) \longrightarrow 
\mathbf{pCS}\left( X,\mathbf{Pro}\left( k\right) \right)
\end{equation*}%
has a right adjoint%
\begin{equation*}
\mathcal{A}\longmapsto \left( U\longmapsto \underleftarrow{\lim }~\mathcal{A}%
\left( U\right) \right) .
\end{equation*}%
Denote the latter precosheaf shortly by $\underleftarrow{\lim }~\mathcal{A}$.

\item ~

\begin{enumerate}
\item The embedding%
\begin{equation*}
\mathbf{CS}\left( X,\mathbf{Mod}\left( k\right) \right) \longrightarrow 
\mathbf{CS}\left( X,\mathbf{Pro}\left( k\right) \right)
\end{equation*}%
has a right adjoint%
\begin{equation*}
R%
{:=}%
\left( \mathcal{B}\longmapsto \left( \underleftarrow{\lim }~\mathcal{B}%
\right) _{\#}^{\mathbf{Mod}\left( k\right) }\right) .
\end{equation*}

\item Moreover,%
\begin{equation*}
\left( \bullet \right) _{\#}^{\mathbf{Mod}\left( k\right) }%
\simeq%
R\circ \left( \bullet \right) _{\#}^{\mathbf{Pro}\left( k\right) }
\end{equation*}
\end{enumerate}

\item Assume now that $X$ is a topological space, and $x\in X$.

\begin{enumerate}
\item 
\begin{equation*}
\mathcal{A}_{\mathbf{Mod}\left( k\right) }^{x}%
\simeq%
\underleftarrow{\lim }~\mathcal{A}_{\mathbf{Pro}\left( k\right) }^{x}.
\end{equation*}

\item If the costalk $\mathcal{A}_{\mathbf{Pro}\left( k\right) }^{x}$ is
rudimentary then $\mathcal{A}_{\mathbf{Mod}\left( k\right) }^{x}%
\simeq%
\mathcal{A}_{\mathbf{Pro}\left( k\right) }^{x}$.

\item The morphism%
\begin{equation*}
\mathcal{A}_{\#}\longrightarrow \mathcal{A}
\end{equation*}%
is a local isomorphism in the sense that%
\begin{equation*}
\left( \mathcal{A}_{\#}\right) _{\mathbf{Mod}\left( k\right) }^{x}%
\simeq%
\mathcal{A}_{\mathbf{Mod}\left( k\right) }^{x}.
\end{equation*}

\item The morphism%
\begin{equation*}
\mathcal{B}_{\#}\longrightarrow \mathcal{B}
\end{equation*}%
is a \textbf{strong} local isomorphism (or a local isomorphism in the sense
of \cite[Definition.V.12.2]{Bredon-Book-MR1481706}), i.e.,%
\begin{equation*}
\left( \mathcal{B}_{\#}\right) _{\mathbf{Pro}\left( k\right) }^{x}%
\simeq%
\mathcal{B}_{\mathbf{Pro}\left( k\right) }^{x}.
\end{equation*}

\item Bredon \cite[Definition VI.3.4 ]{Bredon-Book-MR1481706} calls $%
\mathcal{A}$ \textbf{smooth} iff $\mathcal{A}_{\#}\longrightarrow \mathcal{A}
$ is a strong local isomorphism. We have the following: $\mathcal{A}$ is
smooth iff $\mathcal{A}_{\#}^{\mathbf{Pro}\left( k\right) }$ is rudimentary.
\end{enumerate}
\end{enumerate}
\end{proposition}

\begin{proof}
~

\begin{enumerate}
\item The embedding $\iota :\mathbf{Mod}\left( k\right) \rightarrow \mathbf{%
Pro}\left( k\right) $ preserves colimits. Moreover, $f\in 
\Hom%
_{\mathbf{Mod}\left( k\right) }\left( A,B\right) $ is an epimorphism iff $%
\iota f$ is an epimorphism.

\item The embedding $\iota :\mathbf{Mod}\left( k\right) \rightarrow \mathbf{%
Pro}\left( k\right) $ preserves colimits.

\item See \cite[Theorem 3.1(1)]{Prasolov-Cosheafification-2016-MR3660525}.

\item Let $\mathcal{C}\in \mathbf{CS}\left( X,\mathbf{Mod}\left( k\right)
\right) $. Then%
\begin{eqnarray*}
&&%
\Hom%
_{\mathbf{CS}\left( X,\mathbf{Mod}\left( k\right) \right) }\left( \mathcal{C}%
,\mathcal{A}_{\#}^{\mathbf{Pro}\left( k\right) }\right) 
\simeq%
\Hom%
_{\mathbf{CS}\left( X,\mathbf{Pro}\left( k\right) \right) }\left( \mathcal{C}%
,\mathcal{A}_{\#}^{\mathbf{Pro}\left( k\right) }\right) 
\simeq
\\
&&%
\simeq%
_{\mathbf{pCS}\left( X,\mathbf{Mod}\left( k\right) \right) }\left( \mathcal{C%
},\mathcal{A}\right) 
\simeq%
_{\mathbf{pCS}\left( X,\mathbf{Mod}\left( k\right) \right) }\left( \mathcal{C%
},\mathcal{A}_{\#}^{\mathbf{Mod}\left( k\right) }\right) .
\end{eqnarray*}%
Yoneda's lemma implies the result.

\item Due to Proposition \ref{Prop-Facts-pro-objects}(\ref%
{Prop-Facts-pro-objects-Embedding-Coreflective}), the morphisms $\mathcal{A}%
\left( U\right) \rightarrow \mathcal{B}\left( U\right) $ are naturally in
1-1 correspondence with morphisms $\mathcal{A}\left( U\right) \rightarrow 
\underleftarrow{\lim }~\mathcal{B}\left( U\right) $. Combining those
correspondences, one obtains the desired isomorphism%
\begin{equation*}
\Hom%
_{\mathbf{pCS}\left( X,\mathbf{Pro}\left( k\right) \right) }\left( \mathcal{A%
},\mathcal{B}\right) 
\simeq%
\Hom%
_{\mathbf{pCS}\left( X,\mathbf{Mod}\left( k\right) \right) }\left( \mathcal{A%
},\underleftarrow{\lim }~\mathcal{B}\right) .
\end{equation*}

\item Assume that $\mathcal{A}$ and $\mathcal{B}$ are cosheaves. Notice that 
$\underleftarrow{\lim }~\mathcal{B}$ is \textbf{not} in general a cosheaf.
The following diagram of full embeddings%
\begin{equation*}
\begin{diagram}[size=3.0em,textflow]
\mathbf{CS}\left( X,\mathbf{Mod}\left( k\right) \right) & \rTo^{\iota _{0}} & \mathbf{CS}\left( X,\mathbf{Pro}\left( k\right) \right)
\\ 
\dTo<{\iota _{1}} &  & \dTo>{\iota _{2}} \\ 
\mathbf{pCS}\left( X,\mathbf{Mod}\left( k\right) \right) & \rTo^{\iota
_{3}} & \mathbf{pCS}\left( X,\mathbf{Pro}\left( k\right)
\right) \\
\end{diagram}
%
\end{equation*}%
commutes up to a functor isomorphism. Therefore, the diagram of right
adjoints%
\begin{equation*}
\begin{diagram}[size=3.0em,textflow]
\mathbf{CS}\left( X,\mathbf{Mod}\left( k\right) \right) & \lTo{R} & \mathbf{CS}\left( X,\mathbf{Pro}\left( k\right) \right)
\\ 
\uTo<{\left( \bullet \right) _{\#}^{\mathbf{Mod}\left( k\right) }} &  & 
\uTo>{\left( \bullet \right) _{\#}^{\mathbf{Pro}\left( k\right) }} \\ 
\mathbf{pCS}\left( X,\mathbf{Mod}\left( k\right) \right) & \lTo^{\underleftarrow{\lim }} & \mathbf{pCS}\left( X,\mathbf{Pro}\left( k\right) \right) \\
\end{diagram}
%
\end{equation*}
also commutes up to a functor isomorphism. Therefore,%
\begin{equation*}
R\circ \left( \bullet \right) _{\#}^{\mathbf{Pro}\left( k\right) }%
\simeq%
\left( \bullet \right) _{\#}^{\mathbf{Mod}\left( k\right) }\circ 
\underleftarrow{\lim }.
\end{equation*}%
Since $\left( \mathcal{B}\right) _{\#}^{\mathbf{Pro}\left( k\right) }%
\simeq%
\mathcal{B}$ if $\mathcal{B}$ is a cosheaf, one obtains%
\begin{equation*}
R%
\simeq%
\left( \bullet \right) _{\#}^{\mathbf{Mod}\left( k\right) }\circ 
\underleftarrow{\lim }.
\end{equation*}%
On the other hand, since $\underleftarrow{\lim }~\mathcal{A}%
\simeq%
\mathcal{A}$ if $\mathcal{A}\in \mathbf{pCS}\left( X,\mathbf{Mod}\left(
k\right) \right) $, we have%
\begin{equation*}
\left( \bullet \right) _{\#}^{\mathbf{Mod}\left( k\right) }%
\simeq%
R\circ \left( \bullet \right) _{\#}^{\mathbf{Pro}\left( k\right) }.
\end{equation*}

\item ~

\begin{enumerate}
\item $\mathcal{A}_{\mathbf{Mod}\left( k\right) }^{x}=\underset{x\in U\in
Open\left( X\right) }{\underleftarrow{\lim }}\mathcal{A}\left( U\right) $,
while $\mathcal{A}_{\mathbf{Pro}\left( k\right) }^{x}$ is the pro-module%
\begin{equation*}
\left( \mathcal{A}\left( U\right) \right) _{x\in U\in Open\left( X\right) }.
\end{equation*}

\item For a rudimentary pro-object $\mathbf{C}$, $\underleftarrow{\lim }$ $%
\mathbf{C}%
\simeq%
\mathbf{C}$.

\item See \cite[Proposition 2.21]{Prasolov-Cosheafification-2016-MR3660525}.

\item Again \cite[Proposition 2.21]{Prasolov-Cosheafification-2016-MR3660525}%
.

\item See \cite[Theorem 3.7]{Prasolov-Cosheafification-2016-MR3660525}.
\end{enumerate}
\end{enumerate}
\end{proof}

\subsection{$\mathbf{Set}$-valued cosheaves}

The main purpose here is to construct a counter-example showing that the
coreflection functor%
\begin{equation*}
\left( \bullet \right) _{\#}:\mathbf{pCS}\left( X,\mathbf{Set}\right)
\longrightarrow \mathbf{CS}\left( X,\mathbf{Set}\right)
\end{equation*}%
is \textbf{not} in general exact.

The example is similar to \cite[Example 4.6]%
{Prasolov-Cosheafification-2016-MR3660525}. We choose a topological $T_{1}$
space $X$.

\begin{lemma}
\label{Lemma-Separation}Let $U\subseteq X$ be an open subset, and let $%
x,y\in U$ be two different points. Then there exist two open subsets $V$ and 
$W$ such that:

\begin{enumerate}
\item $U=V\cup W$.

\item $x\in V-W$.

\item $y\in W-V$.
\end{enumerate}
\end{lemma}

\begin{proof}
Choose $V=U-\left\{ y\right\} $ and $W=U-\left\{ x\right\} $.
\end{proof}

Let $T$ be another topological space. Denote by $\mathcal{A}_{T}\in \mathbf{%
pCS}\left( X,\mathbf{Set}\right) $ the following precosheaf:%
\begin{equation*}
\mathcal{A}_{T}\left( U\right) 
{:=}%
U^{T},
\end{equation*}%
where $U^{T}$ is the set of continuous mappings $T\rightarrow U$. For $%
V\subseteq U$ let%
\begin{equation*}
\mathcal{A}_{T}\left( V\rightarrow U\right) :\mathcal{A}_{T}\left( V\right)
\longrightarrow \mathcal{A}_{T}\left( U\right)
\end{equation*}%
assign the composition%
\begin{equation*}
T\overset{f}{\longrightarrow }V\longrightarrow U
\end{equation*}%
to $f\in V^{T}=\mathcal{A}_{T}\left( V\right) $.

$\mathcal{A}_{T}$ is clearly a precosheaf. For a continuous mapping $\varphi
:S\rightarrow T$ one defines naturally a morphism of precosheaves:%
\begin{eqnarray*}
\varphi ^{\ast } &:&\mathcal{A}_{T}\longrightarrow \mathcal{A}_{S}, \\
\varphi ^{\ast }\left( f\right) 
{:=}%
f\circ \varphi &\in &U^{S}, \\
f &\in &U^{T}.
\end{eqnarray*}

\begin{theorem}
\label{Th-Bad-Set-valued-cosheaf}~

\begin{enumerate}
\item $\left( \mathcal{A}_{T}\right) _{\#}$ is the following cosheaf $%
\mathcal{C}$:%
\begin{equation*}
\left( \mathcal{A}_{T}\right) _{\#}\left( U\right) =\mathcal{C}\left(
U\right) 
{:=}%
U
\end{equation*}%
with the evident corestriction mappings for $V\subseteq U$:%
\begin{equation*}
V=\mathcal{C}\left( V\right) \longrightarrow \mathcal{C}\left( U\right) =U.
\end{equation*}

\item The natural morphism%
\begin{equation*}
\eta :\left( \mathcal{A}_{T}\right) _{\#}\longrightarrow \mathcal{A}_{T}
\end{equation*}%
is given by%
\begin{equation*}
\eta \left( x\right) =\left( \text{\textbf{constant} mapping }%
T\longrightarrow \left\{ x\right\} \right) ,
\end{equation*}%
where%
\begin{equation*}
x\in \left( \mathcal{A}_{T}\right) _{\#}\left( U\right) .
\end{equation*}
\end{enumerate}
\end{theorem}

\begin{notation}
Denote by $\Delta ^{s}$ the $s$-dimensional geometric simplex.
\end{notation}

\begin{remark}
Theorem \ref{Th-Bad-Set-valued-cosheaf} can be reformulated like this:%
\begin{equation*}
\left( \mathcal{A}_{T}\right) _{\#}%
\simeq%
\mathcal{A}_{\Delta ^{0}}.
\end{equation*}
\end{remark}

\begin{proof}
Let $\mathcal{B}\in \mathbf{CS}\left( X,\mathbf{Set}\right) $ be a cosheaf.
A morphism $\xi :\mathcal{B}\rightarrow \mathcal{A}_{T}$ is given by a family%
\begin{equation*}
\left( \xi _{U}:\mathcal{B}\left( U\right) \longrightarrow \mathcal{A}%
_{T}\left( U\right) =U^{T}\right) _{U\in Open\left( X\right) }.
\end{equation*}%
We claim that for any $b\in \mathcal{B}\left( U\right) $, the mapping%
\begin{equation*}
\xi _{U}\left( b\right) :T\longrightarrow U
\end{equation*}%
is \textbf{constant}. Indeed, assume that it is not the case:%
\begin{eqnarray*}
\xi _{U}\left( b\right) \left( t\right) &=&x\neq y=\xi _{U}\left( b\right)
\left( s\right) , \\
s,t &\in &T.
\end{eqnarray*}%
Apply Lemma \ref{Lemma-Separation}, and find $V$ and $W$ such that

\begin{enumerate}
\item $U=V\cup W$.

\item $x\in V-W$.

\item $y\in W-V$.
\end{enumerate}

Since $\mathcal{B}$ is a cosheaf,%
\begin{equation*}
\mathcal{B}\left( U\right) =%
\coker%
\left( \mathcal{B}\left( V\cap W\right) \rightrightarrows \mathcal{B}\left(
V\right) \amalg \mathcal{B}\left( W\right) \right) ,
\end{equation*}%
therefore either

\begin{enumerate}
\item $b=\mathcal{B}\left( V\rightarrow U\right) \left( c\right) ,c\in 
\mathcal{B}\left( V\right) $, or

\item $b=\mathcal{B}\left( W\rightarrow U\right) \left( d\right) ,d\in 
\mathcal{B}\left( W\right) $.
\end{enumerate}

Both cases lead to a contradiction:

\begin{enumerate}
\item From the commutative diagram%
\begin{equation*}
\begin{diagram}[size=3.0em,textflow]
\mathcal{B}\left( V\right) & \rTo^{\xi _{V}} & \mathcal{A}_{T}\left( V\right) =V^{T} \\ 
\dTo &  & \dTo \\ 
\mathcal{B}\left( U\right) & \rTo{\xi _{U}} & \mathcal{A}_{T}\left( U\right) =U^{T} \\
\end{diagram}
%
\end{equation*}%
one obtains%
\begin{eqnarray*}
y &=&\xi _{U}\left( b\right) \left( s\right) =\xi _{U}\circ \mathcal{B}%
\left( V\rightarrow U\right) \left( c\right) \left( t\right) = \\
&=&\mathcal{A}_{T}\left( V\rightarrow U\right) \circ \xi _{V}\left( c\right)
\left( t\right) \in V.
\end{eqnarray*}%
Contradiction, since $y\in W-V$.

\item From the commutative diagram%
\begin{equation*}
\begin{diagram}[size=3.0em,textflow]
\mathcal{B}\left( W\right) & \rTo^{\xi _{W}} & \mathcal{A}_{T}\left( V\right) =W^{T} \\ 
\dTo &  & \dTo \\ 
\mathcal{B}\left( U\right) & \rTo{\xi _{U}} & \mathcal{A}_{T}\left( U\right) =U^{T} \\
\end{diagram}%
\end{equation*}%
one obtains%
\begin{eqnarray*}
x &=&\xi _{U}\left( b\right) \left( t\right) =\xi _{U}\circ \mathcal{B}%
\left( W\rightarrow U\right) \left( d\right) \left( t\right) = \\
&=&\mathcal{A}_{T}\left( W\rightarrow U\right) \circ \xi _{W}\left( d\right)
\left( t\right) \in W.
\end{eqnarray*}%
Contradiction, since $x\in V-W$.
\end{enumerate}

It follows that any morphism $\mathcal{B}\longrightarrow \mathcal{A}_{T}$
factors uniquely through $\mathcal{C}$:%
\begin{equation*}
\mathcal{B}\longrightarrow \mathcal{C}\longrightarrow \mathcal{A}_{T},
\end{equation*}%
therefore%
\begin{equation*}
Hom_{\mathbf{CS}\left( X,\mathbf{Set}\right) }\left( \mathcal{B},\left( 
\mathcal{A}_{T}\right) _{\#}\right) 
\simeq%
Hom_{\mathbf{pCS}\left( X,\mathbf{Set}\right) }\left( \mathcal{B},\mathcal{A}%
_{T}\right) 
\simeq%
Hom_{\mathbf{CS}\left( X,\mathbf{Set}\right) }\left( \mathcal{B},\mathcal{C}%
\right)
\end{equation*}%
naturally in $\mathcal{B}$. The Yoneda's lemma implies that%
\begin{equation*}
\left( \mathcal{A}_{T}\right) _{\#}%
\simeq%
\mathcal{C}.
\end{equation*}
\end{proof}

\begin{example}
\label{Ex-Non-exactness-Set-valued}Let%
\begin{equation*}
\Delta ^{0}\rightrightarrows \Delta ^{1}
\end{equation*}%
be the two mappings of a point into the endpoints of $\Delta ^{1}$. Then%
\begin{equation*}
\coker%
\left( \mathcal{A}_{\Delta ^{1}}\rightrightarrows \mathcal{A}_{\Delta
^{0}}\right) 
\simeq%
\pi _{0}
\end{equation*}%
($\mathcal{A}_{\Delta ^{0}}$ and $\pi _{0}$ are cosheaves, see \cite[Example
4.2]{Prasolov-Cosheafification-2016-MR3660525}).%
\begin{eqnarray*}
&&\left( \mathcal{A}_{\Delta ^{1}}\right) _{\#}%
\simeq%
\left( \mathcal{A}_{\Delta ^{0}}\right) _{\#}%
\simeq%
\mathcal{A}_{\Delta ^{0}}, \\
&&%
\coker%
\left( \left( \mathcal{A}_{\Delta ^{1}}\right) _{\#}\rightrightarrows \left( 
\mathcal{A}_{\Delta ^{0}}\right) _{\#}\right) 
\simeq%
\coker%
\left( \mathcal{A}_{\Delta ^{0}}\rightrightarrows \mathcal{A}_{\Delta
^{0}}\right) 
\simeq%
\mathcal{A}_{\Delta ^{0}}.
\end{eqnarray*}%
Let $X$ be any $T_{1}$ space such that the natural mapping $X\rightarrow \pi
_{0}\left( X\right) $ is \textbf{not} bijective. Then%
\begin{equation*}
\mathcal{A}_{\Delta ^{0}}%
\simeq%
\coker%
\left( \left( \mathcal{A}_{\Delta ^{1}}\right) _{\#}\rightrightarrows \left( 
\mathcal{A}_{\Delta ^{0}}\right) _{\#}\right) \ncong \left( 
\coker%
\left( \mathcal{A}_{\Delta ^{1}}\rightrightarrows \mathcal{A}_{\Delta
^{0}}\right) \right) _{\#}%
\simeq%
\left( \pi _{0}\right) _{\#}%
\simeq%
\pi _{0}
\end{equation*}%
because%
\begin{equation*}
\mathcal{A}_{\Delta ^{0}}\left( X\right) 
\simeq%
X\longrightarrow \pi _{0}\left( X\right) =\left( 
\coker%
\left( \mathcal{A}_{\Delta ^{1}}\rightrightarrows \mathcal{A}_{\Delta
^{0}}\right) \right) _{\#}\left( X\right)
\end{equation*}%
is not an isomorphism. Therefore,%
\begin{equation*}
\left( \bullet \right) _{\#}:\mathbf{pCS}\left( X,\mathbf{Set}\right)
\longrightarrow \mathbf{CS}\left( X,\mathbf{Set}\right)
\end{equation*}%
is \textbf{not} right exact.
\end{example}

\subsection{$\mathbf{Ab}$-valued cosheaves}

For a set $X$, let $\mathbb{Z}\left[ X\right] $ be the free abelian group
generated by $X$. The functor%
\begin{equation*}
\mathbb{Z}\left[ \bullet \right] :\mathbf{Set}\longrightarrow \mathbf{Ab,}
\end{equation*}%
being left adjoint to the forgetting functor, preserves colimits. It can be
clearly extended to%
\begin{equation*}
\mathbb{Z}\left[ \bullet \right] :\mathbf{pCS}\left( X,\mathbf{Set}\right)
\longrightarrow \mathbf{pCS}\left( X,\mathbf{Ab}\right) ,
\end{equation*}%
and, since the first $\mathbb{Z}\left[ \bullet \right] $ preserves colimits,
the second $\mathbb{Z}\left[ \bullet \right] $ converts cosheaves into
cosheaves.

\begin{theorem}
\label{Th-Bad-Ab-valued-cosheaf}%
\begin{equation*}
\left( \mathbb{Z}\left[ \mathcal{A}_{T}\right] \right) _{\#}%
\simeq%
\mathbb{Z}\left[ \mathcal{A}_{\Delta ^{0}}\right] .
\end{equation*}
\end{theorem}

\begin{proof}
Analogous to the proof of Theorem \ref{Th-Bad-Set-valued-cosheaf}.

Let $\mathcal{B}\in \mathbf{CS}\left( X,\mathbf{Ab}\right) $ be a cosheaf. A
morphism $\xi :\mathcal{B}\rightarrow \mathcal{A}_{T}$ is given by a family%
\begin{equation*}
\left( \xi _{U}:\mathcal{B}\left( U\right) \longrightarrow \mathcal{A}%
_{T}\left( U\right) =\mathbb{Z}\left[ U^{T}\right] \right) _{U\in Open\left(
X\right) }.
\end{equation*}%
We claim that for any $b\in \mathcal{B}\left( U\right) $, in 
\begin{equation*}
\xi _{U}\left( b\right) =\dsum\limits_{\alpha }^{\text{finite}}n_{\alpha }%
\left[ \varphi _{\alpha }\in U^{T}\right]
\end{equation*}%
all functions $\varphi _{\alpha }$ are \textbf{constant}. Indeed, assume
that it is not the case:%
\begin{eqnarray*}
\varphi _{\alpha }\left( t\right) &=&x\neq y=\varphi _{\alpha }\left(
s\right) , \\
s,t &\in &T.
\end{eqnarray*}%
Apply Lemma \ref{Lemma-Separation}, and find $V$ and $W$ such that

\begin{enumerate}
\item $U=V\cup W$.

\item $x\in V-W$.

\item $y\in W-V$.
\end{enumerate}

Since $\mathcal{B}$ is a cosheaf,%
\begin{equation*}
\mathcal{B}\left( U\right) =%
\coker%
\left( \mathcal{B}\left( V\cap W\right) \rightrightarrows \mathcal{B}\left(
V\right) \oplus \mathcal{B}\left( W\right) \right) ,
\end{equation*}%
therefore%
\begin{eqnarray*}
b &=&\mathcal{B}\left( V\rightarrow U\right) \left( c\right) +\mathcal{B}%
\left( W\rightarrow U\right) \left( d\right) , \\
c &\in &\mathcal{B}\left( V\right) , \\
d &\in &\mathcal{B}\left( W\right) .
\end{eqnarray*}%
It follows that%
\begin{equation*}
\dsum\limits_{\alpha }^{\text{finite}}n_{\alpha }\left[ \varphi _{\alpha
}\in U^{T}\right] =\xi _{U}\left( b\right) =\dsum\limits_{\beta }^{\text{%
finite}}m_{\beta }\left[ \psi _{\beta }\in V^{T}\right] +\dsum\limits_{%
\gamma }^{\text{finite}}k_{\gamma }\left[ \xi _{\gamma }\in W^{T}\right] .
\end{equation*}%
Since $\varphi _{\alpha }\left( T\right) $ intersects both $U-V$ and $U-W$,
no one of the functions $\psi _{\beta }$ or $\xi _{\gamma }$ is equal to $%
\varphi _{\alpha }$. The above equality is therefore impossible.
Contradiction.

It follows that any morphism $\mathcal{B}\longrightarrow \mathcal{A}_{T}$
factors uniquely through $\mathcal{C}$:%
\begin{equation*}
\mathcal{B}\longrightarrow \mathcal{C}\longrightarrow \mathcal{A}_{T},
\end{equation*}%
therefore%
\begin{equation*}
Hom_{\mathbf{CS}\left( X,\mathbf{Ab}\right) }\left( \mathcal{B},\left( 
\mathcal{A}_{T}\right) _{\#}\right) 
\simeq%
Hom_{\mathbf{pCS}\left( X,\mathbf{Ab}\right) }\left( \mathcal{B},\mathcal{A}%
_{T}\right) 
\simeq%
Hom_{\mathbf{CS}\left( X,\mathbf{Ab}\right) }\left( \mathcal{B},\mathcal{C}%
\right)
\end{equation*}%
naturally in $\mathcal{B}$. The Yoneda's lemma implies that%
\begin{equation*}
\left( \mathcal{A}_{T}\right) _{\#}%
\simeq%
\mathcal{C}.
\end{equation*}
\end{proof}

\begin{example}
\label{Ex-Non-exactness-Ab-valued}It is clear that%
\begin{equation*}
\coker%
\left( \mathbb{Z}\left[ \mathcal{A}_{\Delta ^{1}}\right] \rightrightarrows 
\mathbb{Z}\left[ \mathcal{A}_{\Delta ^{0}}\right] \right) 
\simeq%
H_{0}
\end{equation*}%
($\mathbb{Z}\left[ \mathcal{A}_{\Delta ^{0}}\right] $ and $H_{0}$ are
cosheaves, see \cite[Example 4.3]{Prasolov-Cosheafification-2016-MR3660525}).%
\begin{eqnarray*}
&&\left( \mathbb{Z}\left[ \mathcal{A}_{\Delta ^{1}}\right] \right) _{\#}%
\simeq%
\left( \mathbb{Z}\left[ \mathcal{A}_{\Delta ^{0}}\right] \right) _{\#}%
\simeq%
\mathbb{Z}\left[ \mathcal{A}_{\Delta ^{0}}\right] ,\text{ therefore} \\
&&%
\coker%
\left( \left( \mathbb{Z}\left[ \mathcal{A}_{\Delta ^{1}}\right] \right)
_{\#}\rightrightarrows \left( \mathbb{Z}\left[ \mathcal{A}_{\Delta ^{0}}%
\right] \right) _{\#}\right) 
\simeq%
\coker%
\left( \mathbb{Z}\left[ \mathcal{A}_{\Delta ^{0}}\right] \rightrightarrows 
\mathbb{Z}\left[ \mathcal{A}_{\Delta ^{0}}\right] \right) 
\simeq%
\mathbb{Z}\left[ \mathcal{A}_{\Delta ^{0}}\right] .
\end{eqnarray*}%
Let $X$ be any $T_{1}$ space such that the natural mapping $X\rightarrow \pi
_{0}\left( X\right) $ is \textbf{not} bijective. Then%
\begin{eqnarray*}
&&\mathbb{Z}\left[ \mathcal{A}_{\Delta ^{0}}\right] 
\simeq%
\coker%
\left( \left( \mathbb{Z}\left[ \mathcal{A}_{\Delta ^{1}}\right] \right)
_{\#}\rightrightarrows \left( \mathbb{Z}\left[ \mathcal{A}_{\Delta ^{0}}%
\right] \right) _{\#}\right) \\
&\ncong &\left( 
\coker%
\left( \mathbb{Z}\left[ \mathcal{A}_{\Delta ^{1}}\right] \rightrightarrows 
\mathbb{Z}\left[ \mathcal{A}_{\Delta ^{0}}\right] \right) \right) _{\#}%
\simeq%
\left( H_{0}\right) _{\#}%
\simeq%
H_{0}
\end{eqnarray*}%
because%
\begin{equation*}
\mathbb{Z}\left[ \mathcal{A}_{\Delta ^{0}}\left( X\right) \right] 
\simeq%
\mathbb{Z}\left[ X\right] \longrightarrow H_{0}\left( X\right) 
\simeq%
\mathbb{Z}\left[ \pi _{0}\left( X\right) \right]
\end{equation*}%
is not an isomorphism. Therefore,%
\begin{equation*}
\left( \bullet \right) _{\#}:\mathbf{pCS}\left( X,\mathbf{Ab}\right)
\longrightarrow \mathbf{CS}\left( X,\mathbf{Ab}\right)
\end{equation*}%
is \textbf{not} right exact.
\end{example}

\section{\label{Sec-Examples}Examples}

\subsection{Pro-categories}

\begin{example}
\label{Ex-Limits-in-Pro(K)}Below are several examples of limits $%
\underleftarrow{\lim }$ in $\mathbf{Pro}\left( \mathbf{K}\right) $.

\begin{enumerate}
\item \label{Ex-Limits-in-Pro(K)-kernels}Kernels. Given two morphisms%
\begin{equation*}
\alpha ,\beta :\mathbf{X}\longrightarrow \mathbf{Y}
\end{equation*}%
in $\mathbf{Pro}\left( \mathbf{K}\right) $. They may be \textquotedblleft
levelized\textquotedblright , i.e., one may assume that $\mathbf{X}$ and $%
\mathbf{Y}$ are defined with the same index category $\mathbf{I}$, while $%
\alpha $ and $\beta $ are level morphisms:%
\begin{eqnarray*}
\mathbf{X} &\mathbf{=}&\left( X_{i}\right) _{i\in \mathbf{I}},\mathbf{Y=}%
\left( Y_{i}\right) _{i\in \mathbf{I}}, \\
\alpha &=&\left( \alpha _{i}:X_{i}\rightarrow Y_{i}\right) _{i\in \mathbf{I}%
}, \\
\beta &=&\left( \beta _{i}:X_{i}\rightarrow Y_{i}\right) _{i\in \mathbf{I}}.
\end{eqnarray*}%
Take%
\begin{equation*}
Z_{i}%
{:=}%
\ker \left( X_{i}\overset{\alpha ,\beta }{\rightrightarrows }Y_{i}\right) .
\end{equation*}%
It can be proved that%
\begin{equation*}
\mathbf{Z}%
{:=}%
\left( Z_{i}\right) _{i\in \mathbf{I}}
\end{equation*}%
is the kernel of $\left( \alpha ,\beta :\mathbf{X}\longrightarrow \mathbf{Y}%
\right) $ in $\mathbf{Pro}\left( \mathbf{K}\right) $. Indeed, for any $%
\mathbf{T=}\left( T_{s}\right) _{s\in \mathbf{S}}$,%
\begin{eqnarray*}
&&%
\Hom%
_{\mathbf{Pro}\left( \mathbf{K}\right) }\left( \mathbf{T,Z}\right) 
\simeq%
\underleftarrow{\lim }_{i\in \mathbf{I}}~\underrightarrow{\lim }_{s\in 
\mathbf{S}}~%
\Hom%
_{\mathbf{K}}\left( T_{s},Z_{i}\right) 
\simeq
\\
&&%
\simeq%
\underleftarrow{\lim }_{i\in \mathbf{I}}~\underrightarrow{\lim }_{s\in 
\mathbf{S}}~\ker \left( 
\Hom%
_{\mathbf{K}}\left( T_{s},X_{i}\right) \rightrightarrows 
\Hom%
_{\mathbf{K}}\left( T_{s},Y_{i}\right) \right) 
\simeq
\\
&&%
\simeq%
\underleftarrow{\lim }_{i\in \mathbf{I}}~\ker \left( \underrightarrow{\lim }%
_{s\in \mathbf{S}}~%
\Hom%
_{\mathbf{K}}\left( T_{s},X_{i}\right) \rightrightarrows \underrightarrow{%
\lim }_{s\in \mathbf{S}}~%
\Hom%
_{\mathbf{K}}\left( T_{s},Y_{i}\right) \right) 
\simeq
\\
&&%
\simeq%
\ker \left( \underleftarrow{\lim }_{i\in \mathbf{I}}~\underrightarrow{\lim }%
_{s\in \mathbf{S}}~%
\Hom%
_{\mathbf{K}}\left( T_{s},X_{i}\right) \rightrightarrows \underleftarrow{%
\lim }_{i\in \mathbf{I}}~\underrightarrow{\lim }_{s\in \mathbf{S}}~%
\Hom%
_{\mathbf{K}}\left( T_{s},Y_{i}\right) \right) 
\simeq
\\
&&%
\simeq%
\ker \left( 
\Hom%
_{\mathbf{Pro}\left( \mathbf{K}\right) }\left( \mathbf{T},\mathbf{X}\right)
\rightrightarrows 
\Hom%
_{\mathbf{Pro}\left( \mathbf{K}\right) }\left( \mathbf{T},\mathbf{Y}\right)
\right) ,
\end{eqnarray*}%
because $\underleftarrow{\lim }$ preserves all small limits, while \textbf{%
filtered} $\underrightarrow{\lim }$ preserves all \textbf{finite} limits.

\item \label{Ex-Limits-in-Pro(K)-products}Products. Let $\left( \mathbf{X}%
_{\alpha }\right) _{\alpha \in A}$ be a family of pro-objects, and let%
\begin{equation*}
\mathbf{X}=\dprod\limits_{\alpha \in A}\mathbf{X}_{\alpha }
\end{equation*}%
in $\mathbf{Pro}\left( \mathbf{K}\right) $.

\begin{enumerate}
\item $A$ is finite. Proceed as for kernels. One may assume that $\mathbf{X}%
_{\alpha }$ have the same index category:%
\begin{equation*}
\mathbf{X}_{\alpha }=\left( X_{\alpha ,i}\right) _{i\in \mathbf{I}}.
\end{equation*}%
The construction%
\begin{equation*}
\mathbf{X=}\left( \dprod\limits_{\alpha \in A}X_{\alpha ,i}\right) _{i\in 
\mathbf{I}}
\end{equation*}%
gives the desired product.

\item For a general $A$, let $\left( I,\leq \right) $ be the poset of finite
subsets of $A$ with the following ordering%
\begin{equation*}
S\leq T\iff T\subseteq S.
\end{equation*}%
Clearly, $\left( I,\leq \right) $ is \textbf{co}directed, and the
corresponding category $\mathbf{I}$ is \textbf{co}filtered. Moreover,%
\begin{equation*}
\mathbf{X}=\dprod\limits_{\alpha \in A}\mathbf{X}_{\alpha }%
\simeq%
\underset{S\in \left( I,\leq \right) }{\underleftarrow{\lim }}%
~\dprod\limits_{\alpha \in S}\mathbf{X}_{\alpha }.
\end{equation*}%
Notice that $\iota _{\mathbf{K}}:\mathbf{K}\rightarrow \mathbf{Pro}\left( 
\mathbf{K}\right) $ does \textbf{not} in general preserve products. Say, if $%
\mathbf{K}=\mathbf{Ab}$, $\mathbb{N}=\left\{ 0,1,2,\dots \right\} $ and $X$
is a nontrivial abelian group, then%
\begin{equation*}
\mathbf{Y}=\dprod\limits_{n\in \mathbb{N}}X=\left( 0\longleftarrow
X\longleftarrow X^{2}\longleftarrow \dots \longleftarrow X^{n}\longleftarrow
\dots \right)
\end{equation*}%
in $\mathbf{Pro}\left( \mathbf{Ab}\right) $, which is \textbf{not}
isomorphic to the abelian group $\Pi _{n\in \mathbb{N}}X$. In fact, $\mathbf{%
Y}$ is not rudimentary, see Example \ref{Ex-Non-rudimentary-tower}.
\end{enumerate}

\item \label{Ex-Limits-in-Pro(K)-cofiltered-limits}\textbf{Co}filtered
limits.

\begin{enumerate}
\item Proposition \ref{Prop-Facts-pro-objects}(\ref%
{Prop-Facts-pro-objects-admits-cofiltered-limits}) guarantees the existence.
The construction is given in \cite[Proposition A.4.4]{Artin-Mazur-MR883959}.

\item Consider the following simple case. Let%
\begin{equation*}
D:\mathbf{I}\longrightarrow \mathbf{Pro}\left( \mathbf{K}\right)
\end{equation*}%
be a \textbf{co}filtered diagram, and assume that for any $i\in \mathbf{I}$, 
$D\left( i\right) $ is a rudimentary pro-object, $D\left( i\right) 
\simeq%
E_{i}\in \mathbf{K}$. The diagram $\left( E_{i}\right) _{i\in \mathbf{I}}$
represents an object $\mathbf{E}$ in $\mathbf{Pro}\left( \mathbf{K}\right) $%
. It is easy to see that%
\begin{equation*}
\mathbf{E}%
\simeq%
\underset{i\in \mathbf{I}}{\underleftarrow{\lim }}~D\left( i\right) .
\end{equation*}%
It follows that the embedding%
\begin{equation*}
\iota _{\mathbf{K}}:\mathbf{K}\longrightarrow \mathbf{Pro}\left( \mathbf{K}%
\right)
\end{equation*}%
preserves cofiltered limits. However, due to Example \ref%
{Ex-Limits-in-Pro(K)}(\ref{Ex-Limits-in-Pro(K)-products}),%
\begin{equation*}
\iota _{\mathbf{K}}:\mathbf{C}\longrightarrow \mathbf{Pro}\left( \mathbf{K}%
\right)
\end{equation*}%
is \textbf{not} continuous.
\end{enumerate}

\item \label{Ex-Limits-in-Pro(K)-small-limits}Small limits. Any small limit
is a combination of two products and a kernel:%
\begin{equation*}
\underset{i\in \mathbf{I}}{\underleftarrow{\lim }}~X\left( i\right) =\ker
\left( \dprod\limits_{i\in \mathbf{I}}X\left( i\right) \rightrightarrows
\dprod\limits_{\left( i\rightarrow j\right) \in 
\Hom%
_{\mathbf{I}}\left( i,j\right) }X\left( j\right) \right) .
\end{equation*}
\end{enumerate}
\end{example}

\begin{example}
\label{Ex-Colimits-in-Pro(K)}~

\begin{enumerate}
\item \label{Ex-Colimits-in-Pro(K)-cokernels}Cokernels. Given two morphisms%
\begin{equation*}
\alpha ,\beta :\mathbf{X}\longrightarrow \mathbf{Y}
\end{equation*}%
in $\mathbf{Pro}\left( \mathbf{K}\right) $. They may be \textquotedblleft
levelized\textquotedblright , i.e., one may assume that $\mathbf{X}$ and $%
\mathbf{Y}$ are defined with the same index category $\mathbf{I}$, while $%
\alpha $ and $\beta $ are level morphisms:%
\begin{eqnarray*}
\mathbf{X} &\mathbf{=}&\left( X_{i}\right) _{i\in \mathbf{I}},\mathbf{Y=}%
\left( Y_{i}\right) _{i\in \mathbf{I}}, \\
\alpha &=&\left( \alpha _{i}:X_{i}\rightarrow Y_{i}\right) _{i\in \mathbf{I}%
}, \\
\beta &=&\left( \beta _{i}:X_{i}\rightarrow Y_{i}\right) _{i\in \mathbf{I}}.
\end{eqnarray*}%
Take%
\begin{equation*}
Z_{i}%
{:=}%
\coker%
\left( X_{i}\overset{\alpha ,\beta }{\rightrightarrows }Y_{i}\right) .
\end{equation*}%
It can be proved that%
\begin{equation*}
\mathbf{Z}%
{:=}%
\left( Z_{i}\right) _{i\in \mathbf{I}}
\end{equation*}%
is the cokernel of $\left( \alpha ,\beta :\mathbf{X}\longrightarrow \mathbf{Y%
}\right) $ in $\mathbf{Pro}\left( \mathbf{K}\right) $. Indeed, for any $%
\mathbf{T=}\left( T_{s}\right) _{s\in \mathbf{S}}$,%
\begin{eqnarray*}
&&%
\Hom%
_{\mathbf{Pro}\left( \mathbf{K}\right) }\left( \mathbf{Z,T}\right) 
\simeq%
\underleftarrow{\lim }_{s\in \mathbf{S}}~\underrightarrow{\lim }_{i\in 
\mathbf{I}}~%
\Hom%
_{\mathbf{K}}\left( Z_{i},T_{s}\right) 
\simeq
\\
&&%
\simeq%
\underleftarrow{\lim }_{s\in \mathbf{S}}~\underrightarrow{\lim }_{i\in 
\mathbf{I}}~\ker \left( 
\Hom%
_{\mathbf{K}}\left( Y_{i},T_{s}\right) \rightrightarrows 
\Hom%
_{\mathbf{K}}\left( X_{i},T_{s}\right) \right) 
\simeq
\\
&&%
\simeq%
\underleftarrow{\lim }_{s\in \mathbf{S}}~\ker \left( \underrightarrow{\lim }%
_{i\in \mathbf{I}}~%
\Hom%
_{\mathbf{K}}\left( Y_{i},T_{s}\right) \rightrightarrows \underrightarrow{%
\lim }_{s\in \mathbf{S}}~%
\Hom%
_{\mathbf{K}}\left( X_{i},T_{s}\right) \right) 
\simeq
\\
&&%
\simeq%
\ker \left( \underleftarrow{\lim }_{s\in \mathbf{S}}~\underrightarrow{\lim }%
_{i\in \mathbf{I}}~%
\Hom%
_{\mathbf{K}}\left( Y_{i},T_{s}\right) \rightrightarrows \underleftarrow{%
\lim }_{s\in \mathbf{S}}~\underrightarrow{\lim }_{i\in \mathbf{I}}~%
\Hom%
_{\mathbf{K}}\left( X_{i},T_{s}\right) \right) 
\simeq
\\
&&%
\simeq%
\ker \left( 
\Hom%
_{\mathbf{Pro}\left( \mathbf{K}\right) }\left( \mathbf{Y},\mathbf{T}\right)
\rightrightarrows 
\Hom%
_{\mathbf{Pro}\left( \mathbf{K}\right) }\left( \mathbf{X},\mathbf{T}\right)
\right) ,
\end{eqnarray*}%
because $\underleftarrow{\lim }$ preserves all small limits, while \textbf{%
filtered} $\underrightarrow{\lim }$ preserves all \textbf{finite} colimits.

\item \label{Ex-Colimits-in-Pro(K)-coproducts}Coproducts. Let $\left( 
\mathbf{X}_{\alpha }\in \mathbf{Pro}\left( \mathbf{K}\right) \right)
_{\alpha \in A}$ be given by%
\begin{equation*}
\mathbf{X}_{\alpha }=\left( X_{\alpha ,i}\right) _{i\in \mathbf{I}_{\alpha
}}.
\end{equation*}%
Define 
\begin{equation*}
\mathbf{I}=\dprod\limits_{\alpha \in A}\mathbf{I}_{\alpha }.
\end{equation*}%
$\mathbf{I}$ is clearly cofiltered. For $\left( i\left( \alpha \right)
\right) _{\alpha \in A}\in \mathbf{I}$ let%
\begin{equation*}
X_{i}=\dcoprod\limits_{\alpha \in A}X_{\alpha ,i\left( \alpha \right) },
\end{equation*}%
and define $\mathbf{X}=\left( X_{i}\right) _{i\in \mathbf{I}}$. The morphisms%
\begin{equation*}
f_{\alpha }\in 
\Hom%
_{\mathbf{Pro}\left( \mathbf{K}\right) }\left( \mathbf{X}_{\alpha },\mathbf{X%
}\right) =\underleftarrow{\lim }_{i\in \mathbf{I}}~\underrightarrow{\lim }%
_{j\in \mathbf{I}_{\alpha }}~%
\Hom%
_{\mathbf{K}}\left( X_{\alpha ,j},X_{i}\right)
\end{equation*}%
are given by $\left( \left[ \iota _{\alpha ,i\left( \alpha \right) }\right]
\right) _{i\in \mathbf{I}}$ where $\left[ \iota _{\alpha ,i\left( \alpha
\right) }\right] $ is the class (in $\underrightarrow{\lim }_{j\in \mathbf{I}%
_{\alpha }}~$) of the embedding of $X_{\alpha ,i\left( \alpha \right) }$
into the coproduct.

\begin{enumerate}
\item Let $Y\in \mathbf{K}\subseteq \mathbf{Pro}\left( \mathbf{K}\right) $
be a \textbf{rudimentary} pro-object. Let also%
\begin{equation*}
\varphi :%
\Hom%
_{\mathbf{Pro}\left( \mathbf{K}\right) }\left( \mathbf{X},Y\right)
\longrightarrow \dprod\limits_{\alpha \in A}%
\Hom%
_{\mathbf{Pro}\left( \mathbf{K}\right) }\left( \mathbf{X}_{\alpha },Y\right)
\end{equation*}%
be the following mapping: $\varphi \left( g\right) =\left( g\circ f_{\alpha
}\right) _{\alpha \in A}$. We claim that $\varphi $ is an isomorphism.

\begin{enumerate}
\item $\varphi $ is a \textbf{mono}morphism. Indeed, let $g,h\in 
\Hom%
_{\mathbf{Pro}\left( \mathbf{K}\right) }\left( \mathbf{X},Y\right) $ be
represented by $i,j\in \mathbf{I}$ and%
\begin{eqnarray*}
\left( g\left( \alpha \right) :X_{\alpha ,i\left( \alpha \right)
}\rightarrow Y\right) _{\alpha \in A} &:&X_{i}=\dcoprod\limits_{\alpha \in
A}X_{\alpha ,i\left( \alpha \right) }\longrightarrow Y, \\
\left( h\left( \alpha \right) :X_{\alpha ,j\left( \alpha \right)
}\rightarrow Y\right) _{\alpha \in A} &:&X_{j}=\dcoprod\limits_{\alpha \in
A}X_{\alpha ,j\left( \alpha \right) }\longrightarrow Y.
\end{eqnarray*}%
Assume that $\varphi \left( g\right) =\varphi \left( h\right) $. It means
that $g\circ f_{\alpha }=h\circ f_{\alpha }$ for each $\alpha \in A$.
Therefore, for each $\alpha \in A$ there are morphisms%
\begin{eqnarray*}
\left( s\left( \alpha \right) \rightarrow i\left( \alpha \right) \right)
&\in &%
\Hom%
_{\mathbf{I}_{\alpha }}\left( s\left( \alpha \right) ,i\left( \alpha \right)
\right) , \\
\left( s\left( \alpha \right) \rightarrow j\left( \alpha \right) \right)
&\in &%
\Hom%
_{\mathbf{I}_{\alpha }}\left( s\left( \alpha \right) ,j\left( \alpha \right)
\right) ,
\end{eqnarray*}%
such that the diagrams%
\begin{equation*}
\begin{diagram}[size=3.0em,textflow]
X_{\alpha ,s\left( \alpha \right) } & \rTo & X_{\alpha ,i\left(
\alpha \right) } \\ 
\dTo &  & \dTo>{g\left( \alpha \right)} \\ 
X_{\alpha ,j\left( \alpha \right) } & \rTo^{h\left( \alpha \right) } & Y \\
\end{diagram}%
\end{equation*}%
commute in $\mathbf{K}$. The family $\left( s\left( \alpha \right) \right)
_{\alpha \in A}$ determines an object $s$ of $\mathbf{I}$, and we have two
morphisms $s\rightarrow i$ and $s\rightarrow j$ in $\mathbf{I}$. It follows
that the diagram%
\begin{equation*}
\begin{diagram}[size=3.0em,textflow]
X_{s} & \rTo & X_{i} \\ 
\dTo &  & \dTo>{g} \\ 
X_{j} & \rTo^{h} & Y \\
\end{diagram}
%
\end{equation*}%
commutes in $\mathbf{K}$, and therefore $g=h$.

\item $\varphi $ is an \textbf{epi}morphism. Indeed, let%
\begin{equation*}
\left( g\left( \alpha \right) \right) _{\alpha \in A}\in
\dprod\limits_{\alpha \in A}%
\Hom%
_{\mathbf{Pro}\left( \mathbf{K}\right) }\left( \mathbf{X}_{\alpha },Y\right)
,
\end{equation*}%
and let each $g\left( \alpha \right) $ be represented by $i\left( \alpha
\right) \in \mathbf{I}_{\alpha }$ and%
\begin{equation*}
h_{\alpha ,i\left( \alpha \right) }:X_{\alpha ,i\left( \alpha \right)
}\longrightarrow Y.
\end{equation*}%
The $A$-tuple $\,i=\left( i\left( \alpha \right) \right) _{\alpha \in A}$
determines an object of $\mathbf{I}$, and we have%
\begin{equation*}
h=\left( h_{\alpha ,i\left( \alpha \right) }\right) _{\alpha \in
A}:X_{i}=\dcoprod\limits_{\alpha \in A}X_{\alpha ,i\left( \alpha \right)
}\longrightarrow X.
\end{equation*}%
Clearly,%
\begin{equation*}
\varphi \left( h\right) =\left( h\circ f_{\alpha }\right) _{\alpha \in
A}=\left( g\left( \alpha \right) \right) _{\alpha \in A},
\end{equation*}%
and $\varphi $ is an epimorphism.
\end{enumerate}

\item Let now $\mathbf{Y}=\left( Y_{s}\right) _{s\in \mathbf{S}}$.%
\begin{eqnarray*}
&&%
\Hom%
_{\mathbf{Pro}\left( \mathbf{K}\right) }\left( \mathbf{X},\mathbf{Y}\right) 
\simeq%
\underleftarrow{\lim }_{s\in \mathbf{S}}~%
\Hom%
_{\mathbf{Pro}\left( \mathbf{K}\right) }\left( \mathbf{X},Y_{s}\right) 
\simeq
\\
&&%
\simeq%
\underleftarrow{\lim }_{s\in \mathbf{S}}~\dprod\limits_{\alpha \in A}%
\Hom%
_{\mathbf{Pro}\left( \mathbf{K}\right) }\left( \mathbf{X}_{\alpha
},Y_{s}\right) 
\simeq%
\dprod\limits_{\alpha \in A}\underleftarrow{\lim }_{s\in \mathbf{S}}~%
\Hom%
_{\mathbf{Pro}\left( \mathbf{K}\right) }\left( \mathbf{X}_{\alpha
},Y_{s}\right) 
\simeq
\\
&&%
\simeq%
\dprod\limits_{\alpha \in A}%
\Hom%
_{\mathbf{Pro}\left( \mathbf{K}\right) }\left( \mathbf{X}_{\alpha },\mathbf{Y%
}\right) ,
\end{eqnarray*}%
because limits commute with limits, e.g., products.
\end{enumerate}

\item \label{Ex-Colimits-in-Pro(K)-small-colimits}Small colimits. Any small
colimit is a combination of two coproducts and a cokernel:%
\begin{equation*}
\underrightarrow{\lim }~\left( F:\mathbf{I}\rightarrow \mathbf{Pro}\left( 
\mathbf{K}\right) \right) =%
\coker%
\left( \dcoprod\limits_{\left( i\rightarrow j\right) \in 
\Hom%
_{\mathbf{I}}\left( i,j\right) }F\left( i\right) \rightrightarrows
\dcoprod\limits_{i\in \mathbf{I}}F\left( i\right) \right) .
\end{equation*}
\end{enumerate}
\end{example}

Below is a series of examples of rudimentary and non-rudimentary objects. We
will need the following simple lemma:

\begin{lemma}
\label{Lemma-Split-mono-epi}Let $\mathbf{K}$ be a category, and 
\begin{equation*}
\varphi \in 
\Hom%
_{\mathbf{K}}\left( X,Y\right) .
\end{equation*}

\begin{enumerate}
\item If $\varphi $ is a split monomorphism (i.e., has a left inverse) and
an epimorphism, then $\varphi $ is an isomorphism.

\item If $\varphi $ is a split epimorphism (i.e., has a right inverse) and a
monomorphism, then $\varphi $ is an isomorphism.
\end{enumerate}
\end{lemma}

\begin{proof}
~

\begin{enumerate}
\item \label{Lemma-Epi-split-mono}~

\begin{enumerate}
\item It follows that $\xi \circ \varphi =\mathbf{1}_{X}$ for some $\xi
:Y\rightarrow X$. Therefore,%
\begin{equation*}
\varphi \circ \xi \circ \varphi =\varphi =\mathbf{1}_{Y}\circ \varphi .
\end{equation*}%
Since $\varphi $ is an epimorphism, $\varphi \circ \xi =\mathbf{1}_{Y}$, and
we are done.

\item Alternative proof. $\varphi $ is a split monomorphism iff $h^{\varphi
}:h^{Y}\rightarrow h^{X}$ is an \textbf{epi}morphism in $\mathbf{Set}^{%
\mathbf{K}}$, and is an epimorphism, iff $h^{\varphi }:h^{Y}\rightarrow
h^{X} $ is a \textbf{mono}morphism in $\mathbf{Set}^{\mathbf{K}}$.
Therefore, $h^{\varphi }$ is an isomorphism, and $\varphi $ is an
isomorphism as well, due to Yoneda's lemma.
\end{enumerate}

\item \label{Lemma-Mono-split-epi}Dual to (\ref{Lemma-Epi-split-mono}).
\end{enumerate}
\end{proof}

\label{Cor-Non-trivial-pro-object}\label{Cor-Non-rudimentary-pro-object}

\begin{example}
\label{Ex-Non-rudimentary-tower}Let $\mathbf{X}\in \mathbf{Pro}\left( 
\mathbf{K}\right) $ be given by the tower%
\begin{equation*}
X_{0}\overset{\rho _{0}}{\longleftarrow }X_{1}\overset{\rho _{1}}{%
\longleftarrow }X_{2}\overset{\rho _{2}}{\longleftarrow }\dots \overset{\rho
_{i-1}}{\longleftarrow }X_{i}\overset{\rho _{i}}{\longleftarrow }\dots
\end{equation*}%
Assume one of the following:

\begin{enumerate}
\item \label{Ex-Non-rudimentary-tower-epi}All $\rho _{i}$ are epimorphisms,
but not isomorphisms.

\item \label{Ex-Non-rudimentary-tower-mono}$\mathbf{K}$ is abelian, and all $%
\rho _{i}$ are monomorphisms, but not isomorphisms.
\end{enumerate}

Then $\mathbf{X}$ is \textbf{not} rudimentary.
\end{example}

\begin{proof}
For $i\leq j$, denote%
\begin{equation*}
\rho _{i,j}~%
{:=}%
~\rho _{i}\circ \rho _{i+1}\circ \dots \circ \rho _{j-1}.
\end{equation*}%
Assume $\mathbf{X}%
\simeq%
Z\in \mathbf{K}$. There are two morphisms%
\begin{equation*}
\alpha :\mathbf{X}\longrightarrow Z,\beta :Z\longrightarrow \mathbf{X}
\end{equation*}

in $\mathbf{Pro}\left( \mathbf{K}\right) $, such that%
\begin{equation*}
\alpha \circ \beta =1_{Z},\beta \circ \alpha =1_{\mathbf{X}}.
\end{equation*}

$\alpha \in 
\Hom%
_{\mathbf{Pro}\left( \mathbf{K}\right) }\left( \mathbf{X},Z\right) =%
\underrightarrow{\lim }_{i}~%
\Hom%
_{\mathbf{K}}\left( X_{i},Z\right) $ is represented by $\gamma
:X_{k}\rightarrow Z$ for some $k\geq 0$.

$\beta \in 
\Hom%
_{\mathbf{Pro}\left( \mathbf{K}\right) }\left( Z,\mathbf{X}\right) =%
\underleftarrow{\lim }_{i}~%
\Hom%
_{\mathbf{K}}\left( Z,X_{i}\right) $ is represented by the family $\left(
\delta _{i}:Z\rightarrow X_{i}\right) _{i\geq 0}$ where $\rho _{i}\circ
\delta _{i+1}=\delta _{i}$ for all $i\geq 0$. Clearly,%
\begin{equation*}
\rho _{i,j}\circ \delta _{j}=\delta _{i},0\leq i\leq j.
\end{equation*}%
One may ignore the items $X_{i}$, $i<k$, because the index subcategory%
\begin{equation*}
\mathbf{J}=\left\{ k\longleftarrow k+1\longleftarrow k+2\longleftarrow \dots
\right\}
\end{equation*}%
is co-cofinal in the main index category%
\begin{equation*}
\mathbf{I}=\left\{ 0\longleftarrow 1\longleftarrow 2\longleftarrow \dots
\right\} .
\end{equation*}%
See the diagrams below:%
\begin{equation*}
\begin{diagram}[size=3.0em,textflow]
X_{k} & \lTo^{\rho _{k}} & X_{k+1} & \lTo^{\rho _{k+1}} & X_{k+2} & \lTo^{\rho _{k+2}} & \dots
& \lTo^{\rho _{i-1}} & X_{i} & \lTo^{\rho _{i}} & \dots \\ 
\dTo<{\gamma} & \ruTo<{\delta _{k+1}} & & \ruTo(4,2) & & & & \ruTo(8,2)>{\delta _{i}} & & & & & & \\
Z &  &  &  &  & & & \\
\end{diagram}
%
\end{equation*}%
and%
\begin{equation*}
\begin{diagram}[size=3.0em,textflow]
X_{k} & \lTo^{\rho _{k}} & X_{k+1} & \lTo^{\rho _{k+1}} & X_{k+2} & \lTo^{\rho _{k+2}} & \dots
& \lTo^{\rho _{j-1}} & X_{j} & \lTo^{\rho _{j}} & \dots \\ 
\dTo<{\delta _{k}\circ \gamma} & \rdTo>{\delta _{k+1}\circ \gamma} \rdTo(8,2)>{\delta _{j} \circ \gamma} &&&&&&&&&&&&&&&&&& \\
X_{k} & \lTo_{\rho _{k}} & X_{k+1} & \lTo_{\rho _{k+1}} & X_{k+2} & \lTo_{\rho _{k+2}} & \dots
& \lTo_{\rho _{j-1}} & X_{j} & \lTo_{\rho _{j}} & \dots \\ 
\end{diagram}
%
\end{equation*}

$\alpha \circ \beta =\mathbf{1}_{Z}$ means $\gamma \circ \delta _{k}=\mathbf{%
1}_{Z}$. It follows that $Z$ is a \textbf{retract} of $X_{k}$, while $\gamma 
$ and $\delta _{k}$ are the \textbf{retraction} and \textbf{section},
respectively. Therefore, $\gamma $ is a \textbf{split} \textbf{epi}morphism,
while $\delta _{k}$ is a \textbf{split} \textbf{mono}morphism. It follows
also that for any $s\geq k$%
\begin{equation}
\mathbf{1}_{Z}=\gamma \circ \delta _{k}=\gamma \circ \rho _{k,s}\circ \delta
_{s}.  \label{Line-Tower-gamma-delta}
\end{equation}

$\beta \circ \alpha $ is represented (compare to \cite[Remark 2.1.9]%
{Prasolov-Cosheaves-2021-MR4347662}) by $\left( \varepsilon _{j}\right)
_{j\geq k}$, where each $\varepsilon _{j}$ is, in turn, represented by the
equivalence class%
\begin{equation*}
\varepsilon _{j}=\left[ \delta _{j}\circ \gamma \right] \in \Phi _{j}~%
{:=}%
~\underset{i\geq k}{\underrightarrow{\lim }}~%
\Hom%
_{\mathbf{K}}\left( X_{i},X_{j}\right) .
\end{equation*}%
The equality%
\begin{equation*}
\beta \circ \alpha =\mathbf{1}_{\mathbf{X}}\in 
\Hom%
_{\mathbf{Pro}\left( \mathbf{K}\right) }\left( \mathbf{X},\mathbf{X}\right) =%
\underset{j\geq k}{\underleftarrow{\lim }}~\underset{i\geq k}{%
\underrightarrow{\lim }}~%
\Hom%
_{\mathbf{K}}\left( X_{i},X_{j}\right) =\underset{j\geq k}{\underleftarrow{%
\lim }}~\Phi _{j}
\end{equation*}%
means that for each $j\geq k$ the classes $\left[ \delta _{j}\circ \gamma %
\right] $ and $\left[ \mathbf{1}_{X_{j}}\right] $ are \textbf{equal} in $%
\Phi _{j}$, i.e., there exists%
\begin{equation*}
s\left( j\right) \geq \max \left\{ k,j\right\}
\end{equation*}%
such that%
\begin{equation}
\delta _{j}\circ \gamma \circ \rho _{k,s\left( j\right) }=\mathbf{1}%
_{X_{j}}\circ \rho _{j,s\left( j\right) }=\rho _{j,s\left( j\right) }.
\label{Line-Tower-delta-gamma}
\end{equation}

\begin{enumerate}
\item Assume now that all $\rho _{i}$ are \textbf{epi}morphisms, but \textbf{%
not} isomorphisms. All $\rho _{i,j}$, $k\leq i\leq j$, are epimorphisms,
too. Put $j=k+1$ into (\ref{Line-Tower-delta-gamma}):%
\begin{eqnarray*}
\left( \delta _{k+1}\circ \gamma \circ \rho _{k}\right) \circ \rho
_{k+1,s\left( k+1\right) } &=& \\
\delta _{k+1}\circ \gamma \circ \rho _{k,s\left( k+1\right) } &=&\rho
_{k+1,s\left( k+1\right) }=\mathbf{1}_{X_{k+1}}\circ \rho _{k+1,s\left(
k+1\right) }.
\end{eqnarray*}%
Since $\rho _{k+1,s\left( k+1\right) }$ is an epimorphism,%
\begin{equation*}
\delta _{k+1}\circ \gamma \circ \rho _{k}=\mathbf{1}_{X_{k+1}}.
\end{equation*}%
It follows that $\rho _{k}$ is a \textbf{split} \textbf{mono}morphism. Lemma %
\ref{Lemma-Split-mono-epi} implies that $\rho _{k}$ is an isomorphism.
Contradiction.

\item Assume that $\mathbf{K}$ is abelian, and that all $\rho _{i}$ are 
\textbf{mono}morphisms, but \textbf{not} isomorphisms.

Clearly, since $\delta _{k}=\rho _{k,s\left( k\right) }\circ \delta
_{s\left( k\right) }$,%
\begin{equation*}
\Imm%
\delta _{k}\subseteq 
\Imm%
\rho _{k,s\left( k\right) }.
\end{equation*}%
Put $j=k$ into (\ref{Line-Tower-delta-gamma}):%
\begin{equation*}
\delta _{k}\circ \gamma \circ \rho _{k,s\left( k\right) }=\rho _{k,s\left(
k\right) }.
\end{equation*}%
It follows that $%
\Imm%
\rho _{k,s\left( k\right) }\subseteq 
\Imm%
\delta _{k}$, therefore%
\begin{equation*}
\Imm%
\delta _{k}=%
\Imm%
\rho _{k,s\left( k\right) }.
\end{equation*}%
On the other hand, since $\delta _{k}=\rho _{k,s\left( k\right) +1}\circ
\delta _{s\left( k\right) +1}$,%
\begin{equation*}
\Imm%
\rho _{k,s\left( k\right) }=%
\Imm%
\delta _{k}\subseteq 
\Imm%
\rho _{k,s\left( k\right) +1}=%
\Imm%
\left( \rho _{k,s\left( k\right) }\circ \rho _{s\left( k\right) }\right) .
\end{equation*}%
Contradiction, because $\rho _{k,s\left( k\right) }$ and $\rho _{s\left(
k\right) }$ are \textbf{mono}morphisms, but \textbf{not} isomorphisms.
\end{enumerate}
\end{proof}

\begin{remark}
\label{Rem-Not-retract-of-rudimentary}The above proof does \textbf{not} use
equation \ref{Line-Tower-gamma-delta}. This means that in fact we have
proved a \textbf{stronger} assertion: $\mathbf{X}$ is not a \textbf{retract}
of a rudimentary pro-object. See, e.g., Example \ref%
{Prop-Retracts-non-rudimentary}.
\end{remark}

$~$

\begin{definition}
~

\begin{enumerate}
\item A morphism $\varepsilon \in 
\Hom%
_{\mathbf{K}}\left( X,X\right) $ \textbf{splits} in $\mathbf{K}$ iff there
are morphisms%
\begin{equation*}
\varphi :Y\longrightarrow X,\xi :X\longrightarrow Y
\end{equation*}%
satisfying%
\begin{equation*}
\xi \circ \varphi =\mathbf{1}_{Y},\varphi \circ \xi =\varepsilon .
\end{equation*}

\item A morphism $\varepsilon \in 
\Hom%
_{\mathbf{K}}\left( X,X\right) $ is called an \textbf{idempotent} iff $%
\varepsilon \circ \varepsilon =\varepsilon $.

\item Clearly, a morphism $\varepsilon $ is an idempotent if it splits:%
\begin{equation*}
\varepsilon \circ \varepsilon =\varphi \circ \xi \circ \varphi \circ \xi
=\varphi \circ \mathbf{1}_{Y}\circ \xi =\varphi \circ \xi =\varepsilon .
\end{equation*}%
An idempotent is called \textbf{splitting} or \textbf{non-splitting}
depending on whether it splits.
\end{enumerate}
\end{definition}

\begin{example}
\label{Ex-Splitting-and-Non-splitting}~

\begin{enumerate}
\item Categories that do \textbf{not} have non-splitting idempotents.

\begin{enumerate}
\item $\mathbf{Set}$, $\mathbf{Top}$ and $\mathbf{Top}_{\ast }$: sets,
(pointed) topological spaces:

\begin{enumerate}
\item $Y=\left\{ x\in X~|~\varepsilon \left( x\right) =x\right\} $ (with the
subspace topology);

\item $\varphi :Y\longrightarrow X$: $\varphi \left( y\right) 
{:=}%
y$;

\item $\xi :X\longrightarrow Y$: $\xi \left( x\right) 
{:=}%
\varepsilon \left( x\right) $.
\end{enumerate}

\item The category $H\left( \mathbf{Pol}_{\ast }\right) $ of \textbf{pointed}
polyhedra modulo homotopy, see \cite%
{Hastings-Heller-1981-Splitting-homotopy-idempotents-MR643520}.

\item Any pre-abelian category:

\begin{enumerate}
\item Let $Y=\ker \left( \mathbf{1}_{\mathbf{X}}-\varepsilon \right) $.

\item $f:Y\longrightarrow X$: the canonical embedding. $\left( \mathbf{1}%
_{X}-\varepsilon \right) \circ f=0$ implies $\varepsilon \circ f=f$.

\item Since $\left( \mathbf{1}_{\mathbf{X}}-\varepsilon \right) \circ
\varepsilon =0$, there exists (unique!) morphism $g:X\rightarrow Y$ such
that $f\circ g=\varepsilon $:%
\begin{equation*}
X\overset{g}{\longrightarrow }Y=\ker \left( \mathbf{1}_{\mathbf{X}%
}-\varepsilon \right) \overset{f}{\longrightarrow }X\overset{\mathbf{1}_{%
\mathbf{X}}-\varepsilon }{\longrightarrow }X.
\end{equation*}%
Moreover,%
\begin{equation*}
f\circ g\circ f=\varepsilon \circ f=f=f\circ \mathbf{1}_{Y}.
\end{equation*}%
Since $f$ is a monomorphism, $g\circ f=\mathbf{1}_{Y}$.
\end{enumerate}
\end{enumerate}

\item Categories that \textbf{do} have non-splitting idempotents.

\begin{enumerate}
\item The category $H\left( \mathbf{Pol}\right) $ of polyhedra modulo
homotopy, see \cite%
{Hastings-Heller-1981-Splitting-homotopy-idempotents-MR643520}.

\item Assume that there are projective $k$-modules that are not free, e.g.,
if $k=R\times S$. The $k$-module $R$ is projective but not free. The
category $\mathbf{F}$ of \textbf{free} $k$-modules \textbf{does} have
non-splitting idempotents. Let $P\in \mathbf{Mod}\left( k\right) $ be
projective, but not free. There is $Q\in \mathbf{Mod}\left( k\right) $ such
that%
\begin{equation*}
P\oplus Q%
\simeq%
F
\end{equation*}%
where $F$ is free. The composition%
\begin{equation*}
\varepsilon :F\twoheadrightarrow P\rightarrowtail F
\end{equation*}%
is a non-splitting idempotent.
\end{enumerate}
\end{enumerate}
\end{example}

We are interested in non-splitting idempotents because of the following

\begin{example}
\label{Prop-Retracts-non-rudimentary}\label{Ex-Retract-non-rudimentary}Let $%
\mathbf{K}$ be any category that \textbf{does} have non-splitting
idempotents, say, $\varepsilon :Y\rightarrow Y$. Consider $\mathbf{X}\in 
\mathbf{Pro}\left( \mathbf{K}\right) $ given by the tower%
\begin{equation*}
Y\overset{\varepsilon }{\longleftarrow }Y\overset{\varepsilon }{%
\longleftarrow }Y\overset{\varepsilon }{\longleftarrow }\dots \overset{%
\varepsilon }{\longleftarrow }Y\overset{\varepsilon }{\longleftarrow }\dots
\end{equation*}%
Then $\mathbf{X}$:

\begin{enumerate}
\item is a \textbf{retract} of a rudimentary pro-object;

\item is \textbf{not} rudimentary.
\end{enumerate}
\end{example}

\begin{proof}
~

\begin{enumerate}
\item There are two morphisms%
\begin{equation*}
\varphi :\mathbf{X}\longrightarrow Y,\xi :Y\longrightarrow \mathbf{X}
\end{equation*}%
given by the following diagrams:%
\begin{equation*}
\varphi =\left( 
\begin{diagram}[size=3.0em,textflow]
Y & \lTo^{\varepsilon } & Y & \lTo^{\varepsilon } & Y & \lTo^{\varepsilon } & \dots & 
\lTo^{\varepsilon } & Y & \lTo^{\varepsilon } & \dots \\ 
\dTo<{\mathbf{1}_{Y}} &  &  &  &  &  &  &  &  &  &  \\ 
Y &  &  &  &  &  &  &  &  &  & 
\end{diagram}%
\right)
\end{equation*}%
and%
\begin{equation*}
\xi =\left( 
\begin{diagram}[size=3.0em,textflow]
Y &  &  &  &  &  &  & &&&&& \\
\dTo<{\varepsilon} & \rdTo<{\varepsilon} \rdTo(4,2) \rdTo(8,2)>{\varepsilon } &  &  &  &  &  &  &  &  & \dots \\ 
Y & \lTo_{\varepsilon } & Y & \lTo_{\varepsilon } & Y & \lTo_{\varepsilon } & \dots & 
\lTo_{\varepsilon } & Y & \lTo_{\varepsilon } & \dots \\
\end{diagram}%
\right)
\end{equation*}%
Clearly $\varphi \circ \xi =\varepsilon $. Compute $\xi \circ \varphi $:%
\begin{equation*}
\xi \circ \varphi =\left( 
\begin{diagram}[size=3.0em,textflow]
Y & \lTo & Y & \lTo^{\varepsilon } & Y & \lTo^{\varepsilon } & \dots & 
\lTo^{\varepsilon } & Y & \lTo^{\varepsilon } & \dots \\ 
\dTo<{\varepsilon} & \rdTo<{\varepsilon} \rdTo(4,2) \rdTo(8,2)>{\varepsilon } &  &  &  &  &  &  &  &  & \dots \\ 
Y & \lTo_{\varepsilon } & Y & \lTo_{\varepsilon } & Y & \lTo_{\varepsilon } & \dots & 
\lTo_{\varepsilon } & Y & \lTo_{\varepsilon } & \dots \\
\end{diagram}%
\right)
\end{equation*}%
Strangely enough, $\xi \circ \varphi =\mathbf{1}_{\mathbf{X}}$ in $%
\Hom%
_{\mathbf{Pro}\left( \mathbf{K}\right) }\left( \mathbf{X},\mathbf{X}\right) $%
! Indeed,%
\begin{equation*}
\xi \circ \varphi =\left( f_{j}\in \underrightarrow{\lim }_{i}~%
\Hom%
_{\mathbf{K}}\left( X_{i}=Y,X_{j}=Y\right) \right) _{j}
\end{equation*}%
where each $f_{j}$ is represented by $\varepsilon :X_{0}=Y\rightarrow
X_{j}=Y $. Since%
\begin{eqnarray*}
\left( X_{j+1}\overset{\varepsilon }{\longrightarrow }X_{j}\overset{%
\varepsilon }{\longrightarrow }X_{0}\overset{\varepsilon }{\longrightarrow }%
X_{j}\right) &=&\left( X_{j+1}\overset{\varepsilon }{\longrightarrow }X_{j}%
\overset{\mathbf{1}_{Y}}{\longrightarrow }X_{j}\right) , \\
\left( f_{j}\right) _{j} &=&\left( \mathbf{1}_{X_{j}}\right) _{j}=\mathbf{1}%
_{\mathbf{X}}.
\end{eqnarray*}

\item Assume $\mathbf{X}%
\simeq%
Z\in \mathbf{K}$. Since $\mathbf{K}$ is a full subcategory of $\mathbf{Pro}%
\left( \mathbf{K}\right) $, one has a splitting%
\begin{equation*}
\xi \circ \varphi =\mathbf{1}_{Z},\varphi \circ \xi =\varepsilon
\end{equation*}%
of the idempotent $\varepsilon $ in the category $\mathbf{K}$. Contradiction.
\end{enumerate}
\end{proof}

\begin{example}
\label{Ex-Rudimentary-tower}Let now $\varepsilon :Z\rightarrow Z$ be any
splitting idempotent in $\mathbf{K}$, i.e., there are%
\begin{eqnarray*}
\xi &:&Y\longrightarrow Z, \\
\varphi &:&Z\longrightarrow Y, \\
\xi \circ \varphi &=&\varepsilon ,\varphi \circ \xi =\mathbf{1}_{Y}.
\end{eqnarray*}%
Then%
\begin{equation*}
\mathbf{X}=\left( Z\overset{\varepsilon }{\longleftarrow }Z\overset{%
\varepsilon }{\longleftarrow }Z\overset{\varepsilon }{\longleftarrow }\dots 
\overset{\varepsilon }{\longleftarrow }Z\overset{\varepsilon }{%
\longleftarrow }\dots \right) \in \mathbf{Pro}\left( \mathbf{K}\right)
\end{equation*}%
is rudimentary. Indeed, $\varepsilon \circ \xi =\xi \circ \varphi \circ \xi
=\xi $, therefore%
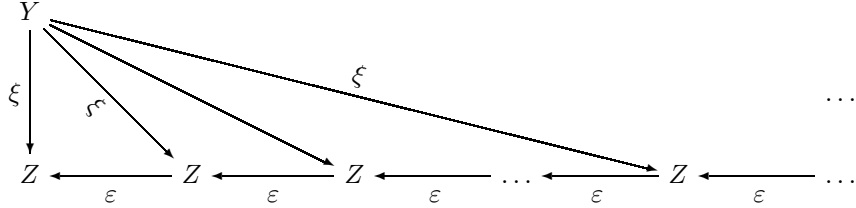
\begin{equation*}
\begin{diagram}[size=3.0em,textflow]
Y &  &  &  &  &  &  & &&&&& \\
\dTo<{\xi} & \rdTo<{\xi} \rdTo(4,2) \rdTo(8,2)>{\xi} &  &  &  &  &  &  &  &  & \dots \\ 
Z & \lTo_{\varepsilon } & Z & \lTo_{\varepsilon } & Z & \lTo_{\varepsilon } & \dots & 
\lTo_{\varepsilon } & Z & \lTo_{\varepsilon } & \dots \\
\end{diagram}
%
\end{equation*}%
defines a well-defined morphism $\Xi :Y\rightarrow \mathbf{X}$. Moreover,
there is a morphism $\Phi :\mathbf{X}\rightarrow Y$ given by%
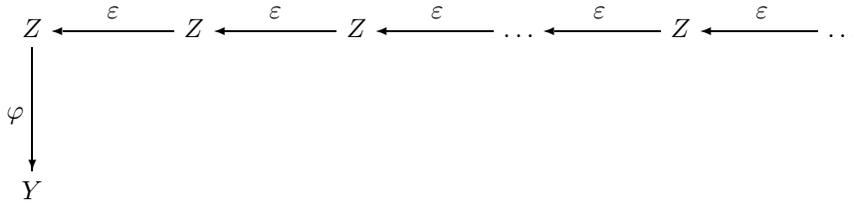
\begin{equation*}
\begin{diagram}[size=3.0em,textflow]
Z & \lTo^{\varepsilon } & Z & \lTo^{\varepsilon } & Z & \lTo^{\varepsilon } & \dots & 
\lTo^{\varepsilon } & Z & \lTo^{\varepsilon } & \dots \\ 
\dTo<{\varphi} &  &  &  &  &  &  &  &  &  &  \\ 
Y &  &  &  &  &  &  &  &  &  & 
\end{diagram}
%
\end{equation*}%
Clearly, $\Phi \circ \Xi =\mathbf{1}_{Y}$. Compute $\Xi \circ \Phi $:%
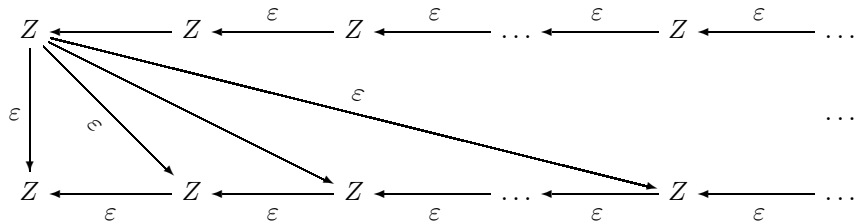
\begin{equation*}
\begin{diagram}[size=3.0em,textflow]
Z & \lTo & Z & \lTo^{\varepsilon } & Z & \lTo^{\varepsilon } & \dots & 
\lTo^{\varepsilon } & Z & \lTo^{\varepsilon } & \dots \\ 
\dTo<{\varepsilon} & \rdTo<{\varepsilon} \rdTo(4,2) \rdTo(8,2)>{\varepsilon } &  &  &  &  &  &  &  &  & \dots \\ 
Z & \lTo_{\varepsilon } & Z & \lTo_{\varepsilon } & Z & \lTo_{\varepsilon } & \dots & 
\lTo_{\varepsilon } & Z & \lTo_{\varepsilon } & \dots \\
\end{diagram}
%
\end{equation*}%
Like in Example \ref{Ex-Retract-non-rudimentary}, $\Xi \circ \Phi =\mathbf{1}%
_{\mathbf{X}}$. Finally, $\mathbf{X}%
\simeq%
Y$.
\end{example}

\subsection{(Pre)cosheaves on Hausdorff paracompact spaces}

\begin{example}
\label{Ex-Singular-homology-cosheaf}Let $G\in \mathbf{Ab}$. The precosheaf%
\begin{equation*}
H_{0}^{sing}\left( \bullet ,G\right) 
{:=}%
U\longmapsto H_{0}^{sing}\left( U,G\right)
\end{equation*}%
where $H_{n}^{sing}$ is the $n$-th singular homology, is in fact a cosheaf 
\cite[Example 3.5]{Prasolov-smooth-cosheaves-2012-MR2879363}.
\end{example}

\begin{example}
\label{Ex-Constant-cosheaf}Let $G\in \mathbf{Ab}$. It follows from \cite[%
Theorem 1.7(3,4)]{Prasolov-smooth-cosheaves-2012-MR2879363} that the
constant cosheaf $G_{\#}$ is isomorphic to the cosheaf%
\begin{equation*}
pro\text{-}H_{0}\left( \bullet ,G\right) 
{:=}%
\left( U\mapsto pro\text{-}H_{0}\left( U,G\right) \right)
\end{equation*}%
where $pro$-$H_{0}$ is the pro-homology groups from Definition \ref%
{Def-Pro-homology-groups}.
\end{example}

\begin{example}
\label{Ex-Converging-sequence}Let $X$ be the following sequence converging
to zero (together with the limit):%
\begin{equation*}
X=\left\{ 0\right\} \cup \left\{ 1,\frac{1}{2},\frac{1}{3},\frac{1}{4}%
,...\right\} \subseteq \mathbb{R}.
\end{equation*}%
Let also $G\in \mathbf{Ab}$.

\begin{enumerate}
\item Let%
\begin{equation*}
\mathcal{A}=H_{0}^{sing}\left( \bullet ,G\right) \in \mathbf{CS}\left( X,%
\mathbf{Ab}\right) \subseteq \mathbf{CS}\left( X,\mathbf{Pro}\left( \mathbf{%
Ab}\right) \right)
\end{equation*}%
from Example \ref{Ex-Singular-homology-cosheaf} and%
\begin{equation*}
\mathcal{B}=G_{\#}=pro\text{-}H_{0}\left( \bullet ,G\right) \in \mathbf{CS}%
\left( X,\mathbf{Pro}\left( \mathbf{Ab}\right) \right)
\end{equation*}%
from Example \ref{Ex-Constant-cosheaf}.

\item Consider costalks.

\begin{enumerate}
\item Clearly, $\mathcal{A}^{x}%
\simeq%
G\in \mathbf{Ab}\subseteq \mathbf{Pro}\left( \mathbf{Ab}\right) $ if $x\in
\left\{ 1,\frac{1}{2},\frac{1}{3},\frac{1}{4},...\right\} $, i.e., $\mathcal{%
A}^{x}$ is a rudimentary (Remark \ref{Rem-Rudimentary}) object of $\mathbf{%
Pro}\left( \mathbf{Ab}\right) $.

\item If $x=0$, $\mathcal{A}^{x}$ depends on whether we consider $\mathcal{A}%
\subseteq \mathbf{CS}\left( X,\mathbf{Ab}\right) $ or $\mathcal{A}\subseteq 
\mathbf{CS}\left( X,\mathbf{Pro}\left( \mathbf{Ab}\right) \right) $. Choose
different notations \cite[Remark 2.18]%
{Prasolov-Cosheafification-2016-MR3660525}:

\begin{enumerate}
\item Let%
\begin{equation*}
U_{n}=\left\{ 0\right\} \cup \left\{ \frac{1}{n},\frac{1}{n+1},\frac{1}{n+2}%
,...\right\} \subseteq X.
\end{equation*}%
Then%
\begin{equation*}
\mathcal{A}_{\mathbf{Ab}}^{x}=\underset{n}{\underleftarrow{\lim }}~\left(
H_{0}^{sing}\left( U_{n},G\right) \right) =\underset{n}{\underleftarrow{\lim 
}}~\left( G\oplus \left( \dbigoplus\limits_{k=n}^{\infty }G\right) \right)
=\dbigcap\limits_{k=n}^{\infty }X_{n}=G
\end{equation*}%
where%
\begin{equation*}
X_{n}=G\oplus \left( \dbigoplus\limits_{k=n}^{\infty }G\right) \subseteq
G\oplus \left( \dbigoplus\limits_{k=0}^{\infty }G\right) .
\end{equation*}

\item However, $\mathcal{A}_{\mathbf{Pro}\left( \mathbf{Ab}\right) }^{x}$ is
an object of $\mathbf{X\in Pro}\left( \mathbf{Ab}\right) $ given by the
diagram%
\begin{equation*}
X_{0}\leftarrowtail X_{1}\leftarrowtail X_{2}\leftarrowtail \dots
\leftarrowtail X_{n}\leftarrowtail \dots
\end{equation*}%
This object is \textbf{not} rudimentary, see Example \ref%
{Ex-Non-rudimentary-tower}(\ref{Ex-Non-rudimentary-tower-mono}). Therefore, $%
\mathcal{A}$ is \textbf{not} a constant cosheaf (with values in $\mathbf{Pro}%
\left( \mathbf{Ab}\right) $).
\end{enumerate}
\end{enumerate}

\item On the contrary, $\mathcal{B}=G_{\#}$ is indeed a constant cosheaf.
Consider costalks. Clearly, $\mathcal{B}^{x}=G$ if $x\in \left\{ 1,\frac{1}{2%
},\frac{1}{3},\frac{1}{4},...\right\} $. Let $x=0$. It follows from
Proposition \ref{Prop-(Co)stalks} that%
\begin{equation*}
\mathcal{B}^{x}=\left( G_{\#}\right) ^{x}%
\simeq%
G^{x}=G.
\end{equation*}%
Let us however consider this example in detail. For each%
\begin{equation*}
U_{n}=\left\{ 0\right\} \cup \left\{ \frac{1}{n},\frac{1}{n+1},\frac{1}{n+2}%
,...\right\} \subseteq X
\end{equation*}%
there is an $H\left( \mathbf{Pol}\right) $-expansion (a tower)%
\begin{equation*}
P_{n}=\left( V_{n,0}\longleftarrow V_{n,1}\longleftarrow
V_{n,2}\longleftarrow \dots \longleftarrow V_{n,k}\longleftarrow \dots
\right)
\end{equation*}%
where%
\begin{equation*}
V_{n,k}=\left\{ 0\right\} \cup \left\{ \frac{1}{n},\frac{1}{n+1},\frac{1}{n+2%
},...,\frac{1}{k}\right\}
\end{equation*}%
and $V_{n,k}\longleftarrow V_{n,k+1}$ sends $\frac{1}{k+1}$ to $\frac{1}{k}$
(remaining the rest unchanged). Therefore, $\mathcal{B}\left( U_{n}\right)
=pro$-$H_{0}\left( U_{n},G\right) $%
\begin{equation*}
=\left( 0\twoheadleftarrow 0\twoheadleftarrow \dots \twoheadleftarrow
0\twoheadleftarrow G\twoheadleftarrow G^{2}\twoheadleftarrow
G^{3}\twoheadleftarrow \dots \twoheadleftarrow G^{\max \left( 1,k-n+1\right)
}\twoheadleftarrow \dots \right)
\end{equation*}%
where the first $n-1$ items in the tower are zeroes. Due to Example \ref%
{Ex-Non-rudimentary-tower}(\ref{Ex-Non-rudimentary-tower-epi}), $\mathcal{B}%
\left( U_{n}\right) $ is \textbf{not} rudimentary. However,%
\begin{equation*}
\mathcal{B}^{x}=\underset{n}{\underleftarrow{\lim }}~\mathcal{B}\left(
U_{n}\right) 
\simeq%
G
\end{equation*}%
\textbf{is} rudimentary. For any $T\in \mathbf{Ab}$,%
\begin{eqnarray*}
&&%
\Hom%
_{\mathbf{Pro}\left( \mathbf{Ab}\right) }\left( \mathcal{B}^{x},T\right) 
\simeq%
\underset{n}{\underrightarrow{\lim }}~%
\Hom%
_{\mathbf{Pro}\left( \mathbf{Ab}\right) }\left( \mathcal{B}\left(
U_{n}\right) ,T\right) 
\simeq
\\
&&%
\simeq%
\underset{n}{\underrightarrow{\lim }}~\underset{k}{\underrightarrow{\lim }}~%
\Hom%
_{\mathbf{Ab}}\left( G^{\max \left( 1,k-n+1\right) },T\right) 
\simeq%
\underset{n}{\underrightarrow{\lim }}~\underset{k}{\underrightarrow{\lim }}%
~\left( 
\Hom%
_{\mathbf{Ab}}\left( G,T\right) \right) ^{\max \left( 1,k-n+1\right) }%
\simeq
\\
&&%
\simeq%
\underset{k}{\underrightarrow{\lim }}~\underset{n}{\underrightarrow{\lim }}%
~\left( 
\Hom%
_{\mathbf{Ab}}\left( G,T\right) \right) ^{\max \left( 1,k-n+1\right) }%
\simeq%
\underset{k}{\underrightarrow{\lim }}~\left( 
\Hom%
_{\mathbf{Ab}}\left( G,T\right) \right) 
\simeq%
\left( 
\Hom%
_{\mathbf{Ab}}\left( G,T\right) \right) .
\end{eqnarray*}%
Therefore $\mathcal{B}^{x}%
\simeq%
G$ indeed.

\item Let%
\begin{equation*}
\mathbf{H}=\left( G\twoheadleftarrow G^{2}\twoheadleftarrow
G^{3}\twoheadleftarrow \dots \twoheadleftarrow G^{n}\twoheadleftarrow \dots
\right)
\end{equation*}%
(a non-rudimentary pro-object). Then%
\begin{equation*}
H_{n}\left( X,\mathcal{B}\right) 
\simeq%
\left\{ 
\begin{array}{ccc}
\mathbf{H} & \text{if} & n=0 \\ 
0 & \text{if} & n>0%
\end{array}%
\right.
\end{equation*}%
See Example \ref{Ex-Hawaiian-ear-ring}.
\end{enumerate}
\end{example}

\begin{example}
\label{Ex-Hawaiian-ear-ring}Let $X_{k}$ be the $k$-dimensional Hawaiian
ear-ring, i.e.%
\begin{equation*}
X_{k}=\dbigcup\limits_{n=1}^{\infty }Y_{n,k}\subseteq \mathbb{R}^{k+1}
\end{equation*}%
where each $Y_{n,k}$ is the $k$-dimensional sphere with the radius $\frac{1}{%
2n}$ and the center in $\left( \frac{1}{2n},0,0,\dots ,0\right) $:%
\begin{equation*}
Y_{n,k}=\left\{ \left( x_{0},x_{1},\dots ,x_{k}\right) \in \mathbb{R}%
^{k+1}~|~\left( x_{0}-\frac{1}{2n}\right) ^{2}+\left( x_{1}\right)
^{2}+\dots +\left( x_{k}\right) ^{2}=\left( \frac{1}{2n}\right) ^{2}\right\}
.
\end{equation*}%
If $k=1$, the space resembles indeed the Hawaiian ear-ring. Notice that $%
X_{0}$ is the converging sequence from Example \ref{Ex-Converging-sequence}.
Let%
\begin{equation*}
\mathbf{H}=\left( G\twoheadleftarrow G^{2}\twoheadleftarrow
G^{3}\twoheadleftarrow \dots \twoheadleftarrow G^{n}\twoheadleftarrow \dots
\right) \in \mathbf{Pro}\left( \mathbf{Ab}\right) .
\end{equation*}%
We have%
\begin{equation*}
H_{n}\left( X_{k},G_{\#}\right) =\left\{ 
\begin{array}{ccccc}
0 & \text{if} & n\neq k & \text{and} & n\neq 0 \\ 
G & \text{if} & k\neq 0 & \text{and} & n=0 \\ 
\mathbf{H} & \text{if} & n=k=0 &  &  \\ 
\mathbf{H} & \text{if} & n=k & \text{and} & k\neq 0%
\end{array}%
\right.
\end{equation*}%
Indeed, consider the following object $\mathbf{Y}^{\left( k\right) }$ of $%
\mathbf{Pro}\left( H\left( \mathbf{Pol}\right) \right) $:%
\begin{equation*}
\mathbf{Y}^{\left( k\right) }=\left( Y_{0}^{\left( k\right) }\longleftarrow
Y_{1}^{\left( k\right) }\longleftarrow Y_{2}^{\left( k\right)
}\longleftarrow \dots \longleftarrow Y_{n}^{\left( k\right) }\longleftarrow
\dots \right)
\end{equation*}%
where%
\begin{equation*}
Y_{n}^{\left( k\right) }=\dbigcup\limits_{s=1}^{n}Y_{s,k}.
\end{equation*}%
The mapping $Y_{n}^{\left( k\right) }\rightarrow Y_{n-1}^{\left( k\right) }$
sends the $n$-th copy of the $k$-sphere into the point $\left( 0,0,\dots
,0\right) $. Let 
\begin{equation*}
\left( f_{n}:X_{k}\rightarrow Y_{n}^{\left( k\right) }\right) _{n\in \mathbb{%
N}}:X_{k}\longrightarrow \mathbf{Y}^{\left( k\right) }
\end{equation*}%
be the morphism in $\mathbf{Pro}\left( H\left( \mathbf{Top}\right) \right) $
such that $f_{n}$ sends all components $Y_{s,k}$, $s>n$, into the point $%
\left( 0,0,\dots ,0\right) $. It can be checked that $X\rightarrow \mathbf{Y}%
^{\left( k\right) }$ is an $H\left( \mathbf{Top}\right) $-expansion. Since $%
X_{k}$ is a metric space, therefore Hausdorff paracompact, Theorem \ref%
{Th-Paracompact-vs-Shape} reads that%
\begin{equation*}
H_{\bullet }\left( X_{k},G_{\#}\right) 
\simeq%
pro\text{-}H_{\bullet }\left( X_{k},G\right) 
\simeq%
H_{\bullet }\left( \mathbf{Y}^{\left( k\right) },G\right) .
\end{equation*}%
Consider now all the cases. Notice that each $Y_{n}^{\left( k\right) }$ is a
polyheder homeomorphic to a bouquet of $n$ copies of the $k$-sphere.

\begin{enumerate}
\item $n\neq k$ and $n\neq 0$. Since $H_{n}\left( Y_{s}^{\left( k\right)
},G\right) =0$ for all $s$, $H_{n}\left( \mathbf{Y}^{\left( k\right)
},G\right) $ is a rudimentary pro-object isomorphic to $0$.

\item $k\neq 0$ and $n=0$. Since $H_{0}\left( Y_{s}^{\left( k\right)
},G\right) =G$ for all $s$, $H_{n}\left( \mathbf{Y}^{\left( k\right)
},G\right) $ is a rudimentary pro-object isomorphic to $G$.

\item $n=k=0$ or ($n=k$ and $k\neq 0$). Since $H_{k}\left( Y_{s}^{\left(
k\right) },G\right) =G^{s}$, $H_{k}\left( \mathbf{Y},G\right) $ is a \textbf{%
non}-rudimentary pro-group%
\begin{equation*}
H_{k}\left( \mathbf{Y},G\right) =\mathbf{H}=\left( G\twoheadleftarrow
G^{2}\twoheadleftarrow G^{3}\twoheadleftarrow \dots \twoheadleftarrow
G^{n}\twoheadleftarrow \dots \right) \in \mathbf{Pro}\left( \mathbf{Ab}%
\right) .
\end{equation*}
\end{enumerate}
\end{example}

\subsection{(Pre)cosheaves on Alexandroff spaces}

\begin{example}
\label{Ex-4-point-circle}Let $\left( X,\tau \right) $ be the following
topological space (the $4$-point circle):%
\begin{eqnarray*}
X &=&\left\{ 0,1,2,3\right\} , \\
\tau &=&\left\{ \varnothing ,X,\left\{ 0\right\} ,\left\{ 2\right\} ,\left\{
0,1,2\right\} ,\left\{ 0,3,2\right\} \right\} .
\end{eqnarray*}%
The corresponding poset is this:%
\begin{equation*}
\begin{diagram}[size=3.0em,textflow]
1 &  &  3 \\ 
\dLine & \rdLine \ldLine & \dLine &  &  \\ 
0 &  & 2 &  & \\
\end{diagram}
%
\end{equation*}%
Apply Theorem \ref{Th-Cosheaf-homology-for-A-spaces} and \ref%
{Th-Poset-vs-cosheaves}(\ref{Th-Poset-vs-cosheaves-Bar-construction}):%
\begin{equation*}
H_{n}\left( X,G_{\#}\right) 
\simeq%
H_{n}^{sing}\left( \left\vert \mathcal{K}\left( X\right) \right\vert
,G\right) 
\simeq%
H_{n}^{sing}\left( X,G\right) 
\simeq%
\underrightarrow{\lim }_{n}~\left( X,G\right) .
\end{equation*}%
One obtains the following description of $H_{\bullet }\left( X,G_{\#}\right) 
$ (identical to that of the ordinary circle):%
\begin{equation*}
H_{n}\left( X,G_{\#}\right) =\left\{ 
\begin{array}{ccc}
G & \text{if} & n=0,1 \\ 
0 & \text{otherwise} & 
\end{array}%
\right.
\end{equation*}%
Notice that all $H_{n}\left( X,G_{\#}\right) $ are \textbf{rudimentary}
pro-objects. Indeed, calculate $\underrightarrow{\lim }_{\bullet }~\left(
X,G\right) $ using the bar-construction. Standard machinery allows us to use
only \textbf{non-degenerate} simplices:%
\begin{equation*}
H_{n}\left( X,G_{\#}\right) 
\simeq%
\underrightarrow{\lim }_{n}~\left( X,G\right) 
\simeq%
H_{n}\left( C_{\bullet }\left( X,G\right) \right)
\end{equation*}%
where $C_{\bullet }\left( X,G\right) =$ 
\begin{equation*}
G\left[ 0\right] \oplus G\left[ 1\right] \oplus G\left[ 2\right] \oplus G%
\left[ 3\right] \longleftarrow G\left[ 0<1\right] \oplus G\left[ 0<3\right]
\oplus G\left[ 2<1\right] \oplus G\left[ 2<3\right] \longleftarrow
0\longleftarrow \dots
\end{equation*}%
The differential $d_{0}:C_{1}\rightarrow C_{0}$ is given by%
\begin{eqnarray*}
d &:&\left[ 0<1\right] \longmapsto \left[ 1\right] -\left[ 0\right] , \\
d &:&\left[ 0<3\right] \longmapsto \left[ 3\right] -\left[ 0\right] , \\
d &:&\left[ 2<1\right] \longmapsto \left[ 1\right] -\left[ 2\right] , \\
d &:&\left[ 2<3\right] \longmapsto \left[ 3\right] -\left[ 2\right] .
\end{eqnarray*}%
The boundary group $B_{0}=d\left( C_{1}\right) $ will be:%
\begin{equation*}
B_{0}=\left\{ \left( g_{i}\right) \in
G^{4}~|~g_{0}+g_{1}+g_{2}+g_{3}=0\right\}
\end{equation*}%
and%
\begin{equation*}
H_{0}\left( X,G_{\#}\right) 
\simeq%
\frac{G^{4}}{B_{0}}%
\simeq%
G.
\end{equation*}%
The cycle group $Z_{1}=\ker d$ will be:%
\begin{equation*}
Z_{1}=\left\{ \left( g_{i}\right) \in G^{4}~|~g_{0}=g_{1}=g_{2}=g_{3}\right\}
\end{equation*}%
and%
\begin{equation*}
H_{1}\left( X,G_{\#}\right) 
\simeq%
Z_{1}%
\simeq%
G.
\end{equation*}
\end{example}

\begin{example}
\label{Ex-6-point-2-sphere}Let now $X$ be the the \textbf{non-Hausdorff
suspension} \cite[before Lemma 10]{McCord-MR0196744} of the $4$-point
circle. The corresponding poset is shown below:%
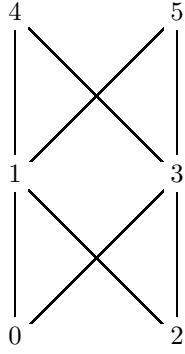
\begin{equation*}
\begin{diagram}[size=3.0em,textflow]
4 & & 5 \\
\dLine & \rdLine \ldLine & \dLine &  &  \\
1 &  &  3 \\ 
\dLine & \rdLine \ldLine & \dLine &  &  \\ 
0 &  & 2 &  & \\
\end{diagram}%
\end{equation*}%
It is a union of two contractible subsets (the \textbf{non-Hausdorff cones})%
\begin{eqnarray*}
U_{0} &=&\left\{ 0,1,2,3,4\right\} , \\
U_{1} &=&\left\{ 0,1,2,3,5\right\} , \\
U_{0}\cap U_{1} &=&\left\{ 0,1,2,3\right\} .
\end{eqnarray*}%
One may use the Mayer-Vietoris exact sequence:%
\begin{eqnarray*}
0 &\leftarrow &H_{0}\left( X,G\right) \leftarrow H_{0}\left( U_{0},G\right)
\oplus H_{0}\left( U_{1},G\right) \leftarrow H_{0}\left( U_{0}\cap
U_{1},G\right) \leftarrow \\
&\leftarrow &H_{1}\left( X,G\right) \leftarrow H_{1}\left( U_{0},G\right)
\oplus H_{1}\left( U_{1},G\right) \leftarrow H_{1}\left( U_{0}\cap
U_{1},G\right) \leftarrow \\
&\leftarrow &H_{2}\left( X,G\right) \leftarrow H_{2}\left( U_{0},G\right)
\oplus H_{2}\left( U_{1},G\right) \leftarrow H_{2}\left( U_{0}\cap
U_{1},G\right)
\end{eqnarray*}%
Use the contractibility of $U_{i}$ and Example \ref{Ex-4-point-circle}:%
\begin{equation*}
0\leftarrow H_{0}\left( X,G\right) \leftarrow G\oplus G\leftarrow
G\leftarrow H_{1}\left( X,G\right) \leftarrow 0\leftarrow G\leftarrow
H_{2}\left( X,G\right) \leftarrow 0,
\end{equation*}%
and obtain the result (identical to that of the ordinary $2$-sphere):%
\begin{equation*}
H_{n}\left( X,G_{\#}\right) =\left\{ 
\begin{array}{ccc}
G & \text{if} & n=0,2 \\ 
0 & \text{otherwise} & 
\end{array}%
\right.
\end{equation*}
\end{example}

\appendix

\section{\label{Sec-quasi-noetherian}Quasi-noetherian rings}

\begin{definition}
\label{Def-quasi-noetherian}$k$ is called \textbf{quasi-noetherian} iff for
any quasi-projective pro-module $\mathbf{P}\in \mathbf{Pro}\left( k\right) $
and any injective module $T\in \mathbf{Mod}\left( k\right) $, the module $%
\Hom%
_{\mathbf{Pro}\left( k\right) }\left( \mathbf{P},T\right) \in \mathbf{Mod}%
\left( k\right) $ is injective.
\end{definition}

\begin{theorem}
\label{Th-(Quasi-)noetherian}The following statements are equivalent:

\begin{enumerate}
\item $k$ is quasi-noetherian.

\item $k$ is noetherian.

\item Any filtered colimit of injective $k$-modules is injective.

\item The direct sum of any family of injective $k$-modules is injective.
\end{enumerate}
\end{theorem}

\begin{proof}
~

\begin{enumerate}
\item $\left( 2\implies 3\implies 4\implies 2\right) $. See \cite[%
Proposition 4.1.]{Chase-1960-Direct-products-modules-MR0120260}.

\item $\left( 3\implies 1\right) $. Let $\mathbf{P=}\left( P_{i}\right)
_{i\in I}\in \mathbf{Pro}\left( k\right) $ be a quasi-projective pro-module.
One can assume (Proposition \ref{Prop-quasi-projective}) that all $P_{i}$
are projective. Let $T\in \mathbf{Mod}\left( k\right) $ be injective. Then
all modules $%
\Hom%
\left( P_{i},T\right) $ are injective, and%
\begin{equation*}
\Hom%
_{\mathbf{Pro}\left( k\right) }\left( \mathbf{P},T\right) =\underset{i\in I}{%
\underrightarrow{\lim }}~%
\Hom%
\left( P_{i},T\right)
\end{equation*}%
is injective as a filtered colimit of injective modules.

\item $\left( 1\implies 4\right) $. Given a family $\left( T_{s}\right)
_{s\in S}$ of injective $k$-modules.

\begin{enumerate}
\item Choose an injective cogenerator $J$ for the category $\mathbf{Mod}%
\left( k\right) $, say,%
\begin{equation*}
J=%
\Hom%
_{\mathbb{Z}}\left( k,\mathbb{Q}/\mathbb{Z}\right) .
\end{equation*}%
For each $M\in \mathbf{Mod}\left( k\right) $, denote%
\begin{equation*}
M^{\ast }%
{:=}%
\Hom%
_{k}\left( M,J\right) .
\end{equation*}

\begin{enumerate}
\item There is a functor morphism%
\begin{equation*}
\iota :1_{\mathbf{Mod}\left( k\right) }\longrightarrow \left( \bullet
\right) ^{\ast \ast }
\end{equation*}%
given by the family%
\begin{eqnarray*}
&&\left( \iota _{M}:M\longrightarrow M^{\ast \ast }\right) _{M\in \mathbf{Mod%
}\left( k\right) }, \\
&&\iota _{M}\left( x\right) \left( f:M\rightarrow J\right) 
{:=}%
f\left( x\right) .
\end{eqnarray*}

\item $\iota _{M}$ is a monomorphism. Indeed, let $x\in \ker \iota _{M}$. It
follows that $f\left( x\right) =0$ for all%
\begin{equation*}
f\in M^{\ast }=%
\Hom%
_{k}\left( M,J\right) 
\simeq%
\Hom%
_{\mathbb{Z}}\left( M,\mathbb{Q}/\mathbb{Z}\right) .
\end{equation*}%
Since $\mathbb{Q}/\mathbb{Z}$ is a cogenerator for $\mathbf{Mod}\left( 
\mathbb{Z}\right) =\mathbf{Ab}$, it follows that $x=0$.

\item If $f:M\rightarrow N$ is an epimorphism, then%
\begin{equation*}
f^{\ast }:N^{\ast }\longrightarrow M^{\ast }
\end{equation*}%
is clearly a monomorphism.
\end{enumerate}

\item For each $s\in S$, let%
\begin{equation*}
\pi _{s}:P_{s}\twoheadrightarrow \left( T_{s}\right) ^{\ast }
\end{equation*}%
be an epimorphism, where $P_{s}$ is projective. The compositions%
\begin{equation*}
\varphi _{s}:T_{s}\longrightarrow \left( T_{s}\right) ^{\ast \ast
}\longrightarrow \left( P_{s}\right) ^{\ast }.
\end{equation*}%
are monomorphisms (see above). Since $T_{s}$ are injective, there exist
splittings%
\begin{eqnarray*}
\psi _{s} &:&\left( P_{s}\right) ^{\ast }\longrightarrow T_{s}, \\
\psi _{s}\circ \varphi _{s} &=&1_{T_{s}}.
\end{eqnarray*}

\item Let%
\begin{eqnarray*}
&&P%
{:=}%
\dbigoplus\limits_{s\in S}\left( P_{s}\right) ^{\ast } \\
&&T%
{:=}%
\dbigoplus\limits_{s\in S}T_{s}, \\
\varphi 
{:=}%
\dbigoplus\limits_{s\in S}\varphi _{s} &:&T\longrightarrow P \\
\psi 
{:=}%
\dbigoplus\limits_{s\in S}\psi _{s} &:&P\longrightarrow T,
\end{eqnarray*}%
Clearly $\psi \circ \varphi =1_{T}$, and $T$ is a direct summand of $P$.

\item It is enough to prove that $P$ is injective. Consider the following 
\textbf{co}filtered system:

\begin{enumerate}
\item $I$ is the poset of \textbf{finite} subsets of $A$ ordered by the 
\textbf{inverse} inclusion.

\item The \textbf{co}filtered system $\left( Q_{i}\right) _{i\in I}$ of
projective modules is given by%
\begin{equation*}
Q_{i}=\dbigoplus\limits_{s\in i}P_{s}.
\end{equation*}%
The resulting pro-module $\mathbf{Q}$ is clearly quasi-projective, and
satisfies%
\begin{equation*}
\Hom%
_{\mathbf{Pro}\left( k\right) }\left( \mathbf{Q},J\right) =\underset{i\in I}{%
\underrightarrow{\lim }}~\left( Q_{i}\right) ^{\ast
}=\dbigoplus\limits_{s\in S}\left( P_{s}\right) ^{\ast }=P.
\end{equation*}%
Since $k$ is quasi-noetherian, $P$ is injective, therefore $T$ is injective
as well.
\end{enumerate}
\end{enumerate}
\end{enumerate}
\end{proof}

\section{\label{Sec-Derived-categories}Derived categories}

\subsection{Left satellites}

We use here the \textquotedblleft classical\textquotedblright\ definition of
an $F$-projective category. The subcategories, which are called
\textquotedblleft $F$-projective\textquotedblright\ in \cite[Definition
13.3.4]{Kashiwara-Categories-MR2182076}, are called \textbf{weak }$F$\textbf{%
-projective} in this paper.

\begin{definition}
\label{Def-F-projective}Let 
\begin{equation*}
F:\mathbf{C}\longrightarrow \mathbf{E}
\end{equation*}%
be an additive functor of abelian categories, and let $\mathbf{P}$ be a full
additive subcategory of $\mathbf{C}$.

\begin{enumerate}
\item \label{Def-F-projective-generating}We say that $\mathbf{P}$ is a 
\textbf{generating} subcategory iff for any $X\in \mathbf{C}$ there exists
an epimorphism $P\twoheadrightarrow X$ with $P\in \mathbf{P}$.

\item \label{Def-F-projective-closed-by-extensions}We say that $\mathbf{P}$
is a \textbf{quasi-projective} subcategory iff for any short exact sequence%
\begin{equation*}
0\longrightarrow X^{\prime }\longrightarrow X\longrightarrow X^{\prime
\prime }\longrightarrow 0
\end{equation*}%
$X,X^{\prime \prime }\in \mathbf{P}$ implies $X^{\prime }\in \mathbf{P}$.

\item \label{Def-F-projective-acyclic}We say that $\mathbf{P}$ is an $F$%
\textbf{-projective} subcategory iff:

\begin{enumerate}
\item $\mathbf{P}$ is generating and quasi-projective.

\item For any short exact sequence%
\begin{equation*}
0\longrightarrow X^{\prime }\longrightarrow X\longrightarrow X^{\prime
\prime }\longrightarrow 0
\end{equation*}%
in $\mathbf{C}$ with $X$, $X^{\prime \prime }\in \mathbf{P}$, the sequence%
\begin{equation*}
0\longrightarrow F\left( X^{\prime }\right) \longrightarrow F\left( X\right)
\longrightarrow F\left( X^{\prime \prime }\right) \longrightarrow 0
\end{equation*}%
is exact.
\end{enumerate}

\item \label{Def-Weak-F-projective}We say that $\mathbf{P}$ is \textbf{weak }%
$F$\textbf{-projective} if $\mathbf{P}$ satisfies the definition of an $F$%
-projective subcategory in \cite[Definition 13.3.4]%
{Kashiwara-Categories-MR2182076}.
\end{enumerate}
\end{definition}

\begin{notation}
\label{Not-Derived-category}\cite[Notation A.3.2]%
{Prasolov-Cosheaves-2021-MR4347662} For an abelian category $\mathbf{E}$,
let:

\begin{enumerate}
\item $~$

\begin{enumerate}
\item $Ch^{+}\left( \mathbf{E}\right) $ denote the category of \textbf{%
bounded below} chain complexes in $\mathbf{E}$:%
\begin{equation*}
X_{\bullet }\in Ch^{+}\left( \mathbf{E}\right) \iff X_{n}=0\text{ if }n\ll 0.
\end{equation*}

\item $Ch^{\geq 0}\left( \mathbf{E}\right) $ denote the category of \textbf{%
positive} chain complexes in $\mathbf{E}$:%
\begin{equation*}
X_{\bullet }\in Ch^{\geq 0}\left( \mathbf{E}\right) \iff X_{n}=0\text{ if }%
n<0.
\end{equation*}
\end{enumerate}

\item A \textbf{qis} denote a \textbf{quasi-isomorphism} in $Ch^{\ast
}\left( \mathbf{E}\right) $ ($\ast =\text{\textquotedblleft }+$%
\textquotedblright\ or \textquotedblleft $\geq 0$\textquotedblright ), i.e.,
a homomorphism%
\begin{equation*}
X_{\bullet }\longrightarrow Y_{\bullet }
\end{equation*}%
inducing an isomorphism of the homologies.

\item A complex $X_{\bullet }$ be \textbf{qis} to $Y_{\bullet }$ iff there
is a qis $X_{\bullet }\rightarrow Y_{\bullet }$.

\item $K^{\ast }\left( \mathbf{E}\right) $ denote the homotopy category of $%
Ch^{\ast }\left( \mathbf{E}\right) $, i.e., morphisms%
\begin{equation*}
X_{\bullet }\longrightarrow Y_{\bullet }
\end{equation*}%
in $K^{\ast }\left( \mathbf{E}\right) $ are \textbf{classes} of chain
homotopic maps $X_{\bullet }\rightarrow Y_{\bullet }$.

\item $D^{\ast }\left( \mathbf{E}\right) $ denote the corresponding derived
category of $K^{\ast }\left( \mathbf{E}\right) $, i.e.,%
\begin{equation*}
D^{\ast }\left( \mathbf{E}\right) =K^{\ast }\left( \mathbf{E}\right)
/N^{\ast }\left( \mathbf{E}\right)
\end{equation*}%
where $N^{\ast }\left( \mathbf{E}\right) $ is the full subcategory of $%
K^{\ast }\left( \mathbf{E}\right) $ consisting of complexes qis to $\mathbf{0%
}$.
\end{enumerate}
\end{notation}

\begin{proposition}
\label{Prop-Left-derived}Let $F:\mathbf{C}\rightarrow \mathbf{E}$ be an
additive functor of abelian categories, and let $\mathbf{P}$ be a full
additive subcategory of $\mathbf{C}$. Assume $\mathbf{P}$ is $F$-projective.
Then:

\begin{enumerate}
\item $\mathbf{P}$ is weak $F$-projective.

\item The left satellite%
\begin{equation*}
\mathbb{L}F:D^{\ast }(\mathbf{C})\longrightarrow D^{\ast }(\mathbf{E})
\end{equation*}%
exists, and%
\begin{equation*}
\mathbb{L}F(X_{\bullet })%
\simeq%
F(Y_{\bullet })
\end{equation*}%
for any qis%
\begin{equation*}
Y_{\bullet }\longrightarrow X_{\bullet }
\end{equation*}%
with $Y_{\bullet }\in K^{\ast }\left( \mathbf{P}\right) $.
\end{enumerate}
\end{proposition}

\begin{proof}
Follows from \cite[dual to Proposition 13.3.5 and Corollary 13.3.8]%
{Kashiwara-Categories-MR2182076}.
\end{proof}

Using $F$-projective subcategories, one can define left \textbf{satellites}
of the functor $F$.

\begin{definition}
\label{Def-Left-derived-functors}\label{Def-Left-satellites}In the
conditions of Proposition \ref{Prop-Left-derived} let $X\in \mathbf{C}$.
Considering $X$ as a complex concentrated in degree $0$, take a qis $%
P_{\bullet }\longrightarrow X$, i.e., a \textbf{resolution}%
\begin{equation*}
0\longleftarrow X\longleftarrow P_{0}\longleftarrow P_{1}\longleftarrow
P_{2}\longleftarrow ...\longleftarrow P_{n}\longleftarrow ...
\end{equation*}%
with $P_{\bullet }\in K^{\geq 0}\left( \mathbf{P}\right) $. Define%
\begin{equation*}
L_{n}F\left( X\right) 
{:=}%
H_{n}\left( P_{\bullet }\right) .
\end{equation*}%
It is easy to check that $L_{n}F$, $n\geq 0$, are additive functors%
\begin{equation*}
L_{n}F:\mathbf{C\longrightarrow E,}
\end{equation*}%
that $L_{n}F=\mathbf{0}$ if $n<0$, and that $L_{0}F%
\simeq%
F$ if $F$ is right exact.

The functors $L_{n}F$ are called the \textbf{left satellites} of $F$.
\end{definition}

\subsection{Cartan-Eilenberg resolutions}

We follow here the scheme presented in \cite[Chapter 5.8]%
{Weibel-1994-An-introduction-to-homological-algebra-MR1269324}. However, the
key ingredients there, the Horseshoe Lemma 2.2.8, and Lemma 5.7.2 are valid
only for categories with enough projective objects. Our situation is more
complicated: we do \textbf{not} have enough projectives, and we deal instead
with generating and quasi-projective subcategories. Therefore, the proofs in
this section are a bit trickier.

Let $\mathbf{P}$ be a full additive subcategory of an abelian category $%
\mathbf{C}$.

\begin{definition}
\label{Def-P-projective-complex}A complex $\left( P_{\bullet },d_{\bullet
}\right) \in Ch^{+}\left( \mathbf{C}\right) $ is called $\mathbf{P}$-\textbf{%
projective} iff:

\begin{enumerate}
\item $P_{s}\in \mathbf{P}$ for each $s\in \mathbb{Z}$.

\item For each $s$, $%
\coker%
\left( d_{s}:P_{s}\rightarrow P_{s-1}\right) \in \mathbf{P}$.
\end{enumerate}
\end{definition}

\begin{remark}
\label{Rem-Homology-quasi-projective}Assume $P_{\bullet }$ is $\mathbf{P}$%
-projective, and $\mathbf{P}$ is quasi-projective, see Definition \ref%
{Def-F-projective}(\ref{Def-F-projective-closed-by-extensions}). Then%
\begin{equation*}
B_{s}\left( P_{\bullet }\right) ,Z_{s}\left( P_{\bullet }\right)
,H_{s}\left( P_{\bullet }\right) \in \mathbf{P}
\end{equation*}%
for all $s$. Indeed, since $%
\coker%
d_{s}\in \mathbf{P}$, the short exact sequence%
\begin{equation*}
0\longrightarrow B_{s}\left( P_{\bullet }\right) \longrightarrow
P_{s}\longrightarrow \left( 
\coker%
d_{s}%
\simeq%
\frac{X_{s}}{B_{s}\left( P_{\bullet }\right) }\right) \longrightarrow 0
\end{equation*}%
implies $B_{s}\left( P_{\bullet }\right) \in \mathbf{P}$. The short exact
sequence%
\begin{equation*}
0\longrightarrow Z_{s}\left( P_{\bullet }\right) \longrightarrow
P_{s}\longrightarrow B_{s-1}\left( P_{\bullet }\right) \longrightarrow 0
\end{equation*}%
implies $Z_{s}\left( P_{\bullet }\right) \in \mathbf{P}$. Finally, the short
exact sequence%
\begin{equation*}
0\longrightarrow \left( H_{s}\left( P_{\bullet }\right) 
\simeq%
\frac{Z_{s}\left( P_{\bullet }\right) }{B_{s}\left( P_{\bullet }\right) }%
\right) \longrightarrow \frac{P_{s}}{B_{s}\left( P_{\bullet }\right) }%
\longrightarrow B_{s-1}\left( P_{\bullet }\right) \longrightarrow 0
\end{equation*}%
implies $H_{s}\left( P_{\bullet }\right) \in \mathbf{P}$.
\end{remark}

\begin{definition}
\label{Def-Cartan-Eilenberg}Given $\left( X_{\bullet },d_{\bullet }\right)
\in Ch^{\geq 0}\left( \mathbf{C}\right) $, its \textbf{Cartan-Eilenberg }$P$%
\textbf{-resolution} is a bicomplex $P_{\bullet \bullet }$ in $\mathbf{P}$
of the first quadrant, i.e., $\left( P_{\bullet \bullet },d_{\bullet
},\delta _{\bullet }\right) \in Ch^{\geq 0,\geq 0}\left( \mathbf{P}\right) $%
, together with a morphism $P_{\bullet \bullet }\rightarrow X_{\bullet }$
such that:

\begin{enumerate}
\item For each $t\geq 0$, $P_{\bullet t}$ is a $\mathbf{P}$-projective
complex.

\item For each $s\geq 0$, the sequence%
\begin{equation*}
0\longleftarrow X_{s}\overset{\delta _{s,0}}{\longleftarrow }P_{s,0}\overset{%
\delta _{s,1}}{\longleftarrow }P_{s,1}\overset{\delta _{s,2}}{\longleftarrow 
}P_{s,2}\overset{\delta _{s,3}}{\longleftarrow }\dots
\end{equation*}%
is exact.

\item For each $s\geq 0$, the sequence%
\begin{equation*}
0\longleftarrow H_{s}X_{\bullet }\overset{\left( \delta _{s,0}\right) _{\ast
}}{\longleftarrow }H_{s}P_{\bullet ,0}\overset{\left( \delta _{s,1}\right)
_{\ast }}{\longleftarrow }H_{s}P_{\bullet ,1}\overset{\left( \delta
_{s,2}\right) _{\ast }}{\longleftarrow }H_{s}P_{\bullet ,2}\overset{\left(
\delta _{s,3}\right) _{\ast }}{\longleftarrow }\dots
\end{equation*}%
is exact.
\end{enumerate}
\end{definition}

\begin{remark}
The sequences above look \textbf{horizontal}, but we will consider those
sequences as \textbf{columns}, i.e., as \textbf{vertical} lines in a
bicomplex, while the complex $X_{\bullet }$ will be considered as a \textbf{%
horizontal} line, see the diagram below:%
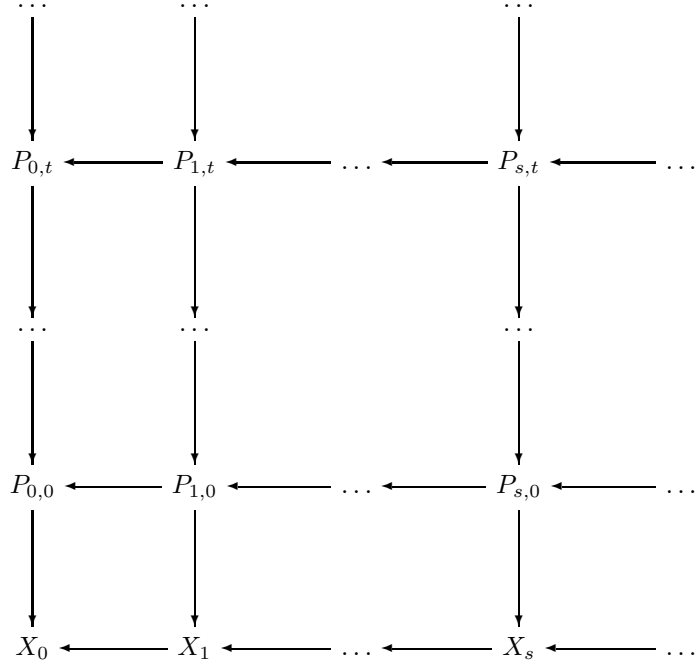
\begin{equation*}
\begin{diagram}[size=3.0em,textflow]
\dots &  & \dots &  &  &  & \dots &  \\ 
\dTo &  & \dTo &  &  &  & \dTo &  \\
P_{0,t} & \lTo & P_{1,t} & \lTo & \dots & \lTo
& P_{s,t} & \lTo &\dots \\ 
\dTo &  & \dTo &  &  &  & \dTo &  \\ 
\dots &  & \dots &  &  &  & \dots &  \\ 
\dTo &  & \dTo &  &  &  & \dTo &  \\ 
P_{0,0} & \lTo & P_{1,0} & \lTo & \dots & \lTo
& P_{s,0} & \lTo & \dots \\ 
\dTo &  & \dTo &  &  &  & \dTo &  \\ 
X_{0} & \lTo & X_{1} & \lTo & \dots & \lTo & 
X_{s} & \lTo & \dots
\end{diagram}%
\end{equation*}
\end{remark}

\begin{lemma}
\label{Lemma-Cartan-Eilenberg}Assume $\mathbf{P}$ is generating in $\mathbf{C%
}$. Then for any $\left( X_{\bullet },d_{\bullet }\right) \in Ch^{\geq
0}\left( \mathbf{C}\right) $ there exists an epimorphism $\varphi _{\bullet
}:\left( P_{\bullet },\delta _{\bullet }\right) \twoheadrightarrow
X_{\bullet }$ with:

\begin{enumerate}
\item $P_{\bullet }\in Ch^{\geq 0}\left( \mathbf{P}\right) $.

\item $P_{\bullet }$ is $\mathbf{P}$-projective.

\item For each $s\geq 0$, $H_{s}\left( P_{\bullet }\right) \longrightarrow
H_{s}\left( X_{\bullet }\right) $ is an epimorphism.
\end{enumerate}
\end{lemma}

\begin{proof}
Let $Z_{\bullet }$, $B_{\bullet }$ and $H_{\bullet }=Z_{\bullet }/B_{\bullet
}$ be the cycles, boundaries and homologies of $X_{\bullet }$, respectively.

\begin{enumerate}
\item For each $s$, choose $P_{s}^{B},P_{s}^{H}\in \mathbf{P}$, and
epimorphisms $\gamma _{s}:P_{s}^{B}\twoheadrightarrow X_{s+1}$, $\eta
_{s}:P_{s}^{H}\twoheadrightarrow Z_{s}$. It follows that each composition of
two epimorphisms%
\begin{equation*}
\beta _{s}=\left( P_{s}^{B}\twoheadrightarrow X_{s+1}\twoheadrightarrow
B_{s}\right)
\end{equation*}%
is an epimorphism as well. Consider the following diagram of complexes:%
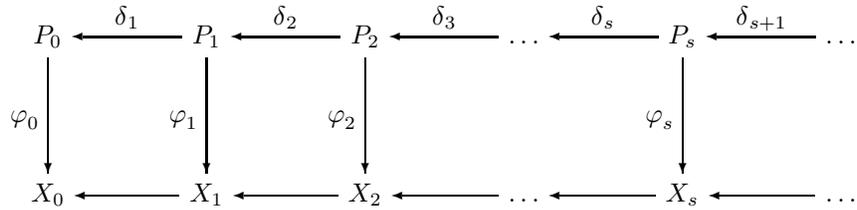
\begin{equation*}
\begin{diagram}[size=3.0em,textflow]
P_{0} & \lTo^{\delta _{1}} & P_{1} & \lTo^{\delta _{2}} & P_{2} & \lTo^{\delta _{3}} & \dots
& \lTo^{\delta _{s}} & P_{s} & \lTo^{\delta _{s+1}} & \dots \\ 
\dTo<{\varphi_{0}} &  & \dTo<{\varphi_{1}} &  & \dTo<{\varphi_{2}} &  &  &  & \dTo<{\varphi_{s}} &  \\ 
X_{0} & \lTo & X_{1} & \lTo & X_{2} & \lTo & 
\dots & \lTo & X_{s} & \lTo & \dots \\
\end{diagram}
%
\end{equation*}%
where:

\begin{enumerate}
\item $P_{0}=P_{0}^{B}\oplus P_{0}^{H}$.

\item $P_{s}=P_{s}^{B}\oplus P_{s}^{H}\oplus P_{s-1}^{B}$, $s\geq 1$.

\item 
\begin{equation*}
\delta _{1}=\left[ 
\begin{array}{ccc}
0 & 0 & 1 \\ 
0 & 0 & 0%
\end{array}%
\right] ,
\end{equation*}

\item 
\begin{equation*}
\delta _{s}=\left[ 
\begin{array}{ccc}
0 & 0 & 1 \\ 
0 & 0 & 0 \\ 
0 & 0 & 0%
\end{array}%
\right] ,s\geq 2.
\end{equation*}

\item $\varphi _{0}=\left[ 
\begin{array}{cc}
\beta _{0} & \eta _{0}%
\end{array}%
\right] $.

\item $\varphi _{s}=\left[ 
\begin{array}{ccc}
\beta _{s} & \eta _{s} & \gamma _{s}%
\end{array}%
\right] $, $s\geq 1$.
\end{enumerate}

\item The complex $P_{\bullet }$ is clearly $\mathbf{P}$-projective. Indeed,%
\begin{eqnarray*}
\coker%
\delta _{1} &=&P_{0}^{H}\in \mathbf{P,} \\
\coker%
\delta _{s} &=&P_{s-1}^{H}\oplus P_{s-2}^{B}\in \mathbf{P},s\geq 2.
\end{eqnarray*}

\item The vertical morphisms $\varphi _{\bullet }$ form a morphism of
complexes:%
\begin{eqnarray*}
\varphi _{0}\circ \delta _{1} &=&\left[ 
\begin{array}{cc}
\beta _{0} & \eta _{0}%
\end{array}%
\right] \left[ 
\begin{array}{ccc}
0 & 0 & 1 \\ 
0 & 0 & 0%
\end{array}%
\right] =\beta _{0}= \\
0+0+\beta _{0} &=&d_{1}\circ \beta _{1}+d_{1}\circ \eta _{1}+d_{1}\circ
\gamma _{1}=d_{1}\circ \varphi _{1}, \\
\varphi _{s-1}\circ \delta _{s} &=&\left[ 
\begin{array}{ccc}
\beta _{s-1} & \eta _{s-1} & \gamma _{s-1}%
\end{array}%
\right] \left[ 
\begin{array}{ccc}
0 & 0 & 1 \\ 
0 & 0 & 0 \\ 
0 & 0 & 0%
\end{array}%
\right] =\beta _{s-1}= \\
0+0+\beta _{s-1} &=&d_{s}\circ \beta _{s}+d_{s}\circ \eta _{s}+d_{s}\circ
\gamma _{s}=d_{s}\circ \varphi _{s},s\geq 2.
\end{eqnarray*}

\item Since $\varphi _{s}|_{P_{s}^{B}}=\beta _{s}$ is an epimorphism, $%
\varphi _{s}$ is an epimorphism, too.

\item Clearly, $H_{s}\left( P_{\bullet }\right) =P_{s}^{H}$, and $%
H_{s}\left( \varphi _{\bullet }\right) $, being the composition of an
epimorphism $\eta _{s}:P_{s}^{H}\twoheadrightarrow Z_{s}$, and the
projection $Z_{s}\twoheadrightarrow H_{s}$, is an epimorphism.
\end{enumerate}
\end{proof}

\begin{proposition}
\label{Prop-Cartan-Eilenberg}~

\begin{enumerate}
\item If $\mathbf{P}$ is generating in $\mathbf{C}$, then for any $%
X_{\bullet }\in Ch^{\geq 0}\left( \mathbf{C}\right) $ there exists a
Cartan-Eilenberg resolution $P_{\bullet \bullet }\rightarrow X_{\bullet }$.

\item Let $\left( Tot\left( P_{\bullet \bullet }\right) ,D_{\bullet }\right) 
$ be the total complex:%
\begin{eqnarray*}
&&\left( Tot\left( P_{\bullet \bullet }\right) \right) _{n}%
{:=}%
\dbigoplus\limits_{s=0}^{n}P_{s,n-s}, \\
D_{n} &=&\dsum\limits_{s=0}^{n}\left( d_{s,n-s}+\left( -1\right) ^{s}\delta
_{s,n-s}\right) :\left( Tot\left( P_{\bullet \bullet }\right) \right)
_{n}\longrightarrow \left( Tot\left( P_{\bullet \bullet }\right) \right)
_{n-1}.
\end{eqnarray*}%
Then $\delta _{\bullet 0}:Tot\left( P_{\bullet \bullet }\right) \rightarrow
X_{\bullet }$ is a qis.
\end{enumerate}
\end{proposition}

\begin{proof}
~

\begin{enumerate}
\item We will define the rows $P_{\bullet ,t}$ using induction on $t$:

\begin{enumerate}
\item $t=0$. Due to Lemma \ref{Lemma-Cartan-Eilenberg}, there exists an
epimorphism%
\begin{equation*}
\delta _{\bullet ,0}:P_{\bullet ,0}\twoheadrightarrow X_{\bullet }
\end{equation*}%
with $P_{\bullet ,0}\in Ch^{\geq 0}\left( \mathbf{P}\right) $, such that $%
H_{s}\left( \delta _{\bullet ,0}\right) $ is an epimorphism for all $s$.

\item Denote $Y_{\bullet }^{\left( -1\right) }%
{:=}%
X_{\bullet }$. For each $t=0,1,2,\dots $, let%
\begin{equation*}
Y_{\bullet }^{\left( t\right) }=\ker \left( P_{\bullet ,t}\longrightarrow
Y_{\bullet }^{\left( t-1\right) }\right) ,
\end{equation*}%
and let, due to Lemma \ref{Lemma-Cartan-Eilenberg},%
\begin{equation*}
\varphi _{\bullet ,t+1}:P_{\bullet ,t+1}\longrightarrow Y_{\bullet }^{\left(
t\right) }
\end{equation*}%
be an epimorphism of complexes, such that $H_{\bullet }\left( \delta
_{\bullet ,t+1}\right) $ is an epimorphism as well. Combining the following
short exact sequences of complexes:%
\begin{eqnarray*}
0 &\longrightarrow &Y_{\bullet }^{\left( 0\right) }\longrightarrow
P_{\bullet ,0}\longrightarrow \left( Y_{\bullet }^{\left( -1\right)
}=X_{\bullet }\right) \longrightarrow 0, \\
0 &\longrightarrow &Y_{\bullet }^{\left( 1\right) }\longrightarrow
P_{\bullet ,1}\longrightarrow Y_{\bullet }^{\left( 0\right) }\longrightarrow
0, \\
0 &\longrightarrow &Y_{\bullet }^{\left( 2\right) }\longrightarrow
P_{\bullet ,2}\longrightarrow Y_{\bullet }^{\left( 1\right) }\longrightarrow
0, \\
&&\dots \\
0 &\longrightarrow &Y_{\bullet }^{\left( t\right) }\longrightarrow
P_{\bullet ,t}\longrightarrow Y_{\bullet }^{\left( t-1\right)
}\longrightarrow 0, \\
&&\dots
\end{eqnarray*}%
one obtains a long exact sequence%
\begin{equation*}
0\longleftarrow X_{\bullet }\overset{\delta _{\bullet ,0}}{\longleftarrow }%
P_{\bullet ,0}\overset{\delta _{\bullet ,1}}{\longleftarrow }P_{\bullet ,1}%
\overset{\delta _{\bullet ,2}}{\longleftarrow }\dots \overset{\delta
_{\bullet ,t}}{\longleftarrow }P_{\bullet ,t}\overset{\delta _{\bullet ,t+1}}%
{\longleftarrow }\dots
\end{equation*}%
where%
\begin{equation*}
\delta _{\bullet ,t}=\left( Y_{\bullet }^{\left( t-1\right) }\longrightarrow
P_{\bullet ,t-1}\right) \circ \varphi _{\bullet ,t}.
\end{equation*}

\item Consider the corresponding long exact sequences of homologies. Since
each $H_{\bullet }\left( \varphi _{\bullet ,t}\right) $ is an epimorphism,
one has a series of short exact sequences%
\begin{eqnarray*}
0 &\longrightarrow &H_{s}\left( Y_{\bullet }^{\left( 0\right) }\right)
\longrightarrow H_{s}\left( P_{\bullet ,0}\right) \longrightarrow
H_{s}\left( Y_{\bullet }^{\left( 1\right) }\right) \longrightarrow 0, \\
0 &\longrightarrow &H_{s}\left( Y_{\bullet }^{\left( 1\right) }\right)
\longrightarrow H_{s}\left( P_{\bullet ,1}\right) \longrightarrow
H_{s}\left( Y_{\bullet }^{\left( 0\right) }\right) \longrightarrow 0, \\
0 &\longrightarrow &H_{s}\left( Y_{\bullet }^{\left( 2\right) }\right)
\longrightarrow H_{s}\left( P_{\bullet ,2}\right) \longrightarrow
H_{s}\left( Y_{\bullet }^{\left( 1\right) }\right) \longrightarrow 0, \\
&&\dots \\
0 &\longrightarrow &H_{s}\left( Y_{\bullet }^{\left( t\right) }\right)
\longrightarrow H_{s}\left( P_{\bullet ,t}\right) \longrightarrow
H_{s}\left( Y_{\bullet }^{\left( t-1\right) }\right) \longrightarrow 0, \\
&&\dots
\end{eqnarray*}%
Combining them, one obtains the long exact sequence%
\begin{equation*}
0\longleftarrow H_{s}\left( X_{\bullet }\right) \longleftarrow H_{s}\left(
P_{\bullet ,0}\right) \longleftarrow H_{s}\left( P_{\bullet ,1}\right)
\longleftarrow \dots \longleftarrow H_{s}\left( P_{\bullet ,t}\right)
\longleftarrow \dots
\end{equation*}
\end{enumerate}

\item Consider the first spectral sequence from \cite[Theorem A.4.3(4a)]%
{Prasolov-Cosheaves-2021-MR4347662} for the bicomplex $P_{\bullet \bullet }$:%
\begin{equation*}
^{ver}E_{s,t}^{2}%
\simeq%
~^{hor}H_{s}\left( ^{ver}H_{t}\left( P_{\bullet ,\bullet }\right) \right)
\implies H_{s+t}\left( Tot\left( P_{\bullet \bullet }\right) \right) .
\end{equation*}%
The sequence degenerates:%
\begin{equation*}
^{ver}E_{s,t}^{2}=~^{hor}H_{s}\left( ^{ver}H_{t}\left( P_{\bullet ,\bullet
}\right) \right) =\left\{ 
\begin{array}{ccc}
H_{s}\left( X_{\bullet }\right) & \text{if} & t=0 \\ 
0 & \text{otherwise} & 
\end{array}%
\right.
\end{equation*}%
therefore%
\begin{equation*}
H_{s}\left( Tot\left( P_{\bullet \bullet }\right) \right) 
\simeq%
H_{s}\left( X_{\bullet }\right) ,
\end{equation*}%
and $Tot\left( P_{\bullet \bullet }\right) \rightarrow X_{\bullet }$ is a
qis.
\end{enumerate}
\end{proof}

\begin{remark}
We have in fact proved the following statement, dual to \cite[Lemma 13.2.1]%
{Kashiwara-Categories-MR2182076}. If $\mathbf{P}\subseteq \mathbf{C}$ is
generating, then for any $X_{\bullet }\in Ch^{\geq 0}\left( \mathbf{C}%
\right) $ there exists a $P_{\bullet }\in Ch^{\geq 0}\left( \mathbf{P}%
\right) $, and a qis $P_{\bullet }\rightarrow X_{\bullet }$.
\end{remark}

\subsection{The Grothendieck spectral sequence}

\begin{definition}
\label{Def-F-acyclic}Let $F:\mathbf{C}\rightarrow \mathbf{E}$ be an additive
functor between abelian categories, and let $\mathbf{P}\subseteq \mathbf{C}$
be an $F$-projective subcategory. We say that an object $X$ is $F$\textbf{%
-acyclic} iff $L_{s}F\left( X\right) =0$ for all $s>0$.
\end{definition}

\begin{theorem}
\label{Th-Grothendieck-spectral-sequence}Let $\mathbf{C}$, $\mathbf{D}$ and $%
\mathbf{E}$ be abelian categories, and let $G:\mathbf{C}\longrightarrow 
\mathbf{D}$ and $F:\mathbf{D}\longrightarrow \mathbf{E}$ be right exact
additive functors. Assume that the additive subcategories $\mathbf{P}%
\subseteq \mathbf{C}$ and $\mathbf{Q}\subseteq \mathbf{D}$ are $G$%
-projective and $F$-projective respectively, and that for any $P\in \mathbf{P%
}$, the object $G\left( P\right) $ is $F$-acyclic. Then there exists a
spectral sequence, natural in $X$,%
\begin{equation*}
E_{s,t}^{2}%
{:=}%
L_{t}F\left( L_{s}G\left( X\right) \right) \implies L_{s+t}\left( F\circ
G\right) \left( X\right) .
\end{equation*}
\end{theorem}

\begin{proof}
For $X\in \mathbf{C}$, choose a resolution%
\begin{equation*}
0\longleftarrow X\longleftarrow P_{0}\longleftarrow P_{1}\longleftarrow
P_{2}\longleftarrow \dots
\end{equation*}%
with $P_{i}\in \mathbf{P}$. Apply $G$ and get a complex%
\begin{equation*}
Y_{\bullet }%
{:=}%
G\left( P_{0}\right) \longleftarrow G\left( P_{1}\right) \longleftarrow
G\left( P_{2}\right) \longleftarrow \dots
\end{equation*}%
Due to Proposition \ref{Prop-Cartan-Eilenberg}, there exists a
Cartan-Eilenberg resolution%
\begin{equation*}
0\longleftarrow Y_{\bullet }\longleftarrow Q_{\bullet ,0}\longleftarrow
Q_{\bullet ,1}\longleftarrow Q_{\bullet ,2}\longleftarrow \dots
\end{equation*}%
where $Q_{st}\in \mathbf{Q}$. See the diagram (remember that the above line
becomes in fact a series of \textbf{columns}):%
\begin{equation*}
\begin{diagram}[size=3.0em,textflow]
\dots &  & \dots &  &  &  & \dots &  \\ 
\dTo &  & \dTo &  &  &  & \dTo &  \\
Q_{0,t} & \lTo & Q_{1,t} & \lTo & \dots & \lTo
& Q_{s,t} & \lTo & \dots \\ 
\dTo &  & \dTo &  &  &  & \dTo &  \\ 
\dots &  & \dots &  &  &  & \dots &  \\ 
\dTo &  & \dTo &  &  &  & \dTo &  \\ 
Q_{0,0} & \lTo & Q_{1,0} & \lTo & \dots & \lTo
& Q_{s,0} & \lTo & \dots \\ 
\dTo &  & \dTo &  &  &  & \dTo &  \\ 
G\left( P_{0}\right) & \lTo & G\left( P_{1}\right) & \lTo
& \dots & \lTo & G\left( P_{s}\right) & \lTo & \dots
\end{diagram}
%
\end{equation*}%
Apply now the two spectral sequences \cite[Theorem A.4.3(4ab)]%
{Prasolov-Cosheaves-2021-MR4347662} to the bicomplex $F\left( Q_{\bullet
\bullet }\right) $:%
\begin{eqnarray*}
^{ver}E_{s,t}^{2}%
\simeq%
~^{hor}H_{s}\left( ^{ver}H_{t}F\left( Q_{\bullet ,\bullet }\right) \right)
&\implies &H_{s+t}\left( Tot\left( F\left( Q_{\bullet ,\bullet }\right)
\right) \right) , \\
^{hor}E_{s,t}^{2}%
\simeq%
~^{ver}H_{t}\left( ^{hor}H_{s}F\left( Q_{\bullet ,\bullet }\right) \right)
&\implies &H_{s+t}\left( Tot\left( F\left( Q_{\bullet ,\bullet }\right)
\right) \right) .
\end{eqnarray*}

\begin{enumerate}
\item ~

\begin{enumerate}
\item $^{ver}H_{t}F\left( Q_{\bullet ,\bullet }\right) =L_{t}F\left(
Y_{\bullet }\right) =L_{t}F\left( G\left( P_{\bullet }\right) \right) $.
Since all objects $G\left( P_{t}\right) $ are $F$-acyclic, one gets%
\begin{equation*}
\left( ^{ver}H_{t}F\left( Q_{\bullet ,\bullet }\right) \right) _{st}=\left\{ 
\begin{array}{ccc}
F\left( G\left( P_{s}\right) \right) & \text{if} & t=0 \\ 
0 & \text{if} & t\neq 0%
\end{array}%
\right.
\end{equation*}

\item 
\begin{eqnarray*}
^{ver}E_{s,t}^{2} &=&~^{hor}H_{s}\left( ^{ver}H_{t}F\left( Q_{\bullet
,\bullet }\right) \right) =~^{hor}H_{s}\left( L_{t}F\left( G\left(
P_{\bullet }\right) \right) \right) = \\
&=&\left\{ 
\begin{array}{ccc}
H_{s}\left( F\left( G\left( P_{\bullet }\right) \right) \right) & \text{if}
& t=0 \\ 
0 & \text{if} & t\neq 0%
\end{array}%
\right. =\left\{ 
\begin{array}{ccc}
L_{s}\left( F\circ G\right) \left( X\right) & \text{if} & t=0 \\ 
0 & \text{if} & t\neq 0%
\end{array}%
\right.
\end{eqnarray*}

\item The sequence degenerates, implying%
\begin{equation*}
H_{n}\left( Tot\left( F\left( Q_{\bullet \bullet }\right) \right) \right) 
\simeq%
L_{n}\left( F\circ G\right) \left( X\right) .
\end{equation*}
\end{enumerate}

\item ~

\begin{enumerate}
\item For each $t\geq 0$, the homologies $^{hor}H_{s}F\left( Q_{\bullet
,t}\right) \in \mathbf{Q}$ due to Remark \ref{Rem-Homology-quasi-projective}.

\item 
\begin{equation*}
0\longleftarrow \left( L_{s}G\left( X\right) =H_{s}\left( G\left( P_{\bullet
}\right) \right) =H_{s}\left( Y_{\bullet }\right) \right) \longleftarrow
H_{s}\left( Q_{\bullet ,0}\right) \longleftarrow H_{s}\left( Q_{\bullet
,1}\right) \longleftarrow \dots \longleftarrow H_{s}\left( Q_{\bullet
,t}\right) \longleftarrow \dots
\end{equation*}%
is a $\mathbf{Q}$-resolution, therefore%
\begin{equation*}
^{hor}E_{s,t}^{2}%
\simeq%
~^{ver}H_{t}\left( F\left( ^{hor}H_{s}\left( Q_{\bullet ,\bullet }\right)
\right) \right) 
\simeq%
L_{t}F\left( L_{s}G\left( X\right) \right) \implies L_{s+t}\left( F\circ
G\right) \left( X\right) .
\end{equation*}
\end{enumerate}

\item See Remark \ref{Rem-Naturality-in-X} below.
\end{enumerate}
\end{proof}

\begin{remark}
\label{Rem-Naturality-in-X}Proof of the naturality in $X$ is complicated,
and is beyond the scope of this paper. One should follow the scheme
described as alternative a) in \cite[p. 146]%
{Grothendieck-Tohoku-1957-MR0102537}. The problem is again that in our
situation we do not have enough projectives. We should do the following:

\begin{enumerate}
\item Consider the abelian category $\mathbf{B}=Ch^{+}\left( \mathbf{C}%
\right) $.

\item Build the derived category $D\left( Ch^{+}\left( \mathbf{B}\right)
\right) $ based on bicomplexes $Ch^{+,+}\left( \mathbf{C}\right) $.

\item Define the full additive subcategory $\mathbf{P}^{\prime }\subseteq 
\mathbf{C}$ consisting of $\mathbf{P}$-projective complexes (Definition \ref%
{Def-P-projective-complex}).

\item Prove the generalization of Proposition \ref{Prop-Cartan-Eilenberg}:
for any complex $X_{\bullet }\in Ch^{+}\left( \mathbf{B}\right) $ there
exists a complex $P_{\bullet }\in Ch^{+}\left( \mathbf{P}^{\prime }\right) $
and a qis $P_{\bullet }\rightarrow X_{\bullet }$.

\item Then the naturality in $X$ would follow:

\begin{enumerate}
\item The \textbf{last} stage of the proof of Theorem \ref%
{Th-Grothendieck-spectral-sequence} would \textbf{not} depend on the
resolution $Q_{\bullet \bullet }$. Given two such resolutions $Q_{\bullet
\bullet }\twoheadrightarrow Y_{\bullet }$ and $Q_{\bullet \bullet }^{\prime
}\twoheadrightarrow Y_{\bullet }$, build, using \cite[Theorem 10.2.3(i) and
Proposition 10.2.7]{Kashiwara-Categories-MR2182076}, the third such
resolution, two qis $Q_{\bullet \bullet }^{\prime \prime }\twoheadrightarrow
Q_{\bullet \bullet }$ and $Q_{\bullet \bullet }^{\prime \prime
}\twoheadrightarrow Q_{\bullet \bullet }^{\prime }$, and a commutative
diagram:%
\begin{equation*}
\begin{diagram}[size=3.0em,textflow]
Q_{\bullet \bullet } & \rOnto & Y_{\bullet } & \lOnto
& Q_{\bullet \bullet }^{\prime } \\ 
& \luOnto & \uOnto & \ruOnto &  \\ 
&  & Q_{\bullet \bullet }^{\prime \prime } &  & 
\end{diagram}
%
\end{equation*}%
Then build the Grothendieck spectral sequence using $Q_{\bullet \bullet
}^{\prime \prime }\twoheadrightarrow Y_{\bullet }$. The result would be
isomorphic to the sequences for $Q_{\bullet \bullet }\twoheadrightarrow
Y_{\bullet }$ and $Q_{\bullet \bullet }^{\prime }\twoheadrightarrow
Y_{\bullet }$.

\item The \textbf{first} stage of Theorem \ref%
{Th-Grothendieck-spectral-sequence} would \textbf{not} depend on the
resolution $P_{\bullet }\twoheadrightarrow X$. Given two such resolutions $%
P_{\bullet }\twoheadrightarrow X$ and $P_{\bullet }^{\prime
}\twoheadrightarrow X$, build the third such resolution, two qis $P_{\bullet
}^{\prime \prime }\twoheadrightarrow P_{\bullet }$ and $P_{\bullet }^{\prime
\prime }\twoheadrightarrow P_{\bullet }^{\prime }$, and a commutative
diagram:%
\begin{equation*}
\begin{diagram}[size=3.0em,textflow]
P_{\bullet } & \rOnto & X & \lOnto & P_{\bullet
}^{\prime } \\ 
& \luOnto & \uOnto & \ruOnto &  \\ 
&  & P_{\bullet }^{\prime \prime } &  & \\ 
\end{diagram}
%
\end{equation*}%
Then build the resolution $Q_{\bullet \bullet }^{\prime \prime
}\twoheadrightarrow G\left( P_{\bullet }^{\prime \prime }\right) $, two qis $%
Q_{\bullet \bullet }^{\prime \prime }\twoheadrightarrow Q_{\bullet \bullet }$
and $Q_{\bullet \bullet }^{\prime \prime }\twoheadrightarrow Q_{\bullet
\bullet }^{\prime }$, and a commutative diagram:%
\begin{equation*}
\begin{diagram}[size=3.0em,textflow]
G\left( P_{\bullet }\right) & \rOnto & G\left( X\right) & 
\lOnto & G\left( P_{\bullet }^{\prime }\right) \\ 
\uTo &  & \uOnto &  & \uTo \\ 
Q_{\bullet \bullet } & \lOnto & Q_{\bullet \bullet }^{\prime \prime }
& \rOnto & Q_{\bullet \bullet }^{\prime }
\end{diagram}
%
\end{equation*}%
Build the Grothendieck spectral sequence using $Q_{\bullet \bullet }^{\prime
\prime }$. The result would be isomorphic to the sequences built for $%
Q_{\bullet \bullet }$ and $Q_{\bullet \bullet }^{\prime }$.

\item Given a morphism $X\rightarrow X^{\prime }$, build two resolutions $%
P_{\bullet }\twoheadrightarrow X$ and $P_{\bullet }^{\prime
}\twoheadrightarrow X^{\prime }$, the third complex $P_{\bullet }^{\prime
\prime }$, a morphism $P_{\bullet }^{\prime \prime }\rightarrow P_{\bullet
}^{\prime }$ and a qis $P_{\bullet }^{\prime \prime }\rightarrow P_{\bullet
} $ such that the diagram%
\begin{equation*}
\begin{diagram}[size=3.0em,textflow]
P_{\bullet }^{\prime \prime } & \rTo & P_{\bullet } & 
\rOnto & X \\ 
& \rdTo &  &  & \dTo \\ 
&  & P_{\bullet }^{\prime } & \rOnto & X^{\prime } \\
\end{diagram}
%
\end{equation*}%
commutes. Then build two resolutions $Q_{\bullet \bullet }\twoheadrightarrow
G\left( P_{\bullet }^{\prime \prime }\right) $ and $Q_{\bullet \bullet
}^{\prime }\twoheadrightarrow G\left( P_{\bullet }^{\prime }\right) $, the
third bicomplex $Q_{\bullet \bullet }^{\prime \prime }$, a morphism $%
Q_{\bullet \bullet }^{\prime \prime }\rightarrow Q_{\bullet \bullet
}^{\prime }$ and a qis $Q_{\bullet \bullet }^{\prime \prime }\rightarrow
Q_{\bullet \bullet }$ such that the diagram%
\begin{equation*}
\begin{diagram}[size=3.0em,textflow]
Q_{\bullet \bullet } & \rOnto & G\left( P_{\bullet }^{\prime
\prime }\right) & \rTo & G\left( P_{\bullet }\right) & 
\rOnto & G\left( X\right) \\ 
\uTo &  &  & \rdTo &  &  & \dTo \\ 
Q_{\bullet \bullet }^{\prime \prime } & \rTo & Q_{\bullet \bullet
}^{\prime } & \rOnto & G\left( P_{\bullet }^{\prime }\right) & 
\rOnto & G\left( X^{\prime }\right) \\
\end{diagram}
%
\end{equation*}%
commutes. Build the Grothendieck spectral sequences using $Q_{\bullet
\bullet }$, $Q_{\bullet \bullet }^{\prime }$, and $Q_{\bullet \bullet
}^{\prime \prime }$. Then one obtains an \textbf{iso}morphism from the the
third sequence to the first one, and a \textbf{morphism} from the third one
to the second one, giving the desired morphism:%
\begin{equation*}
E_{s,t}^{2}\left( f\right) :E_{s,t}^{2}\left( X\right) =L_{t}F\left(
L_{s}G\left( X\right) \right) \longrightarrow L_{t}F\left( L_{s}G\left(
X^{\prime }\right) \right) =E_{s,t}^{2}\left( X^{\prime }\right) .
\end{equation*}

\item Finally:

\begin{enumerate}
\item Prove that the above morphism does not depend on $Q_{\bullet \bullet
}^{\prime \prime }$.

\item Prove that, given $f:X\rightarrow X^{\prime }$ and $g:X^{\prime
}\rightarrow X^{\prime \prime }$,%
\begin{equation*}
E_{s,t}^{2}\left( g\circ f\right) =E_{s,t}^{2}\left( g\right) \circ
E_{s,t}^{2}\left( f\right) .
\end{equation*}
\end{enumerate}
\end{enumerate}
\end{enumerate}
\end{remark}

\bibliographystyle{apalike}
\bibliography{Cosheaves}

\end{document}